\documentclass[10pt]{amsart}
\usepackage{amsmath}
\usepackage{amsthm}
\usepackage{amssymb}
\usepackage{mathtools}
\usepackage{tikz-cd}
\usepackage{tikz}
\usetikzlibrary{3d, decorations.markings, positioning}
\usepackage{enumitem}
\usepackage{stix2}
\usepackage{stackengine}
\usepackage{bbold}
\usepackage{graphics}
\usepackage{float}
\usepackage{parskip}
\usepackage[
left = 3cm,
right = 3cm
]{geometry}
\usepackage[
colorlinks=true,
citecolor=blue,
linkcolor=red,
linktoc=page
]{hyperref}
\usepackage[
  style=alphabetic,
  maxcitenames=10,
  maxbibnames=10,
  giveninits=true,
  backend=biber,
]{biblatex}
\usepackage{cleveref}
\bibliography{bib_twisted_flow.bib}

\numberwithin{equation}{section}

\newtheorem{prop}{Proposition}[subsection]
\newtheorem{theorem}[prop]{Theorem}
\newtheorem{lem}[prop]{Lemma}
\newtheorem{cor}[prop]{Corollary}
\theoremstyle{definition}
\newtheorem{definition}[prop]{Definition}
\newtheorem{ex}[prop]{Example}
\newtheorem{rem}[prop]{Remark}

\Crefname{equation}{}{}
\Crefname{lem}{Lemma}{Lemmas}
\Crefname{theorem}{Theorem}{Theorems}
\Crefname{prop}{Proposition}{Propositions}
\Crefname{cor}{Corollary}{Corollaries}
\Crefname{definition}{Definition}{Definitions}
\Crefname{subsection}{Section}{Sections}
\Crefname{section}{Section}{Sections}

\DeclareMathOperator{\Hom}{Hom}
\DeclareMathOperator{\Map}{Map}
\DeclareMathOperator{\map}{map}
\DeclareMathOperator{\Fun}{Fun}
\DeclareMathOperator{\interior}{int}
\DeclareMathOperator{\Sing}{Sing}

\DeclareMathOperator{\im}{im}

\DeclareMathOperator{\codim}{codim}
\DeclareMathOperator{\rk}{rk}
\DeclareMathOperator{\Pic}{Pic}

\DeclareMathOperator{\Th}{Th}
\DeclareMathOperator{\Mod}{Mod}

\DeclareMathOperator{\ob}{ob}
\DeclareMathOperator{\Set}{Set}

\DeclareMathOperator{\op}{op}
\DeclareMathOperator{\Monoid}{Mon}

\DeclareMathOperator{\can}{can}
\DeclareMathOperator{\Lax}{Lax}
\DeclareMathOperator{\Ar}{Ar}
\DeclareMathOperator{\Span}{Span}
\DeclareMathOperator{\St}{St}
\DeclareMathOperator{\Unst}{Unst}
\DeclareMathOperator{\Fin}{Fin}
\DeclareMathOperator{\con}{con}
\DeclareMathOperator{\CoCart}{CoCart}
\DeclareMathOperator{\Cart}{Cart}
\DeclareMathOperator{\RFib}{RFib}
\DeclareMathOperator{\LFib}{LFib}
\DeclareMathOperator{\Ind}{Ind}
\DeclareMathOperator{\Comp}{Comp}
\DeclareMathOperator{\Pos}{Pos}
\DeclareMathOperator{\Stab}{Stab}
\DeclareMathOperator{\Top}{Top}
\DeclareMathOperator{\Vect}{Vect}
\DeclareMathOperator{\Loc}{Loc}
\DeclareMathOperator{\thom}{th}
\DeclareMathOperator{\Psh}{Psh}

\DeclareMathOperator{\fib}{fib}

\DeclareMathOperator{\PT}{PT}
\DeclareMathOperator{\bVect}{\mathbb{V}ect}
\DeclareMathOperator{\CW}{CW}

\DeclareMathOperator*{\colim}{colim}
\DeclareMathOperator*{\laxcolim}{laxcolim}
\DeclareMathOperator*{\laxlim}{laxlim}
\DeclareMathOperator*{\oplaxcolim}{oplaxcolim}
\DeclareMathOperator*{\oplaxlim}{oplaxlim}

\renewcommand{\Pr}{\mathrm{Pr}}
\newcommand{\bfPr}{\mathbf{Pr}}

\newcommand{\rL}{\mathrm{L}}
\newcommand{\rR}{\mathrm{R}}

\newcommand{\PrL}{\Pr^{\rL}}

\newcommand{\PrLSt}{\Pr^{\rL}_{\St}}

\newcommand{\fr}{\mathrm{fr}}

\newcommand{\Flow}{\cF\mathrm{low}\vphantom{w}}
\newcommand{\bfFlow}{\mathbf{F}\mathrm{low}\vphantom{w}}
\newcommand{\PreFlow}{\mathrm{Flow}\vphantom{w}}
\newcommand{\bfPreFlow}{\mathbf{F}\mathrm{low}\vphantom{w}}
\newcommand{\Bord}{\mathcal{B}\mathrm{ord}\vphantom{d}}
\newcommand{\PreBord}{\mathrm{Bord}\vphantom{d}}
\newcommand{\TwShv}{\mathrm{TwPsh}\vphantom{h}}
\newcommand{\TwSp}{\mathrm{TwSp}\vphantom{h}}
\newcommand{\Cat}{\mathrm{Cat}\vphantom{t}}
\newcommand{\bfCat}{\mathbf{Cat}\vphantom{t}}
\newcommand{\CatHat}{\widehat{\mathrm{Cat}}\vphantom{t}}
\newcommand{\Mfd}{\mathrm{Mfd}\vphantom{d}}
\newcommand{\PreMfd}{\mathrm{PreMfd}\vphantom{d}}
\newcommand{\Alg}{\mathrm{Alg}\vphantom{z}}

\newcommand{\bB}{\mathbb{B}}
\newcommand{\bC}{\mathbb{C}}
\newcommand{\bE}{\mathbb{E}}
\newcommand{\bF}{\mathbb{F}}
\newcommand{\bH}{\mathbb{H}}

\newcommand{\bN}{\mathbb{N}}
\newcommand{\bM}{\mathbb{M}}

\newcommand{\bS}{\mathbb{S}}
\newcommand{\bZ}{\mathbb{Z}}
\newcommand{\bR}{\mathbb{R}}
\newcommand{\bX}{\mathbb{X}}
\newcommand{\bY}{\mathbb{Y}}
\newcommand{\bU}{\mathbb{U}}
\newcommand{\bV}{\mathbb{V}}
\newcommand{\bW}{\mathbb{W}}
\newcommand{\bOne}{\mathbb{1}}
\newcommand{\bZero}{\mathbb{0}}

\newcommand{\Simp}{\mathbb{\Delta}}
\newcommand{\Horn}{\Lambda}

\newcommand{\cA}{\mathcal{A}}
\newcommand{\cB}{\mathcal{B}}
\newcommand{\cC}{\mathcal{C}}
\newcommand{\cE}{\mathcal{E}}
\newcommand{\cI}{\mathcal{I}}
\newcommand{\calD}{\mathcal{D}}

\newcommand{\cK}{\mathcal{K}}
\newcommand{\cJ}{\mathcal{J}}
\newcommand{\cN}{\mathcal{N}}
\newcommand{\cM}{\mathcal{M}}
\newcommand{\cP}{\mathcal{P}}
\newcommand{\cS}{\mathcal{S}}
\newcommand{\cV}{\mathcal{V}}
\newcommand{\cW}{\mathcal{W}}
\newcommand{\cL}{\mathcal{L}}
\newcommand{\cR}{\mathcal{R}}
\newcommand{\cF}{\mathcal{F}}
\newcommand{\cT}{\mathcal{T}}

\newcommand{\cU}{\mathcal{U}}

\newcommand{\IP}{\mathrm{IP}}

\newcommand{\bfA}{\mathbf{A}}

\newcommand{\bfC}{\mathbf{C}}
\newcommand{\bfD}{\mathbf{D}}
\newcommand{\bfM}{\mathbf{M}}
\newcommand{\bfN}{\mathbf{N}}

\newcommand{\sfA}{\mathsf{A}}
\newcommand{\sfB}{\mathsf{B}}
\newcommand{\sfC}{\mathsf{C}}
\newcommand{\sfD}{\mathsf{D}}
\newcommand{\sfE}{\mathsf{E}}
\newcommand{\sfS}{\mathsf{S}}
\newcommand{\sfU}{\mathsf{U}}

\newcommand{\sfI}{\mathsf{I}}

\newcommand{\CSS}{\mathcal{CSS}}

\newcommand{\IK}{\mathcal{IK}}
\newcommand{\Opd}{\mathrm{Opd}^{\mathrm{gen,ns}}_{\infty}}

\newcommand{\Dbl}{\mathrm{Dbl}}
\newcommand{\DblLax}{ \mathrm{Dbl}_{\infty}^{\Lax}  }
\newcommand{\Mon}{\mathrm{Mon}_{\infty}}
\newcommand{\MonLax}{\mathrm{Mon}_{\infty}^{\Lax}}
\newcommand{\Seg}{\mathrm{Seg}_{\infty}}
\newcommand{\Sp}{\mathrm{Sp}\vphantom{p}}

\newcommand{\del}{\partial}
\newcommand{\join}{\star}

\DeclareFontFamily{U}{min}{}
\DeclareFontShape{U}{min}{m}{n}{<-> udmj30}{}
\newcommand{\yo}{\!\text{\usefont{U}{min}{m}{n}\symbol{'207}}\!}

\pgfdeclarearrow{
	name = pxto,
	setup code = {
		\pgfarrowssettipend{1.5\pgflinewidth}
		\pgfarrowssetbackend{-2.5508\pgflinewidth}
		\pgfarrowssetlineend{-.25\pgflinewidth}
		\pgfarrowssetvisualbackend{-0.021\pgflinewidth}
		\pgfarrowsupperhullpoint{1.5\pgflinewidth}{0\pgflinewidth}
		\pgfarrowsupperhullpoint{-2.0085\pgflinewidth}{3.6525\pgflinewidth}
		\pgfarrowsupperhullpoint{-2.5508\pgflinewidth}{3.0763\pgflinewidth}
	},
	drawing code = {
		\pgfsetdash{}{0pt}%
		\pgfpathmoveto{\pgfpoint{1.5\pgflinewidth}{0.0254\pgflinewidth}}%
		\pgfpathlineto{\pgfpoint{-2.0085\pgflinewidth}{3.6525\pgflinewidth}}%
		\pgfpathlineto{\pgfpoint{-2.5508\pgflinewidth}{3.0763\pgflinewidth}}%
		\pgfpathlineto{\pgfpoint{-0.4322\pgflinewidth}{0.5\pgflinewidth}}%
		\pgfpathlineto{\pgfpoint{-0.4322\pgflinewidth}{-0.5\pgflinewidth}}%
		\pgfpathlineto{\pgfpoint{-2.5508\pgflinewidth}{-3.0763\pgflinewidth}}%
		\pgfpathlineto{\pgfpoint{-2.0085\pgflinewidth}{-3.6525\pgflinewidth}}%
		\pgfpathclose%
		\pgfusepathqfill
	}
}
\tikzset{>=pxto}
\tikzcdset{arrow style=tikz}
\renewcommand{\to}{\mathchoice{\longrightarrow}{\rightarrow}{}{}}

\title{Structured Flow Categories and Twisted Presheaves}
\author{Alice Hedenlund}
\author{Trygve Poppe Oldervoll}

\begin{document}

\begin{abstract}
	An orientation theory for flow categories without bubbling is determined by a functor of $\infty$-categories $\mu \colon \cC \to U/O$.
  For any such functor, we construct a stable $\infty$-category $\Flow^{\mu}$ of $\mu$-structured flow categories and bimodules.
  We also construct the expected functors between such $\infty$-categories, giving a tractable framework for manipulating orientations, local systems, and filtrations in exact Floer homotopy theory.
  Classifying spaces for certain bordism theories determined by $\mu$ appear as mapping spaces in $\Flow^{\mu}$, and we use a Pontrjagin--Thom construction to naturally identify $\Flow^{\mu}$ with the $\infty$-category of $\mu$-twisted presheaves on $\cC$.
\end{abstract}

\maketitle

\setcounter{tocdepth}{2}
\tableofcontents

\section*{Introduction}\label{sec:introduction}

\subsection*{Background and Aim}\label{subsec:background_and_aim}
This paper studies the homotopy theory of \emph{flow categories} equipped with extra structure which encodes orientation, filtration and local system data. We start by recalling the relevance of such structures in geometry.

In 1994, Cohen--Jones--Segal proposed a program to understand the homotopy theory underlying Floer theory.
They proposed that a Floer problem (the infinite-dimensional manifold and its Floer--Morse function) should give rise to a ``Floer homotopy type'', refining the associated Floer homology. Moreover, they discussed how such a Floer homotopy type might be constructed via the use of \emph{flow categories}.
A flow category $\mathbb{X}$ has a set of objects and morphism spaces which are smooth manifolds with corners.
The composition maps
\[
\mathbb{X}(x,y) \times \mathbb{X}(y,z) \to \mathbb{X}(x,z)
\]
are diffeomorphisms onto a boundary face of the target, and the boundary of $\bX(x,z)$ is precisely the image of all possible composition maps.
Such a structure can be obtained in finite-dimensional Morse theory by defining a flow category $\bM_{f}$ whose objects are the critical points of the Morse function $f$, and whose morphism spaces $\bM_{f}(p,q)$ are moduli spaces of broken gradient trajectories of $f$.
A Morse flow category can be equipped with extra structure by giving each morphism space a stable framing compatible with composition.
The resulting structure is called a \emph{framed flow category}.

Cohen--Jones--Segal sketch how to obtain a spectrum (à la stable homotopy theory) from a framed flow category.
More recently, Abouzaid--Blumberg show that framed flow categories can be arranged into a stable $\infty$-category, and show that this is a model for the $\infty$-category of spectra~\cite[Proposition 1.10]{AB}.

Far from all flow categories associated with Floer data are frameable, though.
Take, for instance, the case of Lagrangian Floer theory for two exact Lagrangian submanifolds~$L$ and~$K$ in an exact symplectic manifold $X$.
It was shown by \cite{large2021spectral} that this data, together with a Hamiltonian perturbation $H$ and an almost complex structure~$J$, determines a flow category $\bF_{LK}$ whose objects are Hamiltonian chords, and whose morphism spaces are spaces of broken Floer strips with boundary on $L$ and $K$.
In this situation, the tangential structure of $\bF_{LK}$ is controlled via Fredholm index theory by the \emph{Lagrangian difference map}
\begin{align*}
	P_{X}(L, K) \to U/O
\end{align*}
on the space of paths from $L$ to $K$ in $X$.
The non-triviality of this map presents potential obstructions to framing $\bF_{LK}$.
The key fact that relates the framed situation to spectra is that the sphere spectrum $\bS$ is the classifying spectrum for \emph{framed bordism}.
It might, however, be possible to equip all the manifolds $\bX(p,q)$ with $G$-orientations for some $G\to O$.
Such data is then expected to represent a module over the Thom spectrum~$MG$, which classifies $G$-oriented bordism.

An important tool in applications is Floer homology with coefficients in some (possibly derived) local system.
To illustrate how this may be generalized to Floer homotopy, consider the Morse flow category $\bM_{f}$.
Select reference paths connecting each critical point $p$ to a common basepoint $x$ on the relevant manifold $X$.
Any gradient flow line in $\bM_{f}(p,q)$ represents a continuous path in $X$, and by concatenating with the reference paths, a based loop.
This means that each isolated gradient flow line represents an element of $\pi_{1}(X,x)$, and using this, one can define Morse homology with coefficients in any $\bZ[\pi_{1}(X,x)]$-module.
In Floer homotopy theory, we can similarly keep track of the full maps $\bM_{f}(p,q) \to P_{X}(p,q)$ to the space of paths in $X$.
By making a choice of reference paths from the critical points to $x$, we could concatenate and obtain maps to the based loop space $\Omega_{x} X$.
The spectrum classifying framed bordism over $\Omega X$ is the spherical group ring $\bS[\Omega X]$, so by analogy with the above, we expect $\bM_{f}$ to define a module over $\bS[\Omega X]$.
The $\infty$-category of such modules is equivalent to the $\infty$-category $\Sp^{X}$ of \emph{spectral local systems} on $X$.

In quantitative applications, it is often desirable to construct a \emph{filtered} version of Floer homology.
This is done by assigning each generator of the Floer chain complex to the value of the action functional, and considering chain complexes determined by the generators with action in a certain range.
We can capture this type of data in Floer homotopy theory by assigning each object of a flow category $\bX$ to some object of a poset $A$, with the property that all morphisms in $\bX$ must go up in the order of $A$.
This amounts to prescribing a functor $\bX\to A$.
From such data, one can recover an $A^{\op}$-indexed diagram of flow categories $\bX_{a}$, each consisting of the objects with ``action'' greater than $a$. Assuming that $\bX$ can be framed, this data then corresponds to an object of $\Fun(A^{\op},\Sp)$.

We have seen that various flavors of Floer homotopy types are expected to define objects in the stable $\infty$-categories $\Mod_{MG}$, $\Sp^{X}$, and $\Fun(A^{\op},\Sp)$ respectively.
These admit a common generalization: they are all  $\infty$-categories of \emph{twisted presheaves}.
Twisted presheaves are more general than these examples, because the tangential structure, local system data, and filtration data can interact in non-trivial ways.

That some version of twisted stable homotopy theory is needed to deal with non-framed Floer homotopy theory was perhaps first pointed out by Furuta in the unpublished document~\cite{furuta2002prespectrum}.
Furuta was specifically motivated by Seiberg--Witten Floer theory.~\footnote{The construction of a stable homotopy type from Seiberg--Witten Flow theory does not go via flow categories, but uses alternative methods such as finite-dimensional approximation and Conley index theory~\cite{manolescu2003seiberg}.
A direct relationship between these methods and flow categories has not yet been written up.}
A framework for twisted stable homotopy theory was provided by Douglas in his PhD thesis~\cite{Dou}.
Recent work by Moulinos in joint work with the first author recasts and extends Douglas' work in the $\infty$-categorical setting.

A twisted presheaf over an $\infty$-category $\cC$ is a section of a bundle of $\infty$-categories on $\cC^{\op}$ whose fibers are equivalent to the $\infty$-category of spectra, and whose parallel transport is given by tensoring with an invertible spectrum.
The classifying data for such a bundle is a functor $\cC \to B\Pic(\bS)$, where $\Pic(\bS)$ is the grouplike space of tensor-invertible spectra.
Any functor to $U/O$ determines such a bundle by composing with a delooping of the $J$-homomorphism
\begin{align*}
	U/O \simeq B(\bZ\times BO) \to B\Pic(\bS) \,.
\end{align*}
Let us now sketch how we can allow the orientation structure and the local system information to interact when defining a flow category.
Start with some base $\infty$-category~$\cC$ and a functor $\mu \colon \cC \to U/O$.
This gives rise to maps
\begin{align*}
	\Map_{\cC}(c,d) \to P_{\mu(c) \mu(d)} U/O \simeq \Omega U/O \simeq \bZ\times BO \,.
\end{align*}
Note that the first equivalence here is non-canonical, it depends on a choice of paths from $\mu(c)$ and $\mu(d)$ to the chosen basepoint.
Let $\bX$ be a flow category.
A \emph{$\mu$-structure} on~$\bX$ is (roughly) given by the following data:
\begin{enumerate}
	\item For each $p\in \ob(\bX)$ an object $c_{p}$ in $\cC$, and a reference $\gamma_{p}$ path from $\mu(c_{p})$ to the basepoint in $U/O$.
	\item For each pair of objects, a factorization of the stable classifying map of the virtual vector bundle $-T\bX(p,q) \ominus \bR$ through the map $\Map_{\cC}(c_{p},c_{q}) \to \bZ\times BO$ induced by $\mu$ and concatenation with $\gamma_{p}$ and $\gamma_{q}$.
\end{enumerate}
This generalization is perfectly suited to deal with exact Lagrangian Floer theory, because as mentioned earlier, the tangential structure of each moduli space $\bX(p,q)$ is determined via index theory by a map $P_{X}(L,K) \to U/O$ through the maps
\begin{align*}
	\bX(p,q) \to P_{pq} P_{X}(L,K)
\end{align*}
which includes holomorphic strips into all strips.
The preceding examples of Morse theory with local system coefficients, $G$-orientations, and $A$-filtrations also fit naturally in this framework.

Given a diagram of $\infty$-categories
\begin{equation}\label{eq:C->D_overU/O}
	\begin{tikzcd}
		\cC \arrow[rr, "f"] \arrow[dr, "\mu"']    &  &   \calD  \arrow[dl, "\nu"] \\
		&  U/O\,, &
	\end{tikzcd}
\end{equation}
we obtain maps $f\colon \Map_{\cC}(c,d) \to \Map_{\calD}(f(c),f(d))$ compatible with maps to $\bZ\times BO$.
A $\mu$-structure on a flow category $\bX$ can then be pushed forward along $f$ to obtain a $\nu$-structure on $\bX$.
This gives a natural language for manipulating various tangential structures in Floer theory.
Consider for example, the notion of \emph{tangential pairs} as in \cite{porcelli2024spectral}.
A tangential pair is a diagram of spaces
\begin{equation}\label{eq:tangential_pair}
	\begin{tikzcd}
		\Theta \arrow[r] \arrow[d]      &   \Phi  \arrow[d] \\
		BO  \arrow[r]               &  BU \,.
	\end{tikzcd}
\end{equation}
If we can factor the classifying maps for the stable tangent bundles of $L$, $K$ and $X$ through these maps, we get a factorization of the Lagrangian difference map through the homotopy pullback $P_{X}(L,K)\to \Theta\times_{\Phi} \Theta \to U/O$.
We can therefore push forward the natural $P_{X}(L,K)$-structure on the Floer flow category to obtain a $\Theta\times_{\Phi}\Theta$-structure.
In this formalism, it becomes clear how the Floer homotopy types constructed from different orientation data on the same Lagrangians are related.

\subsection*{Main Results}\label{subsec:main_results}

 The main construction of the paper is a family of Abouzaid--Blumberg style $\infty$-categories $\Flow^{\mu}$ of $\mu$-structured flow categories, indexed by $\mu\in \Cat_{\infty/(U/O)}$.
 We moreover construct the expected functoriality between such $\infty$-categories coming from pushforward of tangential structures. We prove that $\Flow^{\mu}$ is compactly generated, stable and presentable.
 We then identify $\Flow^{\mu}$ with the $\infty$-category $\TwShv^{\mu}$ of $\mu$-twisted presheaves on $\cC$, which as we show is the $\infty$-category one would expect by considering the bordism theory captured by $\Flow^{\mu}$.
 Write $\Pr^{\rL}_{\St,\omega}$ for the $\infty$-category of compactly generated stable presentable $\infty$-categories and left adjoint functors which preserve compact objects.
 The main result of this paper may be stated as follows.
\begin{theorem}\label{thm:Flow=TwPsh}
	There is a natural equivalence of functors $\Cat_{\infty /(U/O)} \to \Pr^{\rL}_{\St,\omega}$
	\begin{align*}
		\Flow^{(-)} \simeq \TwShv^{(-)}\,.
	\end{align*}
\end{theorem}

Let us discuss some of the techniques involved in proving this statement. To deal with the coherence involved when defining a $\mu$-structure on a flow category, we will work in the context of \emph{enriched $\infty$-categories} as defined by \cite{GH}. We give an introduction to this theory in \Cref{sec:categorical_preliminaries}. For a monoidal $\infty$-category $\cV$, a $\cV$-category $\sfC$ has a space of objects $\ob(\sfC)$, morphism objects $\sfC(c,d)$ in $\cV$, and composition maps $\sfC(c,d)\otimes \sfC(d,e) \to \sfC(c,e)$ which are associative up to specified homotopies.

In \Cref{subsec:monoids_and_delooping}, we show that a functor $\mu\colon \cC \to U/O$ corresponds to a unique $\cS_{/ (\bZ\times BO)}$-category which by abuse of notation we also call $\mu$.
The objects of $\mu$ are those of $\cC$, together with a reference path from $\mu(c)$ to the basepoint in $U/O$, and the morphism objects are precisely the maps of spaces $\Map_{\cC}(c,d) \to \bZ\times BO$ induced by $\mu$ and concatenation with reference paths.
In \Cref{subsec:double_categories_of_structured_manifolds} we construct a double $\infty$-category $\Mfd^{\mu}_{\diamond}$, whose objects are those of $\mu$, and whose morphism objects are categories $\Mfd^{\mu}_{\diamond}(c,d)$ of \emph{stratified manifolds with corners} with a $\mu(c,d)$-structure. $\Mfd_{\diamond}^{\mu}$ has a forgetful functor to the monoidal $\infty$-category $\Mod_{\diamond}$ of stratifying categories.
For a set $P$, we will define a $\Mod_{\diamond}$-enriched category $\sfA_{P}$ with objects $P$, which captures the boundary structure we expect of a flow category with objects $P$.
Actually, we will assume that $P$ comes equipped with a poset structure, and we will force $\bX(p,q)$ to be empty unless $p\leq q$. This is not restrictive, since such a structure always exists, and the equivalence class of $\bX$ will not depend on the choice. We will further assume that $P$ is \emph{locally finite}, meaning that each principal filter $P\uparrow p = \left\{q\in P, p\leq q\right\}$ is finite. This ensures that $\coprod_{q\in P} \bX(p,q)$ is compact.

\begin{definition}\label{def:intro_flowcategory}
	Let $P$ be a locally finite poset. A \emph{$\mu$-structured flow category with objects $P$} is a lift of $\sfA_{P}$ to a $\Mfd^{\mu}_{\diamond}$-category $\bX$.
\end{definition}
A $\mu$-structure on $\bX$ is essentially homotopic in nature, as we will show, the space of $\mu$-structures on a flow category $\bX$ is equivalent to the space of diagrams
\begin{equation*}
	\begin{tikzcd}
		\Pi\bX \arrow[r] \arrow[dr, "-I\bX"']      &   \cC  \arrow[d,"\mu"] \\
		            &   U/O
	\end{tikzcd}
\end{equation*}
where $\Pi\bX$ is the $\infty$-category presented by the topological category $\bX$, and $I\bX$ is induced by the index bundle which on each morphism space $\bX(p,q)$ is $T\bX(p,q)\oplus \bR$.

 Following \cite{AB}, we will define a semi-simplicial space $\PreFlow^{\mu}$ whose $0$-simplices are precisely $\mu$-structured flow categories up to isomorphism.
 Writing down functors between $\infty$-categories by hand is often not feasible, so we will construct $\PreFlow^{\mu}$ by instead producing a \emph{right fibration} over $\Simp_{s}$, and appeal to straightening/unstraightening in the sense of \cite{HTT}.
 We construct a functor $\cP \to \Simp_{s}$ to the injective simplex category, whose fiber over $[n]$ is the category of order-preserving maps $P\to [n]$ where $P$ is a locally finite poset.
 In particular, the fiber $P_{i}$ over each $i\in [n]$ is a locally finite poset.
 Following \cite{AB}, we then define an enriched category $\sfA_{P}$ in $\Mod_{\diamond}$ which captures the boundary structure of a flow $n$-simplex whose $i$-th vertex is a flow category with poset of objects $P_{i}$.
This construction determines a functor $\cA\colon \cP \to \Alg(\Mod_{\diamond})$.
Pushforward along $\Mfd^{\mu}_{\diamond}\to \Mod_{\diamond}^{\otimes}$ determines a functor $\Alg(\Mfd_{\diamond}^{\mu}) \to \Alg(\Mod_{\diamond})$. In \Cref{subsec:double_categories_of_structured_manifolds} we show that the pullback
\begin{align}\label{eq:intro_algebra_fibration}
	\cA^{*}\Alg(\Mfd_{\diamond}^{\mu}) \to \cP \to \Simp_{s}
\end{align}
is a Cartesian fibration.

\begin{definition}\label{def:intro_Flow^mu}
	Let $\PreFlow^{\mu}\colon \Simp_{s}^{\op} \to \cS$ be the semi-simplicial space corresponding to the underlying right fibration of \eqref{eq:intro_algebra_fibration} under straightening.
\end{definition}
The 1-simplices of $\PreFlow^{\mu}$ are \emph{flow bimodules}.
A bimodule $\bB\colon \bX \to \bY$ consists of a $\mu(c_{p},c_{q})$-structured manifold with corners $\bB(p,q)$ for each $p\in \ob(\bX),q\in \ob(\bY)$.
The boundary strata of $\bB(p,q)$ are images of action maps of the form
\begin{align*}
	\bX(p,r)\times \bB(r,q) \to \bX(p,q) \leftarrow \bB(p,s) \times \bY(s,q) \,,
\end{align*}
which are compatible with the $\mu$-structures of $\bX$ and $\bY$.
The higher simplices of $\PreFlow^{\mu}$ capture composition of bimodules up to bordism.

The construction of $\PreFlow^{\mu}$ is made to be functorial in $\mu \in \Cat_{\infty / (U/O)}$.
A map in this slice is a commutative triangle of the form \eqref{eq:C->D_overU/O}.
This gives rise to a map $f_{!}\colon \PreFlow^{\mu} \to \PreFlow^{\nu}$ which pushes forward $\mu(c,d)$-structures along some $f\colon \mu(c,d) \to \nu(f(c),f(d)  )$.

We say that a semi-simplicial space $X$ is a \emph{quasi-unital inner Kan space} if $X$ admits inner horn fillers, and each vertex $x\in X_{0}$ admits an idempotent equivalence.
The theory of such objects was developed in \cite{oldervoll2026quasi}, where it is shown that $X$ represents an $\infty$-category $\cT(X)$.
Various $\infty$-categorical constructions in $\cT(X)$, such as mapping spaces and slices, can be understood in terms of $X$.
By generalizing the arguments of \cite{AB}, we show that $\PreFlow^{\mu}$ admits inner horn fillers.
To show that $\PreFlow^{\mu}$ is quasi-unital, we show that there exist maps $s_{0},s_{n}\colon \PreFlow^{\mu}_{n} \to \PreFlow^{\mu}_{n+1}$ which behave as degeneracy maps.
The following is then a consequence of \cite[Theorem 1.2]{oldervoll2026quasi}.
\begin{theorem}\label{thm:Flow_is_quik}
	The semi-simplicial space $\PreFlow^{\mu}$ is a quasi-unital inner Kan space.
\end{theorem}
We let $\Flow^{\mu} \coloneq \cT(\PreFlow^{\mu})$ denote the $\infty$-category presented by this quasi-unital inner Kan space. Every object in $\Flow^{\mu}$ can be represented up to equivalence by a $\mu$-structured flow category, and every morphism by a bimodule $\bB\colon \bX\to \bY$. We show that $\Flow^{\mu}$ has the following desirable structural properties.
\begin{theorem}\label{thm:Flow_is_stable_pres}
	$\Flow^{\mu}$ is a compactly generated stable presentable $\infty$-category. The semi-simplicial map $f$ gives rise to a left adjoint functor $f_{!}\colon \Flow^{\mu}\to \Flow^{\nu}$ which moreover preserves compact objects.
\end{theorem}
Showing this amounts to computing certain (co)limits in $\Flow^{\mu}$.
We will use two main techniques for computing such. First, we will use the full structure of the Cartesian fibration \eqref{eq:intro_algebra_fibration}.
Under unstraightening, this corresponds to a semi-simplicial $\infty$-category $\bfFlow^{\mu}\colon \Simp_{s}^{\op} \to \Cat_{\infty}$.
Morphisms in the zeroth level $\bfFlow^{\mu}_{0}$ are given by inclusions of subcategories, and are therefore quite easy to understand.
We will show that certain morphisms in this $\infty$-category give rise to morphisms in $\Flow^{\mu}$ satisfying a universal property generalizing the notion of \emph{companions} in a double $\infty$-category.
Using companion morphisms, we show that the disjoint union of flow categories gives rise to coproducts in $\Flow^{\mu}$, and that any flow category can be written as a filtered colimit over its finite subcategories.
The other technique we will use is that of mapping cones.
We can think of a flow category $\bX$ as a chain complex, and a bimodule $\bB\colon \bX \to \bY$ as a chain map.
In a construction which is completely analogous to the construction of the mapping cone of a chain map, we will produce an explicit representative $c(\bB)$ of the cofiber of a bimodule $\bB$.

An important type of object in $\Flow^{\mu}$ is the one-object flow category $\bOne_{b}$ determined by an object $b$ of $\mu$.
This has a single object which is mapped to $b\in \ob(\mu)$, and no nontrivial morphism spaces.
We let $\Bord^{\mu}\subset \Flow^{\mu}$ denote the full subcategory spanned by such objects. As part of the proof of \Cref{thm:Flow_is_stable_pres}, we show that these are compact generators for $\Flow^{\mu}$.
A morphism $\bOne_{b} \to \bOne_{c}$ is determined by a closed $\mu(c,d)$-structured manifold, and a homotopy between such corresponds to a $\mu(c,d)$-structured bordism.
More generally, we can think of the mapping space $\Map(\bOne_{b}, \bOne_{c})$ as a classifying space for $\mu(c,d)$-structured bordism.
By the classical Pontrjagin--Thom equivalence, we would expect this classifying space to be equivalent to the infinite loop space $\Omega^{\infty}\Th(\mu(c,d))$.
By pushing forward the enrichment of $\mu$ along lax monoidal functors
\begin{align*}
	\cS_{/\bZ\times BO} \xrightarrow{\Th} \Sp \xrightarrow{\Omega^{\infty}} \cS \,,
\end{align*}
the spaces $\Omega^{\infty}\Th(\mu(c,d))$ are the mapping spaces of an $\infty$-category $\Omega^{\infty}\Th(\mu)$.
We will construct a suitably functorial version of the Pontrjagin--Thom collapse, and use this to show the following.

\begin{theorem}\label{thm:PT_equivalence}
	The Pontrjagin--Thom collapse determines an equivalence of $\infty$-categories $\PT^{\mu} \colon \Bord^{\mu} \to \Omega^{\infty}\Th(\mu)$.
\end{theorem}

Because $\Bord^{\mu}$ is closed under suspension in $\Flow^{\mu}$, we actually expect it to contain all the information of the corresponding mapping spectra.
In other words, we would expect to upgrade $\PT^{\mu}$ to an equivalence of categories enriched in $\Sp$.
From the Schwede--Shipley theorem \cite{schwede2003stable}, we then expect to identify $\Flow^{\mu}$ with the $\infty$-category of \emph{spectral presheaves} $\cP_{\Sp}(\Th(\mu))$ on $\Th(\mu)$.
Rather than explicitly identifying mapping spectra of $\Th(\mu)$ and $\Bord^{\mu}$, we show directly that the strong compact generation implies that the composite
\begin{align*}
	\Bord^{\mu} \simeq \Omega^{\infty}\Th(\mu) \xrightarrow{\yo} \cP_{\Sp}(\Th(\mu))
\end{align*}
Kan extends to an equivalence $\Flow^{\mu} \simeq \cP_{\Sp}(\Th(\mu))$.

To compare $\cP_{\Sp}(\Th(\mu))$ to the $\infty$-category of twisted presheaves, we apply some recently developed results about the functor $\cP_{\Sp}$.
Crucially, $\cP_{\Sp}$ lifts to a symmetric monoidal left adjoint, and we use this to show that $\cP_{\Sp}$ also lifts to a 2-functor which preserves oplax colimits.
The $\infty$-category $\TwShv^{\mu}$ can be defined as a certain oplax limit in $\PrLSt$, which by lax additivity \cite{CDW} is also an oplax colimit.
Every object $\mu\colon \cC\to U/O$ of $\Cat_{\infty / (U/O)}$ can be written canonically as the oplax colimit of the diagram
\begin{align*}
	\cC \xrightarrow{\mu}  U/O \to \Cat_{\infty / (U/O)} \,,
\end{align*}
so since the construction $\cP_{\Sp}(\mu)$ preserves oplax colimits, this gives a natural equivalence $\TwShv^{\mu}\simeq \cP_{\Sp}(\Th(\mu))$ which finishes the proof of \Cref{thm:Flow=TwPsh}.

\subsection*{Examples}
Let us illustrate the theory by returning to some examples from geometry.

\begin{ex}\label{ex:morse_homotopy_type}
	Let $f\colon X\to \bR$ be a Morse function on a closed manifold.
  In \cite{cohen2020floer}, the moduli spaces of negative gradient flow lines of $\bM_{f}$ assemble to a framed flow category $\bM_{f}$.
  As we explain in \Cref{ex:morse_flow_cat}, this admits a natural structure by the zero map $X\to U/O$, given roughly by the forgetful map $\bM_{f}(p,q) \to P_{X}(p,q)$ from gradient flow lines to continuous paths.

	We write $\Flow^{X}$ for the $\infty$-category of flow categories structured by the zero map $X\to U/O$.
  The corresponding $\infty$-category of twisted presheaves is the $\infty$-category $\Sp^{X}$ of spectral local systems on $X$.
  The flow category $\bM_{f}$ then corresponds to a finitely generated local system whose total $\bZ$-homology agrees with the homology of the base.
  By a spectral sequence argument, this can only happen for the constant local system $\underline{\bS}$.
  The total homology functor $\Sp^{X}\to \Sp$ corresponds to the functor $\Flow^{X}\to \Flow^{\fr}$ which forgets the $X$-structure, so the above gives a new proof of the fact that $\bM_{f}$ recovers the stable homotopy type $\Sigma^{\infty}_{+}X$.
  This was shown with other methods in \cite{abouzaid2021arnold, blakey2024floerhomotopytheorydegenerate,côté2024equivariantfloerhomotopymorsebott}.
\end{ex}
\begin{ex}
	An unstructured flow category $\bX$ has an index bundle which determines a functor $I\bX\colon \Pi\bX \to U/O$ from the $\infty$-category $\Pi\bX$ presented by the topological category $\bX$.
  The flow category $\bX$ admits a tautological $I\bX$-structure.
  In the case of Morse theory, the index bundle $I\bM_{f}$ is trivial, and by \cite{fourel2026}, the $\infty$-category $\Pi\bM_{f}$ agrees with the \emph{exit path category} $\mathrm{Exit}(M,f)$ associated to the stratification of $M$ by unstable manifolds of $M$.
  Functors $\mathrm{Exit}(M,f)^{\op} \to \Sp$ correspond to \emph{constructible cosheaves} on the stratified space.
  As in \Cref{ex:morse_homotopy_type}, the flow category $\bM_{f}$ with this structure corresponds to the constant cosheaf with value $\bS$.
\end{ex}

\begin{ex}\label{ex:lagrangian_floer_htpy_type}
	In \cite{large2021spectral}, the author constructs flow categories $\bF_{LK}$ lifting the Floer homology of a pair of exact Lagrangians $L,K$ in a Liouville domain $M$.
  In \Cref{sec:appendix}, we explain how the methods of \cite{porcelli2024spectral} can be used to structure $\bF_{LK}$ by the index bundle
	\begin{align*}
		-\mu\colon L\times_{X}K \xrightarrow{TL\times_{TM} TK} BO\times_{BU} BO \xrightarrow{d} U/O \xrightarrow{B\Ind} B(\bZ\times BO) \,.
	\end{align*}
	The map $d$ is the Lagrangian difference map, and the index map comes from modeling~$\Omega (U/O)$ as a space of Cauchy--Riemann operators on a complex vector bundles $E\to D^{2}$ with boundary conditions in totally real subbundles $F\subset E\vert_{S^{1}}$.
  Any such operator has a Fredholm index which is a virtual vector space.
  This construction gives a homotopy equivalence $\Omega(U/O) \simeq \bZ\times BO$, but the precise relationship between this and other models for real Bott periodicity remains unclear.

	We can relate the $\infty$-category $\TwSp^{-\mu}$ to modules over Thom spectra as follows.
  Assume that $X$ is connected, and fix a basepoint.
  Write $F_{L},F_{K}$ for the homotopy fibers of the inclusions $L,K\to X$ over $x$.
  We then have a fiber sequence
	\begin{align*}
		F_{K}\times F_{L} \to L\times_{X} K \to X \,.
	\end{align*}
	Any choice of basepoints in $F_{K}$ and $F_{L}$ gives a basepoint in $L\times_{X}K$, and looping over these basepoints exhibits a relative tensor product
	\begin{align*}
		\Omega(L\times_{X} K) \simeq \Omega F_{K} \otimes_{\Omega^{2}X} \Omega F_{L}
	\end{align*}
	of $\bE_{1}$-monoids in $\cS$.
  By \cite[Proposition 2.18]{bonciocat2025floer}, we have a commutative diagram
	\begin{equation*}
		\begin{tikzcd}
			\Omega^{2}X \arrow[r, "\Omega^{2}TX"] \arrow[d]      &   \Omega^{2}BU  \arrow[d]  \arrow[r, "\text{Bott}", "\simeq"'] & \bZ\times BU \arrow[d] \\
			\Omega F_{L}  \arrow[r]               &   \Omega U/O  \arrow[r, "\Ind", "\simeq"']            & \bZ\times BO \,,
		\end{tikzcd}
	\end{equation*}
	and a similar diagram for $F_{K}$.
  Write $\xi_{L}$ for the index bundle
	\begin{align*}
		\Omega F_{L} \to \Omega U/O \xrightarrow{\Ind} \bZ\times BO \, ,
	\end{align*}
	and $\xi_{K}$ for the corresponding bundle on $K$.
  If we write $\zeta_{X}$ for the Bott bundle
	\begin{align*}
		\zeta_{X} \colon \Omega^{2} X \xrightarrow{\Omega^{2}TX} \Omega^{2}BU \simeq \bZ\times BU \,,
	\end{align*}
	we get a description of $\Omega\mu$ as the following relative tensor product of $\bE_{1}$-algebras in $\cS_{/\bZ\times BO}$
	\begin{align*}
		\Omega \mu \simeq \xi_{L} \otimes_{\zeta_{X}} \left(-\xi_{K} \right)  \,.
	\end{align*}
	Since $\Th$ is a monoidal left adjoint, it preserves relative tensor products.
  By \Cref{ex:modules_over_th}, we can then identify a summand of $\TwSp^{-\mu}$ with the $\infty$-category of modules over $\Th(\xi_{L})^{\op }\otimes_{\Th(\zeta_{X})} \Th(\xi_{K})$.
\end{ex}

\begin{ex}\label{ex:tangential_pair}
	Appropriately modifying \Cref{ex:lagrangian_floer_htpy_type} for the case where we replace $L,K \to X$ by a tangential pair $\Theta \to \Phi$ as in \eqref{eq:tangential_pair}, we instead get modules over the ring spectrum $\Th(\xi)^{\op}\otimes_{\Th(\zeta)} \Th(\xi)$, where $\xi$ is the index bundle on $\Omega\fib(\Theta\to \Phi)$, and $\zeta$ is the Bott bundle on $\Omega^{2}\Phi$.
  Reducing to the case where $\Phi$ is contractible and $\Theta$ connected, we then have the equivalence
	\begin{align}\label{eq:Flow^Theta}
		\Flow^{\Theta} \simeq \Mod_{\Th(\xi)^{\op} \otimes \Th(\xi)} \,.
	\end{align}
	Up to relating the specifics of definitions of orientations for flow categories, this answers \cite[Conjecture 5.10]{porcelli2024spectral}, giving an equivalence between the $\infty$-category $\PreFlow^{\Theta}$ constructed \emph{op.cit} and the homotopy category of \emph{finite} $\Th(\xi)^{\op}\otimes\Th(\xi)$-modules.
  The adjective finite shows up because \cite{porcelli2024spectral} considers flow categories with a finite set of objects.
  Under the equivalence \eqref{eq:Flow^Theta}, the one-object flow category $\bOne$ corresponds to the generator $\Th(\xi)^{\op} \otimes \Th(\xi)$.
  By \Cref{cor:flow_fin_generation}, any finite $\Theta$-structured flow category therefore corresponds to a finite $\Th(\xi)^{\op} \otimes \Th(\xi)$-module.
\end{ex}
\begin{ex}
	Let $A$ be a poset and consider the $\infty$-category $\Flow^{A}$ of flow categories structured by the zero map $A\to U/O$.
  An $A$-structured flow category is a framed flow category $\bX$ with an order-preserving map $f\colon P\to A$ from the poset of objects of $\bX$.
  For each $a\in A$, let $\bOne_{a}$ denote the one-object flow category corresponding to the object $a$ and the virtual vector space $0$.
  The equivalence in \Cref{thm:Flow=TwPsh} is then given by the formula
	\begin{align}\label{eq:Flow^A}
		\Flow^{A} &\xrightarrow{\simeq} \Fun(A^{\op}, \Sp) \\
		\bX &\mapsto  (a\mapsto \map(\bOne_{a}, \bX) ) \,. \nonumber
	\end{align}
	Each morphism object of a bimodule $\bB\colon \bOne_{a} \to \bX$ must map to some relation $a\leq b$ in $A$, so $\bB(\ast, p)$ must be empty unless $p\in P_{a\leq} = f^{-1}(A_{a\leq})$.
	This means that $\bB$ must actually factor through the inclusion of the full subcategory $\bX_{a}\to \bX$ on the objects~$P_{a\leq}.$

	For a pair $a\leq b$, we have an $A$-structured bimodule $\bOne_{a}\to \bOne_{b}$ determined by the one-point manifold.
  The restriction map
	\begin{align*}
		\map(\bOne_{b},\bX) \to \map(\bOne_{a}, \bX)
	\end{align*}
	then corresponds to the inclusion of bimodules which also factor through the inclusion $\bX_{b} \to \bX$.
  By using the theory of companions, we can show that the flow categories $\bX_{a}$ assemble to a functor
	\begin{align*}
		\bX_{(-)} \colon A^{\op} \to \Flow^{A}\to \Flow^{\fr} \,,
	\end{align*}
	and that we have a natural equivalence $\map(\bOne_{(-)}, \bX ) \simeq \map(\bOne, \bX_{(-)})$ of functors $A^{\op}\to \Sp$.
  This can be refined to show that the equivalence \eqref{eq:Flow^A} factors through
	\begin{align*}
		\Flow^{A} &\xrightarrow{\simeq} \Fun(A^{\op}, \Flow^{\fr})\, ,  \\
		\bX &\mapsto (a\mapsto \bX_{a}) \,.
	\end{align*}
\end{ex}

\subsection*{Terminology and Notation}\label{subsec:terminology_and_notation}
We use the following notation in the paper:
\begin{itemize}
	\item We fix strongly inaccessible cardinalities $\kappa_{0} \ll \kappa_{1} \ll \kappa_{2}$.
  We refer to these universes as \emph{small, large} and \emph{very large}, respectively.
  We use a $\widehat{-}$ to indicate the very large $(\infty)$-category of large things.
  For example, we let $\Cat_{\infty}$ be the large $\infty$-category of small $\infty$-categories, and $\CatHat_{\infty}$ the very large $\infty$-category of large $\infty$-categories.
	\item We have fully faithful inclusions
	\begin{equation*}
		\begin{tikzcd}
			\Set \arrow[r] \arrow[d]      &   \cS  \arrow[d] \\
			\Cat  \arrow[r]               &   \Cat_{\infty}
		\end{tikzcd}
	\end{equation*}
	which we will omit from notation, i.e. if $\cC$ is a category, we will also write $\cC$ for the corresponding $\infty$-category with discrete mapping spaces.
	\item $\Simp$ denotes the category of linearly ordered finite sets $[n] = \{0, \dots , n\}$ and order-preserving functions between these. $\Simp_{s}\subset \Simp$ denotes the wide subcategory spanned by injective maps.
	\item If $\cC$ is an $\infty$-category, we use the notation $\Map_{\cC}$ for mapping spaces, $\map_{\cC}$ for mapping spectra (assuming that $\cC$ is stable), and $\overline{\map}$ for internal function objects (assuming that $\cC$ has a closed monoidal structure).
  We drop the subscript~$\cC$ if it is clear from context which $\infty$-category we are working in.
	\item $\Fun$ denotes the ($\infty$)-category of functors between two ($\infty$)-categories.
	\item $\Psh(\cC)= \Fun(\cC^{\op},\cS)$ denotes the $\infty$-category of presheaves.
  We write $s\cS=\Psh(\Simp)$ and $ss\cS = \Psh(\Simp^{\op}_{s})$ for the $\infty$-categories of (semi-)simplicial spaces.
	\item $\ast$ or $[0]$ denotes the contractible space.
	\item $\cC^{\triangleleft}= [0] \join \cC$ and $\cC^{\triangleright}= \cC \join [0]$ denote the cone and cocone on $\cC$, respectively.
	\item $\cC_{c/}$ and $\cC_{c/}$ denote the slices under and over an object $C$ of $\cC$, respectively.
	\item $\Top$ denotes a convenient category of topological spaces, such as compactly generated spaces.
  We will write $\Pi\colon \Top\to\cS$ for the localization which takes a topological space to its fundamental $\infty$-groupoid.
	\item As a rule of thumb, we will use the typeset \texttt{mathcal} for $\infty$-categories ($\mathcal{A} \mathcal{B} \mathcal{C}$), the typeset \texttt{mathbb} for flow categories and flow simplices ($\mathbb{X} \mathbb{Y} \mathbb{Z}$), the typeset \texttt{mathbf} for double $\infty$-categories ($\mathbf{A} \mathbf{B} \mathbf{C}$), and the typeset \texttt{mathsf} for categorical algebras ($\mathsf{A} \mathsf{B} \mathsf{C}$).
\end{itemize}

	\subsection*{Acknowledgments}
	We would like to thank Rune Haugseng for many helpful conversations regarding (enriched) $\infty$-categories, and for reading a draft of this paper.
  We thank Mohammed Abouzaid for helpful comments and suggestions.
  The second author would like to thank Amanda Hirschi and Francesco Morabito for organizing the \emph{Floer Homotopy Spring School} which led to the conception of this paper.

	\section{Categorical Preliminaries}\label{sec:categorical_preliminaries}

	In this section, we introduce results about $\infty$-categories that we shall need in our later constructions.
  We begin by recalling the theory of fibrations of $\infty$-categories.
  We then recall how fibrations can be used to define $\infty$-operads, and recall the setup for enriched $\infty$-categories constructed in \cite{GH}. In \Cref{subsec:slice_double_categories} we study the construction of slice-double $\infty$-categories.
  We use this theory further in \Cref{subsec:monoids_and_delooping} to study categories enriched in the slice monoidal $\infty$-category $\cS_{/X}$.
  In \Cref{subsec:inner_kan_spaces_and_companions} we recall the main results of \cite{oldervoll2026quasi} about presenting $\infty$-categories as semi-simplicial spaces, and set up a theory for companions in a semi-simplicial $\infty$-category.

	\subsection{$\infty$-categories and their models}\label{subsec:infty-categories_and_their_models}

	Before we start with the actual mathematics, we want to say some words about the philosophy of higher category theory that we take in this paper.
	The reader who is comfortable with $\infty$-categories, their different models, and model dependence vs. model independence can safely skip this section.

	When we speak of an $n$-category, we are loosely talking about a structure which has objects, morphisms between objects, 2-morphisms between morphisms, 3-morphisms between 2-morphisms, and so on, up to $n$-morphisms between $(n-1)$-morphisms.
	In particular, a $0$-category is just a set, and a $1$-category is an ordinary category.
	The standard example of a $2$-category is the structure we get by considering categories, functors between categories, and natural transformations between functors.

	In particular, this language allows us to be much looser with what we mean by two objects being ``the same''.
	For example, the right notion of ``sameness'' between categories is two categories being equivalent; to phrase what we mean by two categories being equivalent, we need to use all three levels (categories, functors, and natural transformations).
	Compare this to the notion of ``sameness'' as being ``isomorphic'' which just involves two levels (objects and morphisms).

	An $(n,k)$-category is an $n$-category where all $i$-morphisms are invertible for $i > k$.
	We will abbreviate $(\infty,1)$-category as just $\infty$-category; this means that we should have $i$-morphisms for all $i > 0$ and that these $i$-morphisms should be invertible whenever $i > 1$.
	A simple example to keep in mind is an $\infty$-category of spaces, where objects are spaces, 1-morphisms are maps between spaces, 2-morphisms are homotopies between maps of spaces, 3-morphisms are homotopies between homotopies between maps of spaces, and so on.
	Since homotopies are evidently invertible, this really forms an $(\infty,1)$-category.

	The description of an $\infty$-category above is very loose and not suitable for doing rigorous mathematics.
	When working with them, one often picks a specific model that allows one to keep track of the various hierarchies of morphisms.
	Then one makes sure that the basic concepts from category theory (such as (co)limits, adjunctions, Kan extensions, etc.) that one wants to generalize are encodable in this model by explicitly constructing them.
	There are several models of $\infty$-categories in the literature, but the most common one is arguably quasi-categories, i.e. simplicial sets that satisfy an inner horn condition.
	This model is so ubiquitous currently that it is common to just let ``$\infty$-category'' and ``quasi-category'' be synonymous words.
	However, quasi-categories are not the only model for $\infty$-categories, and it is not always the most suitable one to work with.
	For example, it is typically quite hard to write down non-trivial examples of quasi-categories, i.e., examples that are not just the $\infty$-groupoid of a space or the nerve of an ordinary category.
	It can also be quite hard to check by hand that a given simplicial set is a quasi-category.
	These types of issues are typically solved in some other model of $\infty$-categories, and the results are then imported to quasi-categories.
	Models for $\infty$-categories that are useful for these purposes are, for example:
	\begin{itemize}
		\item Using enriched category theory: topological and simplicial categories.
		\item Complete Segal spaces.
		\item Segal categories.
		\item Relative categories
		\item 1-complicial sets
	\end{itemize}
	These have all been shown to give equivalent models for $\infty$-categories, each one with its own pros and cons.

	Rigorously setting up the theory of $\infty$-categories in a specific model takes a lot of work,but once the work is done, it can in practice be forgotten about.
	Meaning: if one wants to make use of $\infty$-categories as a framework to solve a specific problem in mathematics, one can often treat the theory as a black box.
	This is the philosophy implicitly used in most papers in modern homotopy theory.
	If pressed about what model is used, many authors might say ``quasi-categories'', but in the vast majority of cases, the specific model chosen truly does not matter, and the results would be valid in any model, similar to how most mathematical results are not dependent on whether one uses ZFC set theory or some other version of set theory.
	As extensively as possible, we will work model-agnostically with $\infty$-categories. We might cite results that are proven in a fixed model (notably \cite{HTT}, \cite{HA} and \cite{GH}); these results carry over to give analogous results in any other equivalent model.

	\subsection{Fibrations and straightening-unstraightening}\label{subsec:fibrations_and_straightening}

	In ordinary category theory, writing down a functor $F : \cC \to \mathcal{D}$ is relatively straightforward.
	We write down what it does on objects, what it does on morphisms, and check by hand that this respects identities and compositions.
	If we are instead looking at a functor between $\infty$-categories, this strategy is typically not feasible, since we are dealing with an infinite amount of coherence data.
	The trick here is to use the theory of fibrations.
	Let us remind the reader loosely how to think about these concepts with some examples from ordinary category theory.

	Suppose that we are given a functor $F : \mathcal{B} \to \Set$ from some category $\mathcal{B}$ into the category of sets.
	From this we can construct a new category $\mathcal{E}_{F}$, the category of elements of $F$, as follows:
	\begin{itemize}
		\item Objects are pairs $(b,x)$ where $b$ is an object of $B$ and $x \in F(b)$.
		\item A morphism $(b,x) \to (b',x')$ is a morphism $f : b \to b'$ in $\mathcal{B}$ such that $F(f)(x) = x'$.
	\end{itemize}
	Note that the category of elements comes naturally equipped with a functor $\pi : \mathcal{E}_{F} \to \mathcal{B}$.
	Moreover, the original functor $F$ can be recovered (at least up to isomorphism) from $\pi$.
	However, not all functors $\mathcal{E} \to \mathcal{B}$ arise in this fashion from a functor $\mathcal{B} \to \Set$.
	Indeed, only functors $\mathcal{E} \to \mathcal{B}$ with a specific lifting property arise in this way; we need that for each morphism $f : b \to b'$ and any lift $e \in \mathcal{E}$ of $b$ that there is a unique morphism $\phi : e \to e'$ that lifts $f$.
	Such functors are sometimes called \emph{discrete opfibrations}.

	There is an $\infty$-categorical version of this story in which functors $\mathcal{B} \to \cS$ of $\infty$-categories correspond to functors $\pi : \mathcal{E} \to \mathcal{B}$ with a certain lifting property.

	\begin{definition}\label{def:left/right-fib}
		Given a functor $\pi : \cE \to \cB$ of $\infty$-categories, we say that $p$ is a \emph{left fibration} if
		\[
		\begin{tikzcd}
			\mathrm{Ar}(\mathcal{E}) \arrow[r] \arrow[d] & \mathrm{Ar}(\mathcal{B}) \arrow[d] \\
			\mathcal{E} \arrow[r] & \mathcal{B}
		\end{tikzcd}
		\]
		is a pullback of $\infty$-categories, where the vertical map picks out the source of the arrows.
		We say that $\pi$ is a \emph{right fibration} if the corresponding diagram, with the vertical maps picking out the targets of the arrows, is a pullback.
		The full subcategories of $(\Cat_{\infty})_{/\cB}$ consisting of left and right fibrations over $\cB$ will be denoted $\LFib(\cB)$ and $\RFib(\cB)$, respectively.
	\end{definition}

	\begin{theorem}[{\cite{HTT}}]\label{prop:straightening_LFib}
		There are equivalences of $\infty$-categories.
		\[
		\mathrm{LFib}(\mathcal{B}) \simeq \Fun(\mathcal{B},\cS) \qquad \mathrm{RFib}(\mathcal{B}) \simeq \Fun(\mathcal{B}^{\op},\cS)
		\]
	\end{theorem}

	The functor that takes a fibration and produces a functor into $\cS$ is typically called \emph{straightening}, while the functor going the other way around is called \emph{unstraightening}.
	These will be denoted $\St$ and $\Unst$. In analogy with fiber bundles and classifying spaces, we will sometimes say that a fibration $\cE\to \cC$ is \emph{classified} by its straightening.
	\begin{ex}\label{ex:slice_fibrations}
		The prototypical left fibration is the \emph{slice} $\cC_{c/} \to \cC$ for some object $c$ in $\cC$.
    This is the fibration classified by the representable copresheaf $\Map_{\cC}(c,-)\colon \cC \to \cS$.
    Dually, we have a right fibration $\cC_{/c} \to \cC$ classified by the representable presheaf $\Map_{\cC}(-,c)$.
    The representable left (resp. right) fibrations $\cE \to \cC$ can be uniquely characterized as the ones having an initial (resp. terminal) object.
    If the image of this terminal object in $\cC$ is $c$, there is a uniquely determined equivalence $\cE \simeq \cC_{c/}$ of left fibrations over $\cC$ (resp. $\cE \simeq \cC_{/c}$ of right fibrations over $\cC)$.
	\end{ex}

	We have a similar story to tell about functors into the $\infty$-category $\Cat_{\infty}$.
	We recall this story in the 1-categorical setting.
	Suppose that we are given a functor $F : \mathcal{B} \to \Cat$ from some category $\mathcal{B}$ into the category of categories.
	From this we construct a new category $\mathcal{E}_{F}$, the Grothendieck construction, is the following way:
	\begin{itemize}
		\item Objects are pairs $(b,x)$ where $b$ is an object of $\mathcal{B}$ and $x$ is an object of the category $F(b)$.
		\item A morphism $(b,x) \to (x',b')$ is pair $(f,\alpha)$ where $f : b \to b'$ is a morphism in $\mathcal{B}$ and $\alpha : F(f)(x) \to x'$ is a morphism in the category $F(b')$.
	\end{itemize}
	Again, this category is naturally equipped with a functor $\pi : \mathcal{E}_{\chi} \to \mathcal{B}$ and the functor~$F$ can be recovered up to isomorphism from $\pi$.
	Functors $\pi : \mathcal{E} \to \mathcal{B}$ constructed in this way have a particular property characterizing them; they are coCartesian fibrations.

	\begin{definition}\label{def:cocart_edge}
		Given a functor $\pi : \cE \to \cB$ of $\infty$-categories, we say that a morphism $f : e \to e'$ is \emph{$\pi$-coCartesian} if for every object $e'' \in \cE$ the commutative square
		\[
		\begin{tikzcd}
			\Map_{\cE}(e',e'') \arrow[r,"f^*"] \arrow[d] & \Map_{\cE}(e,e'') \arrow[d] \\
			\Map_{\cB}(\pi e' , \pi e'') \arrow[r,"\pi(f)_*"] & \Map_{\cE}(\pi e'' , \pi e)
		\end{tikzcd}
		\]
		is a pullback.
		Informally, this means that given a morphism $g : e \to e''$ and a factorization
		\[
		\begin{tikzcd}
			\pi e \arrow[rr,"\pi g"] \arrow [dr,"\pi f"'] & & \pi e'' \arrow[dl,"\phi"] \\
			& \pi e' &
		\end{tikzcd}
		\]
		in the $\infty$-category $\cB$, there is a unique lift to a factorization
		\[
		\begin{tikzcd}
			e \arrow[rr,"g"] \arrow [dr,"f"'] & & e'' \arrow[dl] \\
			& e' &
		\end{tikzcd}
		\]
		in the $\infty$-category $\cE$.
		Dually, we say that $f$ is \emph{$\pi$-Cartesian} if the corresponding commutative squares covariantly induced by $f_*$ and $\pi(f)_*$ are pullbacks.
	\end{definition}
	The collection of (co)Cartesian edges is closed under composition in the source. By applying the pasting law to squares like those in \Cref{def:cocart_edge}, one gets the following 2-out-of-3 property for (co)Cartesian edges.
  Given functors
  \[
  \cE\xrightarrow{p} \cF \xrightarrow{q} \cB \,,
  \] an edge~$f$ in~$\cE$ such that $p(f)$ is $q$-(co)Cartesian is $qp$-(co)Cartesian if and only if $f$ is $q$-coCartesian.
	A (co)Cartesian fibrations is a functor which admits a sufficiently many (co)Cartesian lifts.

	\begin{definition}\label{def:cocart_lift}
		Given a morphism $f : b \to b'$ in $\cB$ and an object $e \in \cE$ such that $\pi(e) \simeq b$, a \emph{$\pi$-coCartesian lift of $f$ to $e$} is a $\pi$-coCartesian morphisms $\tilde{f} : e \to e'$ such that $\pi(\tilde{f}) \simeq f$.
		The notion of \emph{$\pi$-Cartesian lifts} is defined dually.
	\end{definition}

	\begin{definition}\label{def:cocart_fib}
		We say that $\pi$ is a \emph{(co)Cartesian fibration} if $\cE$ has $\pi$-(co)Cartesian lifts of all morphisms in $\cB$.
		A morphism of (co)Cartesian fibrations over $\cB$ is a commutative diagram
		\[
		\begin{tikzcd}
			\cE \arrow[rr,"F"] \arrow[dr,"\pi"'] & & \cE' \arrow[dl,"\pi'"] \\
			& \cB
		\end{tikzcd}
		\]
		where $\pi$ and $\pi'$ are (co)Cartesian fibrations and $F$ takes $\pi$-(co)Cartesian morphisms to $\pi'$-(co)Cartesian morphisms.
		The $\infty$-category of coCartesian fibrations over $\cB$ will be denoted $\CoCart(\cB)$ while the $\infty$-category of Cartesian fibrations will be denoted $\Cart(\cB)$.
	\end{definition}

	\begin{theorem}[{\cite{HTT} }]\label{prop:straightening_cocart}
		Let $\cB$ be a small $\infty$-category. There are equivalences
		\[
		\CoCart(\mathcal{B}) \simeq \Fun(\mathcal{B},\Cat_{\infty}) \qquad \text{and} \qquad \Cart(\mathcal{B}) \simeq \Fun(\mathcal{B}^{\op},\Cat_{\infty})
		\]
		of $\infty$-categories, natural in $\cB$.
	\end{theorem}

	Again, the functor that takes a fibration and produces a functor into $\Cat_{\infty}$ will be called straightening and denoted $\St$.
	The functor going the other way will be called unstraightening and denoted $\Unst$.
  As for left/right-fibrations, we say that a (co)Cartesian fibration $\cE\to \cC$ is classified by its straightening.

	By changing universe, straightening also works for fibrations between large $\infty$-categories.
  These give equivalences
	\begin{align*}
		\widehat{\CoCart}(\cB) \simeq \Fun(\cB,\CatHat_{\infty}) \qquad \text{and} \qquad \widehat{\Cart}(\cB) \simeq \Fun(\mathcal{B}^{\op},\widehat{\Cat}_{\infty})
	\end{align*}
	for any large $\infty$-category $\cB$.
  The obvious analogues hold for large left and right fibrations.

	\begin{ex}\label{ex:ev_0_is_cocart}
		The prototypical coCartesian fibration is the target map $ev_{1}\colon \Ar(\cC) \to \cC$.
    A morphism in $\Ar(\cC)$ is $ev_{1}$-coCartesian if and only if it is of the form
		\begin{equation*}
			\begin{tikzcd}
				a \arrow[r, "\simeq"] \arrow[d]      &   a'  \arrow[d] \\
				b  \arrow[r, "f"]               &   c \, .
			\end{tikzcd}
		\end{equation*}
		The fiber of $ev_{1}$ at some $b\in \cC$ is the slice $\infty$-category $\cC_{/b}$, and the functor $\cC_{/b} \to \cC_{/c}$ induced by taking coCartesian lifts of some $f\colon b\to c$ is given by composing with~$f$.
    The straightening of $ev_{1}$ therefore gives us the expected functoriality of slices in $\cC$.
    If~$\cC$ admits pullbacks, each coCartesian transport functor $f_{!}\colon \cC_{/b}\to \cC_{/c}$ admits a right adjoint $f^{*}$ determined by base-change along $f$.
    This implies that $ev_{1}$ is also a Cartesian fibration, whose straightening captures the functoriality of base-change.
	\end{ex}

	Note that any right fibration is also a Cartesian fibration, and any left fibration a coCartesian fibration.
  In each case, these can be characterized as the (co)Cartesian fibrations $\pi\colon \cE\to \cC$ for which each $\cE_{b}$ is a space, or equivalently, those for which every edge in $\cE$ is $\pi$-(co)Cartesian.
  In particular, any morphism of right (left) fibrations must preserve (co)Cartesian lifts, so we have full subcategories
	\begin{align}\label{eq:Lfib->CoCart}
		\LFib(\cC) \to \CoCart(\cC) \quad \text{and} \quad \RFib(\cC) \to \Cart(\cC) \,.
	\end{align}
	Under straightening, these correspond precisely to post-composing with the inclusion $\cS\to \Cat_{\infty}$.
  The inclusions \eqref{eq:Lfib->CoCart} therefore admit right adjoints
	\begin{align*}
		\cR \colon \CoCart(\cC)\to \LFib(\cC) \quad \text{and} \quad  \cL\colon \Cart(\cC) \to \RFib(\cC) \,.
	\end{align*}
	which on the straight side correspond to post-composing with $(-)^{\simeq}\colon \Cat_{\infty}\to \cS$.
  We call $\cR(\cE)\to \cC$ the \emph{underlying right fibration} of $\cE\to \cC$.
	\begin{prop}\label{thm:iterated_fibration}
		Given a morphism of Cartesian fibrations over $\cB$
		\begin{equation*}
			\begin{tikzcd}
				\cE \arrow[rr, "r"] \arrow[dr, "p"'] & & \cF \arrow[dl, "q"] \\
				& \cB &
			\end{tikzcd}
		\end{equation*}
		such that for each edge $b\in \cB$, the map on fibers $r_{b}\colon \cE_{b} \to \cF_{b}$ is a Cartesian fibration, and such that for each morphism $f\colon b\to c$ in $\cB$, the square
		\begin{equation*}
			\begin{tikzcd}
				\cE_{c} \arrow[r, "f^{*}_{\cE}"] \arrow[d, "r_{c}"]      &   \cE_{b}  \arrow[d, "r_{b}"] \\
				\cF_{c}  \arrow[r, "f^{*}_{\cE}"]               &   \cF_{b}
			\end{tikzcd}
		\end{equation*}
		is a morphism of Cartesian fibrations, then $r$ is a Cartesian fibration.
    If each $\cE_{b}\to \cF_{b}$ is a right fibration, the last condition is automatic, and $r$ is also a right fibration.
	\end{prop}
	\begin{proof}
		For a $q$-Cartesian edge $f\colon x\to y$ in $\cF$ and a lift $\tilde{y}$ of $y$ to $\cE$, consider the $p$-Cartesian lift $\tilde{f}$ of $q(f)$ at $\tilde{y}$.
    Because $r$ preserves Cartesian lifts, we must have $r(\tilde{f}) = f$.
    By the 2-out-of-3 property for Cartesian edges, $\tilde{f}$ must be $r$-Cartesian.

		For an edge $f\colon x\to y$ in a fiber $\cF_{b}$ and a lift $\tilde{y}$ of $y$ to $\cE$, we can also factor $\tilde{y}$ through the fiber $\cE_{b}$. We can then find an $r_{b}$-Cartesian lift $\tilde{f}$ of $f$ at $\tilde{y}$. Because fiber transport $\cE_{c} \to \cE_{b}$ takes $r_{c}$-Cartesian edges to $r_{b}$-Cartesian edges, one can check that $\tilde{f}$ is actually $r$-Cartesian.

		Since every morphism in $\cF$ factors by a morphism in a fiber followed by a $q$-coCartesian morphism, and since the collection of $r$-Cartesian edges in $\cE$ is closed under composition, the above implies that $r$ is a Cartesian fibration.

		If each $\cE_{b}\to \cF_{b}$ is a right fibration, all edges in $\cE_{b}$ are $r_{b}$-Cartesian.
    Therefore, the condition that $\cE_{c}\to \cE_{b}$ preserves Cartesian lifts is automatic.
    Any edge in $\cE$ factors as an edge in a fiber $\cE_{b}$ followed by a $p$-Cartesian edge.
    By the first part of the proof, $r_{b}$-Cartesian edges are also $r$-Cartesian, and since $\cE_{b}\to \cF_{b}$ is a right fibration, every edge in $\cE_{b}$ is $r$-Cartesian.
    Every $p$-Cartesian edge is mapped by $r$ to a $q$-Cartesian edge, so by the 2-out-of-3 property for Cartesian edges, every $p$-Cartesian edge is also $r$-Cartesian.
    This shows that every edge in $\cE$ is $r$-Cartesian, so $r$ is a right fibration.
	\end{proof}

	\subsection{Monoidal and double $\infty$-categories}\label{subsec:monoidal_categories}

	With the language of fibrations in place, we can talk about certain fibrations representing operadic structures.
  We begin with defining the broadest class of operads and algebras, before specializing to the cases that will be of interest for the rest of the paper, namely monoidal $\infty$-categories and double $\infty$-categories.

	\begin{definition}[{\cite*[Definition 2.2.3]{GH}}] \label{def:inert}
		A morphism in $\Simp^{\op}$ is \emph{inert} if it is opposite to an order-preserving map $\phi\colon [n] \to [m]$ which is the inclusion of a subinterval, i.e. $\phi(i+1) = \phi(i)+1$.
	\end{definition}

	\begin{ex}\label{ex:rho_i}
		For $0\leq i \leq n-1$, we have an inert morphism $\rho_{i}$ opposite to
		\begin{align*}
			[1] &\to [n]\\
			0 &\mapsto i \\
			1 & \mapsto i+1 \,.
		\end{align*}
	\end{ex}
	Let $\cM\to \Simp^{\op}$ be a functor of $\infty$-categories.
	We denote the fiber over $[n]$ by $\cM_{n}$.
	\begin{definition}[{\cite*[Definition 2.4.1]{GH} }] \label{def:gen_nonsym_operad}
		A functor $\cM\to \Simp^{\op}$ is a \emph{generalized non-symmetric $\infty$-operad} if:
		\begin{enumerate}
			\item For every inert morphism $\phi\colon [n]\to [m] \in \Simp^{\op}$ and every $X\in \cM_{n}$, there exists a coCartesian lift $X\to \phi_!X$ of $\phi$ in $\cM$.
			\item For any $[n]\in \Simp^{\op}$, the transport along the inert morphisms $\rho_{i}$ induce an equivalence
			\begin{align*}
				\cM_{n} \overset{\simeq}\to \cM_{1}\times_{\cM_{0}} \dots \times_{\cM_{0}} \cM_{0} \,,
			\end{align*}
			where the iterated fiber product is taken $n$ times.
			\item For any morphism $\phi \colon [n] \to [m] \in \Simp^{\op}$, and any pair of objects $X\in \cM_{n}$ and $Y\in \cM_{m}$, composition with coCartesian lifts over inert morphisms induce an equivalence
			\begin{align*}
				\Map_{\cM}^{\phi}(X,Y) \overset{\simeq}\to \lim_{\alpha} \Map_{\cM}^{\alpha\circ \phi}(X, \alpha_! Y) \,.
			\end{align*}
			Here, we are taking the limit over the full subcategory of $(\Simp^{\op})_{[m]/}$ spanned by inert morphisms with target $[0]$ or $[1]$, and the superscripts indicate the subspace of morphisms in $\cM$ that lift the indicated morphism in $\Simp^{\op}$.
		\end{enumerate}
	\end{definition}

	\begin{rem}\label{rem:inner_fib}
		In~\cite*[Definition 2.4.1]{GH} there is an extra assumption of $\cM \to \Simp^{\op}$ being an inner fibration.
		This condition is only relevant in the situation where we are working with the quasi-categorical model of $\infty$-categories.
		In particular, it guarantees that the fiber of $\cM\to \Simp^{\op}$ at some $[n]$ taken in simplicial sets really is the categorical fiber that we are expecting it to be.
		Working model-agnostically, we do not have to worry about this.
		Indeed, since we are already assuming access to a model-independent $\infty$-category $\Cat_{\infty}$ in which, in particular, limits can be made sense of internally to $\Cat_{\infty}$, so we simply take $\cM_{n}$ to be the pullback
		\[
		\begin{tikzcd}
			\cM_{n} \arrow[r] \arrow[d] & \cM \arrow[d] \\
			\{[n]\} \arrow[r] & \Simp^{\op}
		\end{tikzcd}
		\]
		in $\Cat_{\infty}$.
	\end{rem}

	\begin{definition}\label{def:algebra}
		For generalized non-symmetric $\infty$-operads $\cM$ and $\cN$, an \emph{$\cM$-algebra in $\cN$} is a diagram
		\begin{equation*}
			\begin{tikzcd}
				\cM \arrow[rr, "A"] \arrow[dr] && \cN \arrow[dl] \\
				&\Simp^{\op} &
			\end{tikzcd}
		\end{equation*}
		in $\Cat_{\infty}$ such that $A$ preserves coCartesian lifts of inert morphisms. Write $\Alg_{\cM}(\cN)$ for the full subcategory of $\Fun_{/\Simp^{\op}}(\cM, \cN)$ spanned by the algebras.
	\end{definition}

	We let $\Opd$ denote the subcategory of $(\Cat_{\infty})_{/\Simp^{\op}}$ spanned by generalized non-symmetric $\infty$-operads and algebras between them.
	Because we can compose algebras, the construction $\Alg_{(-)}(\cN)$ for a fixed $\cN$ determines a functor $(\Opd)^{\op} \to \Cat_{\infty}$.

	\begin{definition}[{\cite*[Definition 3.6.1]{GH} }] \label{def:algebra_fibration}
		For $\cN$ a generalized non-symmetric $\infty$-operad, let $\Alg(\cN) \to \Opd$ be the Cartesian fibration classified by $\cM \mapsto \Alg_{\cM}(\cN)$.
	\end{definition}

	\begin{definition}[{\cite*[Definition 2.2.13]{GH}} ] \label{def:monoidal_infty_cat}
		A \emph{monoidal $\infty$-category} is a generalized non-symmetric $\infty$-operad $\cV^{\otimes} \to \Simp^{\op}$ that is also a coCartesian fibration and such that $\cV_{[0]} \simeq [0]$.
	\end{definition}

	\begin{rem} \label{rem:unwinding_monoidal_cat}
		Let us unpack the definition to convince ourselves that this definition makes sense.
		If $\cV^{\otimes} \to \Simp^{\op}$ is a monoidal $\infty$-category, then we think of $\cV^{\otimes}$ as having an underlying $\infty$-category $\cV\coloneq \cV^{\otimes}_{1}$.
		Sometimes, we will abusively refer to $\cV$ as a monoidal $\infty$-category, remembering that it is obtained in this way.
		As noted in~\cite*[Remark 2.2.15]{GH}, condition (3) of \Cref{def:gen_nonsym_operad} follows from condition (2) when $\cM\to \Simp^{\op}$ is a coCartesian fibration.
		With the assumption of $\cV_{0}$ being trivial, condition (2) becomes that the inert morphism $\rho_{i}$ induces equivalences
		\begin{align*}
			\cV_{n}^{\otimes} \simeq \bigl(\cV_{1}^{\otimes}\bigr)^{n} = \cV^{n} \,.
		\end{align*}
		We also have coCartesian lifts of the morphisms opposite to
		\begin{align*}
			\alpha_n \colon [1] &\to [n]\\
			0 &\mapsto 0\\
			1 &\mapsto n \,.
		\end{align*}
		These specify a functor
		\begin{align*}
			\cV^{n} \simeq \cV_{n}^{\otimes} \to \cV
		\end{align*}
		which tells us how to tensor together a string of $n$ objects of $\cV$.
		The coCartesian lift of the morphism opposite to $[1] \to [0]$ specifies a unit object $I_{\cV} \in \cV$.
	\end{rem}

	\begin{definition} \label{def:Mon_and_MonLax}
		Denote the full subcategory of $\CoCart(\Simp^{\op})$ spanned by the monoidal $\infty$-categories by $\Mon$, and the full subcategory of $\Opd$ spanned by the monoidal $\infty$-categories by $\MonLax$.
		We call the morphisms in these $\infty$-categories \emph{monoidal} and \emph{lax monoidal functors}, respectively.
	\end{definition}

	\begin{rem}\label{rem:unwinding_monoidal_functor}
		Let us see why the names monoidal and lax monoidal functor are justified.
		Consider a diagram
		\begin{equation*}
			\begin{tikzcd}
				\cV^{\otimes} \arrow[rr, "F"] \arrow[dr] && \cW^{\otimes} \arrow[dl] \\
				& \Simp^{\op}&
			\end{tikzcd}
		\end{equation*}
		where $\cV$ and $\cW$ are monoidal $\infty$-categories.
		For a pair of objects $V_0, V_1 \in \cV$, consider the coCartesian lifts
		\begin{align*}
			(V_0,V_1) & \xrightarrow{(\rho_{i})_{!}} V_i, \quad i=0,1 \\
			(V_0,V_1) &\xrightarrow{(\alpha_2)_!} V_0\otimes V_1 \,.
		\end{align*}
		If~$F$ is a monoidal functor, it preserves both of these coCartesian lifts.
		Because $F$ preserves the lifts of $\rho_{i}$, we must have $F( (V_0, V_1)) = (F(V_0), F(V_1))$.
		The coCartesian lifts of $\alpha_2$ in $\cV$ and $\cW$ compute the tensor product in $\cV$ and $\cW$, respectively.
		Hence, since~$F$ preserves this lift, we must have
		\begin{align*}
			F(V_0 \otimes V_1) \simeq F(V_0) \otimes F(V_1) \,.
		\end{align*}
		Now consider the case where $F$ is just lax monoidal.
		The second part of the argument above falls through, but we still know  that the morphism $F\left( (\alpha_2)_! \right)$ factors uniquely through the coCartesian lift of $\alpha_2$ in $\cW$, so we get a diagram
		\begin{equation*}
			\begin{tikzcd}
				F( (V_0, V_{1}) ) = ( F(V_0), F(V_1)) \arrow[dr, "F( (\alpha_2)_! )"']\arrow[rr, "(\alpha_2)_{!}"] & & F(V_0) \otimes F(V_1)  \arrow[dl, "\exists !"] \\
				& F(V_0 \otimes V_1)\,. &
			\end{tikzcd}
		\end{equation*}
	\end{rem}

	\begin{definition} \label{def:double_cat}
		A generalized non-symmetric $\infty$-operad $\bfC\to \Simp^{\op}$ that is also a coCartesian fibration is called a \emph{double $\infty$-category}.
		As mentioned in the introduction, we will typically denote double $\infty$-categories using \texttt{mathbf} font.
	\end{definition}

	\begin{rem}\label{rem:double_is_segal_cat}
		Because double $\infty$-categories are coCartesian fibrations, condition (3) of \Cref{def:gen_nonsym_operad} actually follows automatically from (1) and (2).
		This means that under straightening, a double $\infty$-category is precisely a Segal object $\Simp^{\op} \to \Cat_{\infty}$.
		Simplicial objects in any $\infty$-category with finite limits satisfying the Segal condition can be thought of as an $\infty$-category internal to that $\infty$-category.
		We therefore think of a double $\infty$-category $\bfC$ as having an $\infty$-category $\bfC_{0}$ of objects and an $\infty$-category $\bfC_{1}$ of morphisms. In particular, for any pair of objects $c,d\in \bfC_{0}$, we can think of the pullback
		\begin{equation*}
			\begin{tikzcd}
				\bfC(c,d) \arrow[r] \arrow[d]      &   \bfC_{1}  \arrow[d, "{(d_{1},d_{0})}"] \\
				\ast  \arrow[r, "{(c,d)}"]               &   \bfC_{0}\times \bfC_{0}
			\end{tikzcd}
		\end{equation*}
		as being an $\infty$-category of morphisms between $c$ and $d$. Condition (2) allows us to specify functors
		\[
		\bfC_{n} \simeq \bfC_{1} \times_{\bfC_{0}} \cdots \times_{\bfC_{0}} \bfC_{1} \, ,
		\]
		which tell us how to compose strings of composable morphisms.
		When we compare this definition to \Cref{def:monoidal_infty_cat}, we see that a monoidal $\infty$-category is a double $\infty$-category with a single object.
	\end{rem}

	\begin{definition}\label{def:Dbl_and_DblLax}
		We write $\Dbl_{\infty}$ for the full subcategory of $\CoCart(\Simp^{\op})$ spanned by the double $\infty$-categories, and $\DblLax$ for the full subcategory of $\Opd$ spanned by the double $\infty$-categories.
		We call the morphisms in these $\infty$-categories \emph{double functors} and \emph{lax functors}, respectively.
	\end{definition}
	\begin{definition}\label{def:horizontal_morphisms}
		Let $\bfC\to \Simp^{\op}$ be a double $\infty$-category. We write $h\bfC\to \Simp^{\op}$ for the underlying right fibration of $\bfC$, which is the \emph{Segal space of horizontal morphisms} in $\bfC$.
	\end{definition}
	\begin{rem}\label{rem:horizontal_morph_adjunction}
		By unstraightening, the full subcategory $\Seg \subset s\cS$ of Segal spaces embeds fully faithfully into $\Dbl$. This gives an adjunction
		\begin{align*}
			\iota : \Seg \rightleftarrows \Dbl : h
		\end{align*}
		restricting the underlying right fibration adjunction \eqref{eq:Lfib->CoCart}.
    We may sometimes omit $\iota$ from notation, and thereby think of any Segal space, and in particular, any $\infty$-category~$\cC$ as also being a double $\infty$-category $\cC\to \Simp^{\op}$.
	\end{rem}

	\begin{rem}\label{rem:double_cat_with_space_of_objects}
		Double $\infty$-categories such that $\bfC_{0}$ is a space and such that $h\bfC$ is a \emph{complete} Segal space are usually called $(\infty,2)$-categories.
    While all the double $\infty$-categories that show up in this paper have $\bfC_{0}$ a space, not all of them will have $h\bfC$ complete, a notable example being monoidal $\infty$-categories.
    We will therefore work in the $\infty$-category $\Dbl$, but keep in mind that our objects are more akin to $(\infty,2)$-categories.
	\end{rem}

	\begin{rem}\label{rem:lax_vs_double_functors}
		Note that the difference between double functors and lax functors is that while double functors have to preserve \emph{all} coCartesian lifts, lax double functors only have to preserve coCartesian lifts of the inert morphisms.
		This generalizes the situation of monoidal and lax monoidal functors.
	\end{rem}

	\begin{rem}\label{rem:Mon->Dbl}
		We have a diagram of $\infty$-categories
		\begin{equation*}
			\begin{tikzcd}
				\Mon \arrow[r] \arrow[d] & \MonLax \arrow[d] \\
				\Dbl_{\infty} \arrow[r] & \DblLax \, ,
			\end{tikzcd}
		\end{equation*}
		where the vertical maps are inclusions of full subcategories, the horizontal maps are inclusions of wide subcategories.
	\end{rem}

	\begin{lem}\label{lem:Mon->Dbl->Cat}
		The forgetful functors $\Mon \to \Dbl_{\infty} \to \Cat_{\infty}$ preserve and detect limits over weakly contractible $\infty$-categories.
	\end{lem}

	\begin{proof}
		For the inclusion of monoidal $\infty$-categories into double $\infty$-categories, consider the fact that taking fibers over $[0]$ commutes with limits, and that the limit in $\Cat_{\infty}$ of a constant diagram at $[0]$ is $[0]$.
		For the other statements, consider the composition
		\begin{align*}
			\Dbl_{\infty} \to \CoCart(\Simp^{\op}) \to \Cat_{\infty /\Simp^{\op}} \to \Cat_{\infty} \,.
		\end{align*}
		Since the Segal condition is preserved under limits, the first functor preserves limits.
    The second functor admits a left adjoint by \cite*[\S 1.1]{mazel2015all} and therefore preserves all limits. The forgetful functor from the slice preserves limits over weakly contractible $\infty$-categories.
	\end{proof}

	\subsection{Enriched categories}\label{subsec:categorical_algebras}

	A category enriched in a double $\infty$-category $\bfM$ should have a space of objects $X$. Each object $x\in X$ should be assigned to some object $c_{x}$ of $\bfM$, and for every pair of objects, we should have a morphism object $\sfC(x,y) \in \bfM(c_{x},c_{y})$.
  Moreover, we should have composition morphisms
	\begin{align*}
		\sfC(x,y) \odot \sfC(y,z) \to \sfC(x,z)
	\end{align*}
	in $\bfM(c_{x},c_{z})$, where $\odot$ denotes the horizontal composition of $\bfC$.
  This should be unital and associative up to prescribed homotopy. In this subsection, we will see how this type of structure can be described as an algebra in $\bfM$ over a particularly simple operad.

	\begin{definition}[{\cite*[Definition 4.1.1]{GH}}] \label{def:indiscrete_category}
		For an  $\infty$-category $\cC$, let $\Simp^{\op}_{\cC}\to \Simp^{\op}$ denote the coCartesian fibration classified by
		\begin{align*}
			\Simp^{\op} \to \Cat_{\infty} \,, \quad [n] \mapsto \Fun\left( \left\{ 0,\dots, n\right\}, \cC \right) \simeq \cC^{n+1} \,.
		\end{align*}
		This evidently satisfies the Segal conditions, and therefore defines a double $\infty$-category.
	\end{definition}

	\begin{rem}\label{rem:C->Delta_C_0}
		If $\bfC$ is a double $\infty$-category, the inert morphisms $\rho_{i}$ from \Cref{ex:rho_i} determine a double functor
		\begin{align*}
			\bfC \to \Simp^{\op}_{\bfC_{[0]}}
		\end{align*}
		whose fiber at some $(c,d)$ is precisely the morphism category $\bfC(c,d)$.
	\end{rem}

	By precomposing with the inclusion $\cS\subset \Cat_{\infty}$, the construction $X\mapsto \Simp^{\op}_{X}$ determines a functor $\cS \to \Opd$.
  Since $(\Simp^{\op}_{X})_{1} \simeq X^{2}$, all the morphism $\infty$-categories of $\Simp^{\op}_{X}$ are contractible.
	One should therefore think of $\Simp^{\op}_{X}$ as an indiscrete or chaotic category on the space of objects $X$.
	That is, a category with the same objects as $X$ and with a unique morphism between any two objects.

	\begin{definition}\label{def:Alg_Cat}
		For a generalized non-symmetric $\infty$-operad $\cM$, let $\Alg_{\Cat}(\cM)$ be defined by the pullback
		\begin{equation*}
			\begin{tikzcd}
				\Alg_{\Cat}(\cM) \arrow[r] \arrow[d] & \Alg(\cM) \arrow[d] \\
				\cS \arrow[r, "\Simp^{\op}_{(-)}"] & \Opd \,.
			\end{tikzcd}
		\end{equation*}
	\end{definition}

	We call objects of $\Alg_{\Cat}(\cM)$ \emph{$\cM$-categories}, and morphisms \emph{enriched functors}.
  The construction $\Alg_{\Cat}(-)$ is covariantly functorial in $\Opd$ in an obvious way.
	As mentioned in the introduction, we will typically denote $\cM$-categories using the \texttt{mathsf} typeset.
	If $\mathsf{C}$ is an $\cM$-category, we write $\ob(\mathsf{C})$ for its space of objects.
	That is, the image of $\mathsf{C}$ under the projection $\Alg_{\Cat}(\cM) \to \cS$.

	\begin{rem}\label{rem:unwinding_enriched_cat}
		Let $\bfM$ be a double $\infty$-category, and $\sfC$ an $\bfM$-category.
    Let us unpack the definition to see that the term $\bfM$-category is justified.
		The composite
		\begin{align*}
			\Simp^{\op}_{\ob(\sfC)}\to \bfM \to \Simp^{\op}_{\bfM_{[0]}}
		\end{align*}
		preserves coCartesian lifts of inerts, and is therefore fully determined by its restriction to the fiber over $[0]$, where it gives a functor $c_{(-)}\colon \ob(\sfC) \to \bfM_{0}$.
		Evaluating $\sfC$ at a tuple $(x,y)$ gives an object $\sfC(x,y)$ in the fiber of $\mathbf{M} \to \Simp^{\op}_{\mathbf{M}_{[0]}}$ over $(c_{x}, c_{y})$, which is precisely $\bfM(c_{x}, c_{y})$.

		For a triple $(x,y,z)\in X^{3}$, there are inert morphisms $(x,y,z)\to (x,y)$ and $(x,y,z) \to (y,z)$.
		Because $\sfC$ preserves lifts of these, we must have
		\begin{align*}
			\bfM\left( c_{x}, c_{y}, c_{z} \right) &\xrightarrow{\simeq} \bfM(c_{x}, c_{y}) \times  \mathbf{M}(c_{y},c_{z}) \\
			\sfC(x,y,z) &\mapsto (\sfC(x,y), \sfC(y,z)) \,.
		\end{align*}
		The image of the active morphism $(x,y,z) \to (x,z)$ factors uniquely as the coCartesian lift, which computes horizontal composition in $\mathbf{C}$, followed by a map
		\begin{align*}
			\sfC(x,y) \odot \sfC(y,z) \to \sfC(x,z) \in \bfM(c_{x}, c_{z}) \,.
		\end{align*}
		The unit for the monoidal structure on $\Cat_{\infty}$ is the terminal category $[0]$.
		Therefore, the unital structure of $\bfM$ selects an object $I_m \in \bfM(m,m)$ for each $m\in \bfM_{0}$.
		Evaluating~$\sfC$ at the morphism $(x)\to (x,x)$ in $\Simp^{\op}_{\ob(\sfC)}$ gives a morphism which factors through the coCartesian lift of $(c_{x}) \to (c_{x}, c_{x})$.
		This coCartesian lift selects the unit object at $c_x$, so we get a unit morphism
		\begin{align*}
			I_{c_{x}} \to \sfC(x,x) \in \bfM(c_{x}, c_{x}) \,.
		\end{align*}
		Evaluating $\sfC$ at higher simplices of $\Simp^{\op}_{\ob(\sfC)}$ gives homotopies witnessing associativity and unitality of composition.
		For instance, the 2-cell witnessing that the composite $d_{1}\circ s_{0} = id$ is sent to a 2-cell witnessing
		\begin{equation*}
			\sfC(x,y) \simeq I_{c_{x}} \odot \sfC(x,y) \to \sfC(x,x) \odot \sfC(x,y) \to \sfC(x,y)
		\end{equation*}
		as the identity.
	\end{rem}

	\begin{rem}\label{rem:unwiding_enriched_functor}
		Let $\bfM$ be a double $\infty$-category.
    An enriched functor $F\colon \sfC\to \sfD$ between $\bfM$-categories consists of a map of spaces $f\colon \ob(\sfC) \to \ob(\sfD)$ and a natural transformation
		\begin{equation*}
			\begin{tikzcd}[row sep=large, column sep=large]
				\Simp^{\op}_{\ob(\sfC)} \arrow[r,"f"] \arrow[dr, "\sfC"' ,""{name=T}]      &   \Simp^{\op}_{\ob(\sfD)}  \arrow[d, "\sfD"] \\
		        &   \bfM \, .
				\arrow[from=T, to=1-2, Rightarrow]
			\end{tikzcd}
		\end{equation*}
		In other words, we have morphisms
		\begin{align*}
			F_{xy}\colon \sfC(x,y) \to \sfD( f(x), f(y))
		\end{align*}
		in $\bfC(c_{x}^{\sfC}, c_{y}^{\sfC}) \simeq \bfC (c_{f(x)}^{\sfD}, c_{f(y)}^{\sfD} )$, and homotopies making these compatible with composition.
	\end{rem}

	\begin{rem}\label{rem:enriched_in_monoidal}
		If we specialize to the case where $\cV$ is a monoidal $\infty$-category, the structure of a $\cV$-category is less complicated: all the morphism objects $\sfC(x,y)$ belong to $\cV$, and composition maps are of the form $\sfC(x,y)\otimes \sfC(y,z) \to \sfC(x,z)$.
    An enriched functor $F\colon \sfC\to \sfD$ is determined by a map of spaces $f\colon \ob(\sfC) \to \ob(\sfD)$, and maps
		\begin{align*}
			F_{xy}\colon \sfC(x,y) \to \sfD(f(x),f(y))
		\end{align*}
		 in $\cV$ that are compatible with composition and units up to prescribed homotopy.
	\end{rem}
	\begin{ex}\label{ex:enriched_from_closed}
		Let $\cV^{\otimes}$ be a monoidal $\infty$-category. We say that $\cV$ is \emph{left closed} if for every object $X$ the functor $X \otimes (-)\colon \cV\to\cV$ admits a right adjoint.
    We can think of such a right adjoint as giving internal mapping objects $\overline{\map}(X,-)$ in $\cV$.
    By \cite[Example 7.2.11, Corollary 7.4.10]{GH}, there exists a $\cV$-category $\overline{\cV}\colon\Simp^{\op}_{\cV^{\simeq}} \to \cV^{\otimes}$ with $\overline{\cV}(X,Y) = \overline{\map}(X,Y)$.
    This construction also works if we consider a module $\cC$ over $\cV$ such that the functors $X\otimes (-)\colon \cV \to \cC$ admit right adjoints.
    In particular, any stable presentable $\infty$-category $\cC$ is tensored over $\Sp$, and therefore gives rise to a large $\Sp$-category $\overline{\cC}$ whose morphism objects are mapping spectra in $\cC$.
    Structures of this kind can be generalized to define an equivalent framework for $\cV$-enriched categories using \emph{weak $\cV$-modules} \cite{hinich2018yoneda,heine2025equivalence}.
	\end{ex}

	\begin{rem}\label{rem:pushforward_of_CatAlg}
		Let $F\colon \bfM\to \bfN$ be a lax functor between double $\infty$-categories, and $\sfC$ an $\bfM$-category.
		The algebra $F_*\sfC$ is determined by the composite algebra
		\begin{align*}
			\Simp^{\op}_{\ob(\sfC)} \xrightarrow{\sfC} \bfM \xrightarrow{F} \bfN \,.
		\end{align*}
		The algebra $F_{*}\sfC$ therefore has the same space of objects as $\sfC$, and mapping objects
		\begin{align*}
			F_*\sfC(x,y) = F\left( \sfC(x,y) \right) \,.
		\end{align*}
	\end{rem}

	\begin{rem}\label{rem:monoidal_structure_on_Alg_Cat}
		Having an $\bE_{2}$-monoidal structure on an $\infty$-category means, among other things, that the functor
		\begin{align*}
			\cV \times \cV \to \cV
		\end{align*}
		given by forming the monoidal product itself extends to a monoidal functor
		\begin{align*}
			\cV^{\otimes} \times_{\Simp^{\op}} \cV^{\otimes} \to \cV^{\otimes}.
		\end{align*}
		By \cite*[Corollary 4.3.13]{GH}, an $\bE_{2}$-monoidal structure on $\cV$ induces a monoidal structure on $\Alg_{\Cat}(\cV)$ which is given by taking cartesian product of spaces of objects, and tensor product in $\cV$ of mapping objects, i.e.,
		\begin{align*}
		    \sfC \otimes \sfD \colon \Simp^{\op}_{\ob(\sfC) \times \ob(\sfD)} \simeq \Simp^{\op}_{\ob(\sfC)} \times_{\Simp^{\op}} \Simp^{\op}_{\ob(\sfD)} \xrightarrow{\sfC \times \sfD} \cV^{\otimes} \times_{\Simp^{\op}} \cV^{\otimes} \xrightarrow{ \otimes} \cV^{\otimes} \,.
		\end{align*}
		If $\cV$ has a symmetric monoidal structure, the above extends to a symmetric monoidal structure on $\Alg_{\Cat}(\cV)$.
	\end{rem}

    If $\cV$ is an $\infty$-category with finite limits, the Cartesian product determines a symmetric monoidal structure $\cV^{\times} \to \Simp^{\op}$.
	Describing algebras in such $\infty$-categories requires less data: because the product satisfies a universal property in $\cV$, it is enough to work with \textit{monoid objects}.
	Let $(\Simp_{X}^{\op})_{\interior} \subset \Simp^{\op}_{X}$ be the subcategory of morphisms covering inert morphisms in $\Simp^{\op}$.

	\begin{definition}[{\cite*[Definition 3.5.1]{GH}}]  \label{def:X-monoid}
		For a space $X$ and an $\infty$-category $\cV$ with finite limits, a \emph{$\Simp_{X}^{\op}$-monoid in $\cV$} is a functor
		\begin{align*}
			\sfC \colon \Simp_X^{\op} \to \cV
		\end{align*}
		such that the restriction of $\sfC$ to $(\Simp^{\op}_{X})_{\interior}$ is right Kan extended from the restriction of $\sfC$ to $(\Simp^{\op}_{X})_{1}$.
    We write $\Monoid_{X}(\cV) \subset \Fun(\Simp^{\op}_{X},\cV)$ for the full subcategory spanned by the monoids.
	\end{definition}

	\begin{rem}\label{rem:unwinding_monoid}
		Pointwise, right Kan extensions are computed by limits, so in particular, evaluating a $\Simp_X^{\op}$-monoid $\mathsf{C}$ at some sequence $(x_0,...,x_n)$ gives
		\begin{align*}
			\mathsf{C}(x_0,...,x_n) \simeq \underset{\phi\colon [n] \rightarrow [1] \in \Simp^{\op}_{\interior}}{\lim} \mathsf{C}(\phi_!(x_0,...,x_1)) \simeq \mathsf{C}(x_0,x_1)\times \cC(x_1,x_2) \times ... \times \mathsf{C}(x_{n-1},x_n) \,.
		\end{align*}
	\end{rem}

	\begin{prop}[{\cite*[Propositions 3.5.3 and 3.5.4]{GH}}]
		\label{prop:Mon=Alg}
		If $\cV$ is an $\infty$-category with finite limits and $X$ a space, there exists a functor $\pi \colon \cV^{\times} \to \cV$ such that post-composition with $\pi$ induces an equivalence
		\begin{align*}
			\Alg_{\Simp^{\op}_{X}}(\cV^{\times}) \overset{\simeq}\to \Monoid_{\Simp^{\op}_{X}}(\cV) \,.
		\end{align*}
	\end{prop}
	\begin{ex}\label{ex:double_cat_from_algebra}
		Let $\mathsf{C}$ be a $\Cat_{\infty}$-category
		By \Cref{prop:Mon=Alg}, we can equivalently view $\mathsf{C}$ as a monoid $\Simp^{\op}_{\ob(\sfC)} \to \Cat_{\infty}$.
		By unstraightening, this gives a coCartesian fibration
		\begin{align*}
			\Unst(\mathsf{C}) \to \Simp^{\op}_{\ob(\sfC)} \to \Simp^{\op} \,.
		\end{align*}
		Using the monoid condition for $\sfC$, one can check that this fibration satisfies the Segal condition and therefore determines a double $\infty$-category. In fact, unstraightening determines a fully faithful functor $\Alg_{\Cat}(\Cat_{\infty}^{\times}) \to \Dbl_{\infty}$ whose image is spanned by the double $\infty$-categories $\bfC$ where $\bfC_{0}$ is a space.
	\end{ex}

	\subsection{Completeness for enriched categories}\label{subsec:enriched_categories}
	While we can think of the objects of $\Alg_{\Cat}(\cV)$ as $\cV$-enriched categories, we must perform a localization to obtain the homotopically correct $\infty$-category of $\cV$-categories.
	This is analogous to the localization from Segal spaces to complete Segal spaces, or equivalently, the localization at Dwyer--Kan equivalences. This type of localization makes sense when we work in a \emph{presentably monoidal} $\infty$-category $\cV$.
	\begin{definition}\label{def:presentably_monoidal}
		We say that a large monoidal $\infty$-category $\cV^{\otimes}$ is \emph{presentably monoidal} if the underlying $\infty$-category $\cV$ is presentable, and the monoidal product preserves colimits in each variable separately.
	\end{definition}
	Throughout this subsection, we fix a presentably monoidal $\infty$-category $\cV$.
	\begin{definition}\label{def:morphism_in_enriched_cat}
		Let $\sfC$ be a $\cV$-category, and $c,d\in \ob(\sfC)$ objects of $\sfC$. A \emph{morphism} from $c$ to $d$ is a map $f\colon I_{\cV} \to \sfC(c,d)$ in $\cV$.
	\end{definition}
	We can compose a pair of morphisms in $\sfC$ by forming the following composite in $\cV$
	\begin{align*}
		f\circ g \colon I_{\cV} \simeq I_{\cV} \otimes I_{\cV} \xrightarrow{g\otimes f} \sfC(c,d) \otimes \sfC(d,e) \xrightarrow{\circ} \sfC(c,e) \,.
	\end{align*}
	By definition, each object $c\in \ob(\sfC)$ comes equipped with a morphism $id_{c} \colon I_{\cV} \to \sfC(c,c)$ which we may refer to as the identity on that object.

	\begin{definition}\label{def:equivalence_in_enriched_cat}
		Let $\sfC$ be a $\cV$-category, and $f$ a morphism between objects $c$ and $d$.
		We say that $f$ is an \emph{equivalence} if there exists a morphism $g$ from $d$ to $c$ such that
		\begin{align*}
			[f\circ g] &= [id_{d}] \in \pi_0\Map_{\cV}(I_{\cV}, \sfC(d,d)) \\
			[g\circ f] &= [id_{c}] \in \pi_{0}\Map_{\cV}(I_{\cV}, \sfC(c,c)) \,.
		\end{align*}
	\end{definition}

	\begin{definition}\label{def:ess_surj_enriched_functor}
		An enriched functor $F\colon \sfC \to \sfD$  is \emph{essentially surjective} if for every $d\in \ob(D)$, there exists $c\in \ob(C)$ and an equivalence from $F(c)$ to $d$.
		We say that $F$ is \emph{fully faithful} if for every pair of objects $c,d \in \ob{\sfC}$, the morphism
		\begin{align*}
			F_{cd} \colon \sfC(c,d) \to \sfD(F(c), F(d))
		\end{align*}
		is an equivalence in $\cV$.
	\end{definition}

	Because $\cV$ is presentable, the inclusion of $I_{\cV}$ extends uniquely to an adjunction
	\begin{align*}
		F : \cS \rightleftarrows \cV : U \,.
	\end{align*}
	Because the tensor product in $\cV$ preserves colimits in each variable, $F$ lifts to a monoidal functor, giving a natural lax monoidal structure on $U$.
  By \cite*[Example 4.3.20]{GH}, this gives rise to an adjunction
	\begin{align*}
		F_* : \Alg_{\Cat}(\cS) \rightleftarrows \Alg_{\Cat}(\cV) : U_{*} \,.
	\end{align*}
	Under the equivalence in \Cref{prop:Mon=Alg}, an object $\sfC\in \Alg_{\Cat}(\cS)$ gives rise to a simplicial space.
  By \cite*[Theorem 4.4.7]{GH}, this construction determines an equivalence
	\begin{align}\label{eq:alg(S)-straightening}
		\Alg_{\Cat}(\cS) \simeq \Seg \subset s\cS \,.
	\end{align}
	\begin{definition} \label{def:complete_algebra}
		We say that a $\cV$-category $\sfC$ is \emph{complete} if the Segal space corresponding to $U_{*}\sfC$ is complete.
		Write
		\begin{align}\label{eq:iota_V}
			\iota_{\cV}\colon \Cat^{\cV}_{\infty} \hookrightarrow \Alg_{\Cat}(\cV)
		\end{align}
		for the inclusion of the subcategory spanned by complete $\cV$-categories.
	\end{definition}
	\begin{rem}\label{rem:completeness}
		By \cite*[Proposition 5.1.11]{GH}, the above definition of completeness is equivalent to \cite*[Definition 5.2.2]{GH}.
	\end{rem}
	\begin{rem}\label{rem:S-categories_are_infty_categories}
		Under the unstraightening functor \eqref{eq:alg(S)-straightening}, the subcategory of complete $\cS$-categories is equivalent to the $\infty$-category $\CSS\subset s\cS$ of complete Segal spaces.
    This in turn is equivalent to $\Cat_{\infty}$ via the nerve functor $N\colon \Cat_{\infty}\to s\cS$ defined by $N(\cC)_{n}= \Map([n],\cC)$.
    This means that $\Cat_{\infty}^{\cS}\simeq \Cat_{\infty}$, and that every complete $\cV$-category $\sfC$ has an underlying $\infty$-category $U_{*}\sfC$.
		The mapping spaces in $U_{*}\sfC$ are precisely the spaces
		\begin{align*}
			\sfU_{*}\sfC(c,d) \simeq \Map_{\cV}(I_{\cV}, \sfC(c,d))
		\end{align*}
		of morphisms in $\sfC$.
	\end{rem}
	\begin{rem}\label{rem:2-cats_as_enriched}
		Under the unstraightening functor of \Cref{ex:double_cat_from_algebra}, the functor $U_{*}$ corresponds to the horizontal Segal space functor $h$ of \Cref{def:horizontal_morphisms}. The $\infty$-category of complete $\Cat_{\infty}$-categories is therefore equivalent to the subcategory $\Cat_{(\infty,2)}\subset \Dbl_{\infty}$ of $(\infty,2)$-categories.
	\end{rem}
	\begin{prop}[{\cite*[Corollary 5.5.6]{GH}}] \label{prop:CatV_localization}
		The inclusion \eqref{eq:iota_V} admits a left adjoint
		\begin{align*}
			\cL_{\cV}\colon \Alg_{\Cat}(\cV) \to \Cat^{\cV}_{\infty} \,,
		\end{align*}
		which is also the localization at fully faithful and essentially surjective functors.
	\end{prop}
	\begin{prop}[{\cite*[Proposition 5.7.14]{GH} and \cite{haugseng2023tensor}}] \label{prop:CatV_monoidal_structure}
		If $\cV$ admits a symmetric monoidal structure, there exists a presentable monoidal structure on $\Cat^{\cV}_{\infty}$ such that the localization $\cL_{\cV}$ is monoidal, and the inclusion $i_{\cV}$ is lax monoidal.
	\end{prop}
	\begin{prop}[{\cite*[Proposition 5.7.17]{GH}}] \label{prop:CatV_adjunctions}
		If $F\colon \cV^{\otimes} \to \cW^{\otimes}$ is a lax monoidal functor, the pushforward $F_{*}\colon \Alg_{\Cat}(\cV) \to \Alg_{\Cat}(\cW)$ preserves fully faithful and essentially surjective functors, and therefore descends to a functor $F_* \colon \Cat^{\cV}_{\infty} \to \Cat^{\cW}_{\infty}$.
    If~$F$ is strictly monoidal and the underlying functor admits a right adjoint, then $F$ has a lax monoidal right adjoint $G$, and $F_{*}$ is left adjoint to $G_{*}$.
	\end{prop}

	\subsection{Slice double categories}\label{subsec:slice_double_categories}

	The goal of this subsection is to construct the double $\infty$-category $\cV_{/\sfC}$ for $\cV$ a monoidal $\infty$-category and $\sfC$ a $\cV$-category. The double $\infty$-category $\cV_{/\sfC}$ should have the same space of objects as $\sfC$, and morphism $\infty$-categories
	\begin{align*}
		\cV_{/\sfC}(c,d) \simeq \cV_{/\sfC(c,d)} \,.
	\end{align*}
	We will prove some properties of this construction, and show that for a certain class of algebras, we can also get a double $\infty$-category $\cV_{\sfS/}$ whose morphism $\infty$-categories are the slices $\cV_{\sfS(c,d)/}$.
	\begin{definition}\label{def:exponential_double_cat}
		Let $\bfC$ be a double $\infty$-category and $\cJ$ an $\infty$-category.
		The \emph{exponential double $\infty$-category} $\bfC^{\cJ}$ is defined by the pullback
		\begin{equation*}
			\begin{tikzcd}
				\bfC^{\cJ} \arrow[r] \arrow[d] & \Fun(\cJ, \bfC)\arrow[d] \\
				\Simp^{\op} \arrow[r] & \Fun(\cJ, \Simp^{\op})
			\end{tikzcd}
		\end{equation*}
		in  $\Cat_{\infty}$.
	\end{definition}

	\begin{rem}\label{rem:C^J_is_double_cat}
		Because coCartesian fibrations are stable under forming functor $\infty$-categories and pullbacks, $\bfC^{\cJ}\to \Simp^{\op}$ is a coCartesian fibration.
		Furthermore, by commuting limits, it satisfies the Segal condition.
	\end{rem}

	\begin{rem}\label{rem:C*J_is_functorial}
		The construction of $\bfC^{\cJ}$ is clearly covariantly functorial in $\bfC$ and contravariantly functorial in $\cJ$.
	\end{rem}

	\begin{rem}\label{V^J_is_pointwise_monoidal}
		For a monoidal $\infty$-category $\cV$, the exponential $(\cV^{\otimes})^{\cJ}$ models the \emph{pointwise monoidal structure} on $\Fun(\cJ,\cV)$.
	\end{rem}

	\begin{lem}\label{lem:alg_in_exponential}
		Let $\bfC$ be a double $\infty$-category, $\cM$ a generalized non-symmetric $\infty$-operad, and $\cJ$ a small $\infty$-category.
    Then we have a natural equivalence of $\infty$-categories
		\begin{align*}
			\Alg_{\cM}\left(\bfC^{\cJ} \right) \simeq   \Fun(\cJ, \Alg_{\cM}(\bfC)) \,.
		\end{align*}
	\end{lem}

	\begin{proof}
		For a generalized non-symmetric $\infty$-operad $\cM$, consider the diagram
		\begin{equation*}
			\begin{tikzcd}
				\Fun_{/\Simp^{\op}}(\cM, \bfC^{\cJ}) \arrow[r] \arrow[d]      &    \Fun\left(\cM, \bfC^{\cJ} \right) \arrow[d] \arrow[r] & \Fun\left(\cM, \Fun\left(\cJ, \bfC \right) \right) \arrow[d] \\
				\ast \arrow[r]                & \Fun(\cM, \Simp^{\op} ) \arrow[r] & \Fun(\cM, \Fun(\cJ, \Simp^{\op})) \,.
			\end{tikzcd}
		\end{equation*}
		The left-hand square is a pullback by definition of $\Fun_{/\Simp^{\op}}$, and the right-hand square by definition of $\bfC^{\cJ}$.
		Using the tensor-hom adjunction, the outer square is then equivalent to the following pullback:
		\begin{equation*}
			\begin{tikzcd}
				\Fun\left(\cJ, \Fun_{/\Simp^{\op}}(\cM, \bfC) \right) \arrow[r] \arrow[d]      &  \Fun(\cJ, \Fun(\cM, \bfC))   \arrow[d] \\
				\ast \arrow[r]                & \Fun(\cJ, \Fun(\cM, \Simp^{\op})) \,.
			\end{tikzcd}
		\end{equation*}
		The coCartesian edges of $\Fun(\cJ, \bfC) \to \Fun(\cJ, \Simp^{\op})$ are precisely the natural transformations whose components at each object $J \in \cJ$ is coCartesian for $\bfC\to \Simp^{\op}$.
		Hence
		\begin{align*}
			\Fun\left(\cJ, \Fun_{/\Simp^{\op}}(\cM, \bfC) \right) \simeq \Fun_{/\Simp^{\op}}(\cM, \bfC^{\cJ} )
		\end{align*}
		restricts to an equivalence on the full subcategories of algebras.
	\end{proof}

	\begin{definition}\label{def:slice_double_cat}
		For a double $\infty$-category $\bfC$ and a $\bfC$-category $\sfA$, define $\bfC_{/ \sfA}$ by the following pullback in $\Cat_{\infty}$.
		\begin{equation*}\label{eq:slice_double_cat}
			\begin{tikzcd}
				\bfC_{/ \sfA} \arrow[r] \arrow[d]      &   \bfC^{[1]}  \arrow[d, "ev_{1}"] \\
				\Simp^{\op}_{\ob(\sfA)}  \arrow[r, "\sfC"]               &  \bfC \,.
			\end{tikzcd}
		\end{equation*}
	\end{definition}
	\begin{lem}\label{prop:overcategory}
		The composite $\bfC_{/\sfA} \to \Simp^{\op}_{\ob(\sfA)} \to \Simp^{\op}$ is a double $\infty$-category.
	\end{lem}
	\begin{proof}
		In the fiber over some $(c_{0},\dots c_{n}) \in \Simp^{\op}_{\bfC_{0}}$, the functor $ev_1$ is equivalent to the product
		\begin{align*}
			\prod ev_1 \colon \prod_{i} \Ar(\bfC(c_{i},c_{i+1})) \to  \prod_{i} \bfC(c_{i},c_{i+1})\,.
		\end{align*}
		Since each factor of this is a coCartesian fibration, so is the product. As in \Cref{ex:ev_0_is_cocart}, a morphism $f$ is $\prod ev_1$-coCartesian if and only if $\prod ev_{0}f$ is an equivalence. CoCartesian transport along the lift of a morphism $\phi \colon [n]\to [m]$ is given by a combination of projection and horizontal composition. Because these preserve equivalences in the source, $\phi_!$ is a morphism of coCartesian fibrations.
    \Cref{thm:iterated_fibration} then implies that the right vertical of \eqref{eq:slice_double_cat} is a coCartesian fibration, so by pullback stability, so is the left vertical.
    The Segal condition is stable under pullback, so $\bfC_{/\sfA}$ is a double $\infty$-category.
	\end{proof}

	\begin{rem}\label{rem:slice_is_functorial}
		Again it is clear that $\sfC_{/\sfA}$ is covariantly functorial both in $\bfC$ and $\sfA$.
	\end{rem}

	\begin{rem}\label{rem:unwinding_slice}
		By restricting the pullback square \eqref{eq:slice_double_cat} to the fiber over $[0]$, we see that the objects of $\cV_{/ \sfC}$ are precisely $\ob(\sfC)$.
    By pulling back to the fiber over a tuple $(c,d) \in \Simp^{\op}_{\ob(\sfC)}$, we see that
		\begin{align*}
			\bfC_{/\sfA}(x,y) \simeq \bfC(c_{x},c_{y})_{/\sfA(c,d)} \,.
		\end{align*}
		By chasing coCartesian lifts in \eqref{eq:slice_double_cat}, one can check that the horizontal composition of objects $a\to \sfA(x,y), b\to \sfA(y,z)$ is given by
		\begin{align*}
			a\odot b \to \sfA(x,y) \odot \sfA(y,z) \xrightarrow{\odot} \sfA(y,z) \,.
		\end{align*}
	\end{rem}

	\begin{rem}\label{rem:monoidal_slice}
		We shall mostly be concerned with the slice when $\cV^{\otimes}$ is a monoidal $\infty$-category.
    Then $\cV_{\sfC/}$ is a double $\infty$-category with a space of objects $\ob(\sfC)$, and morphism $\infty$-categories $\cV_{/ \sfC(c,d)}$.
	\end{rem}

	\begin{ex}
		In particular, since $\Cat_\infty \simeq \Cat_\infty^{\cS}$, any $\infty$-category $\cC$ determines a double $\infty$-category $\cS_{/\cC}$, whose objects are $\cC^{\simeq}$ and whose mapping $\infty$-categories are $\cS_{/\Map_{\cC}(x,y)}$. Similarly, any $\Sp$-category $\sfC$ determines a double $\infty$-category $\Sp_{/\sfC}$.
	\end{ex}

	\begin{ex}\label{ex:2cat_from_monoid}
		Let $\sfB X \colon \Simp^{\op} \to \cS$ be a monoid in spaces with underlying space $\sfB X_{1} = X$. The corresponding (non-complete) $\cS$-category has contractible space of objects, so the slice double $\infty$-category $\cS_{/\sfB X}$ is a monoidal $\infty$-category whose underlying $\infty$-category is the slice $\cS_{/X}$. We also write $\cS_{/X}^{\otimes}$ for this monoidal $\infty$-category. We can think of the product in $\cS_{/X}$ as given by the formula
		\begin{align*}
			(S \to X) \otimes (S'\to X) = (S\times S' \to X\times X \to X) \,.
		\end{align*}
		Note that the construction of exponentials and slices for monoidal $\infty$-categories also makes sense when we replace $\Simp^{\op}$ with the $\bE_{n}$ operad, so if $X$ is an $\bE_{n}$-monoid in spaces, this construction gives rise to an $\bE_{n}$-monoidal structure on $\cS_{/X}$.
	\end{ex}

	We will now study the $\infty$-category of algebras in $\cV_{/\sfC}$. For this, it is useful to introduce the notion of right fibrations of double $\infty$-categories. Recall that slice projections $\cC_{/c} \to \cC$ are the prototypical right fibrations, so we would expect that $\cV_{/\sfC} \to \cV^{\otimes}$ is a right fibration of double $\infty$-categories.

	\begin{definition}\label{def:double_right_fib}
		We say that a double functor $F\colon \bfC \to \bfD$ is a \emph{left (resp. right) fibration of double $\infty$-categories} if the following square is a pullback.
		\begin{equation}\label{eq:double_right_fib_pullback}
			\begin{tikzcd}
				\bfC^{[1]} \arrow[r] \arrow[d,"ev_i"]      &  \bfD^{[1]}   \arrow[d,"ev_{i}"] \\
				\bfC \arrow[r]                & \bfD
			\end{tikzcd}
		\end{equation}
		for $i=0$ (resp. $i=1$).
	\end{definition}

	\begin{lem}\label{lem:double_right_fib_condition}
		Let $F\colon \bfC\to \bfD$ be a double functor such that $\bfC_{0}$ and $\bfD_{0}$ are spaces, and for each pair of objects $x,y\in \bfC_{0}$, the functor
		\begin{align*}
			\bfC(x,y) \to \bfD(Fx,Fy)
		\end{align*}
		is a left (resp. right) fibration, then $F$ is a left (resp. right) fibration of double $\infty$-categories.
	\end{lem}
	\begin{proof}
		We treat the case of right fibrations, as the case for left fibrations is dual.
		The diagram
		\begin{equation*}
			\begin{tikzcd}
				\bfC_{n} \arrow[r, "F_{n}"] \arrow[d] & \bfD_{n} \arrow[d] \\
				(\bfC_{0})^{n+1} \arrow[r] & (\bfD_{0})^{n+1}
			\end{tikzcd}
		\end{equation*}
		induces a factorization of $F_{n}$ as
		\begin{align*}
			\bfC_{0} \xrightarrow{F'_{n}} \bfD_{n} \times_{(\bfD_{0})^{n+1}} (\bfC_{0})^{n+1} \xrightarrow{F''_{n}} \bfD_{n} \,.
		\end{align*}
		The functor $F''_{n}$ is a pullback of a map of spaces, and therefore a right fibration.
		Any functor whose target is a space is a Cartesian fibration whose Cartesian edges are the equivalences.
		Therefore $F'_{n}$ is a morphism of Cartesian fibrations over $(\bfC_{0})^{n+1}$.
		The induced morphism on fibers over any $(C_{0},\dots, C_n)$ is equivalent to the product of the right fibrations $\bfC(C_{i}, C_{i+1}) \to \bfD(FC_{i}, FC_{i+1})$, so by \Cref{thm:iterated_fibration}, $F'_{n}$ is a right fibration.
		Now the condition that \eqref{eq:double_right_fib_pullback} is a pullback can be checked fiberwise over $\Simp^{\op}$, where it reduces to the condition that each $F_{n}$ is a right fibration.
	\end{proof}

	\begin{prop}\label{prop:algebra_right_fib}
		If $F\colon \bfC \to \bfD$ is a right fibration of double $\infty$-categories, composition with $F$ determines a right fibration
		\begin{align*}
			F_{*}\colon \Alg(\bfC) \to \Alg(\bfD) \,.
		\end{align*}
	\end{prop}
	\begin{proof}
		$F_{*}$ is a morphism of cartesian fibrations over $\Opd$, so by \Cref{thm:iterated_fibration} it suffices to show that $F_{*}$ induces a right fibration in the fiber over some generalized non-symmetric $\infty$-operad, i.e., to show that the following square is a pullback
		\begin{equation*}
			\begin{tikzcd}
				\Ar(\Alg_{\cM}(\bfC)) \arrow[r] \arrow[d]      &  \Ar(\Alg_{\cM}(\bfD))    \arrow[d] \\
				\Alg_{\cM}(\bfC) \arrow[r]                &\Alg_{\cM}(\bfD) \, ,
			\end{tikzcd}
		\end{equation*}
		By \Cref{lem:alg_in_exponential}, this square is equivalent to
		\begin{equation*}
			\begin{tikzcd}
				\Alg_{\cM}(\bfC^{[1]})  \arrow[r] \arrow[d]      &    \Alg_{\cM}(\bfD^{[1]})  \arrow[d] \\
				\Alg_{\cM}(\bfC)  \arrow[r]                & \Alg_{\cM}(\bfD) \,.
			\end{tikzcd}
		\end{equation*}
		which is a pullback since $F$ is a right fibraiton of double $\infty$-categories and because $\Alg_{\cM}$ preserves limits.
	\end{proof}

	\begin{cor}\label{cor:alg_in_slice}
		Let $\bfC$ be a double $\infty$-category such that $\bfC_{0}$ is a space, and $\sfA$ a $\bfC$-category. The projection $\bfC_{/ \sfA} \to \bfC$ is a right fibration of double $\infty$-categories, and we have a natural equivalence
		\begin{align*}
			\Alg(\bfC_{/\sfA}) \simeq \Alg(\bfC)_{/\sfA}
		\end{align*}
    of right fibrations	over $\Alg(\bfC)$.
	\end{cor}
	\begin{proof}
		The double functor $\bfC_{/ \sfA} \to \bfC$ satisfies the condition of \Cref{lem:double_right_fib_condition} and is therefore a right fibration of double $\infty$-categories.
    By \Cref{prop:algebra_right_fib}, the pushforward $\Alg(\bfC_{/\sfA})\to \Alg(\bfC)$ is a right fibration.
    Now note that the source of this fibration has a terminal object, namely the $\Simp^{\op}_{\ob(\sfA)}$-algebra corresponding to the identity on $\sfA$.
    A right fibration with a terminal object is canonically equivalent to the representable right fibration at the image of the terminal object.
	\end{proof}

	\begin{definition}\label{def:invertible_alg}
		Let $\bfC$ be a double $\infty$-category.
    We say that a $\bfC$-category $\sfS \colon \Simp^{\op}_{\ob(\sfS)} \to \bfC$ is \emph{invertible} if it is also a double functor.
	\end{definition}

	\begin{rem}\label{rem:unwinding_invertible_alg}
		Let $\sfS$ be an invertible $\bfC$-category.
		If we consider coCartesian lifts of an active morphism $\alpha_2 \colon (x,y,z) \to (x,z)$ in $\Simp^{\op}_{\ob(\sfS)}$, the invertibility condition tells us that
		\begin{align*}
			\left(\sfS(x,y), \sfS(y,z) \right) \xrightarrow{\sfS(\alpha_2)} \sfS(x,z)
		\end{align*}
		is a coCartesian lift of $\alpha_{2}\colon [2] \to [1]$; in other words, it computes the tensor product.
		This happens if and only if the composition map
		\begin{align*}
			\sfS(x,y)\odot \sfS(y,z) \to \sfS(x,z)
		\end{align*}
		is an equivalence.
		Similarly, all the unit maps $I_{\cV}\to \sfS(x,x)$ must be equivalences.
		This shows that an algebra is invertible if and only if it factors through the underlying left fibration $h\bfC \to \bfC$.
		Moreover, we have
		\begin{align*}
			\sfS(x,y) \odot \sfS(y,x) \xrightarrow{\sim} \sfS(x,x) \simeq I_{\cV} \\
			\sfS(y,x) \odot \sfS(x,y) \xrightarrow{\sim} \sfS(y,y) \simeq I_{\cV} \, ,
		\end{align*}
		so every morphism object of $\sfS$ must represent an invertible morphism in $h\bfC$.
    In fact, $\sfS$ is uniquely determined by a map of spaces $\ob(\sfS) \to h\bfC^{\simeq}$ into the classifying space of equivalences in $h\bfC$.
    If we apply this at $\bfC= \cV^{\otimes}$, we get that an invertible algebra in $\cV$ necessarily factors through the inclusion $\Pic(\cV)^{\otimes}\to \cV^{\otimes}$ of the \emph{Picard space} of $\cV$.
	\end{rem}

	\begin{definition}\label{def:under_slice_double_cat}
		Let $\sfS$ be an invertible categorical algebra in $\bfC$.
    Then define $\bfC_{\sfS/}$ by the following pullback in $\Dbl$:
		\begin{equation*}
			\begin{tikzcd}
				\bfC_{\sfS/} \arrow[r] \arrow[d] &  \bfC^{[1]} \arrow[d, "d_{1}"]  \\
				\Simp_{\ob(\sfS)}^{\op} \arrow[r, "\sfS"] & \bfC \,.
			\end{tikzcd}
		\end{equation*}

	\end{definition}
	\begin{rem}\label{rem:unwinding_under_slice}
		By \Cref{lem:Mon->Dbl->Cat}, the pullback defining $\bfC_{\sfS/}$ can be computed in $\Cat_{\infty}$. Restricting to the fiber over $[0]$, we can therefore see that the objects of $\bfC_{\sfS/}$ are precisely $\ob(\sfS)$.
    By pulling back to the fiber over $(x,y)\in \Simp^{\op}_{\ob(\sfS)}$, we can identify the mapping $\infty$-category $\bfC_{\sfS/}(x,y)$ with the pullback
		\begin{equation*}
			\begin{tikzcd}
				\bfC(c_{x},c_{y})_{\sfS(x,y)/} \arrow[r] \arrow[d]      &   \Ar( \bfC(c_{x},c_{y}) )  \arrow[d] \\
				{[0]}  \arrow[r, "{\sfS(x,y)}"]               &   \bfC(c_{x},x_{y}) \, .
			\end{tikzcd}
		\end{equation*}
	\end{rem}
	\begin{ex}\label{ex:unit_algebra}
		The unit $I_{\cV}$ in a monoidal $\infty$-category $\cV^{\otimes}$ always defines an invertible algebra.
		In fact, this is the only possible invertible algebra with a single object.
		Hence, we can always form the monoidal $\infty$-category $\cV_{I_{\cV}/}^{\otimes}$ of objects under the unit.
		This is also the natural monoidal structure on $\Alg_{\bE_{0}}(\cV)$.
	\end{ex}

	\begin{definition}\label{def:double_slices}
		Let $\sfC$ and $\sfS$ be $\cV$-categories, and assume that $\sfS$ is invertible.
		Let $\cV_{\sfS/ /\sfC}$ and  $\Map_{\cV}\left(\sfS, \sfC \right)$ be defined by the pullbacks
		\begin{equation}\label{eq:double_slices}
			\begin{tikzcd}
				\cV_{\sfS/ /\sfC} \arrow[r] \arrow[d] & \left(\cV^{\otimes} \right)^{[2]} \arrow[d, "(ev_0 {, } ev_{2})" ]\\
				\Simp^{\op}_{\ob(\sfS)\times \ob(\sfC)} \arrow[r, "\sfS\times \sfC"] & \cV^{\otimes} \times \cV^{\otimes}
			\end{tikzcd}
			\quad \text{and} \quad
			\begin{tikzcd}
				\Map_{\cV}\left(\sfS,\sfC \right) \arrow[r] \arrow[d] & \left(\cV^{\otimes} \right)^{[1]} \arrow[d, "(ev_0 {, } ev_1)"] \\
				\Simp^{\op}_{\ob(\sfS)\times \ob(\sfC)} \arrow[r, "\sfS\times\sfC"] & \cV^{\otimes}\times\cV^{\otimes}
			\end{tikzcd}
		\end{equation}
		in $\Cat_\infty$, respectively.
	\end{definition}

	\begin{lem}\label{lem:double_slices}
		The functors $\cV_{\sfS / / \sfC}\to \Simp^{\op}$ and $\Map_{\cV}(\sfS,\sfC) \to \Simp^{\op}$ exhibit double $\infty$-categories.
		The latter is also a left fibration, i.e. the unstraightening of a Segal space.
		The functor
		\begin{align*}
			\Comp\colon \cV_{\sfS / / \sfC} \to \Map_{\cV}(\sfS, \sfC)
		\end{align*}
		induced by $d_1\colon \left(\cV^{\otimes} \right)^{[2]} \to \left(\cV^{\otimes} \right)^{[1]}$ is a double functor.
	\end{lem}
	\begin{proof}
		The Segal condition is preserved when we take pullbacks, so it suffices to show that the left vertical maps in the squares \eqref{eq:double_slices} are coCartesian fibrations.
		Consider the following diagram, where all squares are pullbacks:
		\begin{equation}\label{eq:V_S//C-diagram}
			\begin{tikzcd}
				\cV_{\sfS/ /\sfC} \arrow[r] \arrow[d] &
				\cW \arrow[r] \arrow[d] &
				\left(\cV^{\otimes} \right)^{[2]} \arrow[d, "(ev_0 {, } ev_{2})" ]\\
				\Simp^{\op}_{\ob(\sfS)\times \ob(\sfC)}
				\arrow[r, "id\times \sfC"] \arrow[d] &
				\Simp^{\op}_{\ob(\sfS)} \times_{\Simp^{\op}} \cV^{\otimes}
				\arrow[r, "\sfS \times id"] \arrow[d] &
				\cV^{\otimes} \times_{\Simp^{\op}} \cV^{\otimes}\\
				\Simp^{\op}_{\ob(\sfC)} \arrow[r, "\sfC"] &
				\cV^{\otimes} \,. &
			\end{tikzcd}
		\end{equation}
		Because $\sfS$ is invertible, it determines a morphism in $\CoCart(\Simp^{\op})$. The product $\sfS\times id$ is therefore a morphism in $\CoCart(\cV^{\otimes})$. Because $\mathrm{ev}_2$ is a coCartesian fibration, the rightmost pullback can be computed in $\CoCart(\cV^{\otimes})$, which is closed under limits in $\Cat_{\infty /\cV^{\otimes}}$.
		Hence, the map $\cW \to \cV^{\otimes}$ is a coCartesian fibration, and therefore so is its pullback $\cV_{\sfS //\sfC} \to \Simp^{\op}_{\ob(\sfC)}$.
		A similar argument works for $\Map_{\cV}(\sfS, \sfC)$ replacing $\left(\cV^{\otimes} \right)^{[2]}$ with $\left(\cV^{\otimes} \right)^{[1]}$. Combining these arguments, we can moreover see that $\Comp$ is a double functor because $d_1$ lifts to a morphism in $\CoCart(\cV^{\otimes})$.
	\end{proof}

	\begin{rem}\label{rem:unwinding_double_slices}
		By restricting to the fiber over $[0]$, we see that the double $\infty$-categories $\cV_{\sfS/ /\sfC}$ and $\Map_{\cV}(\sfS, \sfC)$ both have objects $\ob(\sfS) \times \ob(\sfC)$.
		By restricting to the fiber over some $\left( (s_0,c_0), (s_1,c_1) \right)$ we see that the mapping $\infty$-categories are given as
		\begin{align*}
			\cV_{\sfS/ /\sfC}\left( (s_0,c_0), (s_1, c_1) \right) &= \cV_{\sfS(s_0,s_1)/ / \sfC(c_0,c_1)} \\
			\Map_{\cV}(\sfS,\sfC)\left( (s_0,c_0), (s_1, c_1) \right) &= \Map_{\cV}\left(\sfS(s_0,s_1), \sfC(c_0,c_1) \right) \,,
		\end{align*}
		respectively.
		On these mapping $\infty$-categories, the pullback along $d_1$ gives the composition functor
		\begin{align*}
			\left( \sfS(s_0, s_1) \xrightarrow{f} V \xrightarrow{g} \sfC(c_0,c_1) \right) \mapsto \left(  \sfS(s_0,s_1) \xrightarrow{gf} \sfC(c_0,c_1) \right)
		\end{align*}
		which justifies the notation $\Comp$.
	\end{rem}

	\begin{ex}\label{ex:underlying_from_slice}
		Let $\sfC$ be a categorical algebra in $\cV$, and write $I_{\cV}$ for the invertible algebra determined by the unit. The Segal space $\Map_{\cV}(I_{\cV}, \sfC)$ agrees with the underlying Segal space $U_{*}\sfC$.
	\end{ex}

	\begin{lem}\label{lem:double_slice_pullback}
		Restriction along $d_0$ and $d_2$ determines a pullback
		\begin{equation*}
			\begin{tikzcd}
				\cV_{\sfS / /\sfC} \arrow[r] \arrow[d] & \cV_{\sfS/} \arrow[d] \\
				\cV_{/ \sfC}  \arrow[r]&  \cV
			\end{tikzcd}
		\end{equation*}
     of double $\infty$-categories.
	\end{lem}

	\begin{proof}
		Use the Segal relation $[2] = [1]\coprod_{[0]}[1]$ to obtain
		\begin{align*}
			\left(\cV^{\otimes} \right)^{[2]} \simeq \left(\cV^{\otimes} \right)^{[1]} \times_{\cV^{\otimes}} \left(\cV^{\otimes} \right)^{[1]}.
		\end{align*}
		Under this identification, we see that the $\cW$ of \eqref{eq:V_S//C-diagram} is actually equivalent to $\cV_{\sfS/} \times_{\cV^{\otimes}} \left(\cV^{\otimes}\right)^{[1]}$, and so commuting pullbacks in the top left square of \eqref{eq:V_S//C-diagram} we get the result.
	\end{proof}

	\subsection{Monoids and delooping}\label{subsec:monoids_and_delooping}
	For this subsection, we fix a grouplike $\bE_{\infty}$-monoid in spaces $X$.
  The goal is to nail down the precise relationship between the $\infty$-categories
  \[\Cat_{\infty}^{\cS_{/X}} \quad \text{and} \quad \Cat_{\infty /BX}
  \]
  where we write $BX$ for the realization of the monoid $\sfB X\colon \Simp^{\op} \to \cS$ presenting the $\bE_{1}$-structure of $X$.
  Starting with a functor $f\colon \cC \to BX$, every morphism space $\Map_{\cC}(c,d)$ has a map to some path space in $BX$.
  Each such path space is non-canonically equivalent to $\Omega BX \simeq X$, so up to some indeterminacy, we expect to be able to equip each $\Map_{\cC}(c,d)$ with a map to $X$.
  To get rid of the indeterminacy, we have to add the data of a path $\gamma_{c}$ from $F(c)$ to the basepoint in $BX$.
  This amounts to replacing the space of objects $\cC^{\simeq}$ with $E=\fib(\cC^{\simeq}\to BX)$.
  We now expect to construct a $\cS_{/X}$-category $\sfC$ with space of objects $E$, and mapping objects
	\begin{align*}
		\sfC(c,d) = \cC(\pi c, \pi d) \to \Omega BX \simeq X \,.
	\end{align*}
    Because every morphism in $\sfB X$ is an equivalence, the completion of the Segal space $\sfB X$ is the constant simplicial space $\underline{BX}$.
    We will show that $\sfC$ can be obtained by base-changing $\cC \to BX$ along the completion map $\sfB X \to \underline{BX}$.
    Moreover, we will show that this construction determines a fully faithful functor
	\begin{align}\label{eq:Cat_/BX->Cat^S_/X}
		\Cat_{\infty / BX} \to \Cat_{\infty}^{\cS_{/X}}.
	\end{align}
	Note that we do not expect \eqref{eq:Cat_/BX->Cat^S_/X} to be essentially surjective, because the space of objects~$E^{\simeq}$ has an $X$-action given by concatenating with loops in $BX$.
  Moreover, we expect this action to satisfy
	\begin{align}\label{eq:EX-action}
		\sfC(c,d\odot x) \simeq \sfC(c,d) \otimes (-x) \,.
	\end{align}
	In fact, we will identify the image of \eqref{eq:Cat_/BX->Cat^S_/X} with the full subcategory of $\cS_{/X}$-categories with such action.

	As in \Cref{ex:2cat_from_monoid}, we get a symmetric monoidal $\infty$-category $\cS_{/X}^{\otimes}$ whose underlying $\infty$-category is the slice $\cS_{/X}$.
  By \Cref{cor:alg_in_slice}, we have an equivalence
	\begin{align*}
		\Alg(\cS_{/X}) \simeq \Alg(\cS)_{/\sfB X}\,.
	\end{align*}
	Pulling this back along $\Simp^{\op}_{(-)}\colon \cS \to \Opd$ gives an equivalence
	\begin{align*}
		\Alg_{\Cat}(\cS_{/X}) \simeq \Alg_{\Cat}(\cS)_{/\sfB X}.
	\end{align*}
	The monoidal left adjoint determined by the unit gives rise to an adjunction
	\begin{align*}
		F_{*} :  \Alg_{\Cat}(\cS) \rightleftarrows \Alg_{\Cat}(\cS_{/X}) \simeq \Alg_{\Cat}(\cS)_{/\sfB X} : U_{*} \,.
	\end{align*}
	The functor $F_{*}$ is a morphism between representable right fibrations over $\Alg_{\Cat}(\cS)$, which takes the terminal object $\Simp^{\op}$ to the morphism $e\colon \Simp^{\op} \to \sfB X$ determined by the unit in $X$.
  By the Yoneda lemma, $F_{*}$ is therefore equivalent to the pushforward functor~$e_{*}$ on slices. The right adjoint $U_{*}$ is therefore equivalent to the base change $e^{*}$.
	In particular, $e^{*}\sfB X = \Simp^{\op}$, which is a complete $\cS$-category.
  The identity on $\sfB X$ therefore determines a complete $\cS_{/X}$-category. Note that $\sfB X$ is far from being complete in $\cS$.
	Let $\pi\colon \cS_{/X} \to \cS$ denote the forgetful functor. This lifts to a symmetric monoidal functor, and arguing as above using the Yoneda lemma, the induced left adjoint
	\begin{align*}
		\pi_{*} \colon \Alg(\cS_{/X}) \simeq \Alg(\cS)_{/\sfB X} \to \Alg(\cS)
	\end{align*}
	agrees with the slice projection.
  By the above, the identity on $\sfB X$ determines a terminal object in $\Cat_{\infty}^{\cS_{/X}}$, and its image under the pushforward functor
	\begin{align*}
		\pi_{*}\colon \Cat^{\cS_{/X}}_{\infty} \to \Cat_{\infty}
	\end{align*}
	is the space $BX$. Therefore, $\pi_{*}$ factors through a functor
	\begin{align*}
		\cL \colon \Cat^{\cS_{X/}}_{\infty} \to \Cat_{\infty /BX} \,.
	\end{align*}
	Because the slice projection $\Cat_{\infty /BX} \to \Cat_{\infty}$ preserves and detects colimits, $\cL$ must preserve colimits.

	Recall that we write $N\colon \Cat_{\infty} \to \Seg \simeq \Alg_{\Cat}(\cS)$ for the nerve functor. For the space~$BX$, the nerve $N(BX)$ agrees with the constant simplicial space $\underline{BX}$.

	\begin{lem}\label{lem:monoid_right_adjoint_1}
		The right adjoint $\cR$ of $\cL$ is fully faithful, and for every $\cC \to BX$ in $\Cat_{\infty /BX}$, there is a pullback in $\Alg_{\Cat}(\cS)$ of the form
		\begin{equation}\label{eq:monoid_right_adjoint_1}
			\begin{tikzcd}
				\cR(\cC) \arrow[r] \arrow[d]      &    N(\cC)  \arrow[d] \\
				\sfB X  \arrow[r]               &  \underline{BX} \, .
			\end{tikzcd}
		\end{equation}
	\end{lem}
	\begin{proof}
		Let $\cC\to BX$ be a functor, and consider the following diagram in $\Alg_{\Cat}(\cS)$ where both squares are pullbacks:
		\begin{equation*}
			\begin{tikzcd}
				e^{*}\sfC \arrow[r] \arrow[d]      &    \sfC \arrow [r] \arrow[d] & N(\cC)  \arrow[d] \\
				\Simp^{\op} \arrow[r, "e"]   &      \sfB X \arrow[r]       &  \underline{BX} \,.
			\end{tikzcd}
		\end{equation*}
	    Because $N(\cC)$, $\underline{BX}$ and $\Simp^{\op}$ are complete and because $\Cat_{\infty}^{\cS}$ is closed under limits in $\Alg_{\Cat}(\cS)$, $\sfC\to \sfB X$ determines a complete $\cS_{/X}$-category. Now note that for any other $f\colon \sfD \to \sfB X$, we have a diagram of mapping spaces
		\begin{equation*}
			\begin{tikzcd}
				\Map_{/\sfB X}(\sfD, \sfC) \arrow[r] \arrow[d]      &   \Map(\sfD, \sfC)  \arrow[d]  \arrow[r] & \Map(\sfD, \sfC) \arrow[d] \arrow[r, "\simeq"] & \Map(\cL_{\cS}(\sfC), N(\cC) ) \arrow[d] \\
				\ast   \arrow[r, "\left\{ f\right\}"]               &   \Map(\sfD, \sfB X)  \arrow[r]            & \Map(\sfD, \underline{BX}) \arrow[r, "\simeq"] & \Map(\cL_{\cS}(\sfD),\underline{BX})
			\end{tikzcd}
		\end{equation*}
		where each square is a pullback: the left by definition of $\Map_{/\sfB X}$, the middle because $\sfC$ is a pullback, and the right because $N(\cC)$ and $\underline{BX}$ are complete, so both the horizontal maps are equivalences. This gives rise to equivalences
		\begin{align*}
			\Map_{/\sfB X}(\sfD, \sfC) \simeq \Map_{/\underline{BX}}(\cL_{\cS}(\sfD), N(\cC))
		\end{align*}
		which are clearly natural in $\calD$, giving $\sfC \simeq \cR(\cC)$.

		The counit of the adjunction is given on an object $\cC \to BX$ by applying $\cL_{\cS}$ to the pullback square \eqref{eq:monoid_right_adjoint_1}.
    It therefore suffices to show that the map $\cR(\cC) \to N(\cC)$ is fully faithful and essentially surjective.
    The mapping objects of a Segal space $\sfC$ are computed by the fibers of $(ev_{0},ev_{1})\colon X_{1}\to X_{0}\times X_{0}$, so by commuting limits, the base change of a fully faithful map is fully faithful.
    The map $\sfB X \to \underline{BX}$ is a localization, and therefore fully faithful, and therefore so is $\cR(\cC) \to N(\cC)$.
    Evaluating at $[0]$, the map $\sfB X \to \underline{BX}$ gives $\ast \to BX$, which is an effective epimorphism.
    Effective epimorphisms of spaces are stable under base-change, so $\cR(\cC)_{0} \to N(\cC)_{0}$ is also an effective epimorphism, which implies that $\cR(\cC) \to N(\cC)$ is essentially surjective.
	\end{proof}
	\begin{lem}\label{lem:monoid_right_adjoint_2}
		The right adjoint $\cR$ preserves colimits.
	\end{lem}
	\begin{proof}
		The composite
		\begin{align*}
			\cS \xrightarrow{e_{*}} \cS_{/X} \xrightarrow{\pi_{*}} \cS
		\end{align*}
		is equivalent to the identity, so the same holds for
		\begin{align*}
			\Cat_{\infty}\xrightarrow{e_{*}} \Cat_{\infty}^{\cS_{/X}} \xrightarrow{\pi_{*}} \Cat_{\infty} \,.
		\end{align*}
		Therefore, the composite
		\begin{align*}
			\Cat_{\infty} \xrightarrow{e_{*}} \Cat_{\infty}^{\cS_{/X}} \xrightarrow{\cL} \Cat_{\infty /BX}
		\end{align*}
		is a morphism of right fibrations over $\Cat_{\infty}$, which by Yoneda is uniquely determined by the value on the terminal object.
    This value corresponds to the inclusion of a point $x\in BX$.
    In summary, we have the following diagram of left adjoints:
		\begin{equation*}
			\begin{tikzcd}
				\Cat_{\infty} \arrow[r, "x_{*}"] \arrow[d, "e_{*}"']      &   \Cat_{\infty /BX}  \arrow[d] \\
				\Cat_{\infty}^{\cS_{/X}}  \arrow[r, "\pi_{*}"']    \arrow[ur, "\cL"]           &  \Cat_{\infty} \,.
			\end{tikzcd}
		\end{equation*}
		Because $\cR$ is fully faithful, we get the following commutative diagram by passing to right adjoints in the top left triangle:
		\begin{equation*}
			\begin{tikzcd}
				\Cat_{\infty} & \Cat_{\infty /BX} \arrow[l, "x^{*}"'] \arrow[dl, "\cR"] \arrow[d] \\
				\Cat^{\cS_{/X}}_{\infty} \arrow[u, "e^{*}"] \arrow[r, "\pi_{*}"'] & \Cat_{\infty} \,.
			\end{tikzcd}
		\end{equation*}
		Now we note that since $BX$ is a space, both the slice projection and the pullback $x^{*}$ preserve colimits.
    Since $e^{*}$ reflects essentially surjective functors and $\pi_{*}$ reflects fully faithful functors, any comparison map
		\begin{align*}
			\colim(\cR \circ F) \to \cR( \colim(F))
		\end{align*}
		must be both fully faithful and essentially surjective, and therefore an equivalence in $\Cat^{\cS_{/X}}_{\infty}$.
	\end{proof}

	\begin{lem}\label{lem:monoid_projection_formula}
		The left adjoint $\cL$ lifts to a symmetric monoidal functor such that the adjunction $\cL \dashv \cR$ satisfies the projection formula.
	\end{lem}
	\begin{proof}
		The projection $\pi_{*} \colon \cS_{/X} \to \cS$ lifts to a symmetric monoidal functor, and therefore so does $\pi_{*} \colon \Cat_{\infty}^{\cS_{/X}} \to \Cat_{\infty}$ by \Cref{prop:CatV_adjunctions}.
    The terminal object $\sfB X$ then gets a unique $\bE_{\infty}$-algebra structure, which induces an $\bE_{\infty}$-algebra structure on $\pi_{*}(\sfB X) = BX$.
    By naturality of the slice monoidal structure, $\cL$ then lifts to a symmetric monoidal functor.
    For the projection formula, we must show that the natural map
		\begin{align*}
			\cR(\cC) \otimes \sfD \to \cR(\cC \otimes \cL(\sfD))
		\end{align*}
		is an equivalence in $\Cat_{\infty}^{\cS_{/X}}$.
    By \Cref{prop:CatV_monoidal_structure}, the monoidal structure on $\Cat^{\cS_{/X}}_{\infty}$ is localized from the one on $\Alg_{\Cat}(\cS_{/X})$, so it suffices to show that the corresponding map in $\Alg_{\Cat}(\cS_{/X})$ is fully faithful and essentially surjective.
    To that end, consider the diagram
		\begin{equation*}
			\begin{tikzcd}
				\cR(\cC)\otimes \sfD \arrow[d] \arrow[r] & N(\cC) \otimes \sfD  \arrow[d] \\
				\cR(\cC \otimes \cL(\sfD)) \arrow[r] \arrow[d]      &   N(\cC) \otimes \cL(\sfD)  \arrow[d] \\
				\sfB X  \arrow[r]               &   \underline{BX}
			\end{tikzcd}
		\end{equation*}
		in $\Alg_{\Cat}(\cS)$ where the bottom square is the pullback defining $\cR$.
    The top horizontal as well as the top right vertical are products of fully faithful maps, and therefore fully faithful.
    By pullback stability, the middle horizontal is also fully faithful, and by 2-out-of-3, so is the top left vertical.
    Passing to spaces of objects in the above diagram gives
		\begin{equation*}
			\begin{tikzcd}
				\ob(\cR(\cC) )\times \ob(\calD) \arrow[r] \arrow[d]      &   \cC^{\simeq}\times \ob(\calD)  \arrow[d] \\
				\ob(\cR(\cC)\otimes \cL(\calD))  \arrow[r]  \arrow[d]             &   \cC^{\simeq}\times \ob(\cL(\calD))\arrow[d] \\
				\ast \arrow[r]     &  BX \,.
			\end{tikzcd}
		\end{equation*}
		By the pullback defining $\cR(\cC)$ and pasting, the outer rectangle and bottom squares are pullbacks. By pasting, the top square is a pullback.
    Now the upper right vertical is a product of effective epimorphisms, so by pullback stability, the upper left vertical is also an effective epimorphism.
	\end{proof}
	The above lemma shows that the adjunction $\cL \dashv \cR$ exhibits $\Cat_{\infty/BX}$ as equivalent to the $\infty$-category of modules over the algebra obtained by applying $\cR$ to the unit in $\Cat_{\infty /BX}$. This unit is the map $\ast \to BX$ corresponding to the basepoint, and by \Cref{lem:monoid_right_adjoint_1} we can compute $\cR(\ast)$ by the pullback
	\begin{equation*}
		\begin{tikzcd}
			\sfE X \arrow[r] \arrow[d]      &   \Simp^{\op}  \arrow[d] \\
			\sfB X  \arrow[r]               &   \underline{BX} \,.
		\end{tikzcd}
	\end{equation*}
	By applying $\ob$ to this square, we see that $\ob(\sfE X) \simeq \Omega BX \simeq X$.
  Because $\sfE X \to \Simp^{\op}$ is fully faithful, the morphism spaces $\sfE X(x,y)$ are contractible.
  If we let $e\in X$ denote the unit, the unital structure of $\sfE X$ as an algebra enforces $\sfE(e,e) \to X$ to select the unit in $X$.
  One can then show by functoriality of $\sfE X \to \sfB X$ that $\sfE(x,y) \to X$ corresponds to the point $y-x$ in $X$.

	\begin{cor}\label{cor:monoid_equivalence}
		The adjunction $\cL \dashv \cR$ factors as
		\begin{align*}
			\Cat^{\cS_{X/}}_{\infty} \rightleftarrows \Mod_{\sfE X} (\Cat^{\cS_{X/}}_{\infty}) \simeq \Cat_{\infty /BX} \,.
		\end{align*}
		The equivalence lifts to an equivalence of symmetric monoidal $\infty$-categories.
	\end{cor}

	\begin{proof}
		The adjunction satisfies the assumptions of \cite*[Proposition 5.29]{Mathew_2017}, so the result follows from the identification of $\cR(\ast) = \sfE X$ above.
	\end{proof}

	Note that because the right adjoint $\cR$ is fully faithful, the $\cS_{/X}$-category $\sfE X$ is necessarily an idempotent algebra, meaning that the multiplication map $\sfE X\otimes \sfE X \to \sfE X$ is an equivalence.
  This means that the $\infty$-category of modules over $\sfE X$ is the full subcategory spanned by objects for which the unit map gives an equivalence $\sfC \xrightarrow{\simeq} \sfC \otimes \sfE X$.
  Unwinding the definitions, an $\sfE X$-module structure on $\sfC$ is precisely an action of $X$ on the objects of $\sfC$, satisfying \eqref{eq:EX-action}.

	\subsection{Inner Kan spaces and companions}\label{subsec:inner_kan_spaces_and_companions}
	In this subsection, we recall the main results of \cite{oldervoll2026quasi} about how semi-simplicial spaces satisfying certain lifting conditions can be used to model $\infty$-categories.
  Moreover, we will show how the notion of \emph{companions} in a double $\infty$-category (in the sense of~\cite{ruit2025companion}) generalizes to this setting.

	Recall that we write $ss\cS=\Fun(\Simp_{s}^{\op},\cS)$ for the $\infty$-category of semi-simplicial spaces.
  We write $\Delta^{n}_{s}$ for the representable presheaf at $[n]$ in $ss\cS$.
  The colimit diagrams defining the simplicial sets $\Horn^{n}_{i}$ and $\del \Delta^{n}$ lift to diagrams of semi-simplicial spaces, and we denote the colimits of these as $\Horn^{n}_{i,s}$ and $\del \Delta^{n}_{s}$ respectively.
  These are mapped to their simplicial counterparts by the left Kan extension functor $\cL_{s}\colon ss\widehat{\cS} \to s\widehat{\cS}$.
	We denote the terminal semi-simplicial space by $T$; this is contractible in every degree.

	\begin{definition}\label{def:inner_fibration}
		A map of semi-simplicial spaces $F\colon X\to Y$ is an \emph{inner fibration} if for any $0<i<n$, given the solid part of the following diagram, a dashed lift exists making the diagram
		\begin{equation*}
			\begin{tikzcd}
				\Horn^{n}_{i,s} \arrow[r] \arrow[d]      &   X  \arrow[d, "F"] \\
				\Delta^{n}  \arrow[r] \arrow[ur, dashed]              &   Y
			\end{tikzcd}
		\end{equation*}
    commute.
		We say that a semi-simplicial space $X$ is an \emph{inner Kan space} if the unique map $X\to T$ is an inner fibration. We say that $F$ is a \emph{trivial fibration} if for any solid part of the diagram
		\begin{equation*}
			\begin{tikzcd}
				\del \Delta^{n}_{s} \arrow[r] \arrow[d]      &   X  \arrow[d, "F"] \\
				\Delta^{n}  \arrow[r] \arrow[ur, dashed]              &   Y
			\end{tikzcd}
		\end{equation*}
		there exists a dashed lift.
	\end{definition}
	Inner Kan spaces can be thought of as non-unital versions of quasi-categories. It has been observed in several different contexts that the structure of units in an $\infty$-category is actually a property of some underlying non-unital structure. This phenomenon is usually called \emph{quasi-unitality}, and has been studied in the quasi-category model by \cite{henry2018weak,steimle2018degeneracies,tanaka2018functors}, and in the Segal space model by \cite{harpaz2015quasi,haugseng2021segal}. In this paper, we will use quasi-unitality for inner Kan spaces following \cite{oldervoll2026quasi}.

	\begin{definition}\label{def:eq_in_simplicial}
		An edge $f$ in a semi-simplicial space $X$ is called an \emph{equivalence} if we can solve the lifting problems
		\begin{equation*}
			\begin{tikzcd}
				\Horn^{n}_{0,s} \arrow[r, "F"] \arrow[d, hookrightarrow] & X \\
				\Delta^{n}_{s} \arrow[ur, dotted]
			\end{tikzcd}
			\quad \text{and} \quad
			\begin{tikzcd}
				\Horn^{n}_{n,s} \arrow[r, "G"] \arrow[d, hookrightarrow]& X \\
				\Delta^{n}_{s} \arrow[ur, dotted]
			\end{tikzcd}
		\end{equation*}
		up to homotopy for $2\leq n$, whenever the first edge of $F$ (resp. the last edge of $G$) is $f$.
	\end{definition}

	\begin{definition}\label{def:idempotent}
		We say that an edge $e$ in a semi-simplicial space $X$ is \emph{idempotent} if there exists a 2-simplex $H$ with $d_i H =e$ for $i=0,1,2$.
	\end{definition}
	\begin{definition}\label{def:quasi-unital}
		We say that an inner Kan space $X$ is \emph{quasi-unital} if every vertex $x$ in $X$ is the source of an idempotent equivalence. We say that a morphism $X\to Y$ is \emph{quasi-unital} if it preserves idempotent equivalences. Let $\IK_{q.u.} \subset ss\cS$ denote the subcategory spanned by quasi-unital inner Kan spaces and quasi-unital morphisms.
	\end{definition}

	\begin{prop}[{\cite*[Theorem 1.3, Corollary 5.20]{oldervoll2026quasi}  }]\label{prop:semi-simplicial_results}
		There exists a functor
		\begin{align*}
			\cT\colon \IK_{qu} \to  \Cat_{\infty}
		\end{align*}
		with the following properties:
		\begin{enumerate}
			\item The semi-simplicial nerve functor $N\colon \Cat_{\infty} \to ss\cS$ factors through $\IK_{qu}$, and this determines a fully faithful right adjoint to $\cT$.
      The unit $X \to N\cT(X)$ is a trivial fibration.
			\item For any semi-simplicial set $J$ and quasi-unital inner Kan space $X$, a map $p\colon J \to X$ determines a diagram in $\cT(X)$.
      The slice $\infty$-category under this diagram is naturally equivalent to $\cT(X_{p/})$.
			\item The functor $\cT$ preserves pullbacks where at least one leg is an inner fibration, as well as limits of diagrams $F\colon \cI \to \IK_{qu}$ where $\cI$ is an inverse category, and each matching map
			\begin{align*}
				F(i) \to \lim_{j<i}(F(j))
			\end{align*}
			is an inner fibration.
			\item For an inner Kan space $X$ and two zero simplices $x,y\in X_{0}$, the mapping space $\Map_{\cT(X)}(x,y)$ can be computed as the realization of the Kan semi-simplicial space
			\begin{align*}
				\Hom^{\rR}(x,y)_{n} = \left\{ s_{n}x\right\} \times_{X_{n}} X_{n+1} \times_{X_{o}} \left\{ y\right\} \, ,
			\end{align*}
			where $\left\{ s_{n}x\right\}$ is any extension of $x$ to a map $T\to X$.
		\end{enumerate}
	\end{prop}

	To verify quasi-unitality for flow categories, we will use the following criterion inspired by \cite{AB}. The idea is that instead of constructing all the degeneracy maps $s_{i}\colon X_{n} \to X_{n+1}$, it suffices to construct $s_{0}$ and $s_{n}$, satisfying the expected simplicial identities.
  To get precise control over what this means when we work with semi-simplicial spaces rather than semi-simplicial sets, we will use the notion of free slices.

	\begin{definition}\label{def:augmented_semi-simplicial}
		Let $\Simp_{s,a}$ denote the category of potentially empty totally ordered sets, and order-preserving injections between them. We write $ss\cS_{a} = \Fun(\Simp_{s,a},\cS)$ for the $\infty$-category of \emph{augmented semi-simplicical spaces}.
	\end{definition}
	We denote the objects of $\Simp_{s,a}$ as $[n]$ for $n\geq -1$, with the convention $[-1] = \emptyset$. We have a full subcategory inclusion $ss\cS \subset ss\cS_{a}$ onto the presheaves $X$ such that $X_{-1}=\ast$.

	The category $\Simp_{a}$ carries a monoidal structure given by the join of posets.
  By Day convolution, this extends along the Yoneda embedding to a presentable monoidal structure on $ss\cS_{a}$.
  Explicitly, this is given by the formula
	\begin{align*}
		(X\join Y)_{n} = \coprod_{i+j = n-1} X_{i}\times X_{j} \,,
	\end{align*}
	so the full subcategory $ss\cS$ is closed under product, and contains the unit $\Delta^{0}_{s}$.
  For augmented semi-simplicial spaces $X$ and $J$, we can form the \emph{free slices} $X_{J/}$ and $X_{/J}$ uniquely determined by natural equivalences
	\begin{align*}
		\Map(K, X_{J/}) \simeq \Map(J\join K ,X) \simeq \Map(J, X_{/K}) \,.
	\end{align*}

	\begin{definition}\label{def:outer_degeneracies}
		We say that a semi-simplicial space $X$ \emph{admits initial and terminal degeneracies} if there exists a trivial fibration $\pi \colon \widetilde{X} \to X$ and commutative diagrams
		\begin{equation*}
			\begin{tikzcd}
				\widetilde{X}_{\Delta_{s}^{0}/} \arrow[r, "s_0"] \arrow[dr, "\pi"']      &   \widetilde{X}_{\Delta_{s}^{1}/}  \arrow[d, "d_i"] \\
				&   X_{\Delta_{s}^{0}/}
			\end{tikzcd}
			\qquad
			\begin{tikzcd}
				\widetilde{X}_{/\Delta_{s}^{0}} \arrow[r, "s_{\omega}"] \arrow[dr, "\pi"']      &   \widetilde{X}_{/\Delta_{s}^{1}}  \arrow[d, "d_i"] \\
				&   \widetilde{X}_{/\Delta_{s}^{0}},
			\end{tikzcd}
			\quad i=0,1,
		\end{equation*}
		in $ss\cS_{a}$.
	\end{definition}

	\begin{rem}\label{rem:unpacking_outer_degeneracies}
		After evaluating at $[n-1]\in \Simp_{s,a}$, the map $s_{0}$ corresponds (up to a trivial fibration) to $s_{0}\colon X_{n}\to X_{n+1}$.
    Being a map of augmented semi-simplicial spaces means that $s_{0}$ behaves as we would expect the zeroth degeneracy to behave with respect to $d_{i}, i\geq 2$, namely $d_{i}\circ s_{0} = s_{0}\circ d_{i-1}$.
    The diagram encodes the expected behavior of $s_{0}$ with respect to $d_{0}$ and $d_{1}$. The diagram for $s_{\omega}$ is symmetric, it encodes maps $s_{n}\colon X_{n}\to X_{n+1}$ which behave as the terminal degeneracies.
	\end{rem}

	\begin{prop}\label{prop:outer_degeneracies}
		If $X$ is an inner Kan space which admits initial and terminal degeneracies $s_0$ and $s_{\omega}$, then for every vertex $x\in X$, the edges $s_0x$ and  $s_{\omega} x$ are equivalences and~$X$ is quasi-unital.
	\end{prop}

	\begin{proof}
		This is slightly stronger than \cite*[Theorem 1.2]{oldervoll2026quasi}, which gives the result for $\pi=id$. With minor modifications, the proof \emph{loc.cit.} also goes through in the presence of a trivial fibration $\pi$.
	\end{proof}

	We now turn to semi-simplicial $\infty$-categories, i.e. functors $\Simp^{\op}_{s} \to \Cat_{\infty}$. Such objects satisfying the Segal condition can be thought of as a non-unital analogue of double $\infty$-categories, and these have their own theory of quasi-unitality as explained in \cite{haugseng2021segal}.
  Rather than requiring the Segal condition, we will consider objects $\bfC$ whose underlying semi-simplicial space $h\bfC$ is a quasi-unital inner Kan space. Generalizing the notion of \emph{companions} in a double $\infty$-category \cite*[Definition 4.1]{ruit2025companion}, we will give conditions for when a morphism in $\bfC_{0}$ admits a companion morphism in $\cT(h\bfC)$.

	\begin{definition}\label{def:admits_companions}
		Let $\bfC$ be a semi-simplicial $\infty$-category. We say that $\bfC$ \emph{admits companions} if $h\bfC$ is a quasi-unital inner Kan space, and if for each injective map $\phi \colon [n]\to [m]$, the functors
		\begin{align*}
			(v_{0}^{*}, d_{0}^{*}) &\colon \bfC_{n+1} \to \bfC_{0} \times \bfC_{n} \\
			(id\join \phi)^{*}&\colon \bfC_{m+1} \to \bfC_{n+1}
		\end{align*}
		are morphisms of Cartesian fibrations over $\bfC_{0}$.
	\end{definition}
	\begin{lem}\label{lem:companion_functor}
		Let $\bfC\colon \Simp_{s} \to \Cat_{\infty}$ be a semi-simplicial $\infty$-category which admits companions.
    Then there exists a functor $\cU_{\bfC}\colon \bfC_{0} \to \cT(h\bfC) $ whose restriction to cores agrees with the map induced by the counit of $\cT$, and such that for any $f\colon x \to y \in v\bfC$ and any idempotent equivalence $sy\in \bfC_{1}$ at $y$ there exists a representative $u\in \bfC_{1}$ of $\cU_{\bfC}(f)$ and morphism $\alpha \colon u \to sy$ in $\bfC_{1}$ which is a Cartesian lift of $(f,id_{y})\colon (x,y) \to (y,y)$ along $\bfC_{1}\to \bfC_{0}\times \bfC_{0}$.
	\end{lem}

	\begin{proof}
		Consider the natural transformations $v_{0}$ and $d_{0}$ between functors $\Simp_{s,a} \to \Simp_{s,a}$ with components
		\begin{align*}
			[0] \xrightarrow{id\join \emptyset } [0]\join [n] \xleftarrow{\emptyset \join id} [n] \,.
		\end{align*}
		These induce morphisms
		\begin{align*}
			\underline{\bfC_{0}} \xleftarrow{v_{0}^{*}} \bfC_{\Delta^{0}_{s}/} \xrightarrow{d_{0}^{*}} \bfC
		\end{align*}
		of augmented semi-simplicial $\infty$-categories, where $\underline{\bfC_{0}}$ is the constant functor with value $\bfC_{0}$.
    We therefore get a diagram
		\begin{equation*}
			\begin{tikzcd}
				\bfC_{\Delta^{0}_{s}/} \arrow[r, "{(v_{0}^{*},d_{0}^{*})}"] \arrow[dr, "v_{0}^{*}"']      &   \underline{\bfC_{0}}\times \bfC \arrow[d, "\pi"] \\
				&   \underline{\bfC_{0}}
			\end{tikzcd}
		\end{equation*}
		of augmented semi-simplicial $\infty$-categories.
    This corresponds under adjunction to a natural transformation between functors $\Simp_{s,a}^{\op} \to \Cat_{\infty / \bfC_{0}}$.
    By the assumptions, this natural transformation has components in the subcategory $\Cart(\bfC_{0}) \subset \Cat_{\infty /\bfC_{0}}$, so we may lift to a functor
		\begin{align*}
			\cU_{0}\colon [1] \times \Simp_{s,a}^{\op} \to \Cart(\bfC_{0}) \simeq \Fun(\bfC_{0}^{\op}, \Cat_{\infty})
		\end{align*}
		which under adjunction corresponds to a functor
		\begin{align*}
			\cU'_{0}\colon \bfC_{0}^{\op} \to \Ar(\Fun(\Simp_{s,a}^{\op},\Cat_{\infty}  ) ) \,.
		\end{align*}
		Because the straightening of the Cartesian fibration $\bfC_{n}\times \bfC_{0} \to \bfC_{0}$ is the constant functor $\bfC_{0} \to \Cat_{\infty}$ with value $\bfC_{n}$, we can lift $\cU_{0}'$ to a functor
		\begin{align*}
			\cU_{1}\colon \bfC_{0}^{\op} \to \Fun(\Simp_{s,a}, \Cat_{\infty})_{/\bfC}.
		\end{align*}
		Unwinding the construction, the value of $\cU_{1}$ at an object $x\in \bfC_{0}$ is equivalent to the left vertical map of the pullback
		\begin{equation*}
			\begin{tikzcd}
				\bfC_{x/} \arrow[r] \arrow[d]      &   \bfC_{\Delta^{0}_{s}/}  \arrow[d, "{(v_{0}^{*}, d_{0}^{*})}"] \\
				\bfC  \arrow[r, "\left\{ x\right\}\times id"]               &   \underline{\bfC_{0}} \times \bfC
			\end{tikzcd}
		\end{equation*}
		in $\Fun(\Simp^{\op}_{s,a},\Cat_{\infty})$.
    When we pass to underlying semi-simplicial spaces we get precisely the slice projection $h\bfC_{x/} \to h\bfC$.
    Since $h\bfC$ is a quasi-unital inner Kan space, so is each slice by \Cref{prop:semi-simplicial_results}.
    Moreover, each slice projection $h\bfC_{x/} \to h\bfC$ is a left fibration of semi-simplicial spaces, which is conservative.
     By 2-out-of-3 for conservative maps, any map $h\bfC_{x/} \to h\bfC_{y/}$ over $h\bfC$ is also conservative, and in particular, quasi-unital.
     We can therefore factor through the subcategory $\IK_{qu /h\bfC}$ and compose with $\cT$ to obtain a functor
		\begin{align*}
			\cU_{2}\colon \bfC_{0}^{\op} \to \Cat_{\infty / \cT(h\bfC) } \,.
		\end{align*}
		By \Cref{prop:semi-simplicial_results}, each $\cT(h\bfC_{x/}) \to \cT(h\bfC)$ is identified with the categorical slice projection at the object $x$ of $\cC$.
    Each such is a representable left fibration, so by the Yoneda lemma for copresheaves, we can factor $\cU_{2}$ uniquely through the fully faithful functor
		\begin{align*}
			\cT(h\bfC) \xrightarrow{y} \Fun(\cT(h\bfC), \cS) \simeq \LFib(\cT(h\bfC)) \subset \Cat_{\infty /\cT(h\bfC)}
		\end{align*}
		by (the opposite of) a functor $\cU_{\bfC} \colon \bfC_{0} \to \cT(h\bfC)$.
    Now consider a morphism $f\colon x\to y$ in $\bfC_{0}$.
    By the Yoneda lemma, we can find the morphism $\cU_{\bfC}(f)$ by evaluating
		\begin{align*}
			\cU_{2}(f)\colon \cT(h\bfC)_{y/} \to \cT(h\bfC)_{x/}
		\end{align*}
		at the initial object.
    This initial object can be represented in $\bfC_{1}$ by an idempotent equivalence $sy$ at $y$.
    Hence we can represent $\cU_{\bfC}(f)$ by the image of $sy$ under the functor
		\begin{align*}
			\cU_{1}(f)_{0} \colon (\bfC_{y/})_{0} \to (\bfC_{x/})_{0} \,.
		\end{align*}
		By construction, this functor is the Cartesian transport over $f$ between fibers of the Cartesian fibration $d_{1}^{*}\bfC_{1} \to \bfC_{0}$.
    By assumption, the Cartesian lift $u \to sy$ is mapped to a Cartesian lift of $f$ in $\bfC_{0}\times \bfC_{0}$, and the unique Cartesian lift of $f$ to the product is $(f,id_{y})$.
    By the 2-out-of-3 property for Cartesian lifts, $u\to sy$ is a Cartesian lift of $(f,id_{y})$.
	\end{proof}

	\begin{definition}\label{def:companion_functor}
		If $\bfC$ admits companions, we call $\cU_{\bfC} \colon \bfC_{0} \to \cT(h\bfC)$ the \emph{companion functor}.
	\end{definition}

	\begin{rem}\label{rem:companions_in_double}
		Let $\bfC$ be a double $\infty$-category.
    Then the underlying semi-simplicial $\infty$-category admits companions if and only if
		\begin{equation*}
			\begin{tikzcd}
				\bfC_{1} \arrow[r, "{(d_{1}^{*},d_{0}^{*})}"] \arrow[dr, "d_{1}"']      &   \bfC_{0}\times \bfC_{0}  \arrow[d, "\pi"] \\
				&   \bfC_{0}
			\end{tikzcd}
		\end{equation*}
		is a morphism of Cartesian fibrations over $\bfC_{0}$.
    The Cartesian lift $\alpha$ of some $f\colon x\to y \in \bfC_{0}$ at the horizontal identity $s_{0}y\in \bfC_{1}$ is mapped to a Cartesian lift of $f$ in $\bfC_{0}\times \bfC_{0}$ at $(y,y)$, which is therefore necessarily $(f, id_{y})$. In other words, $\alpha$ is a Cartesian 2-cell
		\begin{equation*}
			\begin{tikzcd}
				X \arrow[r, "F"] \arrow[d, "f"]      &   Y  \arrow[d, "id_{y}"] \\
				Y  \arrow[r, "s_{0}y"]               &   Y.
			\end{tikzcd}
		\end{equation*}
		By \cite*[Corollary 5.16]{ruit2025companion}, such a cell is Cartesian if and only if it exhibits $F$ as the \emph{horizontal companion} of the vertical arrow $f$. So the Cartesian fibration assumption on $\bfC$ implies that each vertical morphism in $\bfC$ admits a horizontal companion. The construction of companions is natural, in the sense that it determines a functor
		\begin{align*}
			\cU \colon \bfC_{0} \to h\bfC
		\end{align*}
		from the vertical $\infty$-category of $\bfC$ to the horizontal $\infty$-category of $\bfC$, which covers the identity on the underlying space of objects $\bfC_{0}^{\simeq}$. The construction of $\cU_{\bfC}$ can be seen as a generalization of the construction of a companion functor.
	\end{rem}
	\begin{rem}\label{rem:conjoints}
		The following dual of \Cref{lem:companion_functor} also holds. If the functors
		\begin{align*}
			(d_{n+1}^{*}, v_{n+1}^{*})\colon \bfC_{n+1} &\to \bfC_{n} \times \bfC_{0} \\
			(\phi\join id)\colon \bfC_{n+1} &\to \bfC_{m+1}
		\end{align*}
		are morphisms of Cartesian fibrations over $\bfC_{0}$, then there exists a functor $\cU'_{\bfC} \colon \bfC_{0}^{\op} \to \cT(h\bfC)$ characterized by $\bfC_{0}^{\simeq} \to \cT(h\bfC)$ agreeing with the counit of $\cT$, and for any $f\colon x\to y$, the morphism $\cU'_{\bfC}(f)$ admitting a representative $u'$ which is the Cartesian transport of $sy$ along $(id_{y}, f)$.
    This construction generalizes the notion of \emph{conjoints} in double $\infty$-categories.
	\end{rem}
	\begin{lem}\label{lem:limts_as_lifting}
		Let $\pi \colon \cE \to \cB$ be a Cartesian fibration, with underlying right fibration $\cR(\pi) \colon \cR(\cE) \to \cB$. Let $F\colon \cI^{\triangleright} \to \cB$ be a diagram in $\cB$.
    Then the composite
		\begin{align*}
			(\cI^{\op})^{\triangleleft} \simeq (\cI^{\triangleright})^{\op} \xrightarrow{F^{\op}} \cB^{\op} \xrightarrow{\St(\cR(\pi))} \cS
		\end{align*}
		is a limit diagram if and only if given the solid part of a diagram
		\begin{equation}\label{eq:cartesian_lifting_problem}
			\begin{tikzcd}
				\cI \arrow[r, "G'"] \arrow[d, "\iota"]      &   \cE  \arrow[d, "\pi"] \\
				\cI^{\triangleright}  \arrow[r, "F"]    \arrow[ur,"G", dashed]          &   \cB
			\end{tikzcd}
		\end{equation}
		where $G'$ is a Cartesian lift of $F\circ \iota$, there exists a unique extension $G$ which is a Cartesian lift of $F$.
	\end{lem}

	\begin{proof}
		The space of Cartesian lifts $G'$ is equivalent to the fiber at $G'$ of the restriction map
		\begin{align*}
			\Map_{/\cI^{\triangleright}}^{\Cart}(\cI^{\triangleright}, F^{*}\cE) \to \Map_{/\cI}^{\Cart}(\cI, \iota^{*}F^{*}\cE) \,.
		\end{align*}
		By adjunction, we could equivalently replace $\cE$ with its underlying right fibration in the above.
    Limits in spaces are computed by sections of the corresponding right fibrations, so this map is precisely the comparison map
		\begin{align*}
			\lim( \St(\cR(\pi))\circ F^{\op}  ) \to \lim( \St(\cR(\pi))\circ F^{\op} \circ \iota )\,,
		\end{align*}
		which is an equivalence if and only if all its fibers are contractible.
	\end{proof}

	\begin{rem}
		We call a diagram of the form \eqref{eq:cartesian_lifting_problem}, where $\cJ \to \cJ^{\triangleright}$ may more generally replaced with an arbitrary functor, a \emph{cartesian lifting problem}.
    A solution to the Cartesian lifting problem is the dashed lift $G$ which takes any morphism $f$ in $\cJ^{\triangleright}$ to a $\pi$-Cartesian lift of $F(f)$.
	\end{rem}

	\begin{lem}\label{lem:companion_colim}
		Let $\bfC$ be a semi-simplicial $\infty$-category which admits companions. Let $\cI$ be an $\infty$-category, and $F\colon \cI^{\triangleright} \to \bfC_{0}$ a diagram such that for any $x\in \bfC_{n}$, any Cartesian lifting problem of the form
		\begin{equation*}
			\begin{tikzcd}
				\cI \arrow[r] \arrow[d]      &   \bfC_{n+1}  \arrow[d] \\
				\cI^{\triangleright}  \arrow[r, "{(F,\left\{ x\right\})}"']     \arrow[ur,dashed]       &   \bfC_{0} \times \bfC_{n}
			\end{tikzcd}
		\end{equation*}
		admits a unique solution.
    Then the diagram $h\circ \cU_{1}\circ F$ is a limit diagram in $ss\cS_{/h\bfC}$.
    Moreover, if either
		\begin{enumerate}
			\item $\cI$ is a weakly contractible direct category, and the pushforward of $h\circ\cU_{1} \circ F$ to $ss\cS$ is Reedy fibrant or,
			\item $\cI= S^{\triangleleft}$ for a finite set $S$, and for every $s\in S$ the morphism $-\infty \to s$ in $\cI$ is taken by $h\circ \cU_{1}\circ F$ to an inner fibration,
		\end{enumerate}
		then $\cU_{\bfC} \circ F$ becomes a colimit diagram in $\cT(h\bfC)$.
	\end{lem}
	\begin{proof}
		The limit condition in $ss\cS_{/\bfC^{\simeq}}$ can be checked degree-wise over $\Simp_{s}$.
    Here it becomes a question of whether the straightening of the underlying right fibration of $\bfC_{n+1} \to \bfC_{0}$ takes $F$ to a limit diagram.
    Because a Cartesian lift of $F$ to $\bfC_{n+1}$ is necessarily a Cartesian lift of some $(F, \left\{ x\right\})$, the space of such lifts is contractible, so the result follows from \Cref{lem:limts_as_lifting}.

		The composite $\cT(h\bfC)^{\op} \to \LFib(\cT(h\bfC)) \to \Cat_{\infty/\cT(h\bfC)}$ detects limits in $\cT(h\bfC)^{\op}$, which are equivalently colimits in $\cT(h\bfC)$.
    To show that $\cU_{\bfC}\circ F$ is a colimit diagram, it therefore suffices to show under the extra assumptions that the limit diagram $h\circ \cU_{1}\circ F$ is preserved by the localization $\cT$.
    Limits over weakly contractible $\infty$-categories in slices are detected underlyingly, so in both cases it suffices to check that the pushforward of $h\circ \cU_{1}\circ F$ to $ss\cS$ is a ``homotopy limit''.
    In case i) this follows immediately from the Reedy fibrancy condition and \Cref{prop:semi-simplicial_results}.
    In case ii), we can compute the limit as the pullback
		\begin{equation*}
			\begin{tikzcd}
				h\bfC_{F(\infty)/} \arrow[r] \arrow[d]      &   \prod_{s\in S} h\bfC_{F(s)/}  \arrow[d] \\
				h\bfC_{F(-\infty)/}  \arrow[r]               &   \prod_{s\in S} h\bfC_{F(-\infty)/} \,.
			\end{tikzcd}
		\end{equation*}
		The map on the right is a product of fibrations, and therefore itself a fibration.
    By \Cref{prop:semi-simplicial_results}, products, as well as pullbacks where one leg is a fibration, are preserved by $\cT$.
	\end{proof}

\section{Structured Flow Categories}\label{sec:structured_flow_categories}

The goal of this section is to define the notion of a $\mu$-structured flow category for some functor $\mu\colon \cC \to U/O$.

The main construction is a right fibration of double $\infty$-categories $\Mfd^{\mu}_{\diamond} \to \Mod_{\diamond}^{\otimes}$.
The $\infty$-category $\Mod_{\diamond}$ is an $\infty$-category of stratifying categories.
Given a poset $P$, we will define a $\Mod_{\diamond}$-category which encodes the boundary structure we expect of a flow category with objects $P$.

By \Cref{cor:monoid_equivalence}, the functor $\mu$ corresponds to a $\cS_{/\bZ\times BO}$-category which we will also denote $\mu$.
The double $\infty$-category $\Mfd^{\mu}_{\diamond}$ has the same space of objects as $\mu$, and morphism $\infty$-categories which are $\infty$-categories of $\mu(c,d)$-structured stratified manifolds with corners.
A $\mu$-structured flow category with objects $P$ will then be a lift of $\sfA_{P}$ to a $\Mfd^{\mu}_{\diamond}$-category.

\subsection{Virtual vector bundles}\label{subsec:virtual_vector_bundles}
Classically, by a virtual vector bundle on a space $X$, we mean a formal difference of vector bundles represented by a pair $E=(E^+, E^-)$.
The correct notion of equivalence for virtual vector bundles should capture the idea that we can ``cancel'' a summand that appears in both $E^+$ and $E^-$.

\begin{definition}\label{def:stable_equivalence}
	A \emph{stable equivalence} of virtual vector bundles $E\to F$ on a space $X$ is determined by a vector bundle $V$, and an isomorphism of vector bundles
	\begin{align*}
		E^+ \oplus V \oplus F^-  \overset{\cong}\to E^- \oplus V \oplus F^+ \,.
	\end{align*}
\end{definition}

Stable equivalence determines an equivalence relation on the set of virtual vector bundles on $X$, because if we have stable equivalences $E \to F$ and $F \to G$, determined by
\begin{align*}
	f\colon E^+ \oplus V \oplus F^-  &\xrightarrow{\cong} E^- \oplus V \oplus F^+ \\
	g\colon F^+ \oplus W \oplus G^- &\xrightarrow{\cong} F^- \oplus W \oplus G^+,
\end{align*}
we can form a stable equivalence $E\to G$ determined by the composition
\begin{align*}
	E^+ \oplus V \oplus F^- \oplus W \oplus G^- &\xrightarrow[f\oplus id]{\cong} E^- \oplus V \oplus F^+ \oplus W \oplus G^- \\
	&\xrightarrow[id \oplus g]{\cong} E^- \oplus V \oplus F^- \oplus W \oplus G^+ .
\end{align*}
The usual operations on vector bundles extend to ones on virtual vector bundles. For vector bundles $E\to X$ and $F\to Y$ we write $E\boxplus F \to X\times Y$ for the exterior sum, which as a map of spaces is just the product $E\times F \to X\times Y$. Because addition has been inverted when dealing with virtual vector bundles, we have access to two new operations.
\begin{definition}\label{def:boxplus_boxminus}
	For $E,F$, virtual vector bundles on $X$, $G$ a virtual vector bundle on $Y$, we define
	\begin{itemize}
		\item  $E\ominus F = (E^+ \oplus F^-, E^-\oplus E^+)$
		\item  $E\boxminus G = (E^+ \boxplus G^-, E^- \boxplus G^+)$
	\end{itemize}
	where the first one is a virtual vector bundle over $X$ and the second one is a virtual vector bundle over  $X \times Y$.
\end{definition}
It is non-trivial to construct an ordinary monoidal category of virtual vector bundles up to stable equivalence. For this, we pass to $\infty$-categories by first defining a certain category enriched in topological spaces.
\begin{definition}\label{def:Top}
	Let $\Top^{\times}$ denote the Cartesian monoidal category of compactly generated spaces. This is closed, so we obtain a $\Top$-enriched category $\overline{\Top}\vphantom{p}^{\times}$ with morphism spaces $\overline{\Hom}(X,Y)$.
\end{definition}
Let $\overline{CW}\subset \overline{\Top}$ be the full enriched subcategory spanned by CW-complexes. $\overline{\CW}$ is the enriched category of fibrant-cofibrant objects in the classical model structure on topological spaces. When we push forward the enrichment along the localization $\Pi \colon \Top\to \cS$, we therefore obtain the $\infty$-category of spaces.
\begin{definition}\label{def:Vect(Top)}
	Let $\Vect(\CW)$ denote the category whose objects are vector bundles $E\to X$ on CW-complexes, and whose morphisms are fiberwise equivalences. We give this the exterior product monoidal structure $\boxplus$. We lift this to an algebra object $\Vect(\overline{\CW})^{\boxplus}$ in $\Cat_{\Top}$ by topologizing the set of morphisms $\Hom(E,F)$ as the subspace
	\begin{align*}
		\overline{\Hom}_{\Vect}(E,F) \subset \overline{\Hom}(E,F)\times_{\overline{\Hom}(E,Y) } \overline{\Hom}(X,Y)
	\end{align*}
	consisting of maps $E\to F$ which fiberwise are linear isomorphisms.
\end{definition}
We have a continuous monoidal functor $\Vect(\overline{\CW})^{\boxplus} \to \overline{\CW}\vphantom{p}^{\times}$ given by remembering only the base. By a classical extension of sections result (see for example \cite[Theorem 5.7]{karoubi2009k}), the forgetful functor induces maps
\begin{align*}
	\overline{\Hom}_{\Vect}(E,F) \to \overline{\Hom}(X,Y)
\end{align*}
which are Serre-fibrations.
For any fiberwise equivalence of vector bundles $\phi\colon F\to G$ covering $f\colon Y\to Z$ and any bundle $E\to X$ we have a pullback of topological spaces
\begin{equation*}
	\begin{tikzcd}
		\overline{\Hom}_{\Vect}(E,F)  \arrow[r, "\phi_{*}"] \arrow[d]      &   \overline{\Hom}_{\Vect}(E,G)  \arrow[d] \\
		\overline{\Hom}(X,Y)  \arrow[r, "f_{*}"]               &  \overline{\Hom}(X,Z)
	\end{tikzcd}
\end{equation*}
given by using the equivalence $F \simeq f^{*}G$ induced by $\phi$. Since the right vertical map is a Serre fibration, this square remains a pullback when we push forward the enrichment along $\Pi$, so the resulting symmetric monoidal functor $\Vect(\cS)^{\boxplus} \to \cS^{\times}$ is a symmetric monoidal right fibration.
\begin{lem}\label{lem:Vect(S)_terminal}
	The universal bundle $\gamma \to \coprod_{k}BO(k)$ is a terminal object in $\Vect(\cS)$.
\end{lem}
\begin{proof}
	For a vector bundle $E\to X$ on a CW-complex, there exists a classifying map
	\begin{align*}
		f_{E}\colon X\to \coprod_{k} BO(k)
	\end{align*}
	exhibiting $E$ as the pullback $E\simeq f_{E}^{*} \gamma$ of the universal bundle $\gamma \to \coprod_{k} BO(k)$. This is moreover unique up to homotopy. We then have a homotopy pullback square
	\begin{equation*}
		\begin{tikzcd}
			\overline{\Hom}_{\Vect}(E,E) \arrow[r] \arrow[d]      &   \overline{\Hom}_{\Vect}(E, \gamma)  \arrow[d] \\
			\ast  \arrow[r, "f_{E}"]               &    \overline{\Hom}(X,\coprod_{k} BO(k)) \, .
		\end{tikzcd}
	\end{equation*}
	The fiber over any other homotopy class is necessarily empty. Reducing to the case where $X$ is connected, we may assume $\rk(E)=n$.
  Then the above square reduces to the homotopy pullback
	\begin{equation*}
		\begin{tikzcd}
			\overline{\Hom}(X, O(n)) \arrow[r] \arrow[d]      &   \overline{\Hom}_{\Vect}(E, \gamma)  \arrow[d] \\
			\ast  \arrow[r, "f_{E}"]               &    \overline{\Hom}(X,BO(n)) \,.
		\end{tikzcd}
	\end{equation*}
	Since $\overline{\Hom}(X, O(n)) \simeq \Omega\overline{\Hom}(X, BO(n))$, this implies that $\overline{\Hom}_{\Vect}(E,\gamma)$ is contractible.
\end{proof}
Now the terminal object $\gamma$ gets a unique $\bE_{\infty}$-algebra structure in $\Vect(\cS)$, which induces an algebra structure on $\coprod_{k} BO(k)$ in $\cS$, and an equivalence
\begin{align*}
	\Vect(\cS)^{\boxplus} \simeq \cS_{/\coprod_{n} BO(n)}^{\otimes}
\end{align*}
of symmetric monoidal right fibrations over $\cS$.
As usual, one can show that the group completion on $\coprod_{n} BO(n)$ is the space $\bZ \times BO$. Consider the continuous monoidal functor $\Vect(\CW)^{\boxplus} \to \Vect(\overline{\CW})^{\boxplus}$ where we view the source as enriched in discrete topological spaces. When pushing forward the enrichment along $\Pi$, this gives rise to symmetric monoidal functors
\begin{align*}
	\Vect(\CW)^{\boxplus} \to \cS_{/\coprod_{k} BO(k)}^{\otimes} \to \cS_{/\bZ\times BO}^{\otimes} \,.
\end{align*}
Now by using the inversion operation on $\bZ\times BO$, a virtual vector bundle $(E^{+}, E^{-})$ on a CW-complex $X$ gives rise to a classifying map
\begin{align*}
	X \xrightarrow{\Delta} X\times X \xrightarrow{E^{+}\times E^{-}} (\bZ\times BO) \times (\bZ\times BO) \xrightarrow{ \ominus} (\bZ\times BO) \,.
\end{align*}
Using the definition of tensor inverses, it is similarly clear that a stable equivalence gives rise to a homotopy of classifying maps.
Indeed, one can show that the abelian group of virtual vector bundles on $X$ up to stable equivalence is isomorphic to the group $[X,\bZ\times BO]$ of homotopy classes of maps.
Working in the $\infty$-category $\cS_{/\bZ\times BO}$ now has the upshot that we can access the whole space of virtual vector bundles on $X$ up to equivalence, and that we have a well-defined symmetric monoidal structure.
\begin{rem}\label{rem:bVect}
	In many instances it will be convenient to have smaller model for virtual vector bundles, i.e. some monoidal category with a functor to $\cS_{/ (\bZ\times BO)}^{\otimes}$. Start with the diagonal functor
	\begin{align*}
		\Delta \colon \Vect \to \Vect(\CW)\times \Vect \,.
	\end{align*}
	This lifts to a monoidal functor, giving an action of the monoid $\Vect$ on $\Vect(\CW)\times \Vect$. Let $\sfB\Vect$ denote the delooping $(2,1)$-category of $\Vect$, i.e. the category wiht a single object whose endomorphism object is the 1-groupoid $\Vect$. Because $\Vect$ has a symmetric monoidal structure, the delooping category $\sfB\Vect$ can also be equipped with a symmetric monoidal structure. The diagonal action is now classified by a lax monoidal functor
	\begin{align*}
		\sfB \Vect \to \Cat^{\times}\,,
	\end{align*}
	which by \cite{ramzi2026monoidal} unstraightens to a symmetric monoidal coCartesian fibration
  \[\bVect(\CW)^{\boxplus} \to \sfB (\Vect)^{\oplus} \,.
  \]
  An object of $\bVect(\CW)$ is a pair $(E,V)$ where $E\to X$ is a vector bundle, and $V$ is a vector space. We should think of this as representing the formal difference $E\ominus V$.
  A morphism $(E,V) \to (E,W)$ covering a morphism $U$ in the base $\sfB \Vect$ factors as a cocartesian lift followed by a morphism in the fiber $\Vect(\CW)\times \Vect$, in other words as
	\begin{align*}
		(E,V) \xrightarrow{U_{!}} (E\oplus U, V\oplus U) \to (F, W) \,.
	\end{align*}
	Now consider the composite of monoidal functors
	\begin{align*}
		\Vect \xrightarrow{\Delta} \Vect(\CW) \times \Vect \to \cS_{/ (\bZ\times BO)} \times (\bZ\times BO) \xrightarrow{ \ominus} \cS_{/\bZ\times BO} \,.
	\end{align*}
	By construction, this factors through
	\begin{equation*}
		\begin{tikzcd}
			\bZ\times BO \arrow[r, "\Delta"] \arrow[d]      &   \bZ\times BO \times \bZ\times BO  \arrow[d,"\ominus"] \\
			\ast \arrow[r, "0"]               &   \bZ\times BO \, ,
		\end{tikzcd}
	\end{equation*}
	so the difference functor $\Vect(\CW) \times \Vect \to \cS_{/\bZ\times BO}$ is invariant under the diagonal action by $\Vect$. In other words, the functor $\sfB \Vect \to \Cat_{\infty}$ classifying the action is constant, so by naturality of unstraightening we get functors
	\begin{align*}
		\bVect(\CW) \to \cS_{/ \bZ\times BO} \times \sfB \Vect \to \cS_{/\bZ\times BO} \,,
	\end{align*}
	where the first is a morphism of cocartesian fibrations. The above construction can easily be generalized to take monoidal structures into account, giving a monoidal functor $\bVect(\CW)^{\boxplus}\to \cS_{/ (\bZ\times BO)}^{\otimes}$. The upshot is that $\bVect(\CW)^{\boxplus}$ is a monoidal $(2,1)$-category, so functors into it can reasonably we written out by hand.
\end{rem}

A virtual vector bundle $E$ on $X$ classifies a cobordism theory of manifolds as follows.

\begin{definition}\label{def:stable_normal_orientation}
	Given a manifold $M$, and a virtual vector bundle $E$ on a space $X$, an \emph{$E$-structure on $M$} consists of a map of spaces $f\colon M \to X$, and a stable equivalence $-TM \simeq f^{*}E$.
\end{definition}

This definition makes sense both with the classical and current notion of virtual vector bundle.
In fact, in the new definition, it is just a map $-TM \to E$ in the $\infty$-category $\cS_{/\bZ\times BO}$.
Hence the mapping space $\Map_{\cS_{/\bZ\times BO}}(-TM, E)$ should be thought of as a \emph{space} of $E$-structures on $M$.

If we let $\Mfd_{n}$ denote the category of smooth $n$-manifolds and embeddings, we get a functor $T\colon \Mfd_{n}\to \Vect(\CW) \to \cS_{/\bZ\times BO}$ by sending a manifold to the classifying map of its tangent bundle.
Then the pullback
\begin{equation*}
	\begin{tikzcd}
		\Mfd_{n/E} \arrow[r] \arrow[d]      &   (\cS_{/\bZ\times BO})_{/E}  \arrow[d] \\
		\Mfd_{n}  \arrow[r, "-T"]               &   \cS_{/\bZ\times BO}
	\end{tikzcd}
\end{equation*}
should be seen as an $\infty$-category of $n$-manifolds with $E$-structures and embeddings between them.
This $\infty$-category is not large enough to deal with cobordism theory, because we have no way of talking about the inclusion of a boundary component.
This can be handled by noting that there is a contractible choice of inward pointing vector field on the boundary $\del X\subset X$.
This amounts to an equivalence of vector bundles
\begin{align*}
	T\del X \oplus \bR \simeq TX\vert_{\del X} \,.
\end{align*}
The correct notion of $E$-structure on $\del X$ should therefore be a map $T\del X \oplus \bR \to E$ in $\cS_{/\bZ\times BO}$.
In what follows, we will generalize this idea to also work for manifolds with corners, and we will do so in a way where we keep track of the monoidal structures involved.

\subsection{Models and stratified manifolds with corners}\label{subsec:models_and_stratified}

By a \emph{smooth manifold with corners}, we mean a paracompact Hausdorff space $X$, with an atlas exhibiting $X$ as locally diffeomorphic to $\bR_{\geq 0}^{n}$.
We say that a point $x$ has \emph{codimension $k$} if there is a diffeomorphism $\psi\colon U_x \to \bR_{\leq 0}^{k} \times \bR^{n-k}$ such that $U_x$ is an open neighborhood of $x$, and $\psi(x)=0$.
The corner structure determines a canonical stratification of $X$.

\begin{definition}\label{def:P^Can}
	For a manifold with corners $X$, let $\cP_X^{\can}$ denote the category given as follows.
	The objects are
	\begin{align*}
		\mathrm{Ob}(\cP_X^{\can}) = \bigcup_k \pi_0\{x\in X \mid \codim(x) =k \} \,.
	\end{align*}
	For an object $\sigma$, write $X_\sigma$ for the closure of the corresponding connected component.
	Then the set of morphisms between $\sigma$ and $\tau$ is defined as
	\begin{align*}
		\Hom_{\cP_X^{\can}}(\sigma, \tau) = \pi_0\left(N(X_\tau) \cap \interior(X_\sigma)\right)
	\end{align*}
	where $N(X_{\tau})$ is any sufficiently small tubular neighborhood of $X_{\tau}$ in $X$.
	The composition
	\[
	\Hom_{\cP_X^{\can}}(\sigma, \tau) \times \Hom_{\cP_X^{\can}}(\tau, \rho) \to \Hom_{\cP_X^{\can}}(\sigma, \rho)
	\]
	is the induced map on $\pi_0$ associated with
	\begin{align*}
		\left(N(X_\tau)\cap \interior( X_{\sigma}) \right) \times_{X_\sigma} \left(N(X_\sigma ) \cap \interior(X_\rho) \right) \to N(X_\tau)\cap \interior(X_\rho) \,.
	\end{align*}
\end{definition}
\begin{rem}
	The category $\cP_{X}^{\can}$ captures the corner strata of $X$ and their adjacency relations.
  As an example, consider $X$ the two-dimensional teardrop, i.e., the connected two-dimensional manifold with corners with a single point of codimension 2.
  This has canonical stratifying category
	\begin{equation*}
		\begin{tikzcd}
			2 \arrow[r,bend left] \arrow[r, bend right] & 1 \arrow[r] & 0 \,.
		\end{tikzcd}
	\end{equation*}
\end{rem}

Each object in $\cP_X^{\can}$ has a well-defined codimension.
It is evident that morphisms increase codimension, giving us a functor
\[
\codim \colon \cP^{\can}_{X} \to \bZ \,.
\]
Following \cite{AB}, we introduce a certain class of combinatorially well-behaved categories.
We will work only with manifolds stratified by such categories.
This structure is required to ensure, for instance, that each stratum itself is a manifold with corners.
In particular, all manifolds with corners arising from Morse theory will be of this form.

\begin{definition}\label{def:manifold_with_corners}
	A category $\cP$ is a  \emph{model for manifolds with corners} (or just \emph{model} for short) if there exists a functor $\codim \colon \cP \to \bZ$ such that for each object $\sigma\in \cP$, the overcategory $\cP_{/\sigma}$ is isomorphic to the poset of subsets of $\{1,...,\codim(\sigma)\}$.
\end{definition}
\begin{definition}\label{def:Q_P(sigma)}
	For a model $\cP$ and an object $\sigma$, we write $Q_{\cP}(\sigma)$ for the set of codimension 1 objects over $\sigma$.
\end{definition}
\begin{rem}\label{rem:codim_is_unique}
	Note that if $\cP$ is a model for manifolds with corners, the functor codim is uniquely determined because it can be reconstructed from the cardinality of $\cP_{/\sigma}$.
\end{rem}

Let $\cP$ be a model for manifolds with corners.
Note that an object $\sigma$ has codimension 0 if and only if $\sigma$ is a minimal object in $\cP$, as the slice $\cP_{/\sigma}$ is then forced to contain only the identity morphisms on $\sigma$.
Moreover, any object $\tau$ of codimension 1 is adjacent to a single codimension 0 object $\sigma$.
Indeed, the overcategory of $\tau$ has precisely two objects, the identity morphism on $\tau$ and the unique morphism $\sigma \to \tau$.
These two properties always hold for the categories $\cP_X^{\can}$.
However, $\cP_{X}^{\can}$ might not be a model for manifolds with corners, see for instance the teardrop.
If $\cP$ is a model for manifolds with corners, the isomorphisms of overcategories are determined by bijections
\begin{align*}
	Q_{\cP}(\sigma) \coloneq \left\{ \tau \to \sigma \in \cP_{/\sigma} \mid \codim(\tau)=1 \right\} \xrightarrow{\cong} \left\{ 1,..., \codim(p) \right\} \,.
\end{align*}
The slogan is that a stratum of codimension $k$ is uniquely determined as an intersection of $k$ codimension 1 strata.

\begin{definition}\label{def:pi_0_of_model}
	For a model $\cP$, write $\pi_{0}\cP$ for the set of minimal objects in $\cP$.
	For any object $\sigma \in \cP$, write $[\sigma]\in \pi_{0}(\cP)$ for the unique minimal object with a morphism to $\sigma$.
\end{definition}

\begin{definition}\label{def:Q_f(sigma)}
	For a functor $f\colon \cP\to \cP'$ between models for manifolds with corners, and some $\sigma \in \pi_{0}\cP$, define
	\begin{align*}
		Q_{f}(\sigma) = Q_{\cP'}(f(\sigma)) \,.
	\end{align*}
\end{definition}

\begin{lem}\label{lem:Q_f}
	For composable functors $f$ and $g$ between models, we have a canonical bijection
	\begin{align*}
		Q_{g\circ f}(\sigma) = Q_{f}(\sigma) \amalg Q_{g}([f(\sigma)]) \,.
	\end{align*}
	For a pair of functors $f$ and $g$ between models, we have a canonical bijection
	\begin{align*}
		Q_{f\times g}( (\sigma, \gamma) ) = Q_{f}(\sigma) \amalg Q_{g}(\gamma) \,.
	\end{align*}
\end{lem}

\begin{definition}
	Let $\Mod^{\otimes}$ denote the monoidal category whose objects are $(\cP,U)$ where $\cP$ is a model for manifolds with corners, and $U=(U^{+},U^{-})$ is a pair of vector spaces.
  A morphism $(\cP_{0},U_{0}) \to (\cP_{1},U_{1})$ in $\Mod$ is given by
	\begin{enumerate}
	    \item 	A functor $f\colon \cP_{0} \to \cP_{1}$.
		\item   For each minimal object $\sigma \in \cP_{0}$, a decomposition $Q_{f}(\sigma) = Q^{+}_{f}(\sigma) \amalg Q^{-}_{f}(\sigma)$, and linear isomorphisms
		\begin{align*}
			U^{+}_{0} \oplus \bR^{Q^{+}_{f}(\sigma)} &\simeq U_{1}^{+} \\
			U^{-}_{0} & \simeq U_{1}^{-}\oplus \bR^{Q^{-}_{f}(\sigma)} \,.
		\end{align*}
	\end{enumerate}
	To compose morphisms $(\cP_{0},U_{0}) \xrightarrow{f} (\cP_{1},U_{1}) \xrightarrow{g} (\cP_{2},U_{2})$ , we set $Q^{\pm}_{g\circ f} = Q^{\pm}_{f} \amalg Q^{\pm}_{g}$ with respect to the decomposition in \Cref{lem:Q_f}, and use the equivalences
	\begin{align*}
		U_{0}^{+} \oplus \bR^{Q_{f}^{+}} \oplus \bR^{Q_{g}^{+}} \simeq U_{1}^{+} \oplus \bR^{Q_{g}^{+}} \simeq U_{2}
	\end{align*}
	and dual for $U_{i}^{-}$.
	We define the monoidal structure by cartesian product of models, and direct sum of $U^{\pm}$.
  For a product of morphisms $f\colon (\cP_{0},U_{0}) \to (\cP_{1},V_{1})$ and $g\colon (\cR_{0},V_{0}) \to (\cP_{1},V_{1})$ we use the decomposition  $Q^{\pm}_{f\times g} = Q^{\pm}_{f} \amalg Q^{\pm}_{g}$ with respect to the bijection in \Cref{lem:Q_f}, the linear isomorphism
	\begin{align*}
		U_{0}^{+} \oplus V_{0}^{+} \oplus \bR^{Q^{+}_{f\times g}} \simeq (U_{0}^{+} \oplus \bR^{Q_{f}^{+}} ) \oplus (V_{0}^{+}\oplus \bR^{Q_{g}^{+}}) \simeq U_{1}^{+} \oplus V_{1}^{+}.
	\end{align*}
\end{definition}

\begin{definition}\label{def:stratified_manifold_with_corners}
	A \emph{stratified manifold with corners} is a triple $\bX = (X,\cP_X, s)$, where $X$ is a manifold with corners, $\cP_X \to \bZ$ is a model for manifolds with corners, and $s\colon \cP_{X}^{\can} \to \cP_X$ is a functor that preserves codimension and induces equivalence on all overcategories.
\end{definition}

\begin{rem}\label{rem:stratified_vs_faces}
	In particular, this definition ensures that $\cP^{\can}_X$ itself is a model for manifolds with corners.
	This ensures that the strata of $X$ are combinatorially well-behaved.
\end{rem}

We will often abuse notation and drop $\cP_X$ and $s$ from the notation.
The functor $\cP_X^{\can}\to \cP_X$ should be seen as allowing a label $\sigma \in \cP_X$ to represent a disjoint union of some connected strata in $\cP_X^{\can}$.
In particular, we denote
\begin{align*}
	\del^{\sigma}X = \bigcup_{\tau \in s^{-1}(\sigma)} X_{\tau} \subset X \,,
\end{align*}
where the $X_\tau$ in the union are the closures of the connected components corresponding to $\tau$, as in \Cref{def:P^Can}.
Now, $\del^{\sigma}X$ is a manifold with corners in its own right, and it inherits a natural stratification by the undercategory
\begin{align*}
	\del^{\sigma}\cP_{X} \coloneq (\cP_{X})_{\sigma/} \,.
\end{align*}
We write $\del^{\sigma} \bX = (\del^{\sigma}X, \del^{\sigma} \cP_{X})$ for this stratified manifold with corners.
Recall the following definition from \cite[Definition 2.25]{AB}.

\begin{definition}\label{def:consistent_normal_framing}
	A \emph{consistent normal framing} of a stratified manifold with corners~$\bX$ is a system of isomorphisms of vector bundles
	\begin{align*}
		\psi_{\sigma}\colon T\del^{\sigma} X \oplus \bR^{Q_{\cP}(\sigma)} \xrightarrow{\simeq} TX\vert_{\del^{\sigma}X} \,,
	\end{align*}
	such that the basis $Q_{\cP}(\sigma)$ of $\bR^{Q_{\cP}(\sigma)}$ is mapped to inward pointing vectors.
	We require that for any $\tau \to \sigma$ in $\cP$, the square
	\begin{equation*}
		\begin{tikzcd}
			T\del^{\tau} X\oplus \bR^{Q_{\cP}(\sigma)} \arrow[r] \arrow[d]      &   T\del^{\tau}X \oplus \bR^{Q_{\cP}(\tau)}  \arrow[d, "\psi_{\tau}"] \\
			T \del^{\sigma}X\vert_{\del^{\tau}X} \oplus \bR^{Q_{\cP}(\sigma)}  \arrow[r, "\psi_{\sigma}"]               &  TX\vert_{\del^{\tau}X}
		\end{tikzcd}
	\end{equation*}
	where the left vertical map is given by the derivative of the inclusion $\del^{\tau}X\to \del^{\sigma}X$, and the top horizontal by the inclusion $Q_{\cP}(\sigma) \subset Q_{\del^{\sigma}\cP}(\tau)\amalg  Q_{\cP}(\sigma) \simeq Q_{\cP}(\tau)$.
\end{definition}

\begin{rem}
	The compatibility requirement says that a consistent normal framing~$\psi$ on~$\bX$ restricts to a normal framing $\del^{\sigma} \psi$ on $\del^{\sigma}\bX$ where $(\del^{\sigma}\psi)_{\tau}$ is the restriction of~$\psi_{\tau}$ to the subbundle where the coordinates $Q_{\cP}(\sigma)$ vanish.
\end{rem}

We will take our flow categories to come equipped with consistent normal framings.
By \cite[Lemma 2.26]{AB}, the choice of such extends over any inclusion of submanifolds, so this choice will become contractible in the $\infty$-category $\Flow^{\mu}$.
By using the consistent normal framing, we can pull back $E$-structures along stratum inclusions, as long as we compensate by adding the trivial factor $\bR^{Q_{\cP}(\sigma)}$.
We will now construct an $\infty$-category that incorporates this notion.
We start with a preliminary version.

\subsection{Structured manifolds with corners}\label{subsec:structured_manifolds}
In this section, we show how a tangential structure can be transported along a stratum inclusion $\del^{\sigma}\bX \to \bX$ between stratified manifolds with corners, as long as one compensates for the normal directions $Q_{\cP}(\sigma)$.

\begin{definition}\label{def:PreMfd}
	Let $\Mfd$ denote the category whose objects are tuples $\bX = (X,\psi,\cP,U)$ where $(\cP,U)$ is an object of $\Mod$, and $\bX$ is a manifold with corners with a stratification by $\cP$, together with a choice of consistent normal framing $\psi$ of $\bX$.
  A morphism $F\colon \bX_{0} \to \bX_{1}$ is determined by:
	\begin{enumerate}
		\item A morphism $f\colon (\cP_{0},U_{0}) \to (\cP_{1},U_{1})$ in $\Mod$.
		\item A diffeomorphism $F_{\sigma \colon} \del^{\sigma}X \to \del^{f(\sigma)}X$ for each minimal element $\sigma \in \cP$.
	\end{enumerate}
	We require that the diagram
	\begin{equation*}
		\begin{tikzcd}
			\cP^{\can}_{X} \arrow[r, "\coprod_{\sigma} F_{\sigma}"] \arrow[d]      &    \cP^{\can}_{X'} \arrow[d] \\
			\cP \arrow[r,"f"]                &  \cP'
		\end{tikzcd}
	\end{equation*}
	commutes, and that for any minimal $\sigma$, the derivative $dF_{\sigma}$ intertwines the consistent normal framing $\del^{\sigma}\psi$ on $\del^{\sigma}X$ with the normal framing $\del^{f(\sigma)}\psi'$ on $\del^{f(\sigma)}X'$.
	We equip $\Mfd$ with the monoidal structure lifting the one on $\Mod$ by taking Cartesian product of manifolds, stratifications and normal framings.
\end{definition}

We now construct the \emph{index bundle} as a monoidal functor $I\colon \Mfd^{\otimes} \to \bVect(\CW)^{\boxplus}$.
We start by defining this on the full subcategory $\Mfd_{\con}$ of objects whose model $\cP$ has a unique minimal object.
We assign an object $\bX = (X,\cP,U,\psi)$ to the pair $(TX \oplus U^{-}, U^{+})$.
We assign a morphism $f\colon \bX_{0} \to \bX_{1}$ to the map covering $\bR^{Q_{f}^{+}}$ which uses the equivalence $U_{0}^{+}\oplus \bR^{Q_{f}^{+}} \simeq U_{1}^{+}$ which is part of the underlying morphism in $\Mod$, and the map of vector bundles
\begin{align*}
	If\colon TX_{0} \oplus U_{0}^{-} \oplus \bR^{Q_{f}^{+}} \simeq TX_{0} \oplus U_{1}^{-} \oplus \bR^{Q_{f}} \xrightarrow{\psi_{1} \oplus id} TX_{1} \oplus U_{1}^{-} \,.
\end{align*}
Because we require morphisms to intertwine normal framings, we have commutative squares
\begin{equation*}
	\begin{tikzcd}
		T\bX \oplus U^{-} \oplus  \bR^{Q_{f}^{+}} \oplus \bR^{Q_{g}^{+}} \arrow[r, "If \oplus id"] \arrow[d]      &   T\bY \oplus V^{-} \oplus \bR^{Q_{g}^{+}}  \arrow[d, "Ig"] \\
		T\bX \oplus U^{-} \bR^{Q_{g\circ f}}  \arrow[r, "I(g\circ f)"']               &   T\bW \oplus W^{-} \, ,
	\end{tikzcd}
\end{equation*}
so $I$ is compatible with composition up to the 2-cell $\bR^{Q_{g\circ f}^{+}} \simeq \bR^{Q_{g}^{+}} \oplus \bR^{Q_{f}^{+}}$. To get the monoidal structure, we use the equivalence $T(X\times Y) \simeq TX\boxplus TY$ and the equivalences of vector spaces that define product in $\Mod$.
Morphisms and products are compatible up to the 2-cell $\bR^{Q_{f\times g}^{+}} \simeq \bR^{Q_{f}^{+}} \oplus \bR^{Q_{g}^{+}}$ because we equip $\bX\otimes \bY$ with the product normal framing.

We will abuse notation and also write $I$ for the composite monoidal functor
\begin{equation*}
    I \colon \Mfd^{\otimes}_{\con} \to \cS_{/\bZ\times BO}^{\otimes} \,.
\end{equation*}
Because the monoidal structure in both source and target of $I$ are compatible with coproducts, and any object in $\Mfd_{\con}$ can be written as the coproduct over its connected components, the monoidal functor $I$ admits an operadic Kan extension, which itself is a strongly monoidal functor
\begin{align*}
	I\colon \Mfd^{\otimes} \to \cS_{/\bZ\times BO} \, ,
\end{align*}
and whose underlying functor preserves coproducts.
Because we shall only need objects of $\Mod_{\con}$ for the present paper, we do not spell out the details of this argument.
Note that the composite $\Mfd^{\otimes} \to \cS_{/\bZ\times BO}^{\otimes} \to \cS^{\times}$ factors through $\Pi\colon \Top^{\times} \to \cS^{\times}$ by the obvious forgetful functor $\Mfd^{\otimes} \to \Top^{\times}$.
As a virtual vector bundle, we can think of $I\bX$ as the difference $TX \ominus U$, where $U= U^{+}\ominus U^{-}$. We will sometimes use this notation as a reminder when we reason informally.

\begin{ex}\label{ex:Mfd_/E}
	Let $E$ be an object in $\cS_{/\bZ\times BO}$, and consider the pullback of $\infty$-categories
	\begin{equation*}
		\begin{tikzcd}
			\Mfd_{/E} \arrow[r] \arrow[d]      &   \left(\cS_{/\bZ\times BO} \right)_{/E}  \arrow[d] \\
			\Mfd   \arrow[r, "-I"]               &   \cS_{/\bZ\times BO} \,.
		\end{tikzcd}
	\end{equation*}
	The fiber over some $\bX$ is the space of shifted $E$-structures $U\ominus T\bX \to E$.
  However, we now have a way to pull back $E$-structures along maps that are not codimension-zero embeddings.
	Namely, we can transport along the inclusion $\del^{\sigma}\bX\to \bX$ of a stratum to obtain
	\begin{align*}
		U \ominus \bR^{Q(\sigma)} \ominus T\del^{\sigma} \bX  \to U\ominus TX \to E \,.
	\end{align*}
\end{ex}

\subsection{Double categories of structured manifolds with corners}\label{subsec:double_categories_of_structured_manifolds}

For the remainder of this section, we fix an $\infty$-category $\cC$ and a functor $\mu \colon \cC\to U/O \simeq B(\bZ \times BO)$.
According to \Cref{cor:monoid_equivalence}, this corresponds uniquely to an $\sfE(\bZ\times BO)$-module in $\Cat_{\infty}^{\cS_{/\bZ\times BO}}$, which we by abuse of notation will also denote by $\mu$.
We can construct a slice double $\infty$-category $(\cS_{/\bZ\times BO})_{/\mu}$ as in \Cref{def:slice_double_cat}.

\begin{definition} \label{def:Mfd_mu}
	Define $\Mfd^{\mu}$ by the pullback
	\begin{equation*}
		\begin{tikzcd}
			\Mfd^{\mu} \arrow[r] \arrow[d] & (\cS_{/\bZ\times BO})_{/\mu} \arrow[d] \\
			\Mfd^{\otimes} \arrow[r,"-I"] & \cS^{\otimes}_{/\bZ \times BO}
		\end{tikzcd}
	\end{equation*}
	in $\Dbl_{\infty}$.
\end{definition}

\begin{rem}\label{rem:unwinding_Mfd_mu}
	Let us try to unpack this definition and understand the structure of the double $\infty$-category $\Mfd^{\mu}$.
	By \Cref{lem:Mon->Dbl->Cat}, we can compute the pullback fiberwise over $\Simp^{\op}$, so in particular the space of objects of $\Mfd^{\mu}$ is the same as the space of objects in the upper right hand corner of the diagram, i.e. $\ob(\mu)$ which is a principal $(\bZ\times BO)$-bundle over $\cC^{\simeq}$.
	Likewise, we can compute mapping $\infty$-categories in $\Mfd^{\mu}$ as pullbacks of mapping $\infty$-categories,
	For a pair of objects $c$ and $d$, we then get by \Cref{rem:unwinding_slice} that the mapping $\infty$-category $\Mfd^{\mu}(c,d)$ is precisely the $\infty$-category $\Mfd_{/\mu(c,d)}$ of \Cref{ex:Mfd_/E}.

	\begin{enumerate}
		\item An object in $\Mfd^{\mu}(c,d)$ is specified by a stratified manifold $\bX$, a virtual vector space $U$, and a shifted $\mu(c,d)$-structure $U\ominus T\bX \to \mu(c,d)$.

		\item  A morphism $(\bX_0,U_0, \phi_{0}) \to (\bX_1,U_1, \phi_{1})$ in $\Mfd^{\mu}(c,d)$ is given by:
		\begin{enumerate}
			\item A morphism $f\colon \bX_0 \to \bX_1$ of stratified manifolds with corners.
			\item A stable equivalence $g\colon U_0\oplus \bR^{Q_f(\sigma)} \to U_1$.
			\item A homotopy between the $\mu(c,d)$-structures $\phi_{0}$ and $f^{*}\phi_{1} \circ g$ on $X_{0}$.
		\end{enumerate}
		\item The higher cells encode homotopy-coherent composition of this data.
	\end{enumerate}
	The horizontal composition
	\begin{align*}
		\odot\colon \Mfd^{\mu}(c,d)\times \Mfd^{\mu}(d,e) \to \Mfd^{\mu}(c,e)
	\end{align*}
	of objects $(X_0,U_0,\phi_0)$ and $(X_1,U_1,\phi_{1})$ is given by the data
	\begin{align*}
		X_{01} &= X_0\times X_1 \\
		U_{01} &= U_0 \oplus U_1 \\
	\end{align*}
	and the shifted $\mu(c,e)$-structure
	\begin{align*}
		U_{01}\ominus TX_{01} \xrightarrow{\phi_{0}\boxplus \phi_{1}} \mu(c,d)\boxplus\mu(d,e) \xrightarrow{\circ} \mu(c,e) \,.
	\end{align*}
\end{rem}

We will distinguish two kinds of morphisms between stratified manifolds, based on their behavior on stratifying categories.

\begin{definition}\label{def:closed_embedding}
	A functor $f\colon \cP\to \cR$ between models for manifolds with corners is called a \emph{closed embedding} if
	\begin{enumerate}
		\item $f$ is fully faithful, and
		\item the image of $f$ is a right ideal, i.e. if $\sigma \in \cP$ and there exists a morphism $f(\sigma) \to \gamma$ in $\cR$, then $\gamma$ is in the image of $f$.
	\end{enumerate}
\end{definition}

\begin{rem}\label{rem:unwinding_closed_embedding}
	Let $f\colon \cP \to \cR$ be a closed embedding.
	Since $f$ is fully faithful, so is the induced functor on slices $f\colon \cP_{\sigma/} \to \cR_{f(\sigma)/}$.
	The right ideal condition implies that each such is also essentially surjective, and hence an equivalence. Applying this at the minimal objects of $\cP$, we can factor $f$ as
	\begin{align*}
		\cP \simeq \coprod_{\sigma \in \pi_{0}\cP} \del^{f(\sigma)} \cR \hookrightarrow \cR \,.
	\end{align*}
	This means that a morphism $F\colon (\bX,U) \to (\bY,V)$ in $\Mfd$ which induces the functor $f$ on stratifying categories factors as
	\begin{align*}
		\bX \simeq \coprod_{\sigma \in \pi_{0}\cP} \del^{f(\sigma)} \bY \hookrightarrow \bY \,.
	\end{align*}
	Furthermore, the strata $f(\sigma)$ must all have the same codimension, namely $\dim(V-U)$.
\end{rem}

\begin{definition}\label{def:open_embedding}
	A functor $e\colon  \cP\to \cR$ between models of manifolds with corners is called an \emph{open embedding} if
	\begin{enumerate}
		\item $e$ is fully faithful,
		\item every minimal object of $\cR$ is in the image of $e$, and
		\item the image of $e$ is a left ideal, i.e. if $\sigma\in \cP$ and there exists a morphism $\gamma \to e(\sigma)$ in $\cR$, then $\gamma$ is in the image of $e$.
	\end{enumerate}
\end{definition}

\begin{rem} \label{rem:open_embedding}
	Dual to the case of closed embeddings, open embeddings induce equivalences on all overcategories. This means that open embeddings preserve codimension, so minimal objects of $\cP$ are mapped to minimal objects of $\cR$.
	The surjectivity condition then implies that a map $(\bX,U) \to (\bY, V)$ in $\Mfd$ covering an open embedding $e$ must have $\dim(U)=\dim(V)$, and must induce a diffeomorphism of manifolds with corners $X\simeq Y$.
	We think of such a map as a relabeling of $\bX$, where the strata corresponding to $\gamma\notin \im(e)$ are set to the empty manifold.
\end{rem}

\begin{definition}\label{def:Mfd_subcategories}
	Let $\Mod_{\diamond}$ be the wide subcategory of $\Mod$ spanned by closed embeddings.
\end{definition}

Because subcategories are stable under pullback, $\Mod_{\diamond}$ pulls back to a wide monoidal subcategory $\Mfd^{\otimes}_{\diamond} \subset \Mfd^{\otimes}$, and a double functor $\Mfd_{\diamond}^{\mu}\to \Mfd^{\mu}$, which on morphism $\infty$-categories is given by wide subcategory inclusions.

\begin{lem}\label{lem:fiberwise_right_fib}
	For each pair of objects $c,d \in \ob(\mu)$, the functor
	\begin{align*}
		\Mfd_{\diamond}^{\mu}(c,d) \to \Mod_{\diamond}
	\end{align*}
	is a right fibration.
\end{lem}

\begin{proof}
	We factor each such functor as
	\begin{align*}
		\Mfd_{\diamond}^{\mu}(c,d) \xrightarrow{F} \Mfd_{\diamond} \xrightarrow{H} \Mod_{\diamond}.
	\end{align*}
	The first factor is pulled back from the representable right fibration $(\cS_{/\bZ\times BO})_{/\mu(c,d)} \to \cS_{/\bZ\times BO}$, and is therefore also a right fibration.
  Given a stratified manifold $(\bX, \cR, V)$ and a morphism $(\cP,U)\to (\cR,V)$ such that $f\colon \cP \to \cR$ is a closed embedding, consider the $\cP$ stratified manifold
	\begin{align*}
		\del^{f} \bX = \bigcup_{\sigma \in \cP} \del^{f(\sigma)} \bX
	\end{align*}
	and the natural inclusion $\del^{f}\bX \to \bX$. This, together with the other data of the morphism $(\cP,U) \to (\cR, V)$ determine a cartesian lift. Moreover, if $f\colon \bX \to \bY$ is a morphism of stratified manifolds with corners covering an isomorphism $\cP_{X} \cong \cP_{Y}$, we must have that $f$ is a diffeomorphism, and hence invertible in $\Mfd_{\diamond}$.
	This shows that $H$ is a Cartesian fibration with all fibers groupoids, and hence a right fibration.
\end{proof}

\begin{cor}\label{cor:Alg(Mfd)->Alg(Mod)_is_right_fib}
	Composition with $\Mfd_{\diamond}^{\mu} \to \Mod_{\diamond}$ induces a right fibration
	\begin{align*}
		\Alg(\Mfd_{\diamond}^{\mu}) \to \Alg(\Mod_{\diamond}) \,.
	\end{align*}
\end{cor}

\begin{proof}
	By \Cref{lem:fiberwise_right_fib,lem:double_right_fib_condition}, the double functor in question is a right fibration, so the result follows from \Cref{prop:algebra_right_fib}.
\end{proof}

\begin{definition}
	Let $\Mod^{[1]_{\circ}} \subset \Mod^{[1]}$ denote the full subcategory spanned by the open embeddings.
	Because this is closed under monoidal product, we get a monoidal subcategory $(\Mod^{\otimes})^{[1]_{\circ}} \subset (\Mod^{\otimes})^{[1]}$.
\end{definition}

\begin{definition}
	Let $(\Mod_{\diamond}^{\otimes})^{[1]_{\circ}}$ be defined by the pullback
	\begin{equation*}
		\begin{tikzcd}
			(\Mod^{\otimes}_{\diamond})^{[1]_{\circ}} \arrow[r] \arrow[d]      &   (\Mod^{\otimes})^{[1]_{\circ}}  \arrow[d, "ev_{0}\times ev_{1}"] \\
			\Mod^{\otimes}_{\diamond} \times_{\Simp^{\op}} \Mod^{\otimes}_{\diamond}  \arrow[r]               &   \Mod^{\otimes}\times_{\Simp^{\op}} \Mod^{\otimes}
		\end{tikzcd}
	\end{equation*}
\end{definition}
in $\Dbl_{\infty}$.
Pulling this constructions back along $\Mfd^{\mu}\to \Mod^{\otimes}$ gives a double functor $\left(\Mfd_{\diamond}^{\mu} \right)^{[1]_{\circ}} \to \left(\Mfd^{\mu}\right)^{[1]_{\circ}}$.

\begin{lem}\label{lem:relabeling_prelim}
	The square
	\begin{equation}\label{eq:relabeling_prelim}
		\begin{tikzcd}
			(\Mfd_{\diamond}^{\mu})^{[1]_{\circ}} \arrow[r] \arrow[d, "ev_{0}"]      &   (\Mod_{\diamond}^{\otimes})^{[1]_{\circ}}  \arrow[d, "ev_{0}"] \\
			\Mfd_{\diamond}^{\mu}  \arrow[r]               &   \Mod_{\diamond}^{\otimes}
		\end{tikzcd}
	\end{equation}
  is a pullback in $\Dbl_{\infty}$.
\end{lem}
\begin{proof}
	This is immediately clear on spaces of objects, so it suffices to show on morphism $\infty$-categories for any pair of objects $c,d\in \ob(\mu)$.
	Consider the cube obtained by including every $\infty$-category in \eqref{eq:relabeling_prelim} into the version without $\diamond$.
	Using the pasting law in this cube, we see that it suffices to show that
	\begin{equation*}
		\begin{tikzcd}
			(\Mfd^{\mu}(c,d))^{[1]_{\circ}} \arrow[r] \arrow[d]      &   \Mod^{[1]_{\circ}}  \arrow[d] \\
			\Mfd^{\mu}(c,d)  \arrow[r]               &   \Mod
		\end{tikzcd}
	\end{equation*}
	is a pullback square. This amounts to showing that every arrow in $\Mfd^{\mu}(c,d)$ covering an open embedding is coCartesian.
  By the 2-out-of-3 property for coCartesian lifts, we can check this for each of the functors
	\begin{align*}
		\Mfd^{\mu}(c,d) \xrightarrow{F} \Mfd \xrightarrow{H} \Mod.
	\end{align*}
	Let $f\colon \bX \to \bY$ be a morphism in $\Mfd$ covering an open embedding $e\colon P\to \cR$. Consider any morphism $g\colon \bX \to \bW$ with a factorization of the underlying map of models as $r\circ e$ for some functor $r\colon \cR \to \cT$. Because $e$ is an open embedding, it preserves minimal strata, so $f$ induces diffeomorphisms $\del^{\sigma}\bX \to \del^{e(\sigma)}\bY$ between minimal strata.
  To lift $r$ to a map $\bY \to \bW$ we simply invert each such diffeomorphism and compose with~$g_{\sigma}$.
  This factorization is clearly unique, so $f$ is a coCartesian morphism.
  The functor~$F$ is pulled back from $(\cS_{/ (\bZ\times BO)})_{/\mu(c,d)} \to \cS_{/ \bZ\times BO}$, and since a morphism in $\Mfd^{\mu}(c,d)$ covering an open embedding gives rise to a diffeomorphism of underlying manifolds, its image in $(\cS_{/ (\bZ\times BO)})_{/\mu(c,d)}$ is an equivalence, which is necessarily coCartesian.
\end{proof}

\begin{cor}\label{cor:relabeling}
	A functor
	\begin{align*}
		F \colon \calD \to \Alg_{\Cat}\left( (\Mod_{\diamond}^{\otimes})^{[1]_{\circ}} \right)
	\end{align*}
	gives rise to a fully faithful map of right fibrations over $\calD$
	\begin{align*}
		\epsilon \colon F^{*}_{0}\Alg_{\Cat}(\Mfd^{\mu}_{\diamond}) \to F^{*}_{1}\Alg_{\Cat}(\Mfd^{\mu}_{\diamond}) \,.
	\end{align*}
	A $\Mfd^{\mu}_{\diamond}$-category $\bX$ lifting $F_{1}(D)$ belongs to the image if and only if for every $p,q \in \ob(F_{1}(D))$, the stratum $\del^{\sigma}\bX(p,q)$ is empty unless $\sigma$ is in the image of the open embedding $F_{0}(D)(p,q)\to F_{1}(D)(p,q)$.
\end{cor}
\begin{proof}
	We apply the pullback-preserving functor $\Alg_{\Cat}$ to the square in \Cref{lem:relabeling_prelim}.
  When we pull back along $F$, this gives a span
	\begin{align*}
		F_{0}^{*}\Alg_{\Cat}(\Mfd^{\mu}_{\diamond}) \xleftarrow[\simeq]{ev_{0}} F^{*}\Alg_{\Cat}( (\Mfd^{\mu}_{\diamond})^{[1]_{\circ}}  ) \xrightarrow{ev_{1}} F_{1}^{*}\Alg_{\Cat}(\Mfd^{\mu}_{\diamond})
	\end{align*}
	of right fibrations over $\calD$.
  Consider the square
	\begin{equation*}
		\begin{tikzcd}
			(\Mfd^{\mu}_{\diamond})^{[1]_{\circ}} \arrow[r] \arrow[d,"ev_{1}"]      &   \Mfd_{\diamond}^{\mu}  \arrow[d] \\
			(\Mod_{\diamond}^{\otimes})^{[1]_{\circ}}  \arrow[r, "ev_{1}"]               &  \Mod_{\diamond}^{\otimes}
		\end{tikzcd}
	\end{equation*}
  of double $\infty$-categories.
	This becomes a pullback on spaces of objects.
  For any pair of objects $c,d\in \ob(\mu)$, the induced functor
	\begin{align*}
		\Mfd^{\mu}_{\diamond}(c,d)^{[1]_{\circ}} \to \Mod_{\diamond}^{[1]_{\circ}}\times_{ \Mod_{\diamond}} \Mfd_{\diamond}^{\mu}
	\end{align*}
	is fully faithful, with image over some open embedding $e\colon \cP \to \cR$ spanned by the $\cR$-stratified manifolds with empty strata outside the image of $e$.
\end{proof}

\subsection{Categories of directed arcs}\label{subsec:categories_of_directed_arcs}
Flow categories are $\Mfd_{\diamond}^{\mu}$-categories whose underlying $\Mod_{\diamond}^{\otimes}$-category has the structure that we expect from Morse theory.
The goal of this subsection is to make this statement precise.
The definitions here are closely related to those of \cite[Section 3.1]{AB} with the monoid $\Gamma$ used \emph{loc.cit.} set to the trivial monoid.

Recall that a partially ordered set, or \emph{poset} for short, is a set $P$ with a binary relation~$\leq$ satisfying transitivity, reflexivity, and antisymmetry.
This is equivalently a complete category enriched in the full monoidal subcategory $[1]\subset \Set$ spanned by $\emptyset$ and $\ast$.
We will take our flow categories to come with an a priori partial order on their set of objects, which will satisfy that $\bX(p,q)$ is empty unless $p\leq q$ in the partial order.
The equivalence class of a flow category in $\Flow^{\mu}$ will not depend on the specific partial order chosen, and by \cite[Remark 1.4]{AB}, any flow category à la Abouzaid--Blumberg can be equipped with a minimal such partial order.
For $\Gamma = 0$, part of the Abouzaid--Blumberg definition of a flow category \cite[Definition 3.4]{AB} is that for all $p\in \ob(\bX)$, the union
\begin{align}\label{eq:compactness}
	\coprod_{q\in \ob(\bX)} \bX(p,q)
\end{align}
is compact. Since we are not working with derived manifolds in our setup, we can assume that each $\bX(p,q)$ is compact.
Then \eqref{eq:compactness} is compact if and only if for a fixed~$p$ there are at most finitely many $q$ such that $\bX(p,q)$ is non-empty.
We can therefore implement the compactness condition as a condition on the poset structure on $\ob(\bX)$.

\begin{definition}\label{def:filter}
	A subset $F\subset P$ of a poset is called a \emph{filter} if for any $p\in F$ and  $q\in P$ such that $p\leq q$ we have that $q\in F$.
	We say that $F$ is a \emph{finite filter} if its underlying set is finite.
\end{definition}

The most elementary example of a filter is the principal filter
\[
P \uparrow p = P_{p\leq} = P_{p/} = \left\{ q\in P \mid p\leq q\right\} \,.
\]
\begin{definition}\label{def:locally_finite}
	A poset $P$ is \emph{locally finite} if the principal filter $P_{p\leq}$ is finite for all $p\in P$.
\end{definition}
One way to see that local finiteness is a reasonable condition is that it will control generation in $\Flow^{\mu}$.
Write $\cR_{\Fin}(P)$ for the poset of finite right ideals of $P$ under inclusion.
Because a finite union of finite filters of $P$ is a finite filter, the category $\cR_{\Fin}(P)$ is filtered.
When $P$ is locally finite, any $p\in P$ belongs to a finite filter, so we can write
\begin{align*}
	P \simeq \colim_{I\in \cR_{\Fin}(P)} I \,.
\end{align*}
 This will be reflected in $\Flow^{\mu}$ by being able to write any flow category $\bX$ as a filtered colimit over finite subcategories, giving $\Flow^{\mu}$ the desirable property of \emph{accessibility}.

A \emph{directed arc} $\gamma$ is a rooted tree such that every vertex is adjacent to exactly two edges. We concatenate arcs $\gamma$ and $\gamma'$ by removing the root of $\gamma$ and the leaf of $\gamma'$ and then identifying the root edge of $\gamma$ with the leaf edge of $\gamma'$.

\begin{definition}\label{def:A_P}
	Given a poset $P$, define a strict 2-category $\sfA_{P}$ with set of objects $P^{\simeq}$ and morphism categories $\sfA_{P}(p,q)$ defined by:
	\begin{enumerate}
		\item Objects are connected directed $\gamma$ with edges labeled by elements of $P$ such that the incoming and outgoing edges are labeled by $p$ and $q$ respectively, and the edge labels appear in strictly increasing order.
		\item We have a morphism $\gamma \to \gamma'$ if the arc $\gamma$ can be obtained from $\gamma'$ by collapsing inner edges.
		\item If $p=q$, we allow the empty arc, which can be seen as an artificially added unit.
	\end{enumerate}
	The horizontal composition $\#\colon \sfA_P(p,r) \times \sfA_P(r,q) \to \sfA_P(p,q)$ is given by concatenating arcs.
\end{definition}
\begin{rem}
	If $P$ is totally ordered, the category $\sfA_{P}(p,q)$ is equivalent to the poset of subsets of the interval $(p,q)\subset P$, with a subset $S\subset P$ corresponding to the arc whose internal edges are labeled by the elements of $S$ in increasing order.
  This is the canonical stratifying category of $[0,\infty)^{ \left\{p+1,\dots, q-1 \right\} }$, with a subset $S$ corresponding to the stratum where the coordinates in $S$ vanish.
  The concatenation of arcs then corresponds to the stratum inclusion
	\begin{align*}
		[0,\infty)^{ \left\{p+1,\dots, r-1 \right\} } \times [0,\infty)^{ \left\{r+1,\dots, q-1 \right\} } \to [0,\infty)^{ \left\{p+1,\dots, q-1 \right\} }
	\end{align*}
	defined by $x_{r}=0$.
\end{rem}

Note first that $\sfA_{P}(p,q)$ is empty unless $p\leq q$.
Then observe that each category $\sfA_P(p,q)$ is a model for manifolds with corners: the overcategory of an arc $\gamma$ with sequence of labels $(p,p_1,\ldots ,p_{l},q)$ is equivalent to the poset of subsets of $\left\{ p_1,\ldots,p_{l}\right\}$.
We have a canonical identification
\begin{align*}
	Q_{\sfA_{P}(p,q)}(\gamma) = \left\{ p_1,\ldots p_{l} \right\} \,.
\end{align*}
For $p\leq q$, $\sfA_P(p,q)$ has exactly one minimal object, namely the arc $\gamma_{pq}$ with no internal edges.
Under the concatenation morphism $\#_{prq}\colon \sfA_{P}(p,r)\times \sfA_P(r,q) \to \sfA_P(q,r)$, the image of the minimal object $(\gamma_{pr}, \gamma_{rq})$ has
\begin{align*}
	Q_{\#_{prq}} = Q_{\sfA_{P}(p,q)}(\gamma_{pr} \# \gamma_{rq}) = \left\{ r\right\} \,.
\end{align*}
We decompose this as $Q_{\#_{prq}}^{-} = \left\{ r \right\}$, $Q_{\#_{prq}}^{+}=\emptyset$. We let $\sfU_P(p,q)$ denote the pair of vector spaces $(0,\bR^{ \left\{ q\right\} })$. We then lift the composition morphism $\#_{prq}$ to a map
\begin{align*}
	(\sfA_{P}(p,r), \sfU_{P}(r,q)) \otimes (\sfA_{P}(r,q), \sfU_{P}(r,q)) \to (\sfA_{P}(p,q), \sfU_{P}(p,q))
\end{align*}
in $\Mod_{\diamond}$ by using the equivalence
\begin{align}\label{eq:Flow_cat_U_iso}
	\sfU^{-}_P(p,r) \oplus \sfU^{-}_P(r,q)  &\simeq \sfU^{-}_{P}(p,q) \oplus \bR^{Q_{\#_{prq}}^{-}} \\
	\bR^{  \left\{ r\right\}} \oplus \bR^{ \left\{ q\right\}}  &\simeq \bR^{ \left\{ q\right\}} \oplus  \bR^{ \left\{ r \right\}}\,. \nonumber
\end{align}
These are clearly associative, defining a $\Mod_{\diamond}$-category $(\sfA_{P},\sfU_{P})$.

\begin{definition}\label{def:structured_flow_cat}
	A \emph{$\mu$-structured flow category} is given by a locally finite poset $P$, and a lift of $(\sfA_P,\sfU_P)$ to a $\Mfd^{\mu}_{\diamond}$-category.
	In other words, it is a diagram in $\DblLax$ of the form:
	\[
	\begin{tikzcd}
		& \Mfd_{\diamond}^{\mu} \arrow[d] \\
		\Simp_{P^{\simeq}}^{\op} \arrow[r, "(\sfA_{P} {,} \sfU_{P})"'] \arrow[ur, dotted, "\bX"] & \Mod_{\diamond}^{\otimes} \, .
	\end{tikzcd}
	\]
\end{definition}

\begin{rem}\label{rem:trivial_spaces_in_flow_cat}
	Let $\bX$ be a flow category with poset of objects $P$.
	For any $p<q$, we have that $\bX(q,p)$ is stratified by $\sfA_{P}(q,p) = \emptyset$, which is only possible if $\bX(q,p)$ is also empty.
	We have a double functor $P\to \Simp^{\op}_{P^{\simeq}}$ given by including the tuples $(p_{0},\dots, p_{n})$ such that $p_{j}\leq p_{j+1}$.
  By the above, a flow category is uniquely determined by its restriction to~$P$.
  Note that $P$ is complete, while $\Simp^{\op}_{P^{\simeq}}$ is not.
  The equivalences in the underlying Segal space of $h\Mfd^{\mu}_{\diamond}$ are the one-point manifolds with maps $\ast \to \mu(c,d)$ selecting equivalences in $\mu$.
  This implies that $\Mfd^{\mu}_{\diamond}$ is an $(\infty,2)$-category.
  The monoidal $\infty$-category $\Mod_{\diamond}$ is however not complete, because its Picard space is $\bZ\times BO$.
\end{rem}

\begin{rem}\label{rem:units_in_flow_cat}
	Let $\bX$ be a $\mu$-structured flow category, and consider the morphism object $\bX(p,p)\in \Mfd_{\diamond}^{\mu}(c_{p},c_{p})$.
  This is stratified by $\sfA_{P}(p,p) = [0]$, and the lax unital structure gives a morphism $I_{c_{p}} \to \bX(p,p)$ from the unit object in $\Mfd_{\diamond}^{\mu}(c_p,c_{p})$.
  This unit is the one point manifold structured by the map $id_{c_{p}}\colon \ast \to \mu(c_{p},c_{p})$, and the map $I_{c_{p}} \to \bX(p,p)$ must be a diffeomorphism over the unique stratum. In practice, we may therefore think of $\bX$ as being a non-unital category, with morphism objects defined only for $p<q$.
\end{rem}

\begin{rem}\label{rem:unwinding_flow_cat}
	Let us use \Cref{rem:unwinding_enriched_cat} to unpack the structure of a flow category so that the reader may compare our definition to~\cite[Definition 3.4 and Definition 3.8]{AB}.
  A $\mu$-structured flow category contains the following data:
	\begin{enumerate}
		\item For each $p\in P$, an object $c_p$ of $\mu$.
		\item For each $p,q\in P$, an object $\bX(p,q)$ of $\Mfd^{\mu}_{\diamond}(c_{p},c_{q})$.
    This is determined by a stratified manifold with corners $(X(p,q), \sfA_P(p,q) )$, and an object $\phi(p,q)$ of $(\cS_{/ \bZ\times BO})_{/\mu}$ of the form
		\begin{align*}
			\phi(p,q)\colon   -(TX(p,q) \oplus \bR^{ \left\{ q\right\}}) \to \mu(c_{p},c_{q}) \,.
		\end{align*}
		\item For each $p,r,q \in P$, a morphism $\bX(p,r)\odot \bX(r,q) \to \bX(p,q)$ in $\Mfd^{\mu}_{\diamond}(c_{p},c_{q})$.
    This amounts to the following data:
		\begin{enumerate}
			\item A map $ \#\colon X(p,r)\times X(r,q) \to X(p,q)$ of manifolds with corners covering $\#\colon \sfA(p,r)\times \sfA(r,q) \to \sfA(p,q)$ and inducing diffeomorphisms onto the stratum labeled by the arc with a single internal edge labeled by $r$.
			\item A map $\phi(p,r) \boxplus \phi(r,q) \to \phi(p,q)$ in $(\cS_{/\bZ\times BO})_{/\mu(c_{p},c_{q})}$, which is equivalently given by a commutative square of the following form in $\cS_{/\bZ\times BO}$:
			\begin{equation*}
				\begin{tikzcd}[column sep=huge]
					(\sfU_{P}(p,r) \ominus T\bX(p,r) ) \boxplus ( \sfU_{P}(r,q) \ominus T\bX(p,r) ) \arrow[r,"{\phi(p,r)\boxplus \phi(r,q)}"] \arrow[d, "\#"]      &   \mu(c_{p},c_{r}) \boxplus \mu(c_{r},c_{q})  \arrow[d, "\circ"] \\
					\sfU_{P}(p,q) \ominus T\bX(p,q)  \arrow[r, "\phi(p{,}q)"]               &   \mu(c_{p},c_{q}) \,.
				\end{tikzcd}
			\end{equation*}
		\end{enumerate}
	\end{enumerate}
    For longer tuples in $P$, one has to specify homotopies witnessing associativity of composition. As noted in \Cref{rem:units_in_flow_cat}, the unital structure is uniquely determined.
\end{rem}

\begin{rem}
	We define an unstructured flow category to be a lift of $\sfA_{P}$ to a $\Mfd_{\diamond}$-category.
  Any unstructured flow category $\bX$ determines a $(\cS_{/ \bZ\times BO})$-category $I\bX$.
  By definition, a lift of $\bX$ to $\Mfd_{\diamond}^{\mu}$ is then the same as a \emph{$\mu$-structure on $\bX$}, i.e. a map $-I\bX \to \mu$ of $\cS_{/ \bZ\times BO}$-categories.
  The $\cS_{/ \bZ\times BO}$-category $\mu$ is defined using the right adjoint of \Cref{lem:monoid_right_adjoint_1}.
  The $\cS$-category $\Pi\bX$ is complete, so by adjunction, a $\mu$-structure on $\bX$ is the same as a commutative diagram of $\infty$-categories
	\begin{equation*}
		\begin{tikzcd}
			\Pi\bX \arrow[r] \arrow[dr,"-I\bX"']      &   \cC  \arrow[d, "\mu"] \\
	               &  U/O\,.
		\end{tikzcd}
	\end{equation*}
\end{rem}

\begin{ex}\label{ex:empty_flow_cat}
	There is a unique $\mu$-structured flow category with empty poset of objects.
	We call this the empty flow category and denote it $\bZero$.
\end{ex}

\begin{ex}\label{ex:one_object_flow_cat}
	For the one-object poset $P=[0]$, the lax functor $\sfA_{[0]} \colon \Simp^{\op} \to \Mod_{\diamond}^{\otimes}$ selects the unit.
  A lift of this to $\Mfd_{\diamond}^{\mu}$ is uniquely determined by a choice of object $c\in \ob(\mu)$.
  We let $\bOne_{c}$ denote this lift, which we call the \emph{one-object flow category} at $c$.
\end{ex}

\begin{ex}[Morse Flow Category]\label{ex:morse_flow_cat}
	Let $f : X \to \bR$ be a Morse function on a closed smooth manifold.
  In \cite{cohen2020floer} it is shown that one can construct a framed flow category whose objects are the critical points of $f$, and whose morphism spaces are the moduli spaces $\bM_{f}(p,q)$ of broken gradient flow lines starting at $p$ and ending at $q$. We can give this flow category an \emph{$X$-structure}, i.e., a structure by the constant map at the unit $X\to U/O$ as follows.
  By \Cref{lem:monoid_right_adjoint_1}, the $\cS_{/\bZ\times BO}$-enriched category associated with $X$ is $X\otimes\sfE(\bZ\times BO)$.
  Its space of objects is $X\times \bZ\times BO$, and the morphism object $X\otimes\sfE(\bZ\times BO)( (x,V), (y,W)  )$ is the trivial bundle $W\ominus V$ on the space of paths $\Map_{X}(x,y)$.
  Since the tangential data is completely independent from the maps to $X$, we simply equip each critical point of $f$ with the negative eigenspace of the Hessian, and each $\bM_{f}$ with the framing constructed as in \cite{cohen2020floer}.
  Consider the $\Top$-enriched category $PX$ whose objects are points in $X$, and whose morphism spaces $PX(p,q)$ are the spaces of paths parametrized by some closed interval $[0,\ell]$ for $0\leq \ell$.
	This has a strictly associative composition given by concatenation and addition $\ell + \ell'$, and we can construct a continuous functor
	\begin{align*}
		\bM_{f} \to PX
	\end{align*}
	by assigning each flow line to the corresponding path in $X$ parametrized by the change in $f$.
	When we localize from topological categories to $\infty$-categories, we simply get a map $\Pi\bM_{f} \to X$, which is the data required to lift $\bM_{f}$ to a $\mu$-structured flow category.
\end{ex}

\begin{ex}\label{ex:tautological_structure}
	For an unstructured flow category $\bX$, pushforward by the index determines a $\cS_{/ (\bZ\times BO)}$-category.
  When we complete as $\cS$-categories, this corresponds to a map $I\bX\colon \Pi \bX \to U/O$. The $\sfE (\bZ\times BO)$-module associated to this map is $(I\bX) \otimes \sfE(\bZ\times BO)$, so the unit gives $\bX$ a tautological $I\bX$-structure.
	This type of structure retains a lot of information about $\bX$, actually so much that it is not invariant under the usual choices involved in Floer or Morse theory.
  Consider for example, the Morse flow category $\bM_{f}$ of \Cref{ex:morse_flow_cat}.
  By \cite{fourel2026}, the $\infty$-category $\Pi\bM_{f}$ agrees with the \emph{exit path category} $\mathrm{Exit}(M,f)$ of the stratification on $M$ by unstable manifolds of $f$.
  In particular, $\bM_{f}$ with its $I\bM_{f}$-structure contains information about the filtration of $\bM_{f}$ by critical values of $f$.
\end{ex}

\begin{ex}\label{ex:filtered_floer_theory}
	In a less extreme variant of \Cref{ex:tautological_structure}, let $A$ be a partially ordered set.
  The structure encoded by the zero map $A\to U/O$ is that of \emph{$A$-filtered framed flow categories}.
  For a framed flow category $\bX$ with poset of objects $P$, an $A$-filtration is simply given by an order-preserving map $P\to A$.
  This type of structure gives a context for filtered Floer homotopy theory, which might be of use to applications in quantitative symplectic topology.
\end{ex}

\section{The $\infty$-category of structured flow categories}\label{sec:the_infty-category_of_structured_flow_categories}

In this section, we build the $\infty$-category of $\mu$-structured flow categories and show that it is a presentable stable $\infty$-category.
Similarly to in \cite{AB}, this is done by first constructing a semi-simplicial space $\PreFlow^{\mu}$ of flow categories and flow simplices.
We then show that this admits inner horn fillers and terminal and initial degeneracies, making it a quasi-unital inner Kan space in the sense of \Cref{def:quasi-unital}.
This means that $\PreFlow^{\mu}$ presents a well-defined $\infty$-category $\Flow^{\mu}$.
We then show that $\PreFlow^{\mu}$ lifts to a semi-simplicial $\infty$-category $\bfPreFlow^{\mu\uparrow}$ which admits companions in the sense of \Cref{def:admits_companions}.
We use these companions and a construction of mapping cones to compute some basic (co)limits in $\Flow^{\mu}$.
The result of these computations is that $\Flow^{\mu}$ is stable and presentable.

\subsection{Flow categories as double functors}\label{subsec:flow_categories_as_double_functors}

For a manifold with corners $\bX$ stratified by~$\cP$, the strata $\del^{\sigma}\bX$ assemble to a diagram $\cP^{\op} \to \Mfd_{\diamond}$.
If $\bX$ is a flow category, we could perform this procedure for all $\bX(p,q)$, and we would expect the functors $\sfA_{P}(p,q)^{\op} \to \Mfd^{\diamond}$ to assemble to a double functor $\bfA_{P} \to \Mfd_{\diamond}$, where $\bfA_{P}$ is the unstraightening of~$\sfA_{P}^{\op}$.
This perspective is useful because certain constructions require us to understand a \emph{lax colimit} of enriched categories.
These are computed by first unstraightening, and then computing an ordinary colimit.
To adapt this technique to $\Mfd_{\diamond}^{\mu}$ requires a bit more work because all the $\Mfd^{\mu}_{\diamond}(c,d)$ are $\infty$-categories.

We have a monoidal functor
\begin{align*}
	\del \colon \Mod^{\otimes}_{\diamond} \to \Cat_{/\Mod_{\diamond}}^{\otimes}
\end{align*}
which takes a model $\cP$ to the diagram $\del_{\cP} \colon \cP^{\op} \to \Mod_{\diamond}$.
This is strictly monoidal because the product in $\Mod$ is just the product of categories, and because $\del^{(\sigma,\tau)}(\cP\times \cR ) \simeq \del^{\sigma}\cP \times \del^{\tau}\cR$.
When we restrict $\del$ to the subcategory $\Mod_{\diamond,\con}$ of models with a unique minimal object, we can factor through the monoidal subcategory $\Cat_{T/\Mod_{\diamond}}$ spanned by categories with a terminal object.
Colimits over categories with terminal object are given by evaluating at the terminal object.
Such colimits always exist and are preserved by any functor, giving a commutative square
\begin{equation}\label{eq:colim_double_functor}
	\begin{tikzcd}
		\Cat_{T/\Mfd_{\diamond,\con}^{\mu}} \arrow[r, "\colim"] \arrow[d]      &   \Mfd_{\diamond,\con}^{\mu}  \arrow[d] \\
		\Cat_{T/\Mod_{\diamond,\con}}  \arrow[r, "\colim"]               &      \Mod_{\diamond,\con}^{\otimes}
	\end{tikzcd}
\end{equation}
of double $\infty$-categories.
To make this precise, we work with \emph{lax arrow categories}.
The pointwise monoidal structure on the lax arrow category gives a monoidal coCartesian fibration
\begin{align*}
	ev_{1}\colon \Ar^{\Lax}(\Cat_{\infty})^{\otimes} \to \Cat_{\infty}^{\times} \,.
\end{align*}
We restrict to the full monoidal subcategory $\Ar^{\Lax}_{T}(\Cat_{\infty})$ spanned by arrows whose source is an ordinary category with a terminal object.
The restriction of $ev_{1}$ is still a coCartesian fibration because coCartesian transport does not affect the source.
Consider also the full monoidal subcategory $\Cat_{\infty}^{\ast} \subset \Ar^{\Lax}_{T}(\Cat_{\infty})$ spanned by arrows whose source is the trivial category. The map $\Cat_{\infty}^{\ast} \to \Cat_{\infty}$ is the \emph{universal coCartesian fibration}, i.e. the unstraightening of the identity. Its fiber at any $\cC$ is $\cC$ itself.
The inclusion $\Cat_{\infty}^{\ast} \to \Ar^{\Lax}_{T}(\Cat_{\infty})$ admits a left adjoint fiberwise given by restriction to the terminal object, which is therefore a global left adjoint.
This left adjoint gets a lax monoidal structure, which in this case is strict because the terminal object of $\cC\times \calD$ is $(\ast_{\cC},\ast_{\calD})$.
We now restrict the left adjoint to the wide subcategory $\Ar_{T}(\Cat_{\infty}) \subset \Ar^{\Lax}_{T}(\Cat_{\infty})$ spanned by morphisms whose lax cell is invertible.
On straightening, this corresponds to a natural transformation $\colim \colon \Cat_{T /(-)} \Rightarrow id$ between lax monoidal functors $\Cat_{\infty}^{\times} \to \CatHat_{\infty}^{\times}$.
Pushing forward the algebras in $\Cat_{\infty}^{\times}$ defining $\Mfd_{\diamond,\con}^{\mu}$ and $\Mod_{\diamond,\con}^{\otimes}$ by this natural transformation gives \eqref{eq:colim_double_functor}.

On spaces of objects, the horizontal maps in \eqref{eq:colim_double_functor} become equivalences. When we pass to mapping categories, the vertical maps are right fibrations, and the comparison map in the fiber over some functor $F\colon \cJ \to \Mod_{\diamond,\con}$ where $\cJ$ has a terminal object $\ast$ is given by the evaluation map
\begin{align*}
	\Map_{/\Mod_{\diamond,\con}}(\cJ, \Mfd^{\mu}_{\diamond,\con}(c,d)) \xrightarrow{ev_{\ast}} \Map_{/\Mod_{\diamond,\con}}(\ast, \Mfd_{\diamond,\con}^{\mu}(c,d)) \,.
\end{align*}
Because $\Mfd_{\diamond,\con}^{\mu}(c,d) \to \Mod_{\diamond,\con}$ is a right fibration, and the inclusion $\ast \to \cJ$ is cofinal, these comparison maps are equivalences, so the square is a pullback.
By noting that the composite along the bottom of the
diagram
\begin{equation*}
	\begin{tikzcd}
		\Mfd_{\diamond,\con}^{\mu} \arrow[r, dashed, "\cF"] \arrow[d]      &   \Cat_{T/\Mfd_{\diamond, \con}^{\mu} }  \arrow[d]  \arrow[r, "\colim"] & \Mfd_{\diamond,\con}^{\mu} \arrow[d] \\
		\Mod_{\diamond,\con}^{\otimes}  \arrow[r,"\del"]               &   \Cat_{T/\Mod_{\diamond,\con}}^{\otimes}  \arrow[r, "\colim"]            & \Mod_{\diamond,\con}^{\otimes}
	\end{tikzcd}
\end{equation*}
is equivalent to the identity, we get an induced dashed map on pullbacks, and by pasting, the left square is also a pullback.
By pasting, we therefore have a pullback square
\begin{equation*}\label{eq:canonical presheaf}
	\begin{tikzcd}
		\Mfd^{\mu}_{\diamond,\con} \arrow[r] \arrow[d]      &   \Cat_{\infty/\Mfd_{\diamond}^{\mu}}  \arrow[d] \\
		\Mod_{\diamond,\con}^{\otimes}  \arrow[r]               & \Cat_{\infty/\Mod_{\diamond}}^{\otimes} \,.
	\end{tikzcd}
\end{equation*}

\begin{rem}
	We expect that a slightly more refined argument would give a similar pullback square for $\Mfd^{\mu}_{\diamond}$, but because we shall only need models with a unique minimal object in this paper we stop here.
\end{rem}

\begin{cor}\label{cor:canonical_presheaf_pullback}
	The functor $\cF$ and unstraightening induces a pullback square
	\begin{equation*}
		\begin{tikzcd}
			\Alg_{\Cat}(\Mfd_{\diamond,\con}^{\mu} )\arrow[r] \arrow[d]      &   \Dbl_{\infty/\Mfd_{\diamond}^{\mu}}  \arrow[d] \\
			\Alg_{\Cat}(\Mod_{\diamond,\con} ) \arrow[r]                & \Dbl_{\infty/\Mod_{\diamond}^{\otimes}} \,.
		\end{tikzcd}
	\end{equation*}
\end{cor}
\begin{proof}
	Apply the pullback-preserving functor $\Alg_{\Cat}(-)$ to \eqref{eq:canonical presheaf}, and use \Cref{cor:alg_in_slice} to commute out the slices. Then compose with the fully faithful unstraightening functor $\Alg_{\Cat}(\Cat_{\infty}) \to \Dbl_{\infty}$ on slices.
  A pasting argument now finishes the proof.
\end{proof}
By pasting, we immediately get the following corollary of \Cref{cor:canonical_presheaf_pullback}.
\begin{cor} \label{cor:lax_lifts_double_functors}
	Let $\sfA$ be a $\Mod_{\diamond,\con}$-category, and $\bfA$ the double category obtained by unstraightening $\sfA^{\op}$.
  Pushforward along $\cF$ induces an equivalence between the space of dotted lifts in the following two diagrams:
	\begin{equation*}
		\begin{tikzcd}
			& \Mfd_{\diamond}^{\mu} \arrow[d] \\
			\Simp^{\op}_{\ob(\sfA)} \arrow[r, "\sf{A}"'] \arrow[ur, dotted] & \Mod_{\diamond}^{\otimes}
		\end{tikzcd}
		\in \DblLax
		\quad \text{and} \quad
		\begin{tikzcd}
			  & \Mfd_{\diamond}^{\mu} \arrow[d] \\
			\mathbf{A} \arrow[ur,dotted] \arrow[r, "\del_{\sfA}"']& \Mod_{\diamond}^{\otimes}
		\end{tikzcd}
		\in \Dbl_{\infty} \,.
	\end{equation*}
\end{cor}

\subsection{Flow simplices}\label{subsec:flow_simplices}

We will represent the $\infty$-category of $\mu$-structured flow categories by a semi-simplicial space. The 0-simplices will be $\mu$-structured flow categories in the sense of \Cref{def:structured_flow_cat}. In this section, we will define $n$-simplices for all $n$.

\begin{definition}\label{def:directed_arc_over_n}
	Let $i\colon P\to [n]$ be a map of posets.
	A \emph{directed arc labeled by} $i$ is a directed arc with the following extra data.
	\begin{enumerate}
		\item Each edge is labeled by an element of $P$ in such a way that the edge labels appear in strictly increasing order with respect to the poset structure of $P$.
		\item A vertex whose adjacent edges are labeled by $p<q$ is labeled by a subset $S\subset \{i(p)+1 , \dots, i(q)-1\}$, with the convention that this is empty whenever $i(q) \leq i(p)+1$.
	\end{enumerate}
	An \emph{edge collapse} of an arc $\gamma$ at an internal edge $e$ labeled by $r\in P$ with adjacent vertices $v_0$ and $ v_1$ labeled by $S_0$ and $S_1$, respectively, is the arc obtained by replacing the subgraph $v_0, e, v_1$ by a single vertex $v$ labeled by a subset of $S_0\cup S_1 \cup \left\{ i(r)\right\}$ with the convention that $i(r)$ is only included if it is legal.
\end{definition}

\begin{definition}\label{def:A_P_over_n}
	Let $i\colon P\to [n]$ be a map of posets.
	For a pair of elements $p,q\in P$, let $\sfA_P(p,q)$ be the poset determined by:
	\begin{itemize}
		\item Objects are directed arcs $\gamma$ labeled by $P$, such that the incoming edge is labeled by $p$ and the outgoing edge is labeled by $q$.
		\item  There is a morphism $\gamma' \to \gamma$ in $\sfA_{P}(p,q)$ if $\gamma'$ can be obtained from $\gamma$ by performing a sequence of edge collapses or if $\gamma'$ and $\gamma$ have the same edge labels and the vertex labels are inclusions.
	\end{itemize}
\end{definition}

\begin{rem}
	Note that the category $\sfA_{P}(p,q)$ depends on the map $P\to [n]$, not only on $P$. We suppress this from the notation.
\end{rem}

\begin{rem}\label{rem:A_P_comparison}
	Let us compare this to~\cite[Definition 4.2]{AB}.
	A sequence $\overrightarrow{P} = (P_0,\dots ,P_n)$ of posets can be viewed a poset over $[n]$ by simply setting
	\[
	P = P_0 \join \cdots \join P_n
	\]
	and declaring the map $i : P \to [n]$ as sending all the elements of $P_j$ to $j$.
	Our definition then almost agrees with the one of \cite{AB}, with the extra requirement that the $P_j$ are posets, as opposed to simply sets, and that the poset structure of the $P_j$s are respected in the labeling of the edges.
	We also allow the arc with no vertices, which should be seen as an artificially added unit.
	The value of a flow category at this identity arc will essentially be forced to be the one-point manifold, and so one can safely ignore it.
\end{rem}

\begin{ex}
	Let us consider a poset $i : P \to [4]$ over $[4]$ and the directed arc $\gamma$ labeled by $P$ that looks like
	
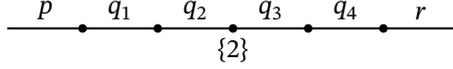
\begin{center}
		\begin{tikzpicture}
			\coordinate (start) at (-3 ,0);
			\coordinate[label= above:$p$] (p) at (-2.5,0);
			\coordinate[label= above:$q_1$] (q1) at (-1.5,0);
			\coordinate[label= above:$q_2$] (q2) at (-0.5,0);
			\coordinate[label= above:$q_3$] (q3) at (0.5,0);
			\coordinate[label= above:$q_4$] (q4) at (1.5,0);
			\coordinate[label= above:$r$] (r) at (2.5,0);
			\coordinate (end) at (3,0);
			\draw[thick] (start) -- (end);
			\coordinate (v1) at (-2,0);
			\draw[fill, color=black] (v1) circle (0.05);
			\coordinate (v2) at (-1,0);
			\draw[fill, color=black] (v2) circle (0.05);
			\coordinate[label= below:$\{2\}$] (v3) at (0,0);
			\draw[fill, color=black] (v3) circle (0.05);
			\coordinate (v4) at (1,0);
			\draw[fill, color=black] (v4) circle (0.05);
			\coordinate (v5) at (2,0);
			\draw[fill, color=black] (v5) circle (0.05);
		\end{tikzpicture}
	\end{center}
	where $i(p) = 0$, $i(q_1) = i(q_2)=1$, and $i(q_3) =  i(q_4) = i(r) =4$.
	If we collapse the edge labeled by $q_2$, the new label can be any subset of $\{2\}$, since adding $i(q_2) = 1$ to the label would not produce a legal directed arc.
	However, if we subsequently also collapse the edge labeled by $q_1$, the new label can be any subset of $\{1,2\}$ since this is indeed a legal directed graph.
\end{ex}

As before, we can concatenate arcs, giving strictly associative composition functors
\begin{align*}
	\#\colon \sfA_P(p,r)\times \sfA_P(r,q) \to \sfA_P(p,q) \,.
\end{align*}
Each $\sfA_P(p,q)$ is a model for manifolds with corners.
Indeed, let $\gamma$ be an arc with edges labeled by the sequence $(p_{0},p_1,\ldots, p_s)$ where $p_0=p$ and  $p_s=q$, and we have that $i(p_{i})= k_i$.
Write $S_i$ for the subset labeling the vertex between $p_i$ and $p_{i+1}$.
Then the overcategory at $\gamma$ is equivalent to the poset of subsets of
\begin{align*}
	Q_{\sfA_P(p,q)}(\gamma) = \left\{ p_{1},\ldots p_{s-1} \right\} \cup  \left( \left\{ j+1,\ldots \ell -1 \right\} \setminus \bigcup_{i} (S_{i}\cup \left\{ k_i \right\})  \right).
\end{align*}
We decompose this by setting $Q^{-}_{\sfA_{P}(p,q)}(\gamma) = \left\{ p_1,\ldots , p_{s-1}\right\}$. The minimal element in $\sfA_P(p,q)$ is the arc $\gamma_{pq}$ with two edges labeled by $p$ and $q$ respectively, and the vertex between them labeled by $\left\{i(p)+1,\ldots i(q)-1 \right\}$.
Under the concatenation map, we have
\begin{align*}
	Q_{\sfA_P(p,q)} (\gamma_{pr} \# \gamma_{rq} ) = \left\{ r \right\} \,.
\end{align*}
To lift $\sfA_{P}$ to a $\Mod_{\diamond}$-category, we eqip $\sfA_{P}(p,q)$ with the pair of vector spaces $U_P(p,q) = \left(\bR^{ \left\{ i(p)+1, \ldots, i(q) \right\}}, \bR^{ \left\{ q\right\}} \right)$.
We lift the composition map by using the linear isomorphism
\begin{align}\label{eq:Flow_simplex_U_iso}
	\sfU^{-}_{P}(p,r)
	\oplus \sfU^{-}_{P}(r,q)
	&\simeq \sfU^{-}_{P}(p,q) \oplus \bR^{Q_{\#}^{-}} \\
	\bR^{ \left\{ r\right\}}
	\oplus \bR^{ \left\{ q\right\}}
	&\simeq  \bR^{ \left\{ q\right\}} \oplus \bR^{ \left\{ r\right\}}  \nonumber
\end{align}
on the negative part, and
\begin{align*}
	\sfU^{+}_{P}(p,r) \oplus \sfU^{+}_{P}(r,q) \oplus \bR^{Q_{\#}^{+}} &\simeq \sfU^{+}_{P}(p,q) \\
	\bR^{ \left\{ i(p)+1,\dots,i(r)\right\}} \oplus \bR^{ \left\{ i(r)+1, \dots, i(q)\right\}} \oplus 0 &\simeq \bR^{ \left\{ i(p)+1,\dots i(q) \right\}}
\end{align*}
on the positive part.
These composition maps are strictly associative, defining a $\Mod_{\diamond}$-category $(\sfA_{P},\sfU_{P})$.

\begin{definition}\label{def:structured_flow_simplex}
	A \emph{$\mu$-structured flow $n$-simplex} is given by an order-preserving map $i : P\to [n]$ from a locally finite poset $P$, and a lift $\bX$ of $\sfA_{P}$ to a $\Mfd_{\diamond}^{\mu}$-category.
	In other words, a $\mu$-structured flow $n$-simplex is a diagram in $\DblLax$ of the form:
	\begin{equation*}
		\begin{tikzcd}
			& \Mfd_{\diamond}^{\mu} \arrow[d] \\
			\Simp_{P^{\simeq}}^{\op} \arrow[ur, dotted, "\bX"] \arrow[r, "\sfA_{P}"'] & \Mod_{\diamond}^{\otimes} \,.
		\end{tikzcd}
	\end{equation*}
\end{definition}

\begin{rem}\label{rem:flow_simplex_as_double_functor}
	Because each $\sfA_{P}(p,q)$ has a unique minimal element, we can view~$\sfA_{P}$ as a $\Mod_{\diamond,\con}$-category.
	Let $\bfA_{P} \in \Dbl_{\infty}$ denote the unstraightening of the algebra $\sfA_{P}^{\op}$.
	Then by \Cref{cor:lax_lifts_double_functors}, a $\mu$-structured flow $n$-simplex with objects $P\to [n]$ is equivalently a diagram in $\Dbl_{\infty}$ of the form:
	\begin{equation*}
		\begin{tikzcd}
			 & \Mfd_{\diamond}^{\mu} \arrow[d] \\
		\bfA_{P} \arrow[ur, dotted, "\bX"] \arrow[r, "\del_{\bfA_{P}}"']	& \Mod_{\diamond}^{\otimes} \,.
		\end{tikzcd}
	\end{equation*}
\end{rem}

\begin{rem}\label{rem:zero_simplex_is_flow_cat}
	Note that a $\mu$-structured flow $0$-simplex is a $\mu$-structured flow category as in the previous sense.
\end{rem}
\begin{rem}
	Just as for flow categories, a simplex $\bX$ is uniquely determined by its restriction along $P\to \Simp^{\op}_{P^{\simeq}}$.
\end{rem}

\begin{ex}[$\mu$-structured flow bimodules]
	Let $\bX$ be a $\mu$-structured flow $1$-simplex over the order-preserving map $i\colon P\to [1]$.
	Let us write $P_j = i^{-1}(j)$ for $j=0,1$.
	Restricting $\bX$ to the objects in $P_{j}$ gives $\mu$-structured flow categories $\bX_{i}$ for $j=0,1$.
	Evaluating $\bX$ at a $(p,q)$ for $p\in P_0, q\in P_{1}$ gives a $\mu(c_{p},c_{q})$-structured manifold with corners $\bX_{01}(p,q)$ stratified by $\sfA_{P}(p,q)$. Note that $\bX(p,q)$ is empty whenever $p\in P_{1}, q\in P_{0}$. We have composition maps
	\begin{align*}
		\bX_{0}(p,p') \odot \bX_{01}(p',q) &\to \bX_{01}(p,q) \\
		\bX_{01}(p,q') \odot \bX_{1}(q',q) &\to \bX_{01}(p,q)
	\end{align*}
	in $\Mfd_{\diamond}^{\mu}(c_{p},c_{q})$ that up to equivalence are given by the inclusion of the strata corresponding to the arcs $(p,p',q)$ and $(p,q',q)$ in $\sfA_{P}(p,q)$.
	In particular, the inclusion of all such arcs exhaust the codimension 1 strata of $\bX(p,q)$, giving the usual master equation
	\begin{align*}
		\del \bX_{01}(p,q) = \bigcup_{p'} \bX_{0}(p,p') \odot \bX_{01}(p',q) \cup \bigcup_{q'} \bX_{01}(p,q')\odot \bX_{1}(q',q) \,.
	\end{align*}
	We will also call such data a $\mu$-\emph{structured flow bimodule}.
\end{ex}

\subsection{The semi-simplicial space}\label{subsec:the_semi-simplicial_space}

In this section, we construct a semi-simplicial space whose $n$-simplices are precisely the flow $n$-simplices.
Let $\Pos$ denote the large category of small posets and order-preserving maps.
There is a natural inclusion $\Simp \to \Pos$, and so we can form the directed pullback
\begin{equation*}
	\begin{tikzcd}
		\Pos / \Simp \arrow[r] \arrow[d]      &   \Ar(\Pos)  \arrow[d, "ev_{0}\times ev_{1}"] \\
		\Pos \times \Simp  \arrow[r]               &  \Pos\times \Pos
	\end{tikzcd}
\end{equation*}
whose objects are precisely order-preserving maps $P\to [n]$.
The projection $\Pos / \Simp \to \Simp$ is trivially a coCartesian fibration with transport over $\phi\colon [m] \to [n]$ given by pushforward $\phi_{*}\colon \Pos_{/[m]} \to \Pos_{/[n]}$.
Each of these functors admits an adjoint given by forming the pullback
\begin{align*}
	\phi^{*}P = P\times_{[n]} [m] \to [m] \,,
\end{align*}
so $\Pos / \Simp \to \Simp$ is also a Cartesian fibration.
\begin{definition}\label{def:cP}
	Let $\cP \subset (\Pos/\Simp)$ denote the subcategory spanned by the objects $P\to [n]$ such that $P$ is locally finite, and the morphisms
	\begin{equation*}
		\begin{tikzcd}\label{eq:cP}
			P \arrow[r, "f"] \arrow[d]      &   P'  \arrow[d] \\
			{[n]}  \arrow[r, "\phi"]               &  {[m]}
		\end{tikzcd}
	\end{equation*}
	where $\phi$ is injective, and the induced map $P \to \phi^{*}P'$ is the inclusion of a convex subposet.
\end{definition}

\begin{lem}\label{lem:sPSinginj_right_fib}
	The projection $\cP \to \Simp_{s}$ is a Cartesian fibration.
\end{lem}

\begin{proof}
	The fibers of $\cP \to \Simp_{s}$ are the subcategories of $\Pos_{/[n]}$ spanned by locally finite posets and convex inclusions.
  Because the pullback of a locally finite poset is locally finite, and the pullback of a convex inclusion is a convex inclusion, the Cartesian transport along any $\phi\colon [n]\to [m]$ in $\Simp_{s}$ induces a factorization
	\begin{equation*}
		\begin{tikzcd}
			\cP_{m} \arrow[r] \arrow[d]      &   \cP_{n}  \arrow[d] \\
			\Pos_{/{[m]}}  \arrow[r, "\phi^{*}"]               &   \Pos_{/[n]}
		\end{tikzcd}
	\end{equation*}
	which is necessarily unique.
  Moreover, the subcategory $\cP$ satisfies the following 2-out-of-3 property: if for composable morphisms $P\xrightarrow{f} P' \xrightarrow{g} P''$, both $g$ and $g\circ f$ belong to~$\cP$, then so does $f$.
  This implies that morphisms in $\cP$ which are Cartesian with respect to $\Pos/\Simp \to \Simp$ are necessarily also Cartesian with respect to $\cP \to \Simp_{s}$.
\end{proof}

\begin{rem}\label{rem:P_S}
	For $\phi$ injective, the pullback $\phi^{*}P$ can be described as follows.
	Write $P_{S} \subset P$ for the preimage of a subset $S\subset [n]$.
	Then for $\phi\colon [n] \to [m]$ injective, we have $\phi^{*}P = P_{\im(\phi)}$ with the unique map to $[m]$ such that
	\begin{align*}
		(\phi^{*}P)_{k} = P_{\phi(k)} \,.
	\end{align*}
\end{rem}

\begin{ex}
	For the map $d_k\colon [n-1] \to [n]$, the poset $d_k^{*}P$ is given by omitting all $p\in P$ with $i(p)=k$.
\end{ex}
A morphism $(f,\phi)$ in $(\Pos/\Simp)$ induces a morphism of $\Cat$-categories
\begin{align*}
	f_{!} \colon \sfA_{P'} \to \sfA_{P}
\end{align*}
given on objects by $f^{\simeq} \colon P^{\simeq} \to (P')^{\simeq}$, and on categories by the functor
\begin{align}\label{eq:A_P_functoriality}
	f_{!}(p,q)\colon \sfA_{P}(p,q) \to \sfA_{P'}(f(p), f(q))
\end{align}
which pushes forward edge labels along $f$ and vertex labels along $\phi$, deleting any vertex labels that are illegal, and merging edges with repeated labels.

\begin{lem}\label{lem:A_P_closed_embedding}
	For a morphism $f$ in the subcategory $\cP\subset \Pos/\Simp$, the functors \eqref{eq:A_P_functoriality} are closed embeddings, and for the unique minimal element $\gamma_{pq} \in A_{P}(p,q)$, there is a canonical identification
	\begin{align*}
		Q_{f_!(p,q)}(\gamma_{p,q}) = \left\{k \in [n] \mid i(p)<k<i(q), k\notin \im(\phi) \right\} \,.
	\end{align*}
\end{lem}

For a morphism $f$ in $\cP$ we use the decomposition $Q_{f_{!}}^{-}=\emptyset$. We then have equivalences
\begin{align*}
	\sfU^{+}_{P }(p,q)
	\oplus \bR^{Q_{f_{!}(p,q)}(\gamma_{pq})}
	& \simeq \sfU^{+}_{P'}(f(p), f(q))
\end{align*}
by identifying both sides with $\bR^{ \left\{i(f(p))+1,\dots i(f(q))\right\}}$. Now we use these stable equivalences, and \Cref{lem:A_P_closed_embedding}, to lift $f_!$ to an enriched functor between $\Mod_{\diamond}^{\otimes}$-categories
\begin{align*}
	f_{!} \colon (\sfA_{P}, \sfU_{P}) \to (\sfA_{P'}, \sfU_{P'}) \,.
\end{align*}
This construction is natural with respect to composition in the sense that it determines a functor
\begin{align*}
	\cA \colon \cP &\to \Alg(\Mod_{\diamond}).
\end{align*}
Consider the composite functor
\begin{align}\label{eq:Alg(Mfd)->Simp}
	\cA^{*}\Alg\left(\Mfd^{\mu}_{\diamond}\right)) \to \cP \to \Simp_{s} \,.
\end{align}
The first factor is a right fibration because it is pulled back from the right fibration in \Cref{cor:Alg(Mfd)->Alg(Mod)_is_right_fib}, and the second is a Cartesian fibration by \Cref{lem:sPSinginj_right_fib}.

\begin{definition}\label{def:PreFlow_mu}
	Let $\bfPreFlow^{\mu}$ denote the large semi-simplicial $\infty$-category classifying the cartesian fibration \eqref{eq:Alg(Mfd)->Simp}. Let $\PreFlow^{\mu}$ be the large semi-simplicial classifying the underlying right fibration of \eqref{eq:Alg(Mfd)->Simp}.
\end{definition}

\begin{rem}\label{rem:PreFlow_is_functorial}
	It should be clear from construction that $\PreFlow^{(-)}$ determines a functor
	\begin{align*}
		\Cat_{\infty /(U/O)} \to ss\widehat{\cS} \,.
	\end{align*}
	The map of semi-simplicial spaces $f_{!}\colon \PreFlow^{\mu} \to \PreFlow^{\nu}$ determined by a morphism $f\colon \mu\to \nu$ in $\Cat_{\infty/(U/O)}$ is given by pushing forward flow categories along the double functor $f_{*}\colon \Mfd^{\mu}_{\diamond} \to \Mfd^{\nu}_{\diamond}$ coming from the naturality of slices in $\cS_{\bZ\times BO}$.
\end{rem}

\begin{definition}\label{def:PreFlow}
	We similarly define a large semi-simplicial space $\PreFlow$ of unstructured flow categories by pulling back the right fibration $\Alg(\Mfd_{\diamond})\to \Alg(\Mod_{\diamond}^{\otimes})$.
  The forgetful functor induces a map of semi-simplicial spaces
	\begin{align*}
		\PreFlow^{\mu} \to \PreFlow.
	\end{align*}
\end{definition}

In view of \Cref{subsec:flow_categories_as_double_functors}, we could also define $\PreFlow^{\mu}$ in terms of double functors by noting that $\cA$ factors through $\Alg_{\Cat}(\Mod_{\diamond,\con})$.
Let $\bfA$ denote the composite
\begin{align*}
	\bfA\colon \cP \xrightarrow{\cA} \Alg_{\Cat}(\Mod_{\diamond,\con}) \xrightarrow{\del} \Dbl_{\infty/\Mod_{\diamond}^{\otimes}} \, .
\end{align*}
We then immediately get the following corollary of \Cref{cor:canonical_presheaf_pullback}.

\begin{cor}\label{cor:flow_cats_are_double_functors}
	The functor $\cF$ induces an equivalence between $\PreFlow^{\mu}$ and the straightening of the underlying right fibration of
	\begin{align*}
		\bfA^{*}\Dbl_{\infty/\Mfd_{\diamond}^{\mu}} \to \cP \to \Simp_{s} \,.
	\end{align*}
\end{cor}

We could also define $\PreFlow^{\mu}$ using flow simplices that are only defined over $P\subset \Simp^{\op}_{P^{\simeq}}$.
The inclusion of such categories determines a natural transformation of functors $\cP \to \Opd$, and since $\Alg(\Mod_{\diamond}) \to \Opd$ is a Cartesian fibration, we can lift this to a natural transformation of functors $\cP \to \Alg(\Mod_{\diamond})$ which we denote $\cA_{\leq} \to \cA$.
\begin{cor}\label{cor:flow_cats_defined_on_P}
	We have an equivalence
	\begin{align}\label{eq:flow_cats_defined_on_P}
		\cA^{*}\Alg(\Mfd^{\mu}_{\diamond}) \xrightarrow{\simeq}   \cA_{\leq}^{*}\Alg(\Mfd^{\mu}_{\diamond})
	\end{align}
	of right fibrations over $\cP$.
  In particular, $\PreFlow^{\mu}$ can be defined in terms of either.
\end{cor}

\begin{proof}
	The map of right fibrations is determined by pulling back along the natural transformation $\cA_{\leq} \to \cA$.
  Since a map of right fibrations is itself a right fibration, it suffices to show that \eqref{eq:flow_cats_defined_on_P} has contractible fibers.
  This follows because a flow simplex $\bX\colon P \to \Mfd^{\mu}_{\diamond}$ has a unique extension over $\Simp^{\op}_{P^{\simeq}}$ by the empty manifold.
\end{proof}

\subsection{Horn fillers}\label{subsec:horn_fillers}
The goal of this subsection is to prove that $\PreFlow^{\mu}$ is an inner Kan space in the sense of \Cref{def:inner_fibration}.
This involves a generalization of the arguments in \cite[Section 5]{AB} which allows us to extend $\mu$-structures over the filling.
The key point is that the fillers constructed in \cite{AB} present a homotopy colimit over the restriction to the faces in the horn, and so the $\mu$-structure extends uniquely.

\begin{definition}\label{def:del_A_p}
	For $P\to [n]$ a poset over $[n]$, let $\del \sfA_{P}(p,q) \subset \sfA_{P}(p,q)$ denote the full subcategory of arcs $\gamma$ which do not contain a segment of the form
	\begin{center}
		\begin{tikzpicture}
			\coordinate (start) at (-1.5 ,0);
			\coordinate[label= above:$r$] (r) at (-1,0);
			\coordinate[label= above:$s$] (s) at (1,0);
			\coordinate (end) at (1.5,0);
			\draw[dotted, thick] (start) -- (end);
			\draw[thick] (r) -- (s) ;
			\coordinate[label= below:{$\{1{,}\dots {,}n-1 \}$}  ] (S) at (0,0);
			\draw[fill, color=black] (S) circle (0.05);
		\end{tikzpicture}
	\end{center}
	for $r\in P_{0}, s\in P_{n}$. Then for $0<i<n$, define $\Horn_{i} \sfA_{P}(p,q) \subset \del \sfA_{P}(p,q)$ to be the full subcategory spanned by arcs $\gamma$ which also do not contain segments of the following form.
	\begin{center}
		\begin{tikzpicture}
			\coordinate (start) at (-1.5 ,0);
			\coordinate[label= above:$r$] (r) at (-1,0);
			\coordinate[label= above:$s$] (s) at (1,0);
			\coordinate (end) at (1.5,0);
			\draw[dotted, thick] (start) -- (end);
			\draw[thick] (r) -- (s) ;
			\coordinate[label= below:{$\{1{,}\dots{,}\widehat{i} {,}\dots {,}n-1 \}$}  ] (S) at (0,0);
			\draw[fill, color=black] (S) circle (0.05);
		\end{tikzpicture}
	\end{center}
	Both of these conditions are stable under concatenation, so they determine subcategories $\Horn_{i}\sfA_{P}\subset \del\sfA_{P} \subset\sfA_{P}$.
\end{definition}
\begin{rem}\label{rem:horn_double_functor}
	Note that while each $\Horn_{i}\sfA(p,q)$ is not a model for manifolds with corners, we still have a functor
	\begin{align*}
		\Horn_{i}\sfA_{P}(p,q)^{\op} \to \sfA_{P}(p,q)^{\op} \xrightarrow{\del} \Mod_{\diamond} \, ,
	\end{align*}
	which unstraightens to a double functor
	\begin{align*}
		\del_{\Horn_{i}} \colon  \Horn_{i}\bfA_{P} \to \bfA_{P} \to \Mod_{\diamond}^{\otimes}.
	\end{align*}
\end{rem}
\begin{lem}\label{lem:horn_A_as_colimit}
	Let $P \to [n] \in \cP$, and $0<i<n$. Under unstraightening, $P\to [n]$ corresponds uniquely to a map of Cartesian fibrations $F(\Delta^{n}_{s}) \to \cP$.
  The colimit of the composite
	\begin{align*}
		G\colon F(\Horn_{i,s}^{n}) \to F(\Delta^{n}_{s}) \to \cP \xrightarrow{\bfA} \Dbl_{\infty/\Mod_{\diamond}^{\otimes}}
	\end{align*}
	is equivalent to $\del_{\Horn_{i}}\colon \Horn_{i}\bfA_{P} \to \Mod_{\diamond}^{\otimes}$.
\end{lem}
\begin{proof}
	It suffices to verify that $\Horn_{i}\bfA_{P}$ is the colimit of $G$ after projecting to $\Dbl_{\infty}$.
  By construction of $\bfA_{P}$, this factors through the category $\Cat_{2} \to \Dbl_{\infty}$ of strict 2-categories by a functor $G'$.
  The category $F(\Horn^{n}_{i,s})$ is the poset of subsets of $[n]$, excluding the two objects $[n]$ and $[n]\setminus \left\{ i\right\}$.
  The functor $G'$ is valued in (non-full) sub-2-categories of $\bfA_{P}$, and inclusions of such.
  For any $K\in F(\Horn^{n}_{i,s})$, the restriction of $G'$ to the cube $[1]^{K}$ is a limit diagram.
  It therefore follows that the colimit of $G'$ is the smallest subcategory of $\bfA_{P}$ generated by all the $\bfA_{d_{i}^{*}P}$ under composition, which is precisely $\Horn_{i}\bfA_{P}$.
  To verify that this colimit is preserved by the inclusion $\Cat_{2} \to \Dbl_{\infty}$, it suffices to argue that~$G$ is a cofibrant diagram.
  The cube with missing vertices is a direct category, and the latching morphism at some $K$ is the map
	\begin{align*}
		\colim_{J\subsetneq K} \bfA_{P_{J}} \to \bfA_{P_{K}} \,.
	\end{align*}
	By similar reasoning as above, this colimit is the subcategory of $\bfA_{P_{K}}$ generated by all the $\bfA_{d_{j}^{*}P_{K}}$ for $j\in K$, which is precisely $\del \bfA_{P_{K}}$.
  Hence the latching map is $\del \bfA_{P_{K}} \to \bfA_{P_{K}}$, which is a subcategory inclusion, and therefore a cofibration in the model structure on $\Cat$-enriched categories.
\end{proof}

\begin{lem}\label{lem:St(P)_inner_Kan}
	The straightening of the underlying right fibration $\cR(\cP) \to \Simp_{s}$ is an inner Kan space.
\end{lem}

\begin{proof}
	Let $F\colon \Horn_{i,s}^{n} \to \St(\cR(\cP)),$ be an inner horn.
	Restricting to each vertex determines a sequence of sets $P_i$, and we pick the minimal partial order on $P = \cup_j P_i$ which extends the one given on $d_{k}^{*}P$ for all $k\neq i $.
\end{proof}
We will now argue as in \cite[Section 5]{AB} that $\PreFlow \to \St(\cR(\cP))$ is an inner fibration.
By \Cref{cor:canonical_presheaf_pullback}, the lifting problem
\begin{equation*}
	\begin{tikzcd}
		\Horn^{n}_{i,s} \arrow[r] \arrow[d]      &   \PreFlow  \arrow[d] \\
		\Delta^{n}_{s}  \arrow[r] \arrow[ur,dashed]              &  \St(\cP)
	\end{tikzcd}
\end{equation*}
is equivalent under unstraightening to
\begin{equation*}
	\begin{tikzcd}
		F(\Horn_{i,s}^{n}) \arrow[r] \arrow[d]      &   \Dbl_{\infty/\Mfd_{\diamond}^{\otimes}}  \arrow[d] \\
		F(\Delta^{n}_{s})  \arrow[r] \arrow[ur,dashed]         &  \Dbl_\infty{/\Mod_{\diamond}^{\otimes}}.
	\end{tikzcd}
\end{equation*}
Recall that a map into a slice $\infty$-category $\cI \to \cC_{/C}$ is equivalently a natural transformation from the underlying diagram in $\cC$, to the constant diagram $\cI \to \cC$ at $C$.
Natural transformations into a constant diagram are equivalent to maps out of a colimit, so because $F(\Delta^{n}_{i,s})$ has a terminal object, and by \Cref{lem:horn_A_as_colimit}, the above is equivalent to
\begin{equation*}
	\begin{tikzcd}
		\Horn_{i}\bfA_{P} \arrow[r, "\bY"] \arrow[d]      &   \Mfd_{\diamond}^{\otimes}  \arrow[d] \\
		\bfA_{P}  \arrow[r]\arrow[ur,dashed]           &  \Mod_{\diamond}^{\otimes} \, .
	\end{tikzcd}
\end{equation*}
Because $\bY$ is a double functor, it is compatible with concatenation, giving $\bY(\gamma \# \gamma') \simeq \bY(\gamma) \otimes \bY(\gamma')$. The component
\begin{align*}
	\bY \colon \Horn_{i}\sfA_{P}(p,q)^{\op} \to \Mfd_{\diamond}
\end{align*}
is therefore given by mapping an arc $\gamma$ with edges labeled by $p,r_{1},\dots ,r_{\ell}, q$ and vertices labeled by $S_0,\dots, S_{\ell}$ to
\begin{align*}
	\bY(\gamma) = \del^{S_0} \bY(q,r_1) \times \dots \times \del^{S_{\ell}}\bY(r_{\ell}, q) \,.
\end{align*}

We write $Y$ for the composite of $\bY$ with the forgetful functor $\Mfd_{\diamond}^{\otimes} \to \Top^{\times}$.
This has components
\begin{align*}
	Y \colon \Horn_{i}\sfA_{P}(p,q)^{\op} \to \Top.
\end{align*}
For each $p,q$, consider also the $L$-block functor
\begin{align*}
	L \colon \Horn_{i}\sfA_{P}(p,r) \to \Top
\end{align*}
of \cite[Equation 5.12]{AB}.
We define $\widetilde{Y}(p,r)$ as the weighted colimit
\begin{align*}
	\widetilde{Y}(p,r) = \left(\coprod_{\gamma'\to \gamma \in \Horn_{i}A_{P(p,r)}} L(\gamma') \times Y(\gamma) \right)/\sim
\end{align*}
in $\Top$, where the equivalence relation identifies the images of
\begin{align*}
	L(\gamma') \times Y(\gamma') \leftarrow L(\gamma') \times Y(\gamma) \rightarrow L(\gamma) \times Y(\gamma) \,.
\end{align*}
In \cite[Lemma 5.8]{AB} it is shown that $\widetilde{Y}(p,r)$ admits the structure of a smooth manifold with corners, and a stratification by $\sfA_{P}(p,r)$ such that $Y(\gamma) \simeq \del^{\gamma}\widetilde{Y}(p,r)$.
We equip $\widetilde{Y}(\gamma_{pr})$ with this stratification and the virtual vector space $U_{P}(p,r)$ to get an object $\widetilde{\bY}(p,r) \in \Mfd_{\diamond}$.
It is also shown in \cite[Corollary 5.10]{AB} that this construction is compatible with the bimodule structure over $\bY_{0}$ and $\bY_{n}$.
We therefore get an unstructured flow $n$-simplex with objects $i\colon P\to [n]$ by the formula
\begin{align*}
	\widetilde{\bY}(p,r) = \begin{cases}
		\widetilde{\bY}(p,r), & p\in P_0, r\in P_n \\
		\del^{i(p),\dots, i(r)} \bY(p,r), & \text{else.}
	\end{cases}
\end{align*}
\begin{lem}\label{lem:horn_filler_as_hocolim}
	In the above situation, the composite
	\begin{align}\label{eq:horn_filler_as_hocolim}
		F(\Horn^{n}_{i,s})^{\triangleright} \subset F(\Delta^{n}_{s}) \xrightarrow{\widetilde{\bY}} \Alg_{\Cat}(\Mfd_{\diamond}) \xrightarrow{-I} \Alg_{\Cat}(\cS_{/\bZ\times BO})
	\end{align}
	is a colimit cone.
\end{lem}
\begin{proof}
	Slice projections are conservative and preserve colimits, so by \Cref{cor:alg_in_slice} it suffices to show that \eqref{eq:horn_filler_as_hocolim} is a colimit after pushing forward to $\Alg(\cS)$.
	The underlying $\cS$-category of a flow simplex is complete, because the only internal equivalences are the identities.
  Therefore, it suffices to show that \eqref{eq:horn_filler_as_hocolim} pushes forward to a colimit cone in $\Cat_{\infty}$.

	Let $\Cat_{\Top}$ denote the category of small categories enriched in simplicial sets.
  This has a Dwyer--Kan model structure which presents $\Cat_{\infty}$.
  Since the $\cS$-category determined by a flow simplex has discrete space of objects, we can factor through this localization functor  by $\del^{(-)}\widetilde{Y}\colon F(\Delta^{n}_{s}) \to \Cat_{\Top}$.
  We will show that $\del^{(-)}Y$ is a projectively cofibrant diagram in this model category, and that the comparison map
	\begin{align*}
		\colim(\del^{(-)}Y) \to \widetilde{Y}
	\end{align*}
	is a weak equivalence.

	We first show that $\del^{(-)}Y$ is a cofibrant diagram.
  Since the category $F(\Horn_{i,s}^{n})$ is direct, it suffices to show that for all proper subsets $K\subsetneq [n]\setminus \left\{ i\right\}$, the latching morphism
	\begin{align*}
		L_{K} Y = \colim_{J \subsetneq K}(\del^{J} Y) \to \del^{K} Y
	\end{align*}
	is a cofibration.
  The topological category $\del^{K} Y$ has the property that any composition map
	\begin{align*}
		\del^{K}Y(p,r) \times \del^{L}Y(r,q) \to \del^{K}Y(p,q)
	\end{align*}
	is an embedding with closed image.
  Similarly, any inclusion map $\del^{J}Y(p,q) \to \del^{K}Y(p,q)$ is an embedding with closed image.
  Moreover, the restriction of $\Sing \del^{(-)}Y$ to the cube~$[1]^{K}$ is a limit diagram of subcategories of $ \del^{K}Y$.
  Each inclusion is an embedding with closed image.
  Using this one can show that the latching morphism $L_{K}Y\to \del^{K}Y$ is precisely the embedded, closed subcategory generated under composition by all the $\del^{d_{j}K}Y$.
  It now suffices to solve the lifting problem
	\begin{equation*}
		\begin{tikzcd}
			L_{K}Y \arrow[r] \arrow[d]      &   \cE  \arrow[d,"p"] \\
			\del^{K} Y  \arrow[r] \arrow[ur,dashed]              &  \calD
		\end{tikzcd}
	\end{equation*}
	where $p$ is essentially surjective on objects, and all $\cE(x,y) \to \calD(px,py)$ are trivial fibrations.
  Because the latching morphism is the identity on objects, the value of the dashed lift on sets of objects is already determined.
  On mapping spaces, each $L_{K}Y(p,q) \to \del^{K}Y(p,q)$ is a cofibration, so we can extend the lift over this mapping space.
  This determines the lift uniquely over the closed and embedded subcategory of $\del^{K}Y$ generated by $\del^{K}Y(p,q)$ and $L_{K}Y$.
  It is clear that we may keep extending one morphism space at the time, each time extending the lift over a new closed and embedded subcategory.

	By a similar reasoning, the colimit $C(Y)$ of the diagram $\del^{(-)}Y\colon F(\Horn_{i,s}^{n}) \to \Cat_{\Top}$ has objects $P$ and morphism spaces
	\begin{align*}
		\underset{\gamma \in \Horn_{i} \sfA_{P}(p,q)^{\op}}{\colim}( \del^{\gamma}Y(p,q)) \,.
	\end{align*}
	Since each $L$-block $L(\gamma')$ is a contractible topological space, the inclusion
	\begin{align*}
		C(Y)(p,q) \to \widetilde{Y}(p,q)
	\end{align*}
	is a homotopy equivalence when $p\in P_{0}, q\in P_{n}$, and a homeomorphism otherwise.
\end{proof}
\begin{cor}\label{cor:horn_filling}
	$\PreFlow^{\mu}\to \St(\cR(\cP))$ is an inner fibration.
\end{cor}
\begin{proof}
	A lifting diagram
	\begin{equation*}
		\begin{tikzcd}
			\Horn_{i,s}^{n} \arrow[r, "\bY"] \arrow[d]      &   \PreFlow^{\mu}  \arrow[d] \\
			\Delta^{n}_{s}  \arrow[r]  \arrow[ur,dashed, "\widetilde{\bY}"]           &  \cR(\cP)
		\end{tikzcd}
	\end{equation*}
	is equivalent under unstraightening to
	\begin{equation*}
		\begin{tikzcd}
			F(\Horn^{n}_{i,s}) \arrow[rr] \arrow[d]      &&   \Alg(\Mfd_{\diamond}^{\mu})  \arrow[d] \\
			F(\Delta^{n}_{s}) \arrow[r] \arrow[urr, dashed] &\cP  \arrow[r]               &   \Alg(\Mod_{\diamond}) \,.
		\end{tikzcd}
	\end{equation*}
	Let $F\bY$ denote the underlying horn of $\bY$ in $\PreFlow$, and construct a filler $F\widetilde{\bY}$ as in \Cref{lem:horn_filler_as_hocolim}.
  By definition of $\Mfd_{\diamond}^{\mu}$ and \eqref{cor:alg_in_slice}, the right square of the following diagram is a pullback.
  It therefore suffices to solve the lifting problems
	\begin{equation*}
		\begin{tikzcd}
			F(\Horn_{i,s}^{n}) \arrow[r, "\bY"] \arrow[d]      & \Alg(\Mfd_{\diamond}^{\mu}) \arrow[r] &   \Alg(\cS_{/ (\bZ\times BO)})_{/\mu}  \arrow[dd] \\
			F(\Horn_{i,s}^{n})^{\triangleright} \arrow[d] \arrow[urr, dashed] && \\
			F(\Delta^{n}_{s})  \arrow[r, "F\widetilde{\bY}"']    \arrow[uurr, dotted]    & \Alg(\Mfd_{\diamond}) \arrow[r, "U\ominus T"].       &   \Alg(\cS_{/\bZ\times BO}) \,.
			\arrow[from=1-2, to=3-2, crossing over]
		\end{tikzcd}
	\end{equation*}
	The dashed lift exists by \Cref{lem:horn_filler_as_hocolim} and the universal property of colimits.
  The dotted lift exists because the latter left vertical is cofinal and the right vertical is a right fibration.
\end{proof}

\subsection{Degeneracies}\label{subsec:degeneracies}

In this subsection, we show that $\PreFlow^{\mu}$ is quasi-unital by using the criterion of \Cref{prop:outer_degeneracies}.
We start with giving a description of the free slice $\PreFlow^{\mu}_{\Delta^{n}_{s}/}$ in terms of a fibration.
Recall that $\Simp_{a}$ denotes the augmented simplicial category, i.e., the full subcategory of $\Pos$ spanned by the totally ordered sets.
Since the only map to the empty poset is the identity, the directed pullback $\Pos/\Simp_{a} \to \Simp_{a}$ has contractible fiber over $[-1]$.

Consider the pullbacks
\begin{equation*}
	\begin{tikzcd}
		(\Pos/\Simp_{a})_{\Delta^{n}/} \arrow[r]  \arrow[d] &
		j^{*}(\Pos/\Simp_{a}) \arrow[r] \arrow[d]      &
		(\Pos/\Simp_{a})  \arrow[d] \\
		\ast \times \Simp_{a} \arrow[r, "{ \left\{ [n]\right\}\times id  }"] &
		\Simp_{a}\times \Simp_{a}  \arrow[r, "\join"]               &
		\Simp_{a} \,.
	\end{tikzcd}
\end{equation*}
A map $\phi\colon [n] \to [m]$ in $\Simp_{a}$ induces a natural transformation  $\left\{ [n]\right\}\times id \to \left\{ [m]\right\}\times id$, which on pullbacks induces a map of Cartesian fibrations $\phi^{*}\colon (\Pos/\Simp)_{\Delta^{m}/} \to (\Pos/\Simp)_{\Delta^{n}/}$.
When we restrict to subcategories corresponding to $\cP$, we get morphisms $\cP_{\Delta^{n}_{s}/} \to \cP_{a}$ of Cartesian fibrations over $\Simp_{s,a}$. Note that
\begin{align*}
	\Simp_{s,a} \xrightarrow{[n]\join-} \Simp_{s,a}
\end{align*}
is an equivalence onto the subcategory of objects $[n+m]$ and injective morphisms which leave the terms $\left\{ 0,\dots, n\right\}$ fixed.
By pullback stability we can therefore think of $\cP_{\Delta^{n}_{s}/} \to \cP_{a}$ as the subcategory of objects $P\to [n+m]$, and morphisms $\phi^{*}P \to P$ where $\phi$ leaves the terms $\left\{ 0,\dots , n\right\}$ fixed.

We write $\cA_{\Delta^{n}_{s}/}$ for the composite
\begin{align*}
	\cP_{\Delta^{n}_{s}/} \to \cP_{a} \xrightarrow{\cA} \Alg(\Mod_{\diamond}) \,.
\end{align*}
The pasting law implies that the outer left rectangle of the diagram
\begin{equation*}
	\begin{tikzcd}
		\cA_{\Delta^{n}_{s}/}^{*}\Alg(\Mfd^{\mu}_{\diamond}) \arrow[r] \arrow[d]  &
		\cA^{*}\Alg(\Mfd^{\mu}_{\diamond})  \arrow[d] \arrow[r] &
		\Alg(\Mfd^{\mu}_{\diamond}) \arrow[d] \\
		\cP_{\Delta^{n}/} \arrow[r] \arrow[d] &
		\cP_{a} \arrow[r, "\cA"] \arrow[d] &
		\Alg(\Mod_{\diamond}) \\
		\Simp_{s,a} \arrow[r, "{[n]\join -}"] &
		\Simp_{s,a}
	\end{tikzcd}
\end{equation*}
is a pullback.
This identifies the free slice $\PreFlow^{\mu}_{\Delta^{n}_{s}/}$ with the straightening of the underlying right fibration of
\begin{align*}
	\cA_{\Delta^{n}_{s}/}^{*}\Alg(\Mfd^{\mu}_{\diamond}) \to \cP_{\Delta^{n}_{s}/} \to \Simp_{s,a} \,.
\end{align*}

To construct initial and terminal degeneracies for $\PreFlow$, we will use the construction of \emph{conic degenerations} from \cite[Section 6.1]{AB}.
We start by noting that the manifold with corners $[0,\infty)^{n}$ is naturally stratified by the poset $[1]^{n}$ of subsets of $\left\{ 1,\dots , n\right\}$ by letting $I\subset \left\{ 1,\dots,n\right\}$ correspond to the stratum where $t_{i} =0 $ for $i\in I$.
The following is a summary of Lemmas 6.2 and 6.3 of \cite{AB}.

\begin{lem}\label{lem:conic_degeneracies}
	For each natural number $n$, there exists a manifold with corners $D(n)$. For $n>0$ there exist smooth maps $D(n) \to [0,\infty)^{n-1}$ satisfying the following.
	\begin{enumerate}
		\item $D(0)$ is a point.
		\item The fiber of $D(n+1)$ over a point $(t_{1},\dots, t_{n})$ is a union of intervals, indexed by one more than the number of $i$ such that $t_{i}=0$.
		\item For a subset $I\subset \left\{ 0, \dots n\right\}$, the inclusion of the subset where $t_{i} >0 $ for $i\in I$ extends to a pullback square
		\begin{equation*}
			\begin{tikzcd}
				{D(n-\vert I \vert + 1) \times (0,\infty)^{I} } \arrow[r] \arrow[d]      &   D(n+1)  \arrow[d] \\
				{[0,\infty)^{n-\vert I \vert} \times (0,\infty)^{I}  } \arrow[r]               &   {[0,\infty)^{n} }
			\end{tikzcd}
		\end{equation*}
		which is natural in the sense that for a chain of subsets $J\subset I \subset \left\{1,\dots, n\right\}$, the obvious diagram commutes.
		\item The preimage in $D(n+1)$ of the stratum $\del^{ \left\{ i \right\}} [0,\infty)^{n}$ is the image of two codimension 1 strata of the form
		\begin{align*}
			D(i) \times [0,\infty)^{n-i} \to D(n+1) \leftarrow [0,\infty)^{i-1}\times D(n-i+1) \,.
		\end{align*}
		The codimension 1 boundary of $D(n+1)$ consists of strata of the above form, as well as two strata
		\begin{align*}
			D(0) \times [0,\infty)^{n} \to D(n+1) \leftarrow [0,\infty)^{n} \times D(0) \,.
		\end{align*}
		The inclusion of these are natural, in the sense that for $1\leq i \leq j \leq n$ we have a commutative diagram of stratum inclusions as follows:
		\begin{equation*}
			\begin{tikzcd}
				{[0,\infty)^{i-1} \times D(j-i) \times [0,\infty)^{n-j} } \arrow[r] \arrow[d]      &   {[0,\infty)^{i-1} \times D(n-i+1)}  \arrow[d] \\
				{ D(j) \times [0,\infty)^{n-j} } \arrow[r]               &   D(n+1) \,.
			\end{tikzcd}
		\end{equation*}
	\end{enumerate}
\end{lem}

\begin{rem}\label{rem:D(S)}
	A finite totally ordered set $S$ is canonically equivalent to $\left\{1, \dots , \vert S \vert \right\}$.
  We therefore write $D(S) \coloneq D(\vert S \vert )$ and $D(S+1) \to [0,\infty)^{S}$ for the map $D(\vert S \vert +1 ) \to [0,\infty)^{\vert S \vert}$.
\end{rem}

For a totally ordered set $S$, we write $\overline{S}$ for the totally ordered set obtained by adjoining a minimal object $-\infty$ and a maximal object $\infty$ to $S$.
For a set $Q$ and a subset $S\subset Q$ equipped with a total order, we define $D([1]^{Q},S)$ by the pullback
\begin{equation}\label{eq:D(1^Q,S)}
	\begin{tikzcd}
		{D([1]^{Q}, S)} \arrow[rr] \arrow[d]      &&   \sfA_{\overline{S}\times [1]}( (-\infty,0), (\infty, 1))  \arrow[d] \\
		{[1]^{Q}}  \arrow[r, "\cap S"]               &   {[1]^{S} } \arrow[r, "\simeq"] & \sfA_{\overline{S}}(-\infty,\infty) \,,
	\end{tikzcd}
\end{equation}
where the vertical map on the right pushes forward edge labels along the projection $\overline{S}\times [1] \to \overline{S}$, and merges edges with identical labels if necessary.
An object of $D([1]^{Q},S)$ is determined by a subset $K\subset Q$ and an arc $\alpha \in A_{\overline{S}\times [1]}( (-\infty ,0),( \infty,1 ))$ such that all elements of $K\cap S$ appear as edge labels in $\alpha$.
For a fixed $K$, $\alpha$ can be written uniquely as a concatenation (at a vertex) $\alpha_{0}\#\alpha_{1}$ where $\alpha_{i}$ has labels in $\overline{S}\times\left\{ i \right\}$.

\begin{lem}\label{lem:D(s+1)_stratification}
	The manifold with corners $D(S+1)$ admits a stratification by $D([1]^{S},S)$ such that the stratum corresponding to $\alpha = \alpha_{0} \# \alpha_{1}$ is
	\begin{align*}
		\del^{\alpha} D(S) \cong \del^{\alpha_{0}} [0,\infty)^{S} \times D( (\max (\alpha_{0}) ,\min(\alpha_{1}) ) \times \del^{\alpha_{1}} [0,\infty)^{S} \,.
	\end{align*}
\end{lem}

Let $\cR$ be a model for manifolds with corners, and let $P \subset \ob(\cR)$ be a subset of codimension 1 objects.
We then define $Q_{P}(\sigma) \subset Q_{\cR}(\sigma)$ to be the preimage of $P$ under the map
\begin{align*}
	Q_{\cR}(\sigma) \subset \ob(\cR_{/\sigma}) \to \ob(\cR) \,.
\end{align*}
We additionally assume that $P$ is equipped with a partial order such that the induced order on each $Q_{P}(\sigma)$ is total.
We define a functor $D(\cR,P) \to \cR$ by the colimit
\begin{align*}
	D(\cR,P) \coloneq \colim_{\sigma\in \cR} D([1]^{Q_{\cR}(\sigma)}, Q_{P}(\sigma)) \to \colim_{\sigma \in \cR}[1]^{Q_{\cR}(\sigma)} \simeq \colim_{\sigma\in \cR} \cR_{/\sigma} \simeq \cR \,.
\end{align*}
This is again a model for manifolds with corners because overcategories can be computed by passing to a local model \eqref{eq:D(1^Q,S)}.
An object in $D(\cR, P)$ is uniquely determined by an object $\sigma \in \cR$ and an arc $\alpha = \alpha_{0}\#\alpha_{1}$  such that all elements of $Q_{P}(\sigma)$ appear in edge labels of $\alpha$.
Such an arc is uniquely determined by the elements $\max \alpha_{0}\leq \min \alpha_{1}\in Q_{P}(\sigma)$.
A morphism $(\sigma, \alpha) \to (\tau, \beta)$ is a morphism $a\colon \sigma \to \tau$ in $\cR$ such that
\begin{align*}
	a_{*} \max \alpha_{0} \leq \max \beta_{0} \leq \min \beta_{1} \leq a_{*} \min \alpha_{1} \,.
\end{align*}
For an object $\gamma =(\sigma, \alpha)$ as above, we define a collection $\del^{\gamma}P$ of codimension 1 objects of $\del^{\sigma} \cR$ by declaring $Q_{\del^{\gamma}P }(\tau)$ to be the preimage of the interval $(\max \alpha_{0},\min \alpha_{1})\subset Q_{P}(\tau)$ under the map
\begin{align*}
	Q_{\del^{\sigma}\cR}(\tau)\to Q_{\del^{\sigma}\cR}(\tau) \amalg Q_{\cR}(\sigma) \simeq  Q_{\cR}(\tau) \,.
\end{align*}
We then have canonical equivalences
\begin{align}\label{eq:D(RP)}
	\del^{\gamma} D(\cR, P) \simeq
	\begin{cases}
		D(\del^{\sigma}\cR, \del^{\gamma}P) & \max \alpha_{0}< \min \alpha_{1} \\
		\del^{\sigma}\cR  & \max \alpha_{0} = \min \alpha_{1}.
	\end{cases}
\end{align}

If $\bX$ is a manifold with corners stratified by $\cR$, and equipped with a compatible system of collars, we get a cover of $\bX$ by open sets of the form $\interior \del^{\sigma} \bX \times [0,\infty)^{Q_{\cR}(\sigma)}$, whose intersections are labeled by morphisms $\sigma \to \tau$ in $\cR$.
We write $Q_{\cR\setminus P}(\sigma) \subset Q_{\cR}(\sigma)$ for the complement of $Q_{P}(\sigma)$ and define a map of manifolds with corners $D(\bX, P) \to \bX$ by clutching together maps
\begin{align*}
	\interior \del^{\sigma} \bX \times D(Q_{P}(\sigma)+1) \times [0,\infty)^{Q_{\cR\setminus P}(\sigma)} &\to \interior \del^{\sigma} \bX \times [0,\infty)^{Q_{P}(\sigma)} \times [0,\infty)^{Q_{\cR\setminus P}(\sigma)} \\
	&\simeq  \interior \del^{\sigma} \bX \times [0,\infty)^{Q_{\cR}(\sigma)}
\end{align*}
over intersections which agree by \Cref{lem:conic_degeneracies}.

\begin{lem}\label{lem:D(XP)_stratification}
	The manifold with corners $D(\bX, P)$ admits a stratification by $D(\cR, P)$ such that for $\gamma = (\sigma, \alpha)$,
	\begin{align*}
		\del^{\gamma}D(\bX, P) \cong \begin{cases}
			D(\del^{\sigma} \bX, \del^{\gamma}P ), &  \max \alpha_{0} < \min \alpha_{1} \\
			\del^{\sigma}\bX, & \max \alpha_{0} = \min \alpha_{1}.
		\end{cases}
	\end{align*}
\end{lem}
\begin{proof}
	Fix an object $\gamma = (\sigma, \alpha_{0}\#\alpha_{1})$ as above.
  A map $a_*\colon \sigma \to \tau$ induces a map
	\begin{align*}
		a_{*} \colon D([1]^{Q_{P}(\sigma)}, Q_{P}(\sigma)) \to D([1]^{Q_{P}(\tau)}, Q_{P}(\tau )) \,.
	\end{align*}
	The stratum of $D(\vert Q_{P}(\tau)\vert +1)$ corresponding to the arc $a_{*} \alpha$ in the stratification of \Cref{lem:D(s+1)_stratification} is then precisely
	\begin{align*}
		\del^{a_{*}\alpha} D(Q_{P}(\tau)+1) =
		\begin{cases}
			\del^{a_{*}\alpha_{0}} [0,\infty)^{Q_{P}(\tau)} \times D(Q_{\del^{\gamma}P}(\tau) +1) \times \del^{a_{*}\alpha_{1}} [0,\infty)^{Q_{P}(\tau)}, & \max \alpha_{0} < \min \alpha_{1} \\
			\del^{a_{*}\alpha} [0,\infty)^{Q_{P}(\tau)}, & \text{else.}
		\end{cases}
	\end{align*}
	We define the stratum inclusion $\del^{\gamma} D(\bX, P) \to D(\bX, P)$ by clutching together the stratum inclusions
	\begin{align*}
		\interior \del^{\tau} \bX \times \del^{a_{*}\alpha} D(Q_{P}(\sigma)+1) \times [0,\infty)^{Q_{\del^{\sigma} \cR\setminus P}(\tau)} \to \interior \del^{\tau} \bX \times D(Q_{P}(\tau)+1) \times [0,\infty)^{Q_{R_{1}}(\tau)},
	\end{align*}
	indexed over $\sigma \to \tau \in \del^{\sigma}\cR$.
\end{proof}
If $\cR$ and $\cR'$, are models and $\cR$ is equipped with a partially ordered set of codimension~1 objects $P$ as above, we equip $\pi_{0}\cR'$ with the discrete poset structure, and thereby get a poset $P\times \pi_{0}\cR'$ of codimension 1 objects of $\cR\times \cR'$.
We have a canonical identification $Q_{P\times \pi_{0}\cR'}(\sigma, \tau) \simeq Q_{P}(\sigma)$, so if the former is totally ordered, so is the latter.

\begin{lem}\label{lem:D(XP)_monoidal}
	In the above situation, if $\bX$ is stratified by $\cR$ and $\bY$ by $\cR'$, we have a diffeomorphism
	\begin{align*}
		D(\bX\times \bY, P\times \pi_{0}\cR') \simeq D(\bX, P) \times \bY \,.
	\end{align*}
\end{lem}

Consider an object $P\to [0]\join [n]$ of $\cP$. We write $P_{0}(p,q)\subset \ob\sfA_{P}(p,q)$ for the codimension one objects corresponding to a single internal edge labeled by some $r\in P_{0}$.
This set is equivalent to the interval $(p,q)\subset P_{0}$, so we equip it with the partial order inherited from $P_{0}$.
For any $\gamma \in \sfA_{P}(p,q)$, the set $Q_{P_{0}(p,q)}(\gamma)$ is the set of internal edge labels in $\gamma$ that belong to  $P_{0}$, and this is totally ordered by the order of appearance in $\gamma$.
We write $\sfA_{P}^{D}$ for the $\Mod_{\diamond}$-category with objects $s_{0}^{*}P$ and morphism objects
\begin{align*}
	\sfA_{P}^{D}(p^{0},q) &= D(\sfA_{P}(p,q), P_{0}(p,q)), & p \in P_{0}, q\in P_{[n]} \\
	\sfA_{P}^{D}(p^{1},q) &= \sfA_{P}(p,q)  p\in P_{0},& q\in P_{[n]} \\
	\sfA_{P}^{D}(p,q) &= \sfA_{P}(p,q), &  p,q\in P_{[n]} \text{ or } p,q\in P_{0}.
\end{align*}
We equip these with the pairs of vector spaces $\sfU_{s_{0}^{*}P}$.
For the multiplication map $\# \sfA_{P}^{D}(p^{0},q) \times \sfA_{P}^{D}(q,r) \to \sfA_{P}^{D}(p^{0},r) $ we use \Cref{lem:D(XP)_monoidal} and consider the stratum inclusion
\begin{align*}
	&D\left(\sfA_{P}(p,q), P_{0}(p,q)  \right) \times \sfA_{P}(q,r) \\
	&\qquad \simeq D\left(\sfA_{P}(p,q)\times \sfA_{P}(q,r), P_{0}(p,r)\cap (-\infty,r) \right)  D\left(\sfA_{D}(p,r), P_{0}(p,r) \right) \,.
\end{align*}
We use the decomposition $Q_{\#}^{+} = \left\{ q \right\}$ and $Q_{\#}=\emptyset$, and the usual linear isomorphisms on $\sfU_{s_{0}^{*}P}$. For the composition map $\# \colon \sfA_{P}^{D}(p^{0},q^{0}) \times \sfA_{P}^{D}(q^{0},r)$ we take
\begin{align*}
	&\sfA_{P}(p,q) \times D\left(\sfA_{P}(q,r), P_{0}(q,r)\right) \\
	&\quad \simeq D\left(\sfA(p,q)\times \sfA_{P}(q,r), P_{0}(p,r)\cap (q,\infty)  \right) \to D\left(\sfA_{P}(p,r), P_{0}(p,r) \right)\,.
\end{align*}
For a map $\phi \colon [0]\join [m] \to [0]\join [n]$ leaving 0 fixed, we associate a map of algebras $\sfA^{D}_{\phi^{*}P} \to \sfA^{D}_{P}$ given by $s_{0}^{*}\phi_{!}\colon s_{0}^{*}\phi^{*}P \to s_{0}^{*}P$ on objects.
We note that $P_0 \simeq (\phi_{*}P)_{0}$, so whenever relevant we use the stratum inclusion
\begin{align*}
	D(\sfA_{\phi^{*}P }(p,q), P_{0}(p,q) ) \to D(\sfA_{P}(p,q), P_{0}(p,q) )
\end{align*}
covering $\phi_{!}\colon \sfA_{\phi^{*}P}(p,q) \to \sfA_{P}(p,q)$ given by noting that the stratum $\phi$ has $Q_{P_{0}(p,q)}(\phi)=\emptyset$.
For the morphism objects not of this form we use the usual $\phi_{!}$.
One can then check that $\sfA_{P}^{D}$ determines a functor
\begin{align*}
	\cA^{D} \colon \cP_{\Delta^{0}_{s}/} \to \Alg(\Mod_{\diamond})\,.
\end{align*}
We associate the minimal and maximal object of $\overline{P}_{0}(p,q)$ with $p$ and $q$ respectively.
For $p\in P_{0}, q\in P_{[n]}$ we have a pullback square
\begin{equation}\label{eq:A_D_2}
	\begin{tikzcd}
		A_{s_{0}^{*}P}( (p,0), q) \arrow[r] \arrow[d]      &   A_{s_{0}^{*}\overline{P_{0}}(p,q)}( (p,0), (q,1) )  \arrow[d] \\
		A_{P}(p,q)  \arrow[r]               &   A_{\overline{P}_{0}(p,q)}(p,q)
	\end{tikzcd}
\end{equation}
where the horizontal maps are defined by deleting all the interior edges labeled by points outside of $P_{0}$, and all vertex labels except possibly the label $\left\{ 1\right\}$ for the top map.
When we pass to slices in the bottom row of \eqref{eq:A_D_2} we have
\begin{equation*}
	\begin{tikzcd}
		{[1]^{Q_{\sfA_{P}(p,q)}(\gamma)} } \arrow[r, "\cap P_{0}"]  \arrow[d, "\simeq"] &  { [1]^{Q_{P_{0}(p,q)}(\gamma) } } \arrow[d, "\simeq"] \\
		\sfA_{P}(p,q)_{/\gamma} \arrow[r] \arrow[d]      &   \sfA_{\overline{P}_{0}(p,q)(\gamma)}  \arrow[d] \\
		\sfA_{P}(p,q)  \arrow[r]               &   A_{\overline{P}_{0}(p,q)}(p,q) \,.
	\end{tikzcd}
\end{equation*}
When we pull these slices back in \eqref{eq:A_D_2}, we recover the square \eqref{eq:D(1^Q,S)}, so we get an equivalence $D(\sfA_{P}(p,q), P_{0}(p,q)) \simeq \sfA_{s_{0}^{*}P}( (p,0),q )$.
Because the pullback diagrams \eqref{eq:A_D_2} are natural with respect to composition, this determines an equivalence of algebras $\sfA_{P}^{D}\simeq \sfA_{s_{0}^{*}P}$.
The diagrams \eqref{eq:A_D_2} are moreover natural with respect to restriction along some $\phi$ which leaves $0$ fixed, so we have an equivalence of functors $\cP_{\Delta^{0}_{s}/}  \to \Alg(\Mod_{\diamond})$ between $\cA_{\Delta^{1}_{s}/} \circ s_{0}^{*}$ and  $\cA_{\Delta^{0}_{s}/}$.

An appropriate domain for the construction $D(\bX,P)$ is the category of stratified manifolds with a choice of collars.
Because collars is a contractible choice, the corresponding semi-simplicial set $\PreFlow_{c}^{\mu}$ of $\mu$-structured flow categories with collars will have a forgetful map $\PreFlow^{\mu}_{c}\to \Flow^{\mu}$ which is a trivial fibration.
Quasi-unitality is preserved under trivial fibrations, so it suffices to construct the degeneracies on $\PreFlow^{\mu}_{c}$.

\begin{definition}\label{def:PreMfd_c}
	Let $\Mfd_{c}^{\otimes}$ denote the monoidal category where
	\begin{enumerate}
		\item  An object is an object $\bX=(X,\cP,U\psi)$ of $\Mfd_{\diamond}$ together with a system of collars $\lambda^{\sigma} \colon \del^{\sigma}X \times [0,\infty)^{Q_{\cR}(\sigma)} \to X$ compatible with the normal framing $\psi$.
		\item  A morphism $(\bX,\lambda) \to (\bY, \kappa)$ is a morphism $\bX \to \bY$ in $\Mfd_{\diamond}$ such that for each minimal $\sigma \in \cP_{\bX}$, the diffeomorphism $\del^{\sigma}\bX \to \del^{f(\sigma)}\bX$ intertwines the collars $\lambda$ and $\kappa$.
		\item Monoidal product is given by the usual product of stratified manifolds, equipped with the product collars.
	\end{enumerate}
\end{definition}

We can inductively construct a system of collars on the $D(n+1)$ such that the stratum inclusions in \Cref{lem:D(s+1)_stratification} are compatible with collars.
Then for an object $(\bX,\bR)$ of $\Mfd_{c}$ and a poset $P$ of codimension 1 objects of $\bR$ we can construct a system of collars on $D(\bX,P)$ by using the collars on $D(n+1)$ in the local model.
This has the property that the stratum inclusions in \Cref{lem:D(XP)_stratification} are compatible with collars, and that the under the identification in \Cref{lem:D(XP)_monoidal}, the collar on $D(\bX\times \bY, P\times \pi_{0}\bR')$ agrees with the product collars.

We now define a functor
\begin{align*}
	s_{0}\colon \cA_{\Delta^{0}_{s}/}^{*}\Alg(\Mfd_{c}) \to \cA_{\Delta^{1}_{s}/}(\Mfd_{c})
\end{align*}
by assigning a simplex $\bX$ with objects $P\to [0]\join [n]$ to the simplex $s_{0}\bX$ with objects $s_{0}^{*}P\to [1]\join [n]$ and morphism objects
\begin{align*}
	s_{0}\bX( (p,0) ,q) &= D(\bX, P_{0}(p,q)), & p\in P_{0}, q\in (s_{0}^{*}P)_{[n]}\\
	s_{0}\bX( (p,1), q ) &= \bX(p,q), & p\in P_{0}, q\in (s_{0}^{*}P)_{[n]} \\
	s_{0}\bX(p,q) &= \bX(p,q), & p,q \in (s_{0}^{*}P)_{0} \text{ or } p,q \in d_{0}^{*}s_{0}^{*}P \,,
\end{align*}
each stratified by $\sfA^{D}_{P}(p,q)$, and equipped with the system of collars induced by the fixed choice of collars on the $D(n+1)$.
We use \Cref{lem:D(XP)_monoidal} and the equivalence of algebras  $\sfA^{D}_{P}\simeq \sfA_{s_{0}^{*}P}$ to define multiplication maps.
Because the construction of $D(\bX,P)$ is natural in diffeomorphisms which intertwine systems of collars, $s_{0}$ is natural in isomorphisms of algebras in $\Mod_{c}$.
Note that $d_{0}s_{0}\bX$ is naturally equivalent to $\bX$ because we remove all the mapping spaces where degeneration happens.
Under the equivalence $\sfA^{D}_{P} \simeq \sfA_{s_{0}^{*}P}$, the inclusion
\begin{align*}
	d_{1} \colon \sfA_{P}(p,q) \to \sfA_{s_{0}P}( (p,0),q )
\end{align*}
has $Q_{P_{0}}(d_{1})=\emptyset$, and corresponds to the minimal stratum under the map $\sfA_{s_{0}P} \to \sfA_{P}$, and so by \Cref{lem:D(XP)_stratification}, the corresponding stratum of $D(\bX(p,q), P_{0}(p,q))$ is equivalent to $\bX$.
This can be made natural, showing that the composites
\begin{align*}
	\cA_{\Delta^{0}_{s}/}^{*}\Alg(\Mfd_{c}) \xrightarrow{s_{0}} \cA_{\Delta^{1}_{s}/}(\Mfd_{c}) \xrightarrow{d_{i}^{*}} \cA_{\Delta^{0}_{s}/}^{*}\Alg(\Mfd_{c})
\end{align*}
are equivalent to the identity. To upgrade this to $\mu$-structures we will show that $I(s_{0}\bX)$ is equivalent to $s_{0}^{*}I\bX$ as $\cS_{/ \bZ\times BO}$-algebras.

Let $\cF \colon \Pos/\Simp \to \Set \to \Opd$ denote the functor which takes $P\to [n]$ to the operad $\Simp^{\op}_{P^{\simeq}}$.
We denote the restriction of $\cF$ to $\cP$ by $s\cF$.
As for $\cA$, we precompose this with $\cP_{\Delta^{n}/} \to \cP$ to obtain $s\cF_{\Delta^{n}_{s}/}$.
Forgetting the map to $\Alg(\Mod)$ and composing with $I\colon \Mfd \to \bVect(\CW)$ determines functors
\begin{align*}
	\cA^{*}_{\Delta^{n}_{s}/}\Alg(\Mfd) \to s\cF^{*}_{\Delta^{n}_{s}/}\Alg(\bVect(\CW))
\end{align*}
that are natural with respect to the functoriality of semi-simplicial slices.
Because $\cF$ extends to $\Pos/\Simp$, we have a natural transformation $ s\cF_{\Delta^{0}_{s}/}\circ s_{0}^{*}  \to s\cF_{\Delta^{0}_{s}/}$, and taking cartesian lifts of this gives a functor
\begin{align*}
	s_{0}^{*} \colon s\cF^{*}_{\Delta^{0}_{s}/}\Alg(\bVect(\CW)) \to s\cF^{*}_{\Delta^{1}_{s}/}\Alg(\bVect(\CW))
\end{align*}
which takes an algebra $\sfA$ over the objects $P\to [0]\join [n]$ to the algebra
\begin{align*}
	\Simp^{\op}_{s_{0}^{*}P} \xrightarrow{s_{0}} \Simp^{\op}_{P} \xrightarrow{\sfA} \bVect(\CW) \,.
\end{align*}
\begin{lem}\label{lem:D_is_htpy_eq}
	There is a natural transformation
	\begin{equation*}
		\begin{tikzcd}
			\cA_{\Delta^{0}_{s}/}^{*}\Alg(\Mfd_{c}) \arrow[r,"s_{0}"] \arrow[d, "{I\oplus (\bR,\bR)}"']      &   \cA_{\Delta_{s}^{1}/}^{*}\Alg(\Mfd_{c}) \arrow[dl, "\pi", Rightarrow, shorten=2em]  \arrow[d, "I"] \\
			s\cF_{\Delta^{0}_{s}/}^{*}\Alg(\bVect(\CW))  \arrow[r, "s_{0}^{*}"']               &  s\cF_{\Delta^{1}_{s}/}^{*}\Alg(\bVect(\CW)) \, ,
		\end{tikzcd}
	\end{equation*}
	which becomes a natural equivalence after composing with $\bVect(\CW) \to \cS_{/\bZ\times BO}$.
\end{lem}
\begin{proof}
		By \cite[Lemma 6.11]{AB}, there exists consistent normal framings on each $[0,\infty)^{n}$ and $D(n+1)$ which are compatible with the identifications \Cref{lem:D(s+1)_stratification}.
    Moreover, there exists isomorphisms $\phi_{n}\colon TD(n+1) \oplus \bR^{ \left\{ q\right\}} \simeq \pi^{*}T[0,\infty)^{n} \oplus \bR^{ \left\{ q\right\}} \oplus \bR^{ \left\{ 1\right\} }$ such that
	\begin{equation*}
		\begin{tikzcd}
			T \del^{\gamma} D(n+1) \oplus \bR^{Q_{D(\cR,P)}(\sigma)} \oplus \bR^{ \left\{ q\right\}} \arrow[r] \arrow[d]      &   T D(n+1) \oplus \bR^{ \left\{ q\right\}} \arrow[d, "\phi_{n}"] \\
			T \del^{\alpha} [0,\infty)^{n} \oplus \bR^{Q_{\cR}(\alpha)} \oplus \bR^{ \left\{ q\right\}} \oplus \bR  \arrow[r]               &   T [0,\infty)^{n} \oplus \bR^{ \left\{ q\right\}} \oplus \bR
		\end{tikzcd}
	\end{equation*}
	commutes, where the horizontal maps are induced by the consistent normal framing, and the left vertical is induced by some other $\phi_{m}$ under the identification in \Cref{lem:D(s+1)_stratification}.
	Because the maps $\pi\colon D(\bX, P) \to \bX$ are pulled back from local models of the form $D(n+1) \to [0,\infty)^{n}$, we can use the above squares to construct equivalences
	\begin{align}\label{eq:phi_X}
		\phi_{\bX}\colon TD(\bX, P) \oplus \bR^{ \left\{ q\right\}} \simeq \pi^{*}T\bX \oplus \bR^{ \left\{ q\right\}}\oplus \bR^{ \left\{ 1\right\}}
	\end{align}
	such that the following diagram commutes:
	\begin{equation}
		\begin{tikzcd}\label{eq:degeneration_is_htpy_eq}
			T \del^{\gamma} D(\bX,P) \oplus \bR^{Q_{D(\cR,P)}(\gamma)} \oplus \bR^{ \left\{ q\right\}} \arrow[r] \arrow[d]      &   T D(\bX, P) \oplus \bR^{ \left\{ q\right\}} \arrow[d] \\
			\pi^{*}T\del^{\sigma}\bX \oplus \bR^{Q_{\cR}(\sigma)} \oplus \bR^{ \left\{ q\right\}} \oplus \bR^{ \left\{ 1\right\}} \arrow[r]               & \pi^{*} T\bX \oplus \bR^{ \left\{ q\right\}} \oplus \bR^{ \left\{ 1\right\}} \,.
		\end{tikzcd}
	\end{equation}
	For a simplex $\bX$ with objects $i\colon P\to [0]\join[n]$, we construct a $\bVect(\CW)$-enriched functor
	\begin{align*}
		\pi_{\bX} \colon I(s_{0}\bX) \to s_{0}^{*}I\bX \oplus (\bR^{ \left\{ 1\right\}},\bR^{ \left\{ 1\right\}})
	\end{align*}
	covering the identity on objects.
  Whenever both $p,q$ are in  $(s_{0}^{*}P)_{0}$ or in $d_{0}^{*}(s_{0}^{*}P)$, we have $s_{0}\bX(p,q)=\bX(p,q)$ as objects of $\Mfd_{c}$, so we let $\pi_{\bX}$ be given by the cartesian lift $I\bX(p,q) \to I\bX(p,q) \oplus (\bR^{ \left\{ 1\right\}},\bR^{ \left\{ 1\right\}})$ over $\bR^{ \left\{ 1\right\}}$.

	At some $s_{0}\bX((p,0),q)$, we have equivalences
	\begin{align}\label{eq:D_is_htpy_eq}
		\sfU_{s_{0}^{*}P}( (p,0), q ) & \simeq \sfU_{P}(p,q) \oplus (\bR^{ \left\{ 1\right\} }, 0) \\
		(\bR^{ \left\{ 1,\dots i(q)+1  \right\}}, \bR^{q}  ) &\simeq (\bR^{ \left\{ 1,\dots i(q) \right\}}, \bR^{ \left\{ q\right\}} ) \oplus (\bR^{ \left\{ 1\right\}}, 0)  \nonumber
	\end{align}
	given by using the map $d_{0} \colon \left\{ 1,\dots , i(q)\right\} \to \left\{ 1,\dots, i(q)+1\right\}$ which misses the element~$1$.
  We therefore let the component of $\pi_{\bX}$ at $( (p,0),q)$ cover $\bR^{ \left\{ 1\right\}}$ in $\sfB\Vect$.
  On the negative part we use the equivalence \eqref{eq:D_is_htpy_eq}.
  On the positive part we use an equivalence of the form \eqref{eq:phi_X}.
  Because the maps $\phi_{\bX}$ are pulled back from a fixed local model, the construction of $\pi_{\bX}$ is natural with respect to isomorphisms of $\Mod_{c}$-categories.
  The diagram \eqref{eq:degeneration_is_htpy_eq} takes care of the naturality with respect to composition and face maps.

	Because the fibers of $D(\bX,P) \to \bX$ are contractible, the transformation $\pi_{\bX}$ becomes an equivalence when we push forward to $\cS$. Since $\cS_{/ (\bZ\times BO)}\to \cS$ is conservative, $\pi_{\bX}$ becomes an equivalence already in $\cS_{/\bZ\times BO}$.
\end{proof}
Note that when we push forward to $\cS_{/\bZ\times BO}$, the $(\bR,\bR)$ factor becomes trivial.
We therefore have a commutative square
	\begin{equation}\label{eq:degeneration_is_htpy_eq_2}
	\begin{tikzcd}
		\cA_{\Delta^{0}_{s}/}^{*}\Alg(\Mfd_{c}) \arrow[r,"s_{0}"] \arrow[d, "{I}"']      &   \cA_{\Delta_{s}^{1}/}^{*}\Alg(\Mfd_{c}) \arrow[dl, "\pi", Rightarrow, shorten=1em]  \arrow[d, "I"] \\
		s\cF_{\Delta^{0}_{s}/}^{*}\Alg(\cS_{/ \bZ\times BO})  \arrow[r, "s_{0}^{*}"']               &  s\cF_{\Delta^{1}_{s}/}^{*}\Alg(\cS_{/ (\bZ\times BO)}) \, .
	\end{tikzcd}
\end{equation}
We define $\Mfd^{\mu}_{c}$ as the pullback
\begin{align*}
	\Mfd^{\mu}_{c} = \Mfd_{\diamond}^{\mu}\times_{ \Mfd_{\diamond}^{\otimes} } \Mfd_{c}^{\otimes}
\end{align*}
of double $\infty$-categories.
Then by a pasting argument and because $\Alg$ preserves pullbacks, we have pullback squares
\begin{equation*}
	\begin{tikzcd}
		\cA_{\Delta^{n}_{s}/}^{*}\Alg(\Mfd_{c}^{\mu}) \arrow[r] \arrow[d]      &   s\cF_{\Delta^{n}_{s}/}^{*}\Alg( (\cS_{/\bZ\times BO})_{/\mu}  )  \arrow[d] \\
		\cA_{\Delta^{n}_{s}/}^{*}\Alg(\Mfd_{c})  \arrow[r,"I"]               &  s\cF_{\Delta^{n}_{s}/}^{*}\Alg(\cS_{/ \bZ\times BO}) \,.
	\end{tikzcd}
\end{equation*}
We can then define the initial degeneracy $s_{0}^{\mu}\colon \PreFlow^{\mu}_{\Delta^{0}_{s}/} \to \PreFlow^{\mu}_{\Delta^{1}_{s}/}$ by straightening the induced map on pullbacks in the cube
\begin{equation*}
	\begin{tikzcd}[row sep=1em, column sep=1em]
		\cA_{\Delta^{0}_{s}/}^{*}\Alg(\Mfd_{c}^{\mu}) \arrow[rr] \arrow[dd] \arrow[dr] & & \cA_{\Delta^{1}_{s}/}^{*}\Alg(\Mfd_{c}^{\mu}) \arrow[dd] \arrow[dr] \\
		&  s\cF_{\Delta^{0}_{s}/}^{*}\Alg( (\cS_{/\bZ\times BO})_{/\mu}  ) \arrow[rr, "s_{0}^{*}"{near start}, crossing over]  & &  s\cF_{\Delta^{1}_{s}/}^{*}\Alg( (\cS_{/\bZ\times BO})_{/\mu}  ) \arrow[dd] \\
		\cA_{\Delta^{0}_{s}/}^{*}\Alg(\Mfd_{c}) \arrow[rr, "s_{0}"{near end}] \arrow[dr] & & \cA_{\Delta^{1}_{s}/}^{*}\Alg(\Mfd_{c}) \arrow[dr] \\
		& s\cF_{\Delta^{0}_{s}/}^{*}\Alg(\cS_{/ \bZ\times BO}) \arrow[rr, "s_{0}^{*}"] & & s\cF_{\Delta^{1}_{s}/}^{*}\Alg(\cS_{/ \bZ\times BO}) \, ,
		\arrow[from=2-2, to= 4-2, crossing over]
	\end{tikzcd}
\end{equation*}
where the bottom square is \eqref{eq:degeneration_is_htpy_eq_2}.
To show that $d_{i}\circ s_{0}^{\mu}= id$ for $i=0,1$, it remains to show that the whiskering of $d_{i}$ with the natural transformation $\pi$ becomes the identity natural transformation once we push forward to $\cS_{\bZ\times BO}$.
For $d_{0}$ this is immediate because on the mapping spaces in question, each $\pi_{\bX}$ is just given by the cartesian transport along $(\bR^{ \left\{ 1\right\}}, \bR^{ \left\{ 1\right\}})$. For $d_{1}$, it follows by construction of the $\phi_{\bX}$ in \cite{AB} that the restriction of $\phi_{\bX}$ to the $d_{1}$ stratum is given a map
\begin{align*}
	T\bX(p,q) \oplus \bR^{Q_{d_{1}}} \oplus \bR^{ \left\{ q\right\}} \simeq \pi^{*}T\bX \oplus \bR^{ \left\{ q\right\}} \oplus \bR^{ \left\{ 1\right\}}
\end{align*}
which splits as the identity on $T\bX(p,q)$, the identity on $\bR^{ \left\{ q\right\}}$, and an equivalence identifying the inward normal direction $Q_{d_{1}}$ with the positive generator $\left\{ 1\right\}$.
This shows that the whiskering of $d_{1}$ and $\pi$ also agrees with the cartesian transport over $(\bR, \bR)$, which becomes homotopic to the identity when we pass to $\cS_{/\bZ\times BO}$.

The construction of $s_{\omega}$ is completely symmetric, and it is clear that when we restrict to $\PreFlow_{0}$, that both are given by mapping a flow category $\bX$ with objects $P$ to the identity bimodule $s\bX \colon \bX \to \bX$ whose morphism spaces are $D(\bX(p,q), P)$.

\begin{cor}\label{cor:flow_is_qu}
	The semi-simplicial space $\PreFlow^{\mu}$ is quasi-unital.
\end{cor}

\begin{proof}
	We have a trivial fibration $\PreFlow^{\mu}_{c} \to \PreFlow^{\mu}$ from a semi-simplicial space which admits outer degeneracies and is therefore quasi-unital by \Cref{prop:outer_degeneracies}.
\end{proof}

\begin{cor}\label{cor:PreFlow_lands_in_IK_qu}
	The functor $\PreFlow^{(-)}$ factors through the full subcategory $\widehat{\IK}_{qu}$.
\end{cor}
\begin{proof}
	By construction of the degeneracy map for $\mu$, pushforward along $f\colon \mu \to \nu$ gives rise to a commutative square
	\begin{equation*}
		\begin{tikzcd}
			\PreFlow^{\mu}_{\Delta^{0}_{s}/} \arrow[r, "s_{0}"] \arrow[d]      &   \PreFlow^{\mu}_{\Delta^{1}_{s}/}  \arrow[d] \\
			\PreFlow^{\nu}_{\Delta^{0}/}  \arrow[r, "s_{0}"]               &   \PreFlow^{\nu}_{\Delta^{1}_{s}/}\,.
		\end{tikzcd}
	\end{equation*}
	This means that for each $\mu$-structured flow category $\bX$, we have a canonical equivalence (as algebras in $\Mfd^{\nu}$)
	\begin{align*}
		f_{!}(s_{0}\bX) \simeq s_{0}(f_{!}\bX) \,.
	\end{align*}
	In particular, $f_{!}$ maps an idempotent equivalence at each vertex of $\PreFlow^{\mu}$ to an equivalence in $\PreFlow^{\nu}$, and so $f_{!}$ is quasi-unital.
\end{proof}
\begin{definition}\label{def:PreFlow_functor}
	We write $\Flow^{(-)} \colon \Cat_{\infty /(U/O)} \to \CatHat$ for the composite of $\PreFlow^{(-)}$ with the localization $\cT$ of \Cref{prop:semi-simplicial_results}.
\end{definition}

\subsection{Generalized companions}\label{subsec:companions}
In this subsection, we will study generalized companions (in the sense of \Cref{def:admits_companions}) in the semi-simplicial $\infty$-category $\bfPreFlow^{\mu}$.
We will restrict to a certain subcategory which admits all companions, and develop some theory for when the corresponding companion functor produces colimits in $\Flow^{\mu}$.

The intuition for the existence of companions is as follows.
Assume that $\bX$ is a flow category with poset of objects $R$, and that $I\subset R$ is a filter.
This convex inclusion is a morphism in $\cP$, and its Cartesian lift at $\bX$ exhibits the restriction of $\bX$ to the objects~$I$ as a flow category $\bX_{I}$.
Now let $\bH$ be an $n$-simplex in the slice $\PreFlow^{\mu}_{\bX/}$.
This means that~$\bH$ is a flow $n+1$-simplex whose $0$-th vertex is $\bX$.
In particular, the poset of objects $P\to [n+1]$ has $P_{0}=R$.
We then let $P_{I} \subset P$ denote the full subposet spanned by objects in $d_{0}^{*}P \subset P$, and objects in $I\subset P_{0} \subset P$.
Because $I$ was a filter, the inclusion $P_{I}\subset P$ is convex, so by restricting $\bH$, we get a flow $n+1$-simplex $\bH_{I}$.
In what follows, we will show that this restriction procedure is sufficiently natural that it defines a semi-simplicial map $\PreFlow^{\mu}_{\bX/} \to \PreFlow^{\mu}_{\bX_{I}/}$, which by the Yoneda lemma determines a unique morphism in $\Flow^{\mu}$.

\begin{definition}\label{def:P_uparrow}
	Let $\cP^{\uparrow}\subset \cP$ denote the wide subcategory spanned by morphisms $f\colon P \to P'$ covering $\phi \colon [n] \to [m]$ such that $P \to \phi^{*}P'$ is the inclusion of a filter.
\end{definition}

\begin{lem}\label{lem:P_has_companions}
	The semi-simplicial category $\cP^{\uparrow}$ admits companions.
  Let $\cI$ be a category and $Q$ an object in $\cP^{\uparrow}_{n}$.
  Any Cartesian lifting problem
	\begin{equation*}
		\begin{tikzcd}
			\cI \arrow[r, "G'"] \arrow[d]      &   \cP^{\uparrow}_{n+1}  \arrow[d] \\
			\cI^{\triangleright}  \arrow[r, "{(F, \left\{ P\right\})}"']  \arrow[ur, dashed, "G"]             &   \cP^{\uparrow}_{0} \times \cP^{\uparrow}_{n}
		\end{tikzcd}
	\end{equation*}
	where $F$ becomes a colimit diagram in $\Pos$ admits a unique solution.
\end{lem}
\begin{proof}
	A map of posets $P\to [n+1]$ is uniquely determined by the poset $P_{0}$, the map $d_{0}^{*}P\to [n]$, and a map of posets $c_{P}\colon P_{0}^{\op}\times d_{0}^{*}P \to [1]$ given by
	\begin{align*}
		c_{P}(p,q) = \begin{cases}
			1, & p<q \in P \\
			0, & \text{else.}
		\end{cases}
	\end{align*}
	A map of posets $f\colon P \to P'$ over $[n+1]$ is similarly determined by the maps of posets~$f_{0}$ and $d_{0}^{*}f$ together with a natural transformation
	\begin{align*}
		c_{P} \to (f_{0} \times d_{0}^{*}f)^{*}c_{P'} \,.
	\end{align*}
	In other words, the functor $(v_{0}^{*},d_{0}^{*}) \colon \Pos_{/[n+1]} \to \Pos \times \Pos_{/[n]}$ is the Cartesian fibration classified by
	\begin{align}\label{eq:cP_as_profunctors}
		\Pos^{\op} \times (\Pos_{/[n]})^{\op}  &\to \Cat \\
		(P\to [n],Q\to [m])&\mapsto \Fun(P^{\op}\times Q, [1]) \,. \nonumber
	\end{align}

	If we assume that $P_{0}$ and $d_{0}^{*}P$ are both locally finite, the local finiteness of $P$ is equivalent to each $p\in P_{0}$ having $p<q$ for only finitely many $q\in d_{0}^{*}P$.
  We write $\Fun_{c}(-,[1])$ for the subposet of morphisms that take the value $0$ on all but finitely many inputs.
  Then $P$ is locally finite if and only if $P_{0}$ and $d_{0}^{*}P$ are, and the adjunct of $c_{P}$ belongs to
	\begin{align*}
		\Fun(P_{0}^{\op}, \Fun_{c}(d_{0}P, [1])) \subset \Fun(P_{0}^{\op}, \Fun(d_{0}^{*}P, [1])).
	\end{align*}

	A map of posets $f\colon P \to P'$ over $[n+1]$ such that $f_{0}$ is a filter and $d_{0}^{*}f$ an equivalence is necessarily convex.
  Therefore, the Cartesian lift of a morphism $g\colon R \to P_{0}$ at some locally finite $P\to [n+1]$ belongs to the subcategory $\cP^{\uparrow}$.
  Moreover, if $Q\to P$ is a morphism in $\cP^{\uparrow}_{n+1}$ and we have a factorization $Q_{0} \to R \to P_{0}$ through $g$, we get a unique factorization
	\begin{align*}
		Q \to g^{*}P \to P
	\end{align*}
	in $\Pos_{/[n+1]}$.
  Because the composite $Q\to P$ and the Cartesian edge $g^{*}P\to P$ are filter inclusions, so is the map $Q\to g^{*}P$.
  This implies that the Cartesian lift $g^{*}P$ is also Cartesian with respect to $\cP^{\uparrow}_{n+1} \to \cP^{\uparrow}_{n}$, and so in particular, $\cP^{\uparrow}_{n+1}\to \cP^{\uparrow}_{0}\times \cP^{\uparrow}_{n}$ is a morhism of Cartesian fibrations.

	The functor
	\begin{align*}
		(id\join \phi)^{*}\colon \Pos_{/[n+1]} \to \Pos_{/[m+1]}
	\end{align*}
	can be obtained by straightening the natural transformation
	\begin{equation*}
		\begin{tikzcd}
			\Pos\times \Pos_{/[n]} \arrow[r] \arrow[d, "id \times \phi^{*}"']      &   \Pos \times \Pos  \\
			\Pos \times \Pos_{/[m]}  \arrow[ur, ""{name=S}]
			\arrow[Rightarrow, from=S, to=1-1]
		\end{tikzcd}
	\end{equation*}
	and is therefore a morphism of Cartesian fibrations.
  Because an edge in $\cP^{\uparrow}_{n+1}$ is Cartesian if and only if it is Cartesian in $\Pos_{/[n+1]}$, the corresponding map $\cP_{n+1}^{\uparrow} \to \cP_{m+1}^{\uparrow}$ is also a morphism of Cartesian fibrations over $\cP^{\uparrow}_{0}$.

	For a fixed $Q$, the classifying functor \eqref{eq:cP_as_profunctors} takes colimits in $\Pos$ to limits of categories, so the uniqueness of Cartesian lifts follows from \Cref{lem:limts_as_lifting}.
\end{proof}
\begin{definition}\label{def:Flow_uparrow}
	Let $\bfPreFlow^{\mu\uparrow}$ be the semi-simplicial $\infty$-category defined by the pullback
	\begin{equation*}
		\begin{tikzcd}
			\bfPreFlow^{\mu \uparrow} \arrow[r] \arrow[d]      &   \bfPreFlow^{\mu}  \arrow[d] \\
			\cP^{\uparrow}  \arrow[r]               &   \cP \,.
		\end{tikzcd}
	\end{equation*}
\end{definition}
\begin{lem}\label{lem:Flow_has_companions}
	The semi-simplicial $\infty$-category $\bfPreFlow^{\mu \uparrow}$ admits companions.
\end{lem}
\begin{proof}
	The morphism of semi-simplicial categories $\bfPreFlow^{\mu\uparrow} \to \cP$ gives rise to a diagram
	\begin{equation*}
		\begin{tikzcd}
			\bfPreFlow^{\mu\uparrow}_{n+1}
			\arrow[r]
			\arrow[d]     &
			\cP^{\uparrow}_{n+1}
			\arrow[d] \\
			\bfPreFlow^{\mu\uparrow}_{0} \times \bfPreFlow^{\mu\uparrow}_{n}
			\arrow[r]     &
			\cP^{\uparrow}_{0}\times \cP^{\uparrow}_{n}
		\end{tikzcd}
	\end{equation*}
	where the horizontal maps are right fibrations.
  We want to show that the left vertical is a morphism of Cartesian fibrations over $\bfPreFlow^{\mu\uparrow}_{0}$.
  By the 2-out-of-3 property for Cartesian lifts, it suffices to show that any morphism of the form $(f,id)$ in $\bfPreFlow^{\mu\uparrow}_{n} \times \bfPreFlow^{\mu\uparrow}_{0}$ admits a Cartesian lift.

	Because the lower horizontal is a right fibration, $(f,id)$ is a Cartesian lift of its image $(f',id)$ in $\cP^{\uparrow}_{0}\times \cP_{0}^{\uparrow}$.
  By \Cref{lem:P_has_companions}, the edge $f'$ admits a Cartesian lift $g'$ in $\cP^{\uparrow}_{n+1}$ which is also a Cartesian lift of $(f',id)$. Then because the top horizontal is a right fibration, any Cartesian lift $g$ of $g'$ to $\bfPreFlow^{\mu\uparrow}$ is necessarily mapped to $(f,id)$, and so by 2-out-of-3 for Cartesian edges, $g$ is a Cartesian lift of $(f,id)$.

	For a map $\phi \colon [m] \to [n]$ in $\Simp_{s}$, we similarly get a diagram
	\begin{equation*}
		\begin{tikzcd}
			\bfPreFlow^{\mu\uparrow}_{n+1}
			\arrow[r]
			\arrow[d]    &
			\bfPreFlow^{\mu\uparrow}_{m+1}
			\arrow[d]
			\arrow[r]    &
			\bfPreFlow^{\mu\uparrow}_{0}
			\arrow[d]    \\
			\cP^{\uparrow}_{n+1}
			\arrow[r]    &
			\cP^{\uparrow}_{m+1}
			\arrow[r]    &
			\cP^{\uparrow}_{0}
		\end{tikzcd}
	\end{equation*}
	where the vertical maps are right fibrations, and the bottom row exhibits a morphism of Cartesian fibrations over $\cP_{0}^{\uparrow}$ by \Cref{lem:P_has_companions}.
  A diagram chase using the 2-out-of-3 property for Cartesian lifts as above shows that the top row also exhibits a morphism of Cartesian fibrations over $\bfPreFlow^{\mu\uparrow}_{0}$.
\end{proof}

\begin{definition}\label{def:flow_companion}
	By \Cref{lem:Flow_has_companions}, the semi-simplicial $\infty$-category $\bfFlow^{\mu \uparrow}$ admits a companion functor which we denote
	\begin{align*}
		\cU\colon \bfFlow^{\mu\uparrow}_{0} \to \Flow^{\mu}.
	\end{align*}
	The natural dual to \Cref{lem:Flow_has_companions} also holds; the subcategory $\bfFlow^{\mu\downarrow}$ spanned by morphisms covering inclusions of cofilters admits conjoints in the sense of \Cref{rem:conjoints}.
  We write
	\begin{align*}
		\cU' \colon (\bfFlow^{\mu\downarrow}_{0})^{\op} \to \Flow^{\mu}
	\end{align*}
	for the conjoint functor.
\end{definition}

\subsection{Stability}\label{subsec:stability}

We now proceed to compute some colimits in $\Flow^{\mu}$. We begin by showing that $\Flow^{\mu}$ is pointed, i.e., that there exists an object which is both terminal and initial.
Recall from \Cref{ex:empty_flow_cat} that the empty flow category $\bZero$ in $\Flow^{\mu}$ is the unique flow category with objects given by the empty poset.
As the notation suggests, this is indeed a zero object in $\Flow^{\mu}$, intuitively because a bimodule between $\bX$ and $\bZero$ contains no manifolds not already specified by $\bX$.

\begin{lem}\label{lem:flow_is_pointed}
	$\Flow^{\mu}$ is pointed by the empty flow category.
\end{lem}

\begin{proof}
	Let us show that $\bZero$ is initial. It suffices to show that the forgetful map
	\begin{align*}
		\Flow^{\mu}_{\bZero/} \to \Flow^{\mu}
	\end{align*}
	is an equivalence.
	By \Cref{prop:semi-simplicial_results}, this functor is represented by the map of semi-simplicial spaces
	\begin{align*}
		\PreFlow^{\mu}_{\bZero/} \to \PreFlow^{\mu}.
	\end{align*}
	We show this is an equivalence levelwise over $\Simp_{s}$, where we have diagrams of spaces
	\begin{equation}\label{eq:flow_is_pointed}
		\begin{tikzcd}
			(\PreFlow^{\mu}_{\bZero/})_{n} \arrow[r] \arrow[d]      &   \PreFlow^{\mu}_{n} \arrow[d] \\
			(\cR(\cP)_{\emptyset/})_{n}  \arrow[r]               &   \cR(\cP)_{n} \,.
		\end{tikzcd}
	\end{equation}
	The bottom map is an equivalence because by the proof of \Cref{lem:P_has_companions}, its fiber at some $P\to [n]$ is the space of maps $\emptyset^{\op}\times P \to [1]$, which is contractible.
  The comparison map in the vertical fibers over some $P$ is the Cartesian transport of $\Alg(\Mfd^{\mu}_{\diamond}) \to \Alg(\Mod_{\diamond})$ along the equivalence
	\begin{align*}
		(d_{0})_{!}\colon \sfA_{P} \to \sfA_{\emptyset\join P} \,.
	\end{align*}
	This shows that \eqref{eq:flow_is_pointed} is a pullback, so its top vertical is also an equivalence.
  A symmetric argument shows that $\bZero$ is also terminal.
\end{proof}

We now show that $\Flow^{\mu}$ admits coproducts by leveraging the companion functor.

\begin{lem}\label{lem:disjoint_unions_exist}
	Let $S$ be a set, and let $F\colon S^{\triangleright} \to \cP^{\uparrow}$ be a diagram which becomes a colimit diagram in $\Pos$.
  Then the lifting problem
	\begin{equation*}
		\begin{tikzcd}
			S \arrow[r] \arrow[d]      &   \bfPreFlow^{\mu\uparrow}_{0}  \arrow[d] \\
			S^{\triangleright}  \arrow[r, "F"]  \arrow[ur, dashed]             &   \cP^{\uparrow}_{0}
		\end{tikzcd}
	\end{equation*}
	admits a unique solution.
\end{lem}

\begin{proof}
	Because we have
	\begin{align*}
		\coprod_{i\in S} \bfA_{P_{i}} \simeq \bfA_{\coprod_{i\in S} P_{i}},
	\end{align*}
	the composite $\bfA \circ F$ is a colimit diagram in $\Dbl_{\infty/\Mod_{\diamond}}$. By the universal property of pullbacks, the lifting problem in question is equivalent to
	\begin{equation*}
		\begin{tikzcd}
			S \arrow[r] \arrow[d]      &   \Dbl_{\infty/\Mfd_{\diamond}^{\mu}}  \arrow[d] \\
			S^{\triangleright}  \arrow[r, "\bfA\circ F"']  \arrow[ur, dashed]             &   \Dbl_{\infty/\Mod_{\diamond}}.
		\end{tikzcd}
	\end{equation*}
	The straightening of a representable right fibration takes colimits to limits, so this admits a unique solution by \Cref{lem:limts_as_lifting}
\end{proof}

For a collection $\bX_{i}, i\in S$ of flow categories, we write $\coprod_{i\in S} \bX_{i}$ for the unique Cartesian lift obtained from \Cref{lem:disjoint_unions_exist}.
We can think of this as the flow category with poset of objects $\coprod_{i} P_{i}$ and
\begin{align*}
	\coprod_{i} \bX_{i}(p,q) = \begin{cases}
		\bX_{i}(p,q), & p,q\in P_{i} \\
		\emptyset, & \text{else,}
	\end{cases}
\end{align*}
with composition lifting that of each $\bX_{i}$.
\begin{lem}\label{lem:coproducts_in_flow}
	Let $S$ be a set, and $F\colon S^{\triangleright} \to \bfPreFlow^{\mu\uparrow}_{0}$ a diagram which becomes a colimit diagram in $\Pos$. Then the companion functor $\cU^{\mu\uparrow}$ maps $F$ to a colimit diagram in $\Flow^{\mu}$.
\end{lem}
\begin{proof}
	Coproducts may equivalently be computed as pushouts over the initial object, so we extend $F$ to $\cI = S^{\triangleleft}$ by setting $F(-\infty) = \bZero$.
  It then suffices to show that this extended~$F$ is mapped to a colimit diagram by the companion functor.
  Because restriction along the unique morphism $\bZero \to \bX$ agrees with the slice projection $\PreFlow^{\mu}_{\bX/} \to \PreFlow^{\mu}_{/\bZero } \simeq \PreFlow^{\mu}$, which is always an inner fibration, it suffices by \Cref{lem:limts_as_lifting} to show that for any flow $n$-simplex $\bH$, the dashed Cartesian lifting problem
	\begin{equation*}
		\begin{tikzcd}[row sep=large]
			\cI \arrow[r]
			\arrow[d]    &
			\bfPreFlow^{\mu\uparrow}_{n+1}
			\arrow[d]
			\arrow[r]   &
			\cP^{\uparrow}_{n+1}
			\arrow[d] \\
			\cI^{\triangleright}
			\arrow[r, "{(F, \underline{\bH}}"']         \arrow[ur, "G",dashed]
			\arrow[urr, "H"{near end}, ,dotted]    &
			\bfPreFlow^{\mu\uparrow}_{0} \times \bfPreFlow^{\mu\uparrow}_{n}
			\arrow[r]    &
			\cP^{\uparrow}_{0} \times \cP^{\uparrow}_{n}
		\end{tikzcd}
	\end{equation*}
	admits a unique solution.
  The space of dotted Cartesian lifts is contractible by \Cref{lem:P_has_companions}, so by 2-out-of-3 for Cartesian lifts, it suffices to show that $H$ admits a unique lift $G$.
  If we write $\ob(\bH)=R$, it follows from the proof of \Cref{lem:P_has_companions} that $H(\infty)$ is the poset over $[n+1]$ corresponding to the functor
	\begin{align*}
		\coprod_{i\in S} c_{P_{i}}\colon \left(\coprod_{i\in S} P_{i}^{\op} \right) \times R \to [1] \,.
	\end{align*}
	This is precisely the pushout
	\begin{equation*}
		\begin{tikzcd}
			\coprod_{i\in S} \emptyset \join R \arrow[r] \arrow[d]      &   \coprod_{i\in S} H(i)  \arrow[d] \\
			\emptyset \join R  \arrow[r]               &   H(\infty) \,,
		\end{tikzcd}
	\end{equation*}
	so in other words, $H$ becomes a colimit diagram in $\Pos_{/[n+1]}$.
  Now because $\bfA$ preserves coproducts in $\Pos$, and takes morphisms in $\cP^{\uparrow}$ to fully faithful double functors, the square
	\begin{equation*}
		\begin{tikzcd}
			\coprod_{i\in S} \bfA_{\emptyset \join R} \arrow[r] \arrow[d]      &   \coprod_{i\in S} \bfA_{H(i)}  \arrow[d] \\
			\bfA_{\emptyset \join R}  \arrow[r]               &   \bfA_{H(\infty)}
		\end{tikzcd}
	\end{equation*}
	in $\Dbl_{\infty/\Mod_{\diamond}}$ has projection to $\Dbl_{\infty}$ which is the image of a homotopy pushout in $\Cat_{2}$.
  This implies that the colimit $H(\infty)$ is preserved by $\bfA$, so $\bfA \circ H$ is a colimit diagram in $\Dbl_{\infty/\Mod_{\diamond}^{\otimes}}$.
  By \Cref{cor:canonical_presheaf_pullback}, the map $\bfFlow^{\mu\uparrow}_{n+1} \to \cP^{\uparrow}_{n+1}$ is the pullback along $\bfA$ of a representable right fibration.
  It therefore follows from \Cref{lem:limts_as_lifting} that there exists a unique lift $G$ of $H$.
\end{proof}

\begin{cor}\label{cor:coproducts_in_flow}
	$\Flow^{\mu}$ admits small coproducts.
\end{cor}

\begin{proof}
	Let $\bX_{s}, s\in S$ be a small collection of flow categories.
  By \Cref{lem:coproducts_in_flow}, the unique Cartesian lift in \Cref{lem:disjoint_unions_exist} is mapped by the companion functor to a colimit diagram in $\Flow^{\mu}$ which exhibits $\coprod_{s\in S} \bX_{s}$ as the coproduct.
\end{proof}

Cofibers in $\Flow^{\mu}$ can not be handled by the companion functor, because there are no maps to $\bZero$ in $\bfPreFlow^{\mu\uparrow}_{0}$.
Instead, we will use relabeling and a relative grading shift to construct cofibers in a manner analogous to mapping cones of chain complexes.

For a map of posets $P\to [1]$, we write $s_{*}P \to [0]$ for the composite $P\to [1] \to [0]$.
The starting point for constructing cofibers is to note that we have equivalences
\begin{align}\label{eq:A_sP=A_P}
	\sfA_{s_*P}(p,q) \simeq \sfA_{P}(p,q)
\end{align}
of models which is moreover compatible with composition.
This means that we can almost directly treat a bimodule $\bB$ with objects $P$ as a flow category $\tilde{c}(\bB)$ with objects $s_{*}P$.
Note however that there is a difference between $\sfU_{s_{*}P}$ and $\sfU_{P}$; the positive part of $\sfU_{P}(p,q)$ contains an extra factor $\bR^{ \left\{ 1\right\}}$ whenever $p\in P_{0},q\in P_{1}$.
There is a $\Vect$-action on $\Mod_{\diamond}$ induced by the monoidal functor
\begin{align*}
	T^{-}\colon \sfB \Vect \to \Mod_{\diamond}^{\otimes} \\
	V\mapsto (\ast, (0,V) ) \,.
\end{align*}
We think of a vector space $V$ as acting by subtracting, and therefore write the action as
\begin{align*}
	(U^{+},U^{-}) \ominus V = (U^{+},U^{-}\oplus V) \,.
\end{align*}
Up to stabilization, we can write $\sfU_{P} = \sfU_{s_{*}P } \ominus \Delta_{*}\bR_{P_{1}}$, where $\Delta_{*}\bR_{P_{1}}$ is the algebra
\begin{align*}
	\Delta_{*}\bR_{P_{1}} \colon P &\to \sfB \Vect \\
	\Delta_{*}\bR_{P_{1}}(p,q) &=\begin{cases}
		\bR, & p\in P_0, q\in P_1 \\
		0, & \text{else}.
	\end{cases}
\end{align*}
As notation suggests, this lifts along $\Delta\colon \sfE \Vect \to \sfB \Vect$ by the algebra
\begin{align*}
	\bR_{P_{1}} \colon P & \to \sfE \Vect \\
	\bR_{P_{1}}(p) &= \begin{cases}
		0, &  p\in P_{0} \\
		\bR, & p\in P_{1}.
	\end{cases}
\end{align*}
\begin{rem}
	For the remainder of this section we will use \Cref{cor:flow_cats_defined_on_P} and work with flow categories as lax functors $\bX \colon P \to \Mfd_{\diamond}^{\mu}$ rather than over $\Simp_{P^{\simeq}}^{\op}$.
  This is necessary because there will typically never be any interesting algebras $\Simp^{\op}_{P^{\simeq}} \to \sfE \Vect$.
\end{rem}

We will now show that the $\sfE(\bZ\times BO)$-module structure on $\mu$ induces a module structure on $\Mfd^{\mu}_{\diamond}$, and that by acting with a certain $\sfE(\bZ\times BO)$-algebra we can modify $\tilde{c}(\bB)$ to get a flow category.

For any symmetric monoidal $\infty$-category $\cV^{\otimes}$, we have a morphism
\begin{equation*}
	\begin{tikzcd}
		(\cV^{\otimes})^{[1]} \arrow[r,"{(ev_{0},ev_{1})}"] \arrow[dr, "ev_{1}"']      &   \cV^{\otimes}\times_{\Simp^{\op}} \cV^{\otimes}  \arrow[d, "pr_{1}"] \\
              & \cV^{\otimes}
	\end{tikzcd}
\end{equation*}
of symmetric monoidal cocartesian fibrations.
Under straightening, this exhibits the slice projection functor as a lax symmetric monoidal functor
\begin{align*}
	\cV^{\otimes} &\to (\Cat_{\infty}^{\times})^{[1]} \\
	C &\mapsto \cV_{/C} \to \cV.
\end{align*}
This induces a monoidal functor on algebras
\begin{align*}
	\Alg_{\Cat}(\cV)^{\otimes} \to \Alg_{\Cat}(\Cat_{\infty}^{[1]}  )^{\otimes} \,.
\end{align*}
When we apply this to the $\sfE(\bZ\times BO)$-module $\mu$ in $\Alg_{\Cat}(\cS_{/ (\bZ\times BO)})$ and unstraighten, we get a $(\cS_{/ (\bZ\times BO)})_{/\sfE(\bZ \times BO)} \to \cS_{/ (\bZ\times BO)}^{\otimes}$-module structure on $(\cS_{/ (\bZ\times BO)})_{/\mu} \to \cS_{/ (\bZ\times BO)}^{\otimes}$.
Inclusion of subcategories over contractible spaces gives a map
\begin{equation*}
	\begin{tikzcd}
		\sfE(\bZ\times BO) \arrow[r] \arrow[d, "\Delta"]      &   (\cS_{/ \bZ\times BO})_{/\sfE(\bZ\times BO)}  \arrow[d] \\
		\sfB(\bZ\times BO)  \arrow[r]               &   \cS_{/ \bZ\times BO}^{\otimes}
	\end{tikzcd}
\end{equation*}
of algebras in $\Dbl_{\infty}^{[1]}$.
We restrict along this map of algebras, and obtain a module structure given by maps
\begin{equation*}
	\begin{tikzcd}
		\sfE(\bZ\times BO)\times_{\Simp^{\op}} (\cS_{/ \bZ\times BO})_{/\mu} \arrow[r] \arrow[d]      &   (\cS_{/\bZ\times BO})_{/\mu}  \arrow[d] \\
		\sfB(\bZ\times BO) \times_{\Simp^{\op}} \cS_{/ \bZ\times BO}^{\otimes}  \arrow[r]               &  \cS_{/ \bZ\times BO}^{\otimes} \,.
	\end{tikzcd}
\end{equation*}
The map $T^{-}$ lifts to a monoidal functor into $\Mfd_{\diamond}$ taking values in the one-point manifolds.
The composite
\begin{align}\label{eq:T-}
	\sfB \Vect \to \Mfd_{\diamond}^{\otimes} \xrightarrow{I} \bVect(\CW)^{\boxplus} \to \cS_{/ (\bZ\times BO)}^{\otimes}
\end{align}
factors through the full monoidal subcategory $\sfB(\bZ\times BO)$, and by construction of the map $\bVect(\CW) \to \cS_{/ (\bZ\times BO)}$, it agrees with the map $\Vect \to \bZ\times BO$ which takes $V$ to $-V$.
We can model $\sfE \Vect$ as the Segal space given by the bar construction of the action of $\Vect$ on itself.
The map of $\bE_{\infty}$-monoids $T^{-}\colon \Vect\to \bZ\times BO$ then gives a commutative square
\begin{equation}\label{eq:EVect->EZxBO}
	\begin{tikzcd}
		\sfE \Vect \arrow[r, "\sfE T^{-}"] \arrow[d, "\Delta"]      &   \sfE \bZ\times BO  \arrow[d, "\Delta"] \\
		\sfB \Vect  \arrow[r,"\sfB T^{-}"]               &   \sfB(\bZ\times BO)
	\end{tikzcd}
\end{equation}
of algebras in $\Seg \subset \Dbl_{\infty}$. Because the composite \eqref{eq:T-} agrees with $T^{-}$, we can restrict the action on $(\cS_{/ (\bZ\times BO)})_{/\mu}$ along \eqref{eq:EVect->EZxBO} and obtain a $\sfE\Vect \to \sfB(\Vect)$-action on $\Mfd^{\mu}_{\diamond}\to \Mfd_{\diamond}$ whose action map $\Sigma^{-}$ is the induced map to the pullback in the cube
\begin{equation*}
	\begin{tikzcd}[row sep=1em, column sep=1em]
		\sfE \Vect \times_{\Simp^{\op}} \Mfd^{\mu}_{\diamond} \arrow[rr, "\Sigma^{-}"] \arrow[dd] \arrow[dr] & & \Mfd_{\diamond}^{\mu} \arrow[dd] \arrow[dr] \\
		& \sfE \Vect \times_{\Simp^{\op}} (\cS_{/ \bZ\times BO})_{/\mu} \arrow[rr, crossing over]  & & (\cS_{/ \bZ\times BO})_{/\mu} \arrow[dd] \\
		\sfB\Vect \times_{\Simp^{\op}} \Mfd_{\diamond}^{\otimes} \arrow[rr] \arrow[dr] & & \Mfd_{\diamond}^{\otimes} \arrow[dr] \\
		& \sfB \Vect \times_{\Simp^{\op}} \cS_{/ \bZ\times BO}^{\otimes}  \arrow[rr] & & \cS_{/\bZ \times BO}^{\otimes} \,.
		\arrow[from=2-2, to= 4-2, crossing over]
	\end{tikzcd}
\end{equation*}
\begin{rem}\label{rem:unwinding_EZxBO_action}
	We can explicitly describe the action of $\sfE(\Vect)$ on $\Mfd^{\mu}_{\diamond}$ as follows.
	Write $\Sigma\colon  \mu \otimes \sfE(\bZ \times BO) \to \mu$ for the action map in $\Alg_{\Cat}(\cS_{/\bZ\times BO})$.
	In particular, if $V\in \ob(\sfE(\bZ\times BO)) = \bZ\times BO$, and $c\in \ob(\mu)$, we write $\Sigma^{V}c \coloneq \Sigma(c,V)$.

	The mapping object between $(c,V)$ and $(d,W)$ in $\mu \otimes \sfE(\bZ\times BO)$ is $\mu(c,d)\oplus W\ominus V$.
  Because $\sfE(\bZ\times BO)$ is an idempotent algebra in $\Cat_{\infty}^{\cS_{/\bZ\times BO}}$, the action map is fully faithful and essentially surjective, so we have equivalences
	\begin{align*}
		\mu(c,d) \oplus W \ominus V \xrightarrow{\sim} \mu(\Sigma^{V}c, \Sigma^{W}d)
	\end{align*}
	in $\cS_{/\bZ\times BO}$.
  Now $\Sigma^{-}$ uses the map $\sfE T^{-}$, so on objects it acts by $V\cdot c =\Sigma^{-V}c$.

	A morphism in $\sfE(V,W)$ is a vector space $T$ and an equivalence $V\oplus T \simeq W$.
  The result of acting by on an object $(\bX, U, \phi)$ of $\Mfd^{\mu}(c,d)$ is then $(\bX, U\ominus T, \phi')$ where $\phi'$ is the composite
	\begin{align*}
		\phi' \colon U\ominus T  \ominus TX \xrightarrow{\phi } \mu(c,d)  \oplus V\ominus W \xrightarrow{\simeq} \mu(\Sigma^{-V}c, \Sigma^{-W}d) \,.
	\end{align*}
\end{rem}

\begin{lem}\label{lem:EVect_act_inv}
	The action of $\sfE \Vect$ on $\Mfd_{\diamond}^{\mu}$ induces a pullback square
	\begin{equation}\label{eq:EVect_acts_invertably}
		\begin{tikzcd}
			\sfE \Vect \times_{\Simp^{\op}} \Mfd_{\diamond}^{\mu} \arrow[r, "\Sigma^{-}"] \arrow[d]      &   \Mfd^{\mu}_{\diamond}  \arrow[d] \\
			\sfE\Vect \times_{\Simp^{\op}}\Mod_{\diamond}^{\otimes}  \arrow[r, "(-)\ominus \Delta"]               & \Mod^{\otimes}_{\diamond} \,.
		\end{tikzcd}
	\end{equation}
\end{lem}
\begin{proof}
	When we pass to spaces of objects, the square in question becomes
	\begin{equation*}
		\begin{tikzcd}
			\Vect \times \ob(\mu) \arrow[r] \arrow[d]      &   \ob(\mu)  \arrow[d] \\
			\Vect  \arrow[r]               &   \ast
		\end{tikzcd}
	\end{equation*}
	which is a pullback.
  When we pass to morphism $\infty$-categories at a pair $c,d\in \ob(\mu)$ and $V,W\in \Vect$, we get
	\begin{equation*}
		\begin{tikzcd}
			\sfE\Vect(V,W) \times \Mfd_{\diamond}^{\mu}(c,d) \arrow[r] \arrow[d]      &   \Mfd_{\diamond}^{\mu}(\Sigma^{-V}c, \Sigma^{-W}d)   \arrow[d] \\
			\sfE\Vect(V,W)\times \Mfd_{\diamond}  \arrow[r, "(-)\ominus \Delta"]  \arrow[d]             &  \Mfd_{\diamond} \arrow[d] \\
			\sfE\Vect(V,W) \times \Mod_{\diamond} \arrow[r, "(-)\ominus \Delta"] & \Mod_{\diamond} \,.
		\end{tikzcd}
	\end{equation*}
	All the vertical maps are right fibrations.
  The bottom square is a pullback because the fiber over $(\cP,U)$ depends only on the stratifying category $\cP$, not on the pair $U$.
  It therefore remains to show that the top square is a pullback.
  For any object $V\oplus T \simeq W$ of $\sfE\Vect(V,W)$ and any $\bX\in \Mfd_{\diamond}$, the fiber of the left hand map is the space of maps $-I\bX \to \mu(c,d)$ in $\cS_{/ \bZ\times BO}$.
  Computing the action, we get the object $\bX\ominus T$, whose index bundle is
	\begin{align*}
		I(\bX \ominus T) = (T\bX \oplus U^{-} \oplus T, U^{+})
	\end{align*}
	as an object of $\bVect$. In $\cS_{/ \bZ\times BO}$ we therefore have $-I(\bX \ominus T) \simeq -I\bX \ominus T$. We have
	\begin{align*}
		\mu(\Sigma^{-V}c,\Sigma^{-W}d) \simeq \mu(c,d) \oplus V \ominus W.
	\end{align*}
	So a point in the fiber over $\bX \ominus T$ is a map
	\begin{align*}
		-I\bX \ominus T \to \mu(c,d) \oplus V \ominus W.
	\end{align*}
	The comparison map on fibers is given by canceling $V\oplus T \simeq W$, which is indeed an equivalence.
\end{proof}
\begin{rem}
	There is also an action $\Sigma^{+}$ constructed using $T^{+}\colon \Vect \to \Mfd_{\diamond}$ taking $V$ to $(\ast, (V,0))$.
  We then replace the lower map of \eqref{eq:EVect_acts_invertably} by $(-)\oplus \Delta$.
\end{rem}
\begin{rem}
	Because all the module structures involved are over $\bE_{\infty}$-algebras, the two actions $\Sigma^{+}$ and $\Sigma^{-}$ commute.
  Moreover, we note that the fibers of $\Mfd_{\diamond}^{\mu}$ over $(\cP,U)$ and $(\cP,U\oplus (V,V))$ are always equivalent by this procedure.
  This means, roughly speaking, that the fiber only depends on $U$ only as a virtual vector space.
  We write
	\begin{align*}
		\Stab^{V} = \Sigma^{-V}\Sigma^{+V}\simeq \Sigma^{+V}\Sigma^{-V} \,.
	\end{align*}
	Note that $\Stab^{V}$ is homotopic to the identity on $(\cS_{/ (\bZ\times BO)})_{/\mu}$.
  This functor is necessary because of the specific strict model for virtual vector spaces baked into $\Mfd_{\diamond}$, but morally does not change the class represented by an object $\bX$.
\end{rem}

\begin{cor}\label{cor:U_compensation}
	Given a functor
	\begin{align*}
		(\cA, \cV) \colon \calD \to \Alg(\Mod \times_{\Simp^{\op}} \sfE \Vect) \,,
	\end{align*}
	the projection and action maps $pr, \Sigma^{\pm} \colon \Mfd_{\diamond}^{\mu} \times_{\Simp^{\op}} \sfE(\bZ\times BO) \to \Mfd_{\diamond}^{\mu}$ determine equivalences
	\begin{align*}
		\Sigma^{+\cV} \colon \cA^{*}\Alg(\Mfd_{\diamond}^{\mu}) &\to  (\cA\oplus \Delta\circ \cV)^{*}\Alg(\Mfd_{\diamond}^{\mu}) \\
		\Sigma^{-\cV} \colon \cA^{*}\Alg(\Mfd_{\diamond}^{\mu}) & \to  (\cA\ominus \Delta\circ \cV)^{*}\Alg(\Mfd_{\diamond}^{\mu})
	\end{align*}
	of right fibrations over $\calD$.
\end{cor}
\begin{proof}
	Apply the limit-preserving functor $\Alg$ to \eqref{eq:EVect_acts_invertably}, and to the following pullback square:
	\begin{equation*}
		\begin{tikzcd}
			\Mfd^{\mu}_{\diamond} \times_{\Simp^{\op}} \sfE\Vect \arrow[r, "pr"] \arrow[d]      &   \Mfd_{\diamond}^{\mu}  \arrow[d] \\
			\Mod_{\diamond} \times_{\Simp^{\op}} \sfE \Vect \arrow[r, "pr"]               & \Mod_{\diamond}   \,.
		\end{tikzcd}
	\end{equation*}
	This gives a zigzag
	\begin{align*}
		\cA^{*}\Alg(\Mfd_{\diamond}^{\mu}) \xleftarrow{pr_{*}} (\cA, \cV)^{*} \Alg(\Mfd^{\mu}_{\diamond} \times_{\Simp^{\op}} \sfE(\bZ\times BO)) \xrightarrow{\Sigma_{*}^{+}} (\cA \oplus \Delta\circ \cV)^{*}\Alg(\Mfd^{\mu}_{\diamond}) \,,
	\end{align*}
	where both maps are equivalences.
  The case for $\Sigma^{-}$ is similar.
\end{proof}

\begin{rem}
	For a poset $P$, a filter $F\subset P$, and a vector space $V$, there exists an algebra
	\begin{align*}
		P \to \sfE\Vect
	\end{align*}
	which on objects takes the value 0 whenever $p\notin F$ and $V$ whenever $p\in F$.
  Any morphism $p\leq q$ from $p\notin F$ to $q\in F$ is mapped to the morphism $V\colon 0\to V$.
  We write~$V_{F}$ for this algebra, and $\Sigma_{F}^{\pm V}\bX$ for the actions of $V_{F}$ on some algebra $\bX\colon P \to \Mfd_{\diamond}^{\mu}$.
  Again the actions commute, and we write
	\begin{align*}
		\Stab^{V}_{F}\bX = \Sigma^{+V}_{F}\Sigma^{-V}_{F}\bX  \simeq \Sigma^{-V}_{F}\Sigma^{+V}_{F} \bX \,.
	\end{align*}
	At least when arguing in sufficiently small contexts, we can argue as if $\Stab^{V}$ is the identity.
  Then $\Sigma^{+V}$ and $\Sigma^{-V}$ become inverse equivalences, and we can make sense of~$\Sigma^{V}$ for a virtual vector space $V$.
  Using various combinations of $\Sigma_{F}^{\pm V}$, we can construct $\Sigma^{V}_{S}\bX$ for any subset $S$ of $P$.
\end{rem}

Consider the equivalence of algebras
\begin{align*}
	\sfU_{s_{*}P} \oplus (\Delta_{*}\bR_{P_{1}}, \Delta_{*}\bR_{P_{1}} ) \simeq \sfU_{P} \oplus (0, \Delta_{*} \bR_{P_{1}})
\end{align*}
given by the obvious identification in the negative part, and by identifying the summand $\bR^{ \left\{ 1\right\}} $ of $\sfU^{+}_{P}(p,q)$ with $\Delta_{*}\bR_{P_{1}}$ whenever $p\in P_0, q\in P_1$.
By \Cref{cor:U_compensation}, a flow 1-simplex with objects $P\to [1]$ therefore corresponds to a unique flow category $c(\bB)$ such that up to relabeling along $\sfA_{P}\to \sfA_{s_{*}P}$, we have
\begin{align*}
	\Sigma_{P}^{-\bR}\Stab_{P_{1}}^{\bR} c(\bB) \simeq \Sigma^{-\bR}_{P_{1}} \bB \,.
\end{align*}
Up to stabilization, we can think of $c(\bB)$ as obtained from relabeling $\Sigma^{\bR}_{P_{0}} \bB$.

\begin{lem}\label{lem:cofibers_in_flow}
	$c(\bB)$ is a cofiber of $\bB$ in $\Flow^{\mu}$.
\end{lem}

\begin{proof}
	Let $\bB\colon \bX\to \bY$ be a flow bimodule with poset of objects $P \to [1]$. Write $S\bB\colon \Span \to \Flow^{\mu}$ for the span
	\begin{align*}
		\bZero \xleftarrow{\bZero_{\bX}} \bX \xrightarrow{\bB} \bY \,.
	\end{align*}
	A cofiber of $\bB$ is a colimit of $S\bB$, which is equivalently an initial object of the slice $\Flow^{\mu}_{S\bB/}$.
	By decomposing $S\bB$ as a colimit, we have
	\begin{align}\label{eq:cofibers_Flow_pullback}
		\Flow^{\mu}_{S\bB/}\simeq \Flow^{\mu}_{\bB /}\times_{\Flow^{\mu}_{\bX /} } \Flow^{\mu}_{\bZero_{\bX}/} \,.
	\end{align}
	Because $\bZero$ is an initial object of $\Flow^{\mu}$, the forgetful functors determine equivalences
	\begin{align*}
		\Flow^{\mu}_{\bZero_{\bX}/} \to \Flow^{\mu}_{\bZero/} \to \Flow^{\mu} \,.
	\end{align*}
	Under these equivalences, the induced section $z\colon \Flow^{\mu} \to \Flow^{\mu}_{\bX/}$ is the one that at each object $\bZ$ selects the zero morphism $\bX\to \bZ$.
	We will now represent $z$ in terms of a monomorphism of semi-simplicial spaces.
	Consider the subspace of $(\PreFlow^{\mu}_{\bX/})_{n}$ spanned by the flow ($n+1$)-simplices $\bH$ with objects $R$ such that $\bH(p,q)$ is the empty manifold whenever $p\in R_{0}=P_{0}, q\in d_{0}^{*}P$.
	This is indeed a subspace at each simplicial level because the empty manifold has contractible space of diffeomorphisms.
	These subspaces are preserved under face maps in $\Delta^{n}_{s}$, so taken for all $n$, they determine an monomorphism
	\begin{align}\label{eq:cofibre_slice_0_inclusion}
		X \subset \PreFlow^{\mu}_{\bX/}.
	\end{align}
	It is now easy to see that the composite
	\begin{align*}
		X \to \PreFlow^{\mu}_{\bX/} \to \PreFlow^{\mu}
	\end{align*}
	is a levelwise equivalence, and that the functor $z$ may be represented by inverting this equivalence and applying the inclusion \eqref{eq:cofibre_slice_0_inclusion}.
	Now consider the pullback
	\begin{align*}
		B = \PreFlow^{\mu}_{\bB/} \times_{\PreFlow^{\mu}_{\bX /}} X \,.
	\end{align*}
	The forgetful map is an inner fibration, so $\cT$ preserves this pullback by Proposition~\ref{prop:semi-simplicial_results}.
	Also by Proposition~\ref{prop:semi-simplicial_results}, the semi-simplicial slices are mapped by $\cT$ to the categorical slices, so $B $ is a representative for $\Flow^{\mu}_{S\bB /}$.

	For a map of posets $R\to [n+1]$, write $s_{*}R$ for the composite $R \to [n+1] \xrightarrow{s_{0}} [n]$. Consider the functor
	\begin{align*}
		Z_{R}(p,q)\colon \sfA_{s_*R}(p,q) \to \sfA_{R}(p,q)
	\end{align*}
	defined on arcs by leaving edge labels the same, but pushing forward vertex labels according to the map
	\begin{align*}
		d_0 \colon [n] \to [n+1]
	\end{align*}
	and adding a $1$ to any vertex label where it is possible.
	This functor is an equivalence whenever $p\notin R_{0}$ or $q\in R_{0}$.
	Otherwise, it is an open embedding whose complement is precisely the image of
	\begin{align*}
		(d_{1})_{!}(p,q) \colon \sfA_{d_{1}^{*}R}(p,q) \to \sfA_{R}(p,q) \,.
	\end{align*}
	The functors $Z_{R}$ respect concatenation of arcs, and therefore give a morphism of algebras $Z_{R}\colon \sfA_{s_{*}R} \to \sfA_{R}$.
  For an injective morphism $\psi \colon [m+1] \to [n+1]$ which leaves~$0$ an~ $1$ fixed, there is a unique injective morphism $\phi\colon [m] \to [n]$ leaving $0$ fixed and making the commutative diagram
	\begin{equation*}
		\begin{tikzcd}
			{[m+1]} \arrow[r, "\psi"] \arrow[d, "s_{0}"]      &   {[n+1]}  \arrow[d, "s_{0}"] \\
			{[m]}  \arrow[r, "\phi"]               &   {[n]}
		\end{tikzcd}
	\end{equation*}
  into a pullback square.
	For any $R\to [n+1]$, we therefore have by pasting $s_{*}(\psi^{*}R) \simeq \phi^{*}(s_{*}R)$.
  This shows that the construction $R\mapsto s_{*}R$ determines a functor $\cP_{\Delta^{1}_{s}/} \to \cP_{\Delta^{0}_{s}/}$.
	Moreover, we have a commutative square
	\begin{equation*}
		\begin{tikzcd}
			\sfA_{\phi^{*}s_{*}R} \simeq \sfA_{s_{*}\psi^{*}R}\arrow[r, "Z_{\psi^{*}R}"] \arrow[d, "\phi_{!}"]      &   \sfA_{\psi^{*}R}  \arrow[d, "\psi_{!}"] \\
			\sfA_{s_{*}R}  \arrow[r, "Z_{R}"]               &   \sfA_{R}
		\end{tikzcd}
	\end{equation*}
	which is natural with respect to composition in $\psi$.
  For every $R\to [n+1]$ we have equivalences
	\begin{align*}
		\sfU_{s_{*}P} \oplus (\Delta_{*}\bR_{d_{0}R}, \Delta_{*}\bR_{d_{0}R} ) \simeq \sfU_{P} \oplus (0, \Delta_{*}\bR_{d_{0}R})
	\end{align*}
	given by the obvious identification of the negative parts.
  On the positive part, we use the equivalence
	\begin{align*}
		\Delta_{*}\bR_{R_{[n]}} \oplus \sfU^{+}_{s_{*}P} \simeq \sfU^{+}_{P}
	\end{align*}
	which is given by shifting labels up according to $d_{0}\colon [n]\to [n+1]$, and identifying $\Delta_{*}\bR_{d_{0}R}$ with $\bR^{ \left\{ 1\right\}}$ whenever $p\in R_{0}$, $q\in d_{0}R$.
  We write
	\begin{align*}
		\cV \colon \cP_{\Delta^{0}_{s}/} \to \Alg(\sfE\Vect )
	\end{align*}
	for the functor which takes $R\to [n+1]$ to $\bR_{d_{0}R}$ which is evidently natural in face maps which leave $0$ fixed.
  We have constructed a natural transformation
	\begin{equation*}
	\begin{tikzcd}[row sep=large]
		\cP_{\Delta^{1}_{s}/} \arrow[r, "{\cA_{\Delta^{1}_{s}/} \oplus (0, \Delta_{*}\cV )}"] \arrow[d, "s_{*}"']      &   \Alg(\Mod) \\
		\cP_{\Delta^{0}_{s}/}  \arrow[ur, "{\cA_{\Delta^{0}_{s}/} \oplus (\Delta_{*}\cV, \Delta_{*}\cV)}"' , ""{name=S} ]
		\arrow[from=S, to=1-1, Rightarrow, "Z"]
	\end{tikzcd}
    \end{equation*}
	whose source and target factor through the subcategory $\Alg(\Mod_{\diamond})$, and with all components open embeddings.
  The transformation $Z$ becomes a natural equivalence in $\Opd$, so by \Cref{lem:alg_in_exponential} it corresponds to a functor into $\Alg(\Mod^{[1]})$, which by construction factors through the subcategory $\Alg(\Mod_{\diamond}^{[1]_{\circ}})$.
  Let $\cW$ denote the functor
	\begin{align*}
		\cW \colon \cP_{\Delta^{1}_{s}/} &\to \Alg(\sfE\Vect) \\
		& R \mapsto \bR_{R} \,.
	\end{align*}

	Combining \Cref{cor:relabeling,cor:U_compensation} we therefore have maps
	\begin{align*}
		(\cA_{\Delta^{0}_{s}/} )^{*}\Alg(\Mfd^{\mu}_{\diamond}) &\xrightarrow[\simeq]{\Stab_{\cV}} (\cA_{\Delta^{0}_{s}/} \oplus (\Delta_{*}\cV,\Delta_{*}\cV  ))^{*}\Alg(\Mfd^{\mu}_{\diamond}) \\
		&\xhookrightarrow{Z_{*}} (\cA_{\Delta^{1}_{s}/}\oplus (0,\Delta_{*}\cV)  )^{*}\Alg(\Mfd_{\diamond}^{\mu}) \\
		&\xrightarrow[\simeq]{\Sigma^{-\cW} (\Sigma^{-\cV})^{-1}} \cA_{\Delta^{1}_{s}/}^{*}\Alg(\Mfd_{\diamond}^{\mu})
	\end{align*}
  of right fibrations over $\cP_{\Delta^{1}_{s}/}$.
	The composite $F$ is fully faithful with image over $R\to [n+1]$ precisely those flow $(n+1)$-simplices $\bH$ such that $d_{1}\bH(p,q) = \emptyset$ for $p\in R_{0}, q\in, R_{k} k>1$.

	When we project away from the slices, the components of $Z_{R}$ become the identity, and the algebras $\cV$ and $\cW$ agree.
  The map $F$ therefore commutes with projection to $\cA^{*}\Alg(\Mfd_{\diamond}^{\mu})$.

	In the fiber over $[-1] \in \Simp_{s,a}$, $F$ gives an equivalence inverse to the cone construction, so in particular $F(c(\bB)) \simeq \bB$.
  This means that we can straighten $F$ over $\Simp_{s,a}$ and pull back to the slice under $c(\bB)$ to get an equivalence $\PreFlow^{\mu}_{c(\bB)/} \simeq B$ over $\PreFlow$.
  When we apply the localization $\cT$, we therefore get the equivalence
	\begin{align*}
		\Flow^{\mu}_{c(\bB)/} \simeq \Flow^{\mu}_{S\bB/} \,,
	\end{align*}
	over $\PreFlow$.
  The diagonal bimodule on $c(\bB)$ gives an extension of $S\bB$ to a pushout square exhibiting $c(\bB)$ as the cofiber of $\bB$.
\end{proof}

\begin{cor}\label{cor:flow_is_stable}
	$\Flow^{\mu}$ is a stable $\infty$-category.
\end{cor}

\begin{proof}
	Let $\bX$ be a flow category with poset of objects $P$.
  By \Cref{lem:cofibers_in_flow}, the cofiber of the zero morphism $\bZero_{\bX}\colon \bX \to \bZero$ is equivalent to $\Sigma_{P}^{\bR}\bX$.
  This is also the categorical suspension of $\bX$ in $\Flow^{\mu}$, so we denote it $\Sigma\bX$.
	By the proof of \Cref{lem:cofibers_in_flow}, the identity-bimodule on $\Sigma\bX$, which is an initial object of $\Flow^{\mu}_{\Sigma\bX/}$, corresponds to the square
	\begin{equation} \label{eq:diagram_pullback_pushout}
		\begin{tikzcd}[row sep=tiny, column sep= tiny]
			\bX \arrow[rrr]\arrow[ddd] \arrow[dddrrr] &&& \bZero \arrow[ddd] \\
			&&F(s\Sigma\bX) & \\
			&\bZero&& \\
			\bZero \arrow[rrr] &&& \Sigma \bX \, ,
		\end{tikzcd}
	\end{equation}
	which is therefore a pushout in $\Flow^{\mu}$.
  Up to stabilization, $F(s\Sigma \bX)$ is obtained from relabeling $\Sigma^{\bR}_{P\times \left\{ 1\right\}} s\bX$.
	By \Cref{lem:D_is_htpy_eq}, the construction of conic degeneracies is compatible with the $\sfE(\bZ\times BO)$-action. We therefore get
	\begin{align*}
		\Sigma_{P\times \left\{ 0\right\}}^{-\bR} s (\Sigma_{P}^{\bR} \bX) \simeq \Sigma^{-\bR}_{P\times \left\{ 0\right\}} \Sigma_{P\times [1]}^{\bR} s\bX \simeq \Sigma_{P\times \left\{ 1\right\}}^{\bR} s\bX \,.
	\end{align*}
	Under the order reversal symmetry on $\Simp_{s}$, the proof of \Cref{lem:cofibers_in_flow} gives rise to a dual statement about fibers.
  If $\bB$ is a bimodule with objects $R\to [1]$, it shows that the flow category $f(\bB)$ obtained by relabeling $\Sigma^{-\bR}_{R_{1}}$ along the equivalence $\sfA_{P} \simeq \sfA_{s_*P}$ is a fiber of $\bB$.
  In particular, the fiber of the morphism $\bZero_{\Sigma \bX} \colon \bZero \to \Sigma_{P}^{\bR} \bX$ is equivalent to $\Sigma^{-\bR}_{P}\Sigma^{\bR}_{P}\bX \simeq \bX$.
  Moreover, relabeling and degree shifting determines a monomorphism
	\begin{align*}
		(\PreFlow^{\mu}_{/\bX}) \to \PreFlow^{\mu}_{/\bZero_{\Sigma \bX}}
	\end{align*}
	which takes the terminal object $s(\Sigma_{P}^{\bR})\bX$ to the relabeling of $\Sigma^{-\bR}_{P\times \left\{ 1\right\}} s(\Sigma_{P}^{\bR}\bX) \simeq \Sigma_{P\times \left\{ 1\right\}}^{\bR} s\bX$.
  Hence, the square \eqref{eq:diagram_pullback_pushout} is also a pullback.
  This implies that for every $\bX$, the adjunction unit $\bX \to \Omega \Sigma \bX$ is an equivalence, which implies that the suspension functor on $\Flow^{\mu}$ is fully faithful.
  From the above, we also have
	\begin{align*}
		\Sigma \Omega \bX \simeq \Sigma^{\bR}_{P} \Sigma^{-\bR}_{P} \bX \simeq \bX \,,
	\end{align*}
	which shows that $\Sigma$ is essentially surjective.
  Combining this with \Cref{lem:flow_is_pointed,cor:coproducts_in_flow}, $\Flow^{\mu}$ satisfies the criteria of \cite*[Corollary 1.4.2.27]{HA}, and is therefore stable.
\end{proof}

\subsection{Generation} \label{sec:presentability}

In this subsection, we show that $\Flow^{\mu}$ is compactly generated and presentable.
It follows from \Cref{lem:coproducts_in_flow,lem:cofibers_in_flow} that $\Flow^{\mu}$ is cocomplete.
It therefore remains to show that $\Flow^{\mu}$ has a small set of compact generators.
This set will be the set of one-object flow categories $\bOne_{b}$ as in \Cref{ex:one_object_flow_cat}.

\begin{lem}\label{lem:cofiber_of_a_filter}
	Let $\bX$ be a flow category with poset of objects $P$, and let $i \colon I\to P$ be the inclusion of a filter.
  Then the complement $j\colon I^{c} \to \cP$ of $I$ is a cofilter in $P$.
  The Cartesian lifts $i_{!}\colon i^{*}\bX \to \bX$ and $j_{!}\colon j^{*}\bX \to \bX$ of $i$ and $j$ at $\bX$ determine morphisms in $\bfPreFlow^{\mu\uparrow}_{0}$ and $\bfPreFlow^{\mu\downarrow}_{0}$, respectively.
  The companion and conjoint of these, respectively, fit in a cofiber sequence
	\begin{align*}
		i^{*}\bX \xrightarrow{\cU(i_{!})} \bX \xrightarrow{\cU'(j_{!})} j^{*}\bX \,.
	\end{align*}
\end{lem}
\begin{proof}
	Define a map of posets $P\to [1]$ by mapping any $p\in I$ to $1$, and any $p\in I^{c}$ to $0$. This is order-preserving because $I$ is a filter and $I^{c}$ is a cofilter.
  As in the construction of cofibers, we have an equivalence $\sfA_{P} \simeq \sfA_{s_{*}P}$, and we define a bimodule $\bB\colon \Omega \bX_{I^{c}} \to \bX_{I}$ by relabeling $\Sigma^{-\bR}_{I^{c}}\bX = \Sigma^{-\bR}_{P}(\Sigma_{I}^{-\bR})^{-1}\bX $ along this equivalence.
  In other words, we apply the relabeling functor $F$ of \Cref{lem:cofibers_in_flow} to the object $\bX$ with objects $s_{*}P$.
  By construction this is now a bimodule whose cone $c(\bB)$ is precisely $\bX$.
  The structure map $i^{*}\bX \to c(\bB)$ is the image of the diagonal bimodule on $c(\bB) \simeq \bX$ under the relabeling map $F$.
  Unwinding this construction, we see that $s\bX$ is mapped to the restriction of $s\bX$ to the subposet of $P\times [1]$ spanned by $I\times \left\{ 0\right\}$ and $P\times \left\{ 1\right\}$.
  This restriction is precisely the Cartesian lift of the morphism $(i_{!},id)$ at $s\bX$, and is therefore the companion of $i_{!}$.
  A similar argument using the dual of \Cref{lem:cofibers_in_flow} shows that the structure map $\bX \to j^{*}\bX$ is precisely the restriction of $s\bX$ to the subposet of $P\times [1]$ spanned by $P\times \left\{ 0\right\}$ and $I^{c} \times \left\{ 1\right\}$, which is precisely the Cartesian lift of $(id, j_{!})$ at $s\bX$, i.e. the conjoint of $j_{!}$.
\end{proof}

We say that a flow category $\bX$ is finite if its poset of objects $P$ is finite, and write $\Flow^{\mu}_{\Fin}$ for the full subcategory of $\Flow^{\mu}$ spanned by finite flow categories.
This subcategory contains the zero object and is closed under finite coproducts and cofibers, and is therefore a stable subcategory.

\begin{cor}\label{cor:flow_fin_generation}
	$\Flow^{\mu}_{\Fin}$ is generated by the objects $\bOne_{b}$ under finite colimits.
\end{cor}

\begin{proof}
	Take a finite flow category $\bX$ and embed its poset of objects $P$ into a finite total order $T\simeq [n]$.
  Then have an increasing sequence of filter inclusions
	\begin{align*}
		\emptyset \to P_{\leq 0 } \to P_{\leq 1} \to \dots P_{\leq n} = P
	\end{align*}
	where the complement of $P_{\leq k} \to P_{\leq k+1}$ consists of a single point.
  We take a Cartesian lift of this diagram at $\bX$ and apply the companion functor to obtain a finite filtration
	\begin{align*}
		\bZero  \to \bX_{[0,0]} \to \dots \bX_{[0,1]} \to \dots \to \bX_{[0,n]} \simeq \bX
	\end{align*}
  of $\bX$.
	By \Cref{lem:cofiber_of_a_filter}, each part of the associated graded is a flow category with a single object.
\end{proof}

\begin{lem}\label{lem:restriction_is_fibration}
	For a morphism $i^{*}\bX \to \bX$ in $\bfPreFlow^{\mu\uparrow}_{0}$ covering a filter inclusion $I\to P$, the Cartesian transport map induces a left fibration of semi-simplicial spaces
	\begin{align*}
		\PreFlow^{\mu}_{\bX/} \to \PreFlow^{\mu}_{i^{*}\bX/} \,.
	\end{align*}
\end{lem}

\begin{proof}
	Let $j\colon I^{c} \to P$ denote the complement of $I$, and let $\bB \colon \Omega j^{*}\bX \to i^{*}\bX$ denote the attaching bimodule as in the proof of \Cref{lem:cofiber_of_a_filter}.
  We write $\del \bB\colon \del \Delta^{1}_{s} \to \PreFlow^{\mu}$ for the restriction of $\bB$ to the boundary of $\Delta^{1}$.
  Consider the solid part of the diagram
	\begin{equation*}
		\begin{tikzcd}
			B \arrow[r, "f", dashed] \arrow[d] \arrow[rr, bend left]      &   \PreFlow^{\mu}_{i^{*}\bX/}  \arrow[d, "z'", dashed]  \arrow[r] & \PreFlow^{\mu} \arrow[d, "z"] \\
			\PreFlow^{\mu}_{\bB/}  \arrow[r, "d"]               &   \PreFlow^{\mu}_{\partial \bB /}  \arrow[r] \arrow[d]           & \PreFlow^{\mu}_{\Omega j^{*}\bX/}\arrow[d] \\
			&
			\PreFlow^{\mu}_{i^{*}\bX/} \arrow[r] & \PreFlow^{\mu} \, ,
		\end{tikzcd}
	\end{equation*}
	where we use the notation $B$ as in the proof of \Cref{lem:cofibers_in_flow}, and the map $z$ is the composite
	\begin{align*}
		\PreFlow^{\mu} \simeq X \subset \PreFlow^{\mu}_{\Omega j^{*}\bX/} \,.
	\end{align*}
	The lower square is a pullback because the slice functor $\PreFlow^{\mu}_{(-)/}$ takes pushouts to pullbacks.
  The dashed map $z'$ can then be defined by the universal property.
  The outer right rectangle is a pullback because both vertical maps are equivalences, so the top right square is a pullback by pasting.
  The dashed map $f$ then exists by the universal property, and the outer top rectangle is a pullback by definition of $B$.
  The pasting law now implies that the upper left square is a pullback.
  The map $d$ is given by restricting along a monomorphism of semi-simplicial sets and is therefore a left fibration, so by pullback stability, so is $f$.
  By commutativity of the diagram, we see that $f$ must agree with the composite
	\begin{align*}
		B \subset \PreFlow^{\mu}_{\bB/} \xrightarrow{d_{0}^{*}} \PreFlow^{\mu}_{i^{*}\bX/} \,.
	\end{align*}
	Composing this further with the levelwise equivalence constructed in the proof of \Cref{lem:cofibers_in_flow}, we see that the composite
	\begin{align*}
		\PreFlow^{\mu}_{\bX/} \simeq B \xrightarrow{d_{0}^{*}} \PreFlow^{\mu}_{i^{*}\bX/}
	\end{align*}
	is precisely the map given by restricting a flow $(n+1)$-simplex $\bH$ with objects $R$ such that $v_0^{*}\bH =\bX$ along the convex map $i^{*}R \to R$, which is precisely the functor induced by Cartesian transport along $i^{*}\bX \to \bX$.
\end{proof}
For a poset $P$, the slice category $(\cP^{\uparrow}_{0})_{/P}$ is equivalently the poset of all filters $I\subset P$.
We let $\cJ_{P} \subset (\cP^{\uparrow}_{0})_{/P}$ denote the subposet spanned by the finite filters.
Because the union of finitely many finite filters is a finite filter, $\cJ_{P}$ is a filtered category.
This means that the natural comparison map
\begin{align*}
	\colim_{I\in \cJ_{P}} I \to P
\end{align*}
in $\Pos$ is the inclusion of a full subposet.
Since $P$ is locally finite, each object $p\in P$ belongs to the finite filter $P_{p\leq}$, so this map is actually an equivalence.
Note in particular that the colimit diagram for the above colimit factors through the inclusion $\cP^{\uparrow}_{0}\to \Pos$ by a functor $\cJ_{P}^{\triangleright} \to \cP^{\uparrow}_{0}$.

\begin{lem}\label{lem:filtered_colim_lifting_condition}
	Let $P$ be a locally finite poset. Any lift of $\cJ_{P}^{\triangleright} \to \cP^{\uparrow}_{0}$  to $\bfPreFlow^{\mu\uparrow}_{0}$ is taken by $\cU$ to a colimit diagram in $\Flow^{\mu}$.
\end{lem}
\begin{proof}
	Because $\bfPreFlow^{\mu\uparrow}_{0}\to \cP_{0}^{\uparrow}$ is a right fibration and $\cJ_{P}^{\triangleright}$ has a terminal object, such a lift is uniquely determined by a flow category $\bX$ with objects $P$.
  For a finite filter $I\subset P$, we write $\bX_{I} \to \bX$ for the Cartesian lift of $I\to P$ at $\bX$.
  We begin by showing that for any flow $n$-simplex $\bH$, the dashed Cartesian lifting problems
	\begin{equation*}
		\begin{tikzcd}[row sep=6ex]
			\cJ_{P}
			\arrow[r]
			\arrow[d, "\iota"]    &
			\bfPreFlow^{\mu\uparrow}_{n+1}
			\arrow[r]     &
			\cP^{\uparrow}_{n+1}
			\arrow[d]     \\
			\cJ_{P}^{\triangleright}
			\arrow[r, "{\left(F,\underline{\bH} \right)}"']
			\arrow[ur, dashed, "G"]
			\arrow[urr, dotted, "H",pos=0.4]   &
			\bfPreFlow_{0}^{\mu\uparrow}\times\bfPreFlow_{n}^{\mu\uparrow}
			\arrow[r]    &
			\cP^{\uparrow}_{n}\times \cP^{\uparrow}_{n}
			\arrow[from=1-2,to=2-2, crossing over]
		\end{tikzcd}
	\end{equation*}
	admit unique solutions.
  There exists a unique dotted lift $H$ by \Cref{lem:P_has_companions}.
  This lift is a filtered diagram of subposets of $H(\infty)$, and because each $p\in H(\infty)$ either belongs to $R$ or to a finite filter $I\subset P$, each $p$ belongs to some $H(i)$, and so $H$ becomes a colimit diagram in $\Pos_{/[n+1]}$.
  Now $\bfA\circ H\circ \iota$ is similarly a filtered diagram of full subcategories of $\bfA_{H(\infty)}$, so since every object belongs to some $\bfA_{ H(i)}$, $\bfA\circ H$ is a colimit diagram in $\Dbl_{\infty/\Mod_{\diamond}^{\otimes}}$.
  By \Cref{cor:canonical_presheaf_pullback}, the functor $\bfFlow^{\mu\uparrow}_{n+1}\to \cP^{\uparrow}_{n+1}$ is the pullback along $\bfA$ of a representable right fibration, so by \Cref{lem:limts_as_lifting} there exists a unique lift $G$ of $H$.
  By \Cref{lem:companion_colim}, we get a limit diagram $\cU_{1}\circ H$ witnessing
	\begin{align*}
		\lim_{I \in \cJ_{P}^{\op}} \PreFlow^{\mu}_{\bX_{I}/} \simeq \PreFlow^{\mu}_{\bX/} \,.
	\end{align*}
	Mapping a finite filter $I\subset P$ to its cardinality determines a functor $\cJ_{P} \to \bN$ such that each non-identity morphism is mapped to a non-identity morphism.
  This means that $\cJ_{P}^{\triangleright}$ is a direct category, so by \Cref{lem:companion_colim} it suffices to show that the diagram is Reedy fibrant.
  For each finite filter $I\subset P$, the matching morphism at $I$ is
	\begin{align*}
		\PreFlow^{\mu}_{\bX_{I}/} \to \lim_{J\subsetneq I} \PreFlow^{\mu}_{\bX_{J}/} \,.
	\end{align*}
	By an argument similar to the first part of this proof, this morphism is given by composing with the companion of the inclusion $\bX_{K}\to \bX_{I}$, where $K$ is the union of all proper filters $J\subsetneq I$.
  This is an inner fibration by \Cref{lem:restriction_is_fibration}.
\end{proof}

\begin{lem}\label{lem:1_is_compact}
	For any $b\in \ob(\mu)$, the one-object flow category $\bOne_{b}$ is a compact object of $\Flow^{\mu}$.
\end{lem}
\begin{proof}
	To show that $\bOne_{b}$ is compact, it suffices to show that the mapping spectrum functor
	\begin{align*}
		\map(\bOne_{b}, -) \colon \Flow^{\mu} \to \Sp
	\end{align*}
	preserves small coproducts, and therefore all small colimits.
  For any small set $S$ and any collection of flow categories $\bX_{s}, s\in S$, consider the comparison map
	\begin{align*}
		\coprod_{s\in S} \map(\bOne_{b}, \bX_{s}) \to \map(\bOne_{b}, \coprod_{s\in S} \bX_{s}) \,.
	\end{align*}
	It suffices to show that this is an equivalence on homotopy groups.
  Because $\Flow^{\mu}$ is stable, and because $\Sigma^{n} \bOne_{b} \simeq \bOne_{\Sigma^{n}b}$ for $n\in \bZ$, it suffices to show that the comparison map is an equivalence on $\pi_{0}$ for all $b$.
  By \Cref{prop:semi-simplicial_results}, the set of homotopy classes of maps $[\bX, \bY]$ between flow categories $\bX$ and $\bY$ can be described as $\pi_{0}$ of the realization of the right mapping space $\Hom^{\rR}(\bX, \bY)$ for any choice of idempotent equivalence at $\bY$.
  The right mapping space satisfies the Kan condition, so $\pi_{0}$ of its realization is the set of flow bimodules $\bB\colon \bX \to \bY$ up to equivalence, modulo the relation $\bB \sim \bB'$ if there exists a flow-2-simplex $\bH$ of the form
	\begin{equation*}
		\begin{tikzcd}
			\bX \arrow[r, "s\bX"] \arrow[dr, "\bB"']      &   \bX  \arrow[d, "\bB'"] \\
			&   \bY \,.
		\end{tikzcd}
	\end{equation*}
	Any bimodule $\bB\colon \bOne_{b} \to \coprod_{s} \bX_{s}$  has a locally finite poset of objects $P\to [1]$.
  This means that the manifolds $\bB(\ast, p)$ have to be empty for all but finitely many $p$, and so in particular $\bB$ is in the image of $\coprod_{s}[\bOne_{b}, \bX_{s}]$.
  Similarly, any 2-simplex $\bH$ of the form
	\begin{equation*}
		\begin{tikzcd}
			\bOne_{b} \arrow[r, "s\bOne_{b}"] \arrow[dr, "\bB"']      &   \bOne_{b}  \arrow[d, "\bB'"] \\
			&   \underset{s\in S}{\coprod} \bX_{s}
		\end{tikzcd}
	\end{equation*}
	has $\bH(\ast, p)= \emptyset$ for all but finitely many $p$, and so in particular also defines a relation between the corresponding objects in $\coprod_{s} [\bOne_{b}, \bX_{s}]$.
\end{proof}
\begin{definition}
	Let $\PrLSt$ denote the $\infty$-category of stable presentable $\infty$-categories and colimit-preserving functors. Let $\Pr^{\rL}_{\St,\omega}\subset \PrLSt$ denote the subcategory of $\infty$-categories which are also compactly generated, and functors which preserve compact objects.
\end{definition}
\begin{cor}\label{cor:flow_is_cg_pres}
	The $\infty$-category $\Flow^{\mu}$ is stable, presentable, and compactly generated.
	Any $f_!\colon \Flow^{\mu}\to \Flow^{\nu}$ preserves colimits and compact objects, so $\Flow^{(-)}$ determines a functor
	\begin{align*}
		\Cat_{\infty/ (U/O)} \to \Pr^{\rL}_{\St,\omega} \,.
	\end{align*}
\end{cor}
\begin{proof}
	By \Cref{lem:filtered_colim_lifting_condition}, each $\bX$ in $\Flow^{\mu}$ can be written as a filtered colimit of finite flow categories.
  By \Cref{cor:flow_fin_generation}, each finite flow category belongs to the category generated under finite colimits by the objects $\bOne_{b}$, so by \Cref{lem:1_is_compact}, these objects are compact generators for all of $\Flow^{\mu}$.
	A map $f$ in $\Cat_{\infty/(U/O)}$ determines a map of $\sfE(\bZ\times BO)$-modules in $\Cat_{\infty}^{\cS_{/\bZ\times BO}}$.
  We can therefore lift the double functor
	\begin{align*}
		f_{!} \colon \Mfd^{\mu}_{\diamond} \to \Mfd^{\nu}_{\diamond}
	\end{align*}
	to a map of $\sfE(\bZ\times BO)$ modules in $\Dbl_{\infty}$.
  In particular, we get equivalences $f_{!}( \Sigma^{\pm V}_{F} \bX) \simeq \Sigma^{V}_{F} f_{!}\bX$ for any algebra $\bX \colon P \to \Mfd^{\mu}_{\diamond}$, any filter $F\subset P$ and any virtual vector space~$V$.
  It is also clear that pushing forward the framing by $f_{!}$ is compatible with relabeling along open embeddings.
  The functor $f_{!}\colon \Flow^{\mu}\to \Flow^{\nu}$ therefore preserves coproducts and cofibers, and therefore all colimits.
  Furthermore, $f_{!}\bOne_{b} \simeq \bOne_{f_{!}b}$, and since the compact objects of $\Flow^{\mu}$ are generated under retracts and finite colimits by the~$\bOne_{b}$, $f_{!}$ preserves all compact objects.
\end{proof}

\begin{rem}\label{rem:strongly_cont}
	A functor $f_{!}\colon \cC \to \calD$ in $\Pr^{\rL}_{\St,\omega}$ has a right adjoint $f^{*}$.
  This right adjoint preserves limits, so since $\cC$ and $\calD$ are stable, $f^{*}$ is exact and therefore preserves finite colimits.
  Since $f_{!}$ also preserves compact objects, one can verify that $f^{*}$ preserves filtered colimits, and therefore all colimits.
  By presentability $f^{*}$ therefore admits a further right adjoint $f_{*}$.
  By this argument, we can identify $\Pr^{\rL}_{\St, \omega} \subset \PrLSt$ with the subcategory of compactly generated categories and internal left adjoints, i.e., functors whose right adjoint also belongs to $\PrLSt$.
\end{rem}

\section{Twisted Presheaves}\label{sec:twisted_presheaves}
In this section, we will introduce stable $\infty$-categories of \emph{twisted presheaves}.
These are a common generalization of the module $\infty$-categories $\Mod_{\Th(\xi)}$ and the presheaf $\infty$-categories $\Fun(\cC^{\op},\Sp)$.
We will see how these $\infty$-categories can be defined either in terms of an oplax (co)limit, or as \emph{enriched presheaves} on a small $\Sp$-category whose mapping spectra are Thom spectra.
The enriched presheaves perspective is crucial for the comparison with flow categories because it allows us to identify the mapping spectra of the compact generators.

\subsection{Twisted presheaves as an oplax colimit}\label{subsec:twisted_presheaves_as_an_oplax}
In this section, we will define $\TwShv$ as a certain oplax colimit in an $(\infty,2)$-category.
By \Cref{rem:2-cats_as_enriched}, we can think of an $(\infty,2)$-category either as a double $\infty$-category $\bfC$ such that $\bfC_{0}$ is a space and such that~$h\bfC$ is a complete Segal space, or as a complete $\Cat_{\infty}$-category.
We say that an $(\infty,2)$-category is \emph{locally} of some size if its morphism $\infty$-categories have that size.
\begin{ex}
	The large $\infty$-category $\Cat_{\infty}$ of small $\infty$-categories admits a closed monoidal structure, and can therefore be enriched in itself by \Cref{ex:enriched_from_closed}.
  This enrichment defines a large, locally small $(\infty,2)$-category $\bfCat_{\infty}$ whose mapping objects are $\Fun(\cC, \calD)$.
  Similarly, we can enrich $\CatHat_{\infty}$ in itself to obtain a very large, locally large $(\infty,2)$-category $\widehat{\bfCat}_{\infty}$.
\end{ex}
\begin{ex}
	Let $\PrLSt \subset \CatHat_{\infty}$ denote the subcategory spanned by stable presentable $\infty$-categories and colimit-preserving functors.
  We equip this with the Lurie tensor product, which gives $\PrLSt$ a closed monoidal structure.
  The mapping objects are then the $\infty$-categories $\Fun^{\rL}(\cC,\calD)$ of colimit-preserving functors.
  By \Cref{ex:enriched_from_closed}, we can therefore enrich $\PrLSt$ in itself, and by pushing this enrichment forward along the lax monoidal inclusion $\PrLSt \to \CatHat_{\infty}$, we obtain a very large, locally large $(\infty,2)$-category $\bfPr^{\rL}_{\St}$.
\end{ex}
By viewing $(\infty,2)$-categories as double $\infty$-categories, it makes sense to talk about lax functors between $(\infty,2)$-categories.
Dually, one could unstraighten a $\Cat_{\infty}$-category to a Cartesian fibration over $\Simp$, and define \emph{oplax functors} as functors over $\Simp$ which preserve Cartesian lifts of inerts.

(Op)Lax limits were first introduced in the $\infty$-categorical setting in~\cite{GHN}.
In the $(\infty,2)$-category $\bfCat_{\infty}$, we can define the lax limit of a functor $F\colon \cJ \to \bfCat_{\infty}$ to be the $\infty$-category of sections of the corresponding coCartesian fibration $\cE \to \cJ$, and the oplax limit of $F$ to be the $\infty$-category of sections of the corresponding Cartesian fibration $\cE'\to \cJ^{\op}$.
In a general $(\infty,2)$-category $\bfC$, we can define the (op)lax limit of a diagram $F\colon \cJ \to \bfC$ to be an object $X$ together with a universal \emph{(op)lax cone} on $F$ with vertex $X$, i.e. an object
\begin{align*}
	C\in (\op)\laxlim_{j\in \cJ} \bfC(X, F(j))
\end{align*}
such that for any other object $Y$, the functor
\begin{align*}
	\bfC(Y,X) \to (\op)\laxlim_{j\in \cJ} \bfC(Y,F(j))
\end{align*}
induced over each $j$ by composing with the morphism $C_{j}\in \bfC(X,F(j))$ is an equivalence. (Op)lax colimits may be defined dually by taking cocones.

Consider the locally large $(\infty,2)$-category $\bfPr^{\rL}_{\St}$.
By \cite[Corollary 4.15, Example 5.7]{CDW}, this admits (op)lax (co)limits indexed by small $\infty$-categories.
Moreover, $\PrLSt$ is \emph{(op)lax additive}, meaning that for any functor $F\colon \cJ \to \bfPr^{\rL}_{\St}$ the comparison functor
\begin{align*}
	(\op)\laxcolim_{\cJ} F \to (\op)\laxlim_{\cJ} F
\end{align*}
is an equivalence.
Furthermore, (op)lax limits in $\bfPr^{\rL}_{\St}$ are preserved by the inclusion to $\widehat{\bfCat}_{\infty}$, so these may be computed by sections of the corresponding Cartesian fibration.

Recall that the Picard space $\Pic(\bS)$ acts invertibly on $\Sp$ by the tensor product, and each action $X\otimes - \colon \Sp \to \Sp$ preserves colimits.
This action is classified by a functor $\tau \colon B\Pic(\bS) \to \PrLSt$, which is equivalently the inclusion of the component of $\Sp$ in $(\PrLSt)^{\simeq}$.

\begin{definition}\label{def:twisted_presheaves}
	For a small $\infty$-category $\cC$ and a functor $\mu \colon \cC \to B \Pic(\bS)$, we define the $\infty$-category $\TwShv^{\mu}$ of \emph{$\mu$-twisted presheaves on $\cC$}  as
	\[
	\TwShv^\mu = \oplaxcolim(\cC \xrightarrow{\mu} B\Pic(\bS) \xrightarrow{\tau} \bfPr^{\rL}_{\St}) \,.
	\]
\end{definition}

\begin{rem}\label{rem:lax_additivity}
	This definition is justified as $\bfPr^{\rL}_{\St}$ admits small oplax colimits.
  By lax additivity, we could equivalently define $\TwShv^{\mu}$ as the oplax limit of the same diagram.
  This oplax limit can be computed in $\widehat{\bfCat}_{\infty}$ as the sections of the corresponding Cartesian fibration $\cE \to \cC^{\op}$.
  Note that the fibers of $\cE$ are all equivalent to $\Sp$, and the Cartesian transports are given by tensoring with some invertible spectrum.
  This shows, in particular, that for a null-homotopic map $\cC \to B\Pic(\bS)$, the $\infty$-category of twisted presheaves is equivalent to $\Fun(\cC^{\op},\Sp)$.
\end{rem}

Twisted presheaves are a generalization of \emph{twisted spectra} as introduced in Douglas' PhD thesis~\cite{Dou}, and later put into the $\infty$-categorical framework by the first author together with Moulinos~\cite{HM}.
In particular, if we restrict to the case where the $\infty$-category $\cC$ is actually a space, the oplax limit reduces to a limit, and so we recover the definition
\begin{align*}
	\TwSp(\mu) = \colim(\cC\xrightarrow{\mu} B\Pic(\bS) \to \PrLSt)
\end{align*}
of \cite{HM}.

A functor $f\colon \cC \to \calD$ gives rise to a restriction functor $f^{*}\colon \Fun(\calD^{\op},\Sp) \to \Fun(\cC^{\op},\Sp)$ which has left and right adjoints $f_{!}$ and $f_{*}$ given by left and right Kan extension, respectively.
The $\infty$-categories $\TwShv$ enjoy similar functoriality in $\Cat_{\infty /B\Pic(\bS)}$, which can be constructed using oplax additivity.
\begin{prop}\label{prop:adjoint_under_lax_add}
	Suppose we are given a functor $f : \cC \to \mathcal{D}$ of $\infty$-categories and a functor $\mu : \mathcal{D} \to \PrLSt$.
	The following holds:
	\begin{enumerate}
		\item There is a functor
		\[
		f^* : \oplaxlim_{\calD} \mu \to \oplaxlim_{\cC} f^* \mu
		\]
		which is a left adjoint.
		\item There is a functor
		\[
		f_! : \oplaxcolim_{\cC} f^* \mu \to \oplaxcolim_{\calD} \mu
		\]
		which is a left adjoint. Moreover, under oplax additivity its right adjoint can be identified with the functor $f^*$ above.
	\end{enumerate}
\end{prop}
\begin{proof}
	The functors $f^{*}$ and $f_{!}$ exist by the functoriality of oplax (co)limits. We can obtain the adjunction $f_{!}\dashv f^{*}$ in a manner similar to the proof of~\cite[Proposition 1.5.6]{HM}.
\end{proof}
\begin{ex}
	Any object $c\in \cC$ gives rise to an adjunction
	\begin{align*}
		c_{!} : \Sp \rightleftarrows \TwShv^{\mu} : c^{*}.
	\end{align*}
	When we compute $\TwShv$ as an oplax limit in $\widehat{\bfCat}_{\infty}$, the right adjoint $c^{*}$ corresponds to evaluation of sections at $c$.
  We should therefore think of $c_{!}\bS$ as a twisted generalization of the Yoneda presheaf at $c$ in $\Fun(\cC^{\op},\Sp)$.
\end{ex}
\begin{definition}\label{def:PrLSt_omega}
	Let $\bfPr^{\rL}_{\St,\omega} \subset \bfPr^{\rL}_{\St}$ denote the sub-2-category spanned by compactly generated stable $\infty$-categories and left adjoints which preserve compact objects (or equivalently left adjoints whose right adjoint admits a further right adjoint).
\end{definition}
\begin{lem}\label{lem:oplax_colim_in_PrLSt_omega}
	The $(\infty,2)$-category $\bfPr^{\rL}_{\St,\omega}$ admits small oplax colimits, and the inclusion $\bfPr^{\rL}_{\St,\omega} \to \bfPr^{\rL}_{\St}$ preserves these.
\end{lem}
\begin{proof}
	We will show that the oplax colimit of any diagram
	\begin{align*}
		F\colon \cJ \to \bfPr^{\rL}_{\St,\omega} \to \bfPr^{\rL}_{\omega}
	\end{align*}
	is also an oplax colimit in $\bfPr^{\rL}_{\St,\omega}$.
  First note that by \Cref{prop:adjoint_under_lax_add}, the induced map $j_{!} \colon F(j) \to \oplaxcolim(F)$ has a right adjoint which itself is a left adjoint.
  This means that $j_{!}$ preserves compact objects.
  By computing $\oplaxlim(F)$ in $\widehat{\bfCat}_{\infty}$, we can identify the right adjoint $j^{*}$ of $j_{!}$ with evaluation of sections at $j\in \cJ^{\op}$.
  Since sections can be distinguished by their values on each equivalence class of objects in $\cJ$, the set
	\begin{align}\label{eq:oplax_colimi_in_PrLSt_omega}
		\bigcup_{j\in \pi_{0}\cJ^{\simeq}} \left\{j_{!} C^{j}_{\alpha}  \right\}_{\alpha \in A_{j}}
	\end{align}
	is a small set of compact generators for $\oplaxlim(F)$.
  This shows that the oplax colimit cone on $F$ factors through $\bfPr^{\rL}_{\St,\omega}$.
  For any other oplax cone $G$ on $F$ which factors through $\bfPr^{\rL}_{\St,\omega}$, the induced left adjoint $L\colon \oplaxcolim F \to G(\infty)$ takes the set \eqref{eq:oplax_colimi_in_PrLSt_omega} of compact generators to compact objects, and since any compact object is a retract of a finite colimit of such, $L$ preserves all compact objects.
\end{proof}

\subsection{Twisted presheaves are enriched presheaves}\label{subsec:twisted_presheaves_as_enriched_presheaves}

In this subsection, we will show that $\TwShv^{\mu}$ is naturally equivalent to the $\infty$-category of \emph{enriched presheaves} on a certain $\Sp$-category.
A good theory of enriched presheaves requires a discussion of \emph{weighted colimits}, and it is not yet clear how to handle these in the framework for enriched categories presented in \cite{GH} and that we have been using so far.
However, there are other equivalent definitions of enriched categories where this can be handled \cite{heine2025equivalence,heine2024higher,reutter2025enriched}.
We will not explore these definitions here, but rather use the construction of enriched presheaves as a black box.
We note that (op)lax colimits are a certain instance of $\Cat_{\infty}$-weighted colimits.

\begin{theorem}\label{thm:enriched_presheaves}
	There exists a symmetric monoidal functor $\cP_{\Sp}\colon (\Cat^{\Sp}_{\infty})^{\otimes} \to (\Pr^{\rL}_{\St, \omega})^{\otimes}$ with right adjoint $\chi$ such that:
	\begin{enumerate}
		\item For $\cC\in \Cat^{\Sp}_{\infty}$, the unit $\yo: \cC \to \chi(\cP_{\Sp}(\cC))$ is fully faithful.
		\item For $\calD\in \Pr^{\rL}_{\St, \omega}$, the underlying $\infty$-category $\Omega^{\infty}\chi(\calD)$ naturally equivalent to $\calD^{\omega}$.
		\item The functor $\Omega^{\infty}\yo \colon \Omega^{\infty}\cC \to (\cP_{\Sp}(\cC))$ takes values in compact objects, and generates $\cP_{\Sp}(\cC)$ under shifts and colimits.
		\item The morphism objects in $\chi(\calD)$ are mapping spectra in $\calD$.
	\end{enumerate}
\end{theorem}
\begin{proof}
	The existence and monoidality of the functor $\cP_{\Sp}\colon \Cat_{\infty}^{\Sp} \to \PrLSt$ follows from \cite*[Prop 1.7]{heine2024higher}.
  The factorization through $\Pr^{\rL}_{\St,\omega}$, and the description of the right adjoint follows from \cite*[Theorem B]{ben-moshe} by observing that internal left adjoints in $\PrLSt$ are precisely those functors that preserve compact objects, and atomic objects are precisely compact objects.
\end{proof}
We shall also use the following refinement of \Cref{ex:enriched_from_closed}.
\begin{theorem}\label{thm:2-cat_from_module}
	The forgetful functor $\PrL_{\Cat_{\infty}}\to \CatHat_{\infty}$ factors through a functor
	\begin{align*}
		\chi\colon \Pr^{\rL}_{\Cat_{\infty}} \to \CatHat_{(\infty,2)}
	\end{align*}
	which is an equivalence onto the subcategory of locally small $(\infty,2)$-categories whose underlying $\infty$-category is presentable and which admit small $\Cat_{\infty}$-weighted colimits, and 2-functors which preserve small $\Cat_{\infty}$-weighted colimits.
\end{theorem}

\begin{proof}
	This is \cite[Theorem 1.2]{heine2025equivalence} applied to $\cV = \Cat_{\infty}$ together with the enriched adjoint functor theorem \cite[Theorem 3.73]{heine2024higher}.
\end{proof}

\begin{cor}\label{cor:presheaves_on_SigmaC}
	Let $\cC$ be an $\infty$-category.
  We have a natural equivalence
	\begin{align*}
		\cP_{\Sp}(\Sigma^{\infty}_{+}\cC) \simeq \Fun(\cC^{\op},\Sp) \,.
	\end{align*}
\end{cor}
\begin{proof}
	By \Cref{thm:enriched_presheaves}, the composite
	\begin{align*}
		\Pr^{\rL}_{\St,\omega} \xrightarrow{\chi} \Cat_{\infty}^{\Sp} \xrightarrow{\Omega^{\infty}_{*}} \Cat_{\infty}
	\end{align*}
	agrees with $(-)^{\omega}$, so the result follows from uniqueness of adjoints.
\end{proof}

\begin{rem}\label{rem:PrLSt_omega_from_module}
	By \Cref{lem:oplax_colim_in_PrLSt_omega}, $\bfPr^{\rL}_{\St,\omega}$ admits all (op)lax colimits.
  In particular, it admits \emph{tensors} over $\Cat_{\infty}$.
  The tensor $\cC\otimes \calD$ is just the oplax colimit of the constant diagram $\cC \to \bfPr^{\rL}_{\St,\omega}$ with value $\calD$, which we know is $\Fun(\cC^{\op},\calD)$.
  This tensoring is compatible with colimits in $\Pr^{\rL}_{\St,\omega}$ because these can be computed as limits in $\widehat{\Cat}_{\infty}$.
  By \cite[Proposition 3.64]{heine2024higher}, this is enough to guarantee that $\bfPr^{\rL}_{\St,\omega}$ admits all $\Cat_{\infty}$-weighted colimits and therefore is in the image of the functor $\chi$ in \Cref{thm:2-cat_from_module}.
  Moreover, we can recover the corresponding $\Cat_{\infty}$-module structure on $\Pr^{\rL}_{\St,\omega}$ as the one given by forming $\Cat_{\infty}$-tensors in $\bfPr^{\rL}_{\St}$.
  This module structure is equivalent to the one induced by the monoidal left adjoint
	\begin{align}\label{eq:Pr^L_omega_tensoring}
		\Fun( (-)^{\op},\Sp ) \colon \Cat_{\infty}^{\times} \to (\Pr^{\rL}_{\St,\omega})^{\otimes} \,.
	\end{align}
\end{rem}
\begin{rem}\label{rem:cat_/BPic-2-category}
	The monoidal left adjoint
	\begin{align*}
		e_{*}\colon \Cat_{\infty}^{\times} \to \Cat_{\infty/B\Pic(\bS)}^{\otimes}
	\end{align*}
	gives rise to a $\Cat_{\infty}$-module structure on $\Cat_{\infty /B(\Pic(\bS))}$ given by
	\begin{align*}
	(\cC \xrightarrow{\mu} B\Pic(\bS) ) \otimes (\calD \xrightarrow{\underline{e}} B\Pic(\bS) ) \simeq (\cC \times \calD  \to \cC \xrightarrow{\mu} B\Pic(\bS)  ) \,.
	\end{align*}
	Under straightening, this corresponds to the pointwise $\Cat_{\infty}$-module structure, which is equivalently the free $\Cat_{\infty}$-module
	\begin{align*}
		\Fun(B\Pic(\bS)^{\op}, \Cat_{\infty} ) \simeq B\Pic(\bS)\otimes \Cat_{\infty} \,.
	\end{align*}
	We write $\bfCat_{\infty/B\Pic(\bS)}$ for the corresponding $(\infty,2)$-category.
\end{rem}
\begin{cor}\label{cor:TwShv_as_kan_ext}
	The functor $\TwShv$ is the underlying functor of the unique oplax colimit preserving 2-functor
	\begin{align*}
		\bfCat_{\infty/B\Pic(\bS)} \to \bfPr^{\rL}_{\St,\omega}
	\end{align*}
	extending $\tau \colon B\Pic(\bS) \to \bfPr^{\rL}_{\St,\omega}$.
\end{cor}

\begin{proof}
	By \Cref{rem:cat_/BPic-2-category}, the functor $\tau$ extends to a unique morphism in $\Pr^{\rL}_{\Cat_{\infty}}$, which by \Cref{thm:2-cat_from_module} gives rise to an oplax colimit-preserving 2-functor
	\begin{align*}
		\tau_{!} \colon \bfCat_{\infty / B\Pic(\bS)} \to \bfPr^{\rL}_{\St, \omega}.
	\end{align*}
	Under straightening, oplax colimits in $\bfCat_{\infty/B\Pic(\bS)}$ are computed pointwise over $B\Pic(\bS)$.
  For an object $\mu \colon \cC \to B\Pic(\bS)$, each fiber $\cC_{x}$ is equivalent to the oplax colimit of the constant functor from $\cC_{x}$ to $\bfCat_{\infty}$ with value $[0]$, so $\mu$ can be written as the oplax colimit
	\begin{align*}
		\mu\simeq \oplaxcolim(\cC\xrightarrow{ \mu} B\Pic(\bS) \to \bfCat_{\infty /B\Pic(\bS)}) \,.
	\end{align*}
    By \Cref{lem:oplax_colim_in_PrLSt_omega} and oplax additivity, we have
	\begin{align*}
		\tau_{!}(\mu) \simeq \oplaxcolim(\cC \xrightarrow{\mu} B\Pic(\bS) \xrightarrow{\tau} \bfPr^{\rL}_{\St, \omega}  ) \simeq \TwShv^{\mu}.
	\end{align*}
\end{proof}
By monoidal Kan extension, the inclusion $\Pic{\bS} \to \Sp$ gives rise to a symmetric monoidal left adjoint
\begin{align*}
	\Th\colon \cS_{/\Pic(\bS)}^{\otimes} \to \Sp^{\otimes}.
\end{align*}
This takes an invertible local system $\xi\colon X\to \Pic(\bS)$ to its \emph{Thom spectrum}, i.e. the colimit
\begin{align*}
	\Th(\xi) = \colim(X\xrightarrow{\xi} \Pic(\bS) \to \Sp).
\end{align*}
\begin{rem}
	The definition of $\TwShv^{\mu}$ is structurally similar to that of $\Th(\xi)$.
  In fact, we  can think of $\TwShv^{\mu}$ as a categorified Thom spectrum.
  We have replaced a space $B$ with an $\infty$-category $\cC$, an $\bS$-line bundle $\xi$ with a $\Sp$-line bundle $\mu$, the stable $\infty$-category of spectra $\Sp$ with the 2-stable $(\infty,2)$-category $\Pr^{\rL}_{\St,\omega}$, and a colimit with an oplax colimit.
\end{rem}
\begin{ex}
	If $\xi$ is nullhomotopic, a choice of nullhomotopy gives an equivalence $\Th(\xi) \simeq \Sigma^{\infty}_{+} X$.
\end{ex}
Consider the composite
\begin{align*}
	\Cat_{\infty}^{\cS_{/\Pic(\bS)}} \xrightarrow{\Th_{*}} \Cat_{\infty}^{\Sp} \xrightarrow{\cP_{\Sp}} \PrLSt \,.
\end{align*}
We have symmetric monoidal structures on each of these functors, so we get an algebra structure on the image $\cP_{\Sp}(\Th_{*} \sfE \Pic(\bS) )$ of the idempotent algebra $\sfE \Pic(\bS)$ considered in \Cref{subsec:monoids_and_delooping}.

\begin{lem}\label{lem:presheaves_on_th(pic(s))}
	The unit map $\Sp \to \cP_{\Sp}(\Th_{*}\sfE \Pic(\bS))$ is an equivalence.
\end{lem}

\begin{proof}
	The restriction of $\Th$ to the full subcategory $\Pic(\bS) \subset \cS_{/\Pic(\bS)}$ is given by the inclusion $\Pic(\bS)\to \Sp$.
  The enriched category $\Th_{*}\sfE\Pic(\bS)$ is therefore equivalent to the full subcategory $\overline{\Pic}(\bS) \subset \chi(\Sp)$ spanned by the tensor invertible objects.
  The algebra structure on $\sfE\Pic(\bS)$ induces an algebra structure on $\overline{\Pic}(\bS)$.
  The unit
	\begin{align*}
		\Sigma^{\infty}_{+}\Simp^{\op} \to \overline{\Pic}(\bS)
	\end{align*}
	for this algebra structure is given by including the object $\Th(0)=\bS$.
  The $\infty$-category of enriched presheaves on $\overline{\Pic}(\bS)$ has a fully faithful enriched Yoneda embedding
	\begin{align*}
		\yo \colon \overline{\Pic}(\bS) \to \chi(\cP_{\Sp}( \overline{\Pic}(\bS)) \,,
	\end{align*}
	and the unit map
	\begin{align*}
		i_{!}\colon \Sp \to \cP_{\Sp}( \overline{\Pic}(\bS) )
	\end{align*}
	for the induced algebra structure is determined by mapping $\bS$ to $\yo(\bS)$.
  By \Cref{thm:enriched_presheaves}, $\cP_{\Sp}(\overline{\Pic}(\bS))$ is generated under colimits by the objects $\yo(\Sigma^{n}\bS)$ for $n\in \bZ$.
  Because $\yo$ is fully faithful, we have the following equivalences of mapping spectra in $\chi(\cP_{\Sp}( \overline{\Pic}(\bS))$:
	\begin{align*}
		\map(\Sigma^{n}\yo(\bS), \yo( \Sigma^{n} \bS)  ) \simeq \Sigma^{-n} \map(\yo(\bS), \yo(\Sigma^{n}\bS) ) \simeq\Sigma^{-n}\Sigma^{n}\bS \simeq \bS \\
		\map(\yo(\Sigma^{n}\bS), \Sigma^{n}\bS) \simeq \Sigma^{n} \map(\yo(\Sigma^{n}\bS), \yo(\bS)) \simeq \Sigma^{n}\Sigma^{-n}\bS \simeq \bS.
	\end{align*}
	With these identifications, one can check that the unit $S^{0}\to \Omega^{\infty}\bS$ gives rise to inverse equivalences $\Sigma^{n}\yo(\bS) \simeq \yo(\Sigma^{n}\bS)$ in $\cP_{\Sp}(\overline{\Pic}(\bS))$.
  This implies that $\cP_{\Sp}$ is generated by the single compact object $\yo(\bS)$, whose mapping spectrum is $\bS$.
  It therefore follows from the Schwede--Shipley Theorem \cite[Theorem 7.1.2.1]{HA} that $i_{!}$ is an equivalence.
\end{proof}
\begin{rem}
	The generation result in \Cref{lem:presheaves_on_th(pic(s))} is an instance of the enriched Yoneda embedding preserving \emph{weighted limits}.
  Cotensors, which in $\chi(\Sp)$ are given by mapping spectra, are an example of weighted limits.
  The subcategory $\overline{\Pic}(\bS)$ is closed under cotensoring with invertible spectra, which can equivalently be seen as tensoring with the inverse.
  By \cite[Corollary 3.67]{heine2024higher}, weighted limits, and in particular cotensors, are preserved by the enriched Yoneda embedding $\yo$.
  This gives the equivalence $\yo(\Sigma^{n}\bS) \simeq \Sigma^{n}\yo(\bS)$ used in the proof of \Cref{lem:presheaves_on_th(pic(s))}.
\end{rem}

\begin{cor}\label{cor:TwSh_is_enriched_presheaves}
	We have a natural equivalence $\TwShv(-) \simeq \cP_{\Sp}(\Th(-))$ of functors $\Cat_{\infty /B\Pic(\bS)} \to \PrLSt$.
\end{cor}
\begin{proof}
	By \Cref{lem:presheaves_on_th(pic(s))}, we have left adjoint functors
	\begin{align*}
		\Cat_{\infty} \xrightarrow{e_{*}}
		\Cat_{\infty/B\Pic(\bS)} \simeq
		\Mod_{\sfE \Pic(\bS)}(\Cat_{\infty}^{\cS_{/\Pic(\bS)}}) \xrightarrow{\cP_{\Sp} \circ \Th}
		\Mod_{\Sp}(\Pr^{\rL}_{\St,\omega}) \simeq
		\Pr^{\rL}_{\St,\omega} \,,,
	\end{align*}
	all of which lift to monoidal functors.
  By pulling back module structures, this gives rise to a map
	\begin{align*}
	    \cP_{\Sp}(\Th(-))\colon \Cat_{\infty/B\Pic(\bS)}\to	\Pr^{\rL}_{\St,\omega}
	\end{align*}
	in $\Pr^{\rL}_{\Cat_{\infty}}$, which by \Cref{thm:2-cat_from_module}, \Cref{rem:PrLSt_omega_from_module}, and \Cref{rem:cat_/BPic-2-category} gives rise to an oplax colimit-preserving 2-functor
	\begin{align*}
		\cP_{\Sp}(\Th(-))\colon \bfCat_{\infty /B\Pic(\bS)} \to \bfPr^{\rL}_{\St, \omega} \,.
	\end{align*}
	By \Cref{cor:TwShv_as_kan_ext}, it now suffices to show that the restriction of $\cP_{\Sp}(\Th(-))$ to $B\Pic(\bS)$ agrees with $\tau$.
  By monoidality of $\cP_{\Sp}(\Th(-))$, this restriction classifies the $\Pic(\bS)$-action on $\cP_{\Sp}(\overline{\Pic}(\bS)) \simeq \Sp$ determined by the Yoneda embedding and tensor products of spectra, which is precisely the definition of $\tau$.
\end{proof}

\begin{ex}\label{ex:modules_over_th}
	Let $X$ be a pointed connected space with a pointed map $\mu\colon X\to B\Pic(\bS)$.
  The basepoint $x$ gives an object of the $\cS_{/\Pic(\bS)}$-category corresponding to $\mu$, whose endomorphism object is
	\begin{align*}
		\mu(x,x) =  \Omega X \xrightarrow{\Omega\mu} \Pic(\bS) \,.
	\end{align*}
	When we push the enrichment forward along $\Th$, we get $\Th(\mu)(x,x) \simeq \Th(\Omega \mu)$.
  This object is now a compact generator of $\cP_{\Sp}(\Th(\mu)) \simeq \TwSp^{\mu}$, so because $\yo$ is fully faithful as a  spectral functor, the Schwede--Shipley theorem gives an equivalence
	\begin{align*}
		\TwSp^{\mu} \simeq \Mod_{\Th(\Omega \mu)} \,.
	\end{align*}
	This recovers \cite[Theorem D]{CCRY}.
  Given a map of pointed connected spaces $f\colon Y\to X$, we get an $\bE_{1}$-map of Thom spectra $\Th(\Omega f^{*}\mu) \to \Th(\Omega\mu)$.
  By \cite[Theorem 3.7.6]{HM}, we can identify the functors $f_{!},f^{*},f_{*}$ on twisted spectra with extension, restriction and coextension of scalars along $f$, respectively.
\end{ex}

\subsection{The $J$-homomorphism}\label{subsec:the_j-homomorphism}
The $J$-homomorphism has a long tradition in homotopy theory.
It manifests as a map $O\to G$ from the infinite orthogonal group to the group of automorphisms of the sphere spectrum.
In this section, we fix a model for a (non-connected) delooping of the $J$-homomorphism as a map $\bZ\times BO \to \Pic(\bS)$ defined by taking a virtual vector bundle to its Thom spectrum.

Let $\cR(\overline{\CW})$ denote the topologically enriched category of fibrant-cofibrant \emph{retractive spaces} in the sense of \cite{may2006parametrized}.
The objects are diagrams
\begin{align*}
	B \xrightarrow{s} E \xrightarrow{p} B
\end{align*}
such that $p\circ s =id_{B}$, and where $p$ is a (Serre-)fibration and $s$ a cofibration.
These are the fibrant-cofibrant objects of a topological model category representing the total space of the Cartesian fibration $\cR(\cS) \to \cS$ classified by $\Fun(-, \cS_{*})$.
The stabilization of $\cR(\cS)$ is then the Cartesian fibration $\Loc(\cS)\to \cS$ classified by $\Fun(-,\Sp)$, with stabilization functor $\cR(\cS) \to \Loc(\cS)$ given by pushforward along $\Sigma^{\infty}$ .

Fiberwise one-point compactification determines a continuous functor
\begin{align*}
	\Vect(\overline{\CW}) \to \cR(\overline{\CW}) \,.
\end{align*}
This presents a functor between $\infty$-categories, which we further compose with stabilization to obtain
\begin{align*}
	\cJ\colon \Vect(\cS) \to \cR(\cS) \to \Loc(\cS) \,.
\end{align*}
Note that for any vector bundle $E\to B$, the corresponding local system over $B$ is fiberwise a sphere, and therefore fiberwise a tensor invertible object of $\Sp$.
Moreover, morphisms in $\Vect(\cS)$ correspond to fiberwise equivalences, so $\cJ$ factors through the subcategory inclusion $\cS_{/\Pic(\bS)} \to \Loc(\cS)$ determined by pushforward along  $\Pic(\bS) \to \Sp$.
By \Cref{lem:Vect(S)_terminal} and the Yoneda lemma, the factorization
\begin{align*}
	\cJ\colon \Vect(\cS) \to \cS_{/\Pic(\bS)}
\end{align*}
is uniquely determined by a map of spaces $\coprod_{k} BO(k) \to \Pic(\bS)$.

We equip $\Vect(\overline{\CW})$ with the exterior direct sum monoidal structure $\boxplus$, retractive spaces with the pointwise monoidal structure
\begin{align*}
	(E\to X) \boxtimes (F\to Y) \simeq E\times F \to X\times Y \,,
\end{align*}
and $\cR(\cS)$ and $\Loc(\cS)$ with the exterior product monoidal structure $\boxtimes$.
The functor $\cJ$ then lifts to a symmetric monoidal functor, giving a map of $\bE_{\infty}$-monoids $\coprod_{k} BO(k) \to \Pic(\bS)$, whose target is grouplike.
We define the $J$-homomorphism to be the unique factorization through the group completion
\begin{align*}
	\coprod_{k} BO(k) \to \bZ \times BO \xrightarrow{J} \Pic(\bS) \,.
\end{align*}

Recall that $\cR(\CW) \to \CW$ is a Cartesian fibration with transport along $X\to Y$ given by forming pullbacks.
Cartesian transport along $X\to \ast$ admits a left adjoint given by quotienting out the section.
The fiber over the terminal object is the category $\CW_{\ast}$ of cofibrantly pointed CW-complexes.
The functor
\begin{align*}
	\pi_{!}\colon \cR(\CW) \to \CW_{*} \\
	E \rightleftarrows X \mapsto E/X
\end{align*}
lifts to a continuous symmetric monoidal functor, which presents the symmetric monoidal functor
\begin{align*}
	\cR(\cS)^{\boxtimes} \to \cS_{*}^{\otimes}
\end{align*}
given by forming the colimit in $\cS_{*}$.

\begin{lem}\label{lem:J_hom}
	There exists a commutative diagram
	\begin{equation*}
		\begin{tikzcd}
			\Vect(\CW)^{\boxplus} \arrow[r, "(-)_{+}"] \arrow[d]      &  \cR(\CW)^{\boxtimes}  \arrow[r] & \CW_{\ast}^{\otimes} \arrow[d, "\Sigma^{\infty}"] \\
			\cS_{/ \bZ\times BO}^{\otimes}  \arrow[r, "J_{*}"]               &   \cS_{/\Pic(\bS)}^{\otimes}  \arrow[r, "\Th"]            & \Sp^{\otimes}
		\end{tikzcd}
	\end{equation*}
  of symmetric monoidal $\infty$-categories.
\end{lem}

\begin{proof}
	Consider the diagram
	\begin{equation*}
		\begin{tikzcd}
			\Vect(\CW)^{\boxplus} \arrow[r] \arrow[d]      &   \cR(\CW)^{\boxtimes}  \arrow[d]  \arrow[r] & \CW_{\ast}^{\otimes} \arrow[d] \\
			\cS_{/\coprod_{k} BO(k)}^{\otimes}  \arrow[r]  \arrow[d]             &  \cR(\cS)^{\boxtimes} \arrow[r]  \arrow[d ,"\Sigma^{\infty}"]          &  \cS_{*}^{\otimes} \arrow[d,"\Sigma^{\infty}"] \\
			\cS_{/\bZ\times BO}^{\otimes} \arrow[r, "J_{*}"] & \Loc(\cS)^{\boxtimes} \arrow[r, "\pi_{!}"] & \Sp^{\otimes}
		\end{tikzcd}
	\end{equation*}
  where the two top rows are obtained by including the categories on the top row into their topologically enriched counterparts, and pushing forward the enrichment along $\CW \to \cS$.
	The bottom right square commutes because $\Sigma^{\infty}$ preserves colimits and therefore commutes with left Kan extension along $X\to \ast$.
  The bottom left square commutes by construction of $J$.
\end{proof}

\begin{rem}\label{rem:thom_spectra}
	For a vector bundle $E\to X$, we define the Thom spectrum $\Th(E)$ in one of two ways.
  We can first form the \emph{Thom space} by the quotient $D(E)/S(E)$ of the disc bundle by the sphere bundle, or equivalently by one-point compactifying the fibers, and then quotienting out the section at $\infty$.
  The Thom spectrum can then be defined as the infinite suspension of the Thom space:
	\begin{align*}
		\Th(E) = \Sigma^{\infty} D(E)/S(E) \,.
	\end{align*}
	Alternatively, we take $E$ to its stable classifying map $X\to \bZ\times BO$, push this forward along the $J$-homomorphism to obtain the \emph{Thom local system} $\thom(E)\colon X\to \Pic(\bS)$, and then define the Thom spectrum of $E$ to be the total homology
	\begin{align*}
		\Th(E) = \colim(X\xrightarrow{\thom(E)} \Pic(\bS)\to \Sp ) \,.
	\end{align*}
	\Cref{lem:J_hom} implies that these two perspectives are equivalent, even as symmetric monoidal functors $\Vect(\CW)^{\boxplus} \to \Sp^{\otimes}$.
\end{rem}

Since $J$ is a map of grouplike $\bE_{\infty}$-spaces, we may deloop to get a map $BJ\colon U/O \to B\Pic(\bS)$. Pushforward along $BJ$ determines a functor
\begin{align*}
	BJ_{!}\colon \Cat_{\infty/ (U/O) } \to \Cat_{\infty/B\Pic(\bS)} \,.
\end{align*}
By \Cref{cor:TwSh_is_enriched_presheaves}, we can compute $\TwShv^{\mu}$ either as
\begin{align*}
	\TwShv^{\mu} = \oplaxcolim(\cC \xrightarrow{\mu} U/O \xrightarrow{BJ} \Pic(\bS) \to \bfPr^{\rL}_{\St,\omega} )
\end{align*}
or as enriched presheaves on $\Th(\mu)$.
\begin{ex}
	Let $f\colon G \to O$ be a map of $\bE_{2}$-spaces, for example coming from a continuous map of abelian topological groups.
  The Thom spectrum $MG$ of the vector bundle classified by $Bf\colon BG\to BO$ gives rise to the multiplicative homology theory of $G$-oriented bordism.
  Delooping yet again, we get an object $B^{2}f$ of $\Cat_{\infty / (U/O)}$.
  As in \Cref{ex:modules_over_th}, we get an equivalence
	\begin{align*}
		\TwShv^{B^{2}f} \simeq \Mod_{MG} \,.
	\end{align*}
\end{ex}

\iffalse
Suppose now that we are given a closed $n$-manifold.
Via the Whitney embedding theorem we find an embedding $e : M \to \bR^{n+k}$ and we let $\nu_M$ denote the normal bundle of this embedding.
By the tubular neighborhood theorem we find a tubular neighborhood of $M$, i.e an embedding $D\nu_M \to \bR^{n+k}$.
We define the Thom--Pontryagin collapse map
\[
\eta : S^{n+k} \to \Th(\nu_M)
\]
by essentially the identify on the tubular neighborhood and sending everything else to the basepoint.
The homotopy class of $\eta$ turns out to be independent of the choice of embedding and tubular neighborhood, and hence depends solely on the manifold $M$ itself.

\ALICE{Can also phrase this more abstractly as the collapse $r: M \to \ast$ induces functors $r_!, r_* : \Sp^M \to \Sp$ and $r^* : \Sp \to \Sp^M$. The PT collapse is precisely the unit $\bS \to r_* r^* \bS \simeq r_! \nu_M = \Th(\nu_M)$ (equivalence is norm map).}

\begin{theorem}(Atiyah duality)
	If $M$ is a closed manifold, then the Spanier--Whitehead dual of its suspension spectrum is the Thom spectrum of its (stable) normal bundle-
\end{theorem}

\ALICE{different structure groups $BG \to BO$.}

Let us return to twisted spectra.
Note that given a twist $\mu : B \to B\Pic(\bS)$, we can look at the loop $\Omega_b B \to \Pic(\bS)$ at the point $b \in B$.
\fi

\section{Structured Flow Categories are Twisted Presheaves}\label{sec:structured_flow_categories_and_twisted_presheaves}

The goal of this section is to finish the proof of \Cref{thm:Flow=TwPsh}, giving a natural equivalence between the stable $\infty$-categories $\Flow^{\mu}$ and $\TwShv^{\mu}$.
We saw in the previous section that the $\infty$-categories $\TwShv^{\mu}$ arise as enriched presheaves on the small $\Sp$-categories $\Th(\mu)$.
In particular, there is a full subcategory inclusion
\begin{align*}
	\Omega^{\infty}_{*}\Th(\mu) \to \TwShv^{\mu},
\end{align*}
whose image gives compact generators of $\TwShv^{\mu}$.
We have similarly identified compact generators for $\Flow^{\mu}$, namely the one-object flow categories $\bOne_{b}$.
We will consider the full subcategory $\Bord^{\mu}$ spanned by these objects.
The notation reflects the fact that the mapping space $\Bord^{\mu}(\bOne_{c},\bOne_{d})$ is a classifying space for $\mu(c,d)$-framed bordism.
We will construct a sufficiently functorial version of the Pontrjagin--Thom construction to give an equivalence $\Bord^{\mu} \simeq \Omega^{\infty}\Th(\mu)$.
We will then argue that this equivalence actually extends to the desired equivalence $\Flow^{\mu}\simeq\TwShv^{\mu}$.

Let $E\to B$ be a vector bundle over a space, or more generally an object in $\cS_{/\bZ\times BO}$.
The classical Pontrjagin--Thom construction gives an equivalence
\begin{align*}
	\Omega^{E}_{*} \overset{\simeq}\to \pi_*(\Th(E))
\end{align*}
between the graded abelian group of cobordism classes of manifolds with $E$-structure and the homotopy groups of $\Th(E)$.
Let us sketch the construction of this map.
Start with a closed manifold $M$ with an $E$ structure $-TM \to E$.
Pushing this map forward along $\Th$ gives a map
\begin{align*}
	R(M) \colon \Th(-TM) \to \Th(E)
\end{align*}
called the \emph{record} map.
Now embed $M$ in a euclidean space $V$ of high dimension, pick a tubular neighborhood $N$, pass to one-point compactifications, and form the \emph{scanning map}
\begin{align*}
	V_{+} &\to N_{+} \\
	x&\mapsto \begin{cases}
		x, & x\in N \\
		\infty, & \text{else.}
	\end{cases}
\end{align*}
The desuspension $\Sigma^{-V}\Sigma^{\infty}N_{+}$ is equivalent to the Thom spectrum of the negative tangent bundle of $M$, giving the collapse map
\begin{align*}
	C(M)\colon \bS \to \Th(-TM) \,.
\end{align*}
Then the composite
\begin{align*}
	R(M) \circ C(M) \colon \bS \to \Th(E)
\end{align*}
represents the class of $[M]$ in $\pi_{0}(\Th(E))\simeq \Omega^{E}_{0}$.
The record map is easily captured by $\Mfd_{\diamond}^{\mu}$ as follows.
By naturality of slices, the symmetric monoidal functor $\Th$ induces a double functor
\begin{align*}
	\Th \colon \left(\cS^{\otimes}_{\bZ\times BO} \right)_{/\mu} \to \Sp_{/\Th(\mu)} \,.
\end{align*}
We can then let $R_{\mu}$ denote the composite along the top row of the following diagram.
\begin{equation*}
	\begin{tikzcd}
		\Mfd_{\diamond}^{\mu} \arrow[r] \arrow[d] &  \left(\cS_{/\bZ\times BO}^{\otimes} \right)_{/\mu} \arrow[r, "\Th"] \arrow[d] & \Sp_{/\Th(\mu)} \arrow[d] \\
		\Mfd_{\diamond}^{\otimes} \arrow[r, "U\ominus T"] & \cS_{/\bZ\times BO}^{\otimes} \arrow[r, "\Th"] & \Sp^{\otimes} \,.
	\end{tikzcd}
\end{equation*}
The result is a double functor
\begin{align*}
	R_{\mu} \colon \Mfd_{\diamond}^{\mu} \to \Sp_{/\Th(\mu)} \,,
\end{align*}
which lifts $\Th(U\ominus T)$.
We call this the \emph{double categorical record map}.
In particular, on an object $\bX\in \Mfd_{\diamond}^{\mu}(c,d)$ with $U_{\bX}\simeq 0$, and $\cP_{\bX} = \Delta^{0}$, this recovers the record map of $-TX \to \mu(c,d)$.
Constructing the collapse map is a bit more involved, as it involves embeddings which are not part of the data of $\Mfd_{\diamond}$.

\subsection{The $\mu$-structured bordism category}\label{subsec:the_structured_bordism_category}

We begin by giving a model for the full subcategory of $\Flow^{\mu}$ spanned by the objects $\bOne_{b}$ in terms of a semi-simplicial space.
Let $\iota\colon \Simp_{s} \to \cP$ denote the section determined by mapping $[n]$ to the identity $[n]\to [n]$.

\begin{definition}\label{def:PreBord}
	Let $\PreBord^{\mu}$ be the semi-simplicial space defined by unstraightening the right fibration
	\begin{align*}
		\iota^{*}\cA^{*}\Alg(\Mfd_{\diamond}^{\mu})\to \Simp_{s} \,.
	\end{align*}
\end{definition}

\begin{lem}\label{lem:Bord->Flow}
	The semi-simplicial space $\PreBord^{\mu}$ is a quasi-unital inner Kan space. The natural inclusion $\PreBord^{\mu}\to \PreFlow^{\mu}$ is quasi-unital and presents a fully faithful functor.
\end{lem}

\begin{proof}
	The construction of horn fillers and degeneracies for $\PreFlow^{\mu}$ restricts to horn fillers and degeneracies for $\PreBord^{\mu}$.
  Using the description of mapping spaces in \Cref{prop:semi-simplicial_results} one easily sees that the induced functor $\Bord^{\mu}\to \Flow^{\mu}$ is fully faithful.
\end{proof}
The natural transformation
\begin{align*}
	\iota^{*}\cA^{*}\Alg(\Mfd^{(-)}_{\diamond}) \to \cA^{*}\Alg(\Mfd^{(-)}_{\diamond})
\end{align*}
of functors $\Cat_{\infty / (U/O)} \to \Cart(\Simp_{s})$ gives rise to a natural transformation $\PreBord^{(-)} \to \PreFlow^{(-)}$, which, as in the proof of \Cref{cor:PreFlow_lands_in_IK_qu} takes values in quasi-unital maps, and therefore gives a natural transformation $\Bord^{(-)}\to \Flow^{(-)}$.

\subsection{Embedded manifolds with corners}\label{subsec:embedded_manifolds}
In this section, we construct a monoidal $\infty$-category $\Mfd_{e}^{\otimes}$, which is a version of $\Mfd_{\diamond}^{\otimes}$ where all manifolds with corners are embedded in some high dimensional vector space.
We will show that this yields an equivalent theory of one point flow categories to $\Mfd_{\diamond}$, and construct a monoidal collapse map $\Mfd_{e}^{\otimes} \to \Sp_{\bS/}^{\otimes}$ lifting $\Th(U\ominus T)$.

Let $I=[0,1]$ denote the unit interval.
For each $n$, the manifold $I^{n}$ with its canonical stratification is a stratified manifold with corners.
Moreover, we can canonically identify the tangent bundle of $I^{n}$ with the trivial $\bR^{n}$ bundle.
In particular, any inclusion $I^{k}\to I^{n}$ of a corner stratum determines a canonical morphism $(I^{k},\bR^{k}) \to (I^{n},\bR^{n})$ in $\Mfd_{\diamond}$.

\begin{definition}
	For each $n$, let $\sfI_{[n]} \colon [n] \to \Mfd_{\diamond}$ be the lax functor determined by
	\begin{align*}
		\sfI_{[n]}(i,j) = (I^{ \left\{i+1,\dots j-1 \right\}}, (\bR^{ \left\{ i+1,\dots j-1\right\}},0 ) \,,
	\end{align*}
	with composition $\sfI_{[n]}(i,j) \times \sfI_{[n]}(j,k)\to \sfI_{[n]}(i,k)$ determined by the inclusion
	\begin{align*}
		I^{ \left\{ i+1,\dots , j-1 \right\}  } \times I^{ \left\{ j+1, \dots , k-1 \right\}} \to I^{ \left\{ i+1,\dots, k-1\right\} }
	\end{align*}
	which sets the missing coordinate to $0$.
  For a map $\phi \colon [m] \to [n]$ in $\Simp_{s}$, let the natural transformation $\phi_{!}\colon \sfI_{[m]} \to \phi^{*}\sfI_{[n]}$ be determined on each $(i,j)$ by the map which sets the missing coordinates to $1$.
  Let $\cI\colon \Simp_{s}\to \Alg(\Mfd_{\diamond})$ denote the functor determined by this construction.
\end{definition}

\begin{lem}\label{lem:A=I}
	There is an equivalence of cartesian fibrations
	\begin{align*}
		\cI^{*}\Alg(\Mfd_{\diamond}^{\mu}) \simeq (\cA\circ \iota)^{*}\Alg(\Mfd_{\diamond}^{\mu}) \,.
	\end{align*}
\end{lem}

\begin{proof}
	We identify the codimension one strata of $\sfA_{[n]}(i,j)$ and $I^{ \left\{ i+1,\dots , j-1\right\}}$ as follows:
	\begin{enumerate}
		\item The codimension 1 arc with a single internal edge labeled by $k$ corresponds to the stratum $\left\{ x_k=0 \right\}$.
		\item The codimension 1 arc with a single vertex labeled by $\left\{i+1,\dots \hat{k},\dots  j-1\right\}$ corresponds to the stratum $\left\{ x_k = 1 \right\}$.
	\end{enumerate}
	We then see that since composition $\sfA_{[n]}(i,k)\times \sfA_{[n]}(k,j)\to \sfA_{[n]}(i,j)$ sends the minimal object to a codimension one object of type (1), our identification is compatible with composition in $\sfI_{[n]}$.
	For $i<k<j$, the pushforward functor $(d_{k})_{!}\colon \sfA_{[n-1]}(i,j) \to \sfA_{[n]}(i,j+1)$ is given by relabeling along $d_{k}$, and this sends the minimal object to a codimension one object of type (2), so this is again compatible with the pushforward on $\sfI_{[n]}$.
  For the virtual vector spaces, note that we have a canonical identification
	\begin{align*}
		U_{[n]}(i,j) = \left(\bR^{ \left\{ i+1, \dots , j\right\}}, \bR^{ \left\{ j\right\}} \right) \simeq \bR^{ \left\{ i+1,\dots j-1 \right\} }
	\end{align*}
	given by canceling the summand $\bR^{ \left\{ j\right\}}$.
  In particular, the stabilization functor of \Cref{cor:U_compensation} induces the desired equivalence.
\end{proof}

\begin{definition}\label{def:Mfd_e}
	Let $\PreMfd_{e}$ denote the category where:
	\begin{enumerate}
		\item An object is given by a map $p \colon M\to I^{k}$ of manifolds with corners, a vector bundle $q\colon N\to M$, a vector space $V$ and smooth map $e\colon N\to V\times I^{k}$.
    We require that $p$ is transverse to all the strata of $I^{k}$, and that the product $(e,pq)\colon N \to V\times I^{k}$ is a codimension-zero embedding.
		\item A morphism is given by a pullback square
		\begin{equation*}
			\begin{tikzcd}
				M_{0} \arrow[r] \arrow[d, "p_{0}"]      &   M_{1}  \arrow[d,"p_{1}"] \\
				I^{k_{0}}  \arrow[r, "f"]               &   I^{k_{1}}
			\end{tikzcd}
		\end{equation*}
		where the bottom map is the inclusion of a stratum, and a linear isomorphism $\phi \colon V_0 \to V_1$ such that
		\begin{equation*}
			\begin{tikzcd}
				N_{0} \arrow[r, "{(e_{0},p_{0}q_{0})}"] \arrow[d]      &   V_{0} \times I^{k_{0}}  \arrow[d] \\
				N_{1}  \arrow[r, "{(e_{1}, p_{1}q_{1})}"]               &   V_{1} \times I^{k_{1}}
			\end{tikzcd}
		\end{equation*}
		commutes.
	\end{enumerate}

\end{definition}
We give $\PreMfd_{e}$ the monoidal structure determined by exterior product of vector bundles, and Cartesian product of embeddings.
The formation of scanning maps gives rise to a monoidal functor
\begin{align*}
	\PreMfd_{e}^{\otimes} &\to (\CW_{\ast}^{\otimes})^{[1]} \\
	M\leftarrow N \rightarrow V\times I^{k} &\mapsto (V\times I_{k})_{+} \to N_{+} \,.
\end{align*}
We push this forward along the monoidal functor $\Sigma^{\infty} \colon \CW_{\ast} \to \Sp$, and note that since $I^{k}$ is contractible, the result factors through the monoidal subcategory $(\Pic(\bS)/\Sp )^{\otimes} \subset (\Sp^{\otimes})^{[1]}$ of arrows where the source is a tensor invertible object, and morphisms where the map on the source is an equivalence.

By pulling back the section $s\colon \Sp^{\otimes} \to (\Sp^{\otimes})^{[1]}$ selecting identity morphisms, we get a section $s\colon \Pic(\bS)^{\otimes} \to (\Pic(\bS)/\Sp)^{\otimes}$, which is equivalent to the inclusion of the full subcategory spanned by arrows which are equivalences in $\Sp$.
Using the symmetric monoidal structure on $\Sp$, we can then form the composite
\begin{align*}
	(\Pic(\bS)/\Sp)^{\otimes} \xrightarrow{(id, s\circ (-)^{-1}\circ ev_0 )} (\Pic(\bS)/\Sp)^{\otimes}\times_{\Simp^{\op}} (\Pic(\bS)/\Sp)^{\otimes} \xrightarrow{\otimes} \Sp_{\bS/}^{\otimes} \, .
\end{align*}
The full monoidal subcategory of $\PreMfd_{e}^{\otimes}$ spanned by objects where $M$ is a point, $k$ is zero, and $e\colon N \to V$ is a linear isomorphism is equivalent to $\Vect^{\oplus}$.
The scanning map on an object in this subcategory is clearly an equivalence, so we have a unique factorization
\begin{equation}\label{eq:vect_factorization}
	\begin{tikzcd}
		\Vect^{\oplus} \arrow[r] \arrow[d]      &   \Pic(\bS)^{\otimes}  \arrow[d, "s"]  \arrow[r] & \Simp^{\op}  \arrow[d, "\bS"] \\
		\PreMfd_{e}^{\otimes}  \arrow[r]               &   (\Pic(\bS)/\Sp )^{\otimes}  \arrow[r]            & \Sp_{\bS/} \,.
	\end{tikzcd}
\end{equation}

The monoidal $\Vect$-action on $\PreMfd_{e}$ corresponds to a functor $F \colon \sfB \Vect \to \Mon$ where $\sfB\Vect$ is the delooping $\infty$-category of $\Vect^{\oplus}$, i.e., the one with a single object, whose monoid of endomorphisms is $\Vect^{\oplus}$.
The factorization \eqref{eq:vect_factorization} means that we can make the functor $\PreMfd_{e}^{\otimes} \to \Sp_{\bS/}^{\otimes}$ invariant with respect to the $\Vect$-action, meaning that it descends to a functor
\begin{align*}
	C\colon \Mfd_{e}^{\otimes} \coloneq \colim F \to \Sp_{\bS/}^{\otimes} \,.
\end{align*}
We have a functor $\PreMfd_{e} \to \Mfd_{\diamond}$ defined on $(M\leftarrow N \to V\times I^{k})$ by equipping $M$ with the stratification induced by the map $M\to I^{k}$, and the virtual vector space $\bR^{k}$.
This functor is also invariant under tensoring with an object in $\Vect$, and hence defines a monoidal functor $\pi\colon \Mfd_{e}^{\otimes} \to \Mfd_{\diamond}^{\otimes}$.

\begin{rem}
	When working with the category $\Mfd_{e}$ it is crucial that we treat flow categories as defined over the poset $[n]$ rather than over all of $\Simp^{\op}_{[n]}$: the stratifying category $\emptyset$ is not in the image of $\pi$, so we could not get any flow simplices over $\Simp^{\op}_{[n]^{\simeq}}$.
  This could alternatively be solved by artificially adding an initial object to $\Mfd_{e}$.
\end{rem}

\begin{lem}\label{lem:Mfd_e->Mfd_is_right_fib}
	Pushforward by $\pi$ determines a right fibration
	\begin{align*}
		(\cA_{\leq}\circ \iota)^{*}\Alg(\Mfd_{e}) \to (\cA_{\leq}\circ \iota)^{*}\Alg(\Mfd_{\diamond})\,.
	\end{align*}
\end{lem}

\begin{proof}
	Let $\Mod_{\con}\subset \Mod_{\diamond}$ denote the full subcategory of models with a unique minimal object.
  Since $\cA_{\leq}\circ \iota$ factors through $\Alg(\Mod_{\con})$ and $\pi$ factors through the corresponding subcategory $\Mfd_{\con}\subset \Mfd_{\diamond}$, it suffices by \Cref{lem:double_right_fib_condition,prop:algebra_right_fib} to show that the underlying functor $\Mfd_{e}\to \Mfd_{\con}$ is a right fibration.
  Start by noting that $\PreMfd_{e} \to \Mfd_{\con}$ is a right fibration with Cartesian lifts given by restriction to a stratum.
  The colimit over $\sfB \Vect$ can be computed by a bar complex, which is a sifted diagram.
  Sifted colimits are preserved by the forgetful functor $\Mon \to \Cat_{\infty}$, so the underlying functor is a colimit of objects $\Vect^{n}\times \Mfd_{e} \to \Mfd_{\con}$.
  Since $\Vect$ is a space, each such is a right fibration.
  The forgetful functor $\RFib(\Mfd_{\con})\to \Cat_{\infty}$ preserves and detects sifted colimits.
\end{proof}

\begin{lem}\label{lem:Bord_e->Bord_is_trivial}
	The map $\PreBord_{e} \to \PreBord$ of semi-simplicial spaces obtained from unstraightening the right fibration in \Cref{lem:Mfd_e->Mfd_is_right_fib} is a trivial fibration.
\end{lem}

\begin{proof}
	We need to solve the lifting problem
	\begin{equation*}
		\begin{tikzcd}
			\del \Delta^{n}_{s} \arrow[r] \arrow[d]      &    \PreBord_{e} \arrow[d] \\
			\Delta^{n}_{s}  \arrow[r]     \arrow[ur,dashed]           &     \PreBord \,.
		\end{tikzcd}
	\end{equation*}
	For $n=0$, this is trivial.
  For $n>0$, it suffices as in the proof of horn filling to extend the lift over the minimal object of $\sfI_{[n]}(0,n)$, i.e., to solve the lifting problem
	\begin{equation*}
		\begin{tikzcd}
			\del \sfA_{[n]}(0,n)^{\op} \arrow[r] \arrow[d]      &   \Mfd_{e}  \arrow[d] \\
			\sfA_{[n]}(0,n)^{\op}  \arrow[r]   \arrow[ur,dashed]             &    \Mfd \,.
		\end{tikzcd}
	\end{equation*}
	As in the proof of \Cref{lem:Mfd_e->Mfd_is_right_fib}, the forgetful functor $\Mon \to \Cat_{\infty}$ preserves sifted colimits, so the colimit $\Mfd_{e}$ of $F$ can be computed in $\Cat_{\infty}$ as follows.
  First unstraighten $F$ to a coCartesian fibration $\widetilde{\Mfd}_{e} \to \sfB\Vect$.
  A morphism $N_{0} \to N_{1}$ in $\widetilde{\Mfd}_{e}$ covering the morphism $V$ in $\sfB\Vect$ is precisely a map $N_{0}\oplus V \to N_{1}$ in $\PreMfd_{e}$.
  Then we can obtain the colimit $\Mfd_{e}$ by localizing $\widetilde{\Mfd}_{e}$ at coCartesian edges.
  The coCartesian transport of any $V\in \Vect$ is given by acting with $V$, and since any $V$ is equivalent to some $\bR^{n}$, we have
	\begin{align*}
		\Mfd_{e} \simeq \colim\left(\widetilde{\Mfd}_{e} \xrightarrow{\oplus \bR} \widetilde{\Mfd}_{e} \xrightarrow{\oplus \bR} \dots \right).
	\end{align*}
	Since $\del \sfA_{[n]}(0,n)$ is a finite category, any map to a filtered colimit factors through some finite step.
  The inclusion $\Vect \to \PreMfd_{e}$ is split by the functor which remembers only the vector space $V$.
  We let $\sfE\Vect \to \sfB\Vect$ denote the unstraightening of the action of $\Vect$ on itself.
  We then have an iterated pullback diagram
	\begin{equation*}
		\begin{tikzcd}
			\PreMfd_{e} \arrow[r] \arrow[d]      &   \Vect  \arrow[d]  \arrow[r] &    \ast \arrow[d] \\
			\widetilde{\Mfd}_{e}  \arrow[r]               &   \sfE\Vect  \arrow[r]            &  \sfB \Vect \, ,
		\end{tikzcd}
	\end{equation*}
	so by pasting, the left-hand square is a pullback.
  The functor $\widetilde{\Mfd}_{e}\to \sfE \Vect$ is a morphism of coCartesian fibrations over $\sfB \Vect$, which in the fiber over a point in $\sfB\Vect$ becomes $\PreMfd_{e}\to \Vect$ whose target is a space and therefore a coCartesian fibration.
  Moreover, coCartesian transport along a morphism $V$ in $\sfB\Vect$ induces the square
	\begin{equation*}
		\begin{tikzcd}
			\PreMfd_{e} \arrow[r, "-\oplus V"] \arrow[d]      &   \PreMfd_{e}  \arrow[d] \\
			\Vect  \arrow[r, "-\oplus V"]               &   \Vect
		\end{tikzcd}
	\end{equation*}
	which is a morphism of coCartesian fibrations.
  It then follows from \Cref{thm:iterated_fibration} that $\widetilde{\Mfd}_{e}\to \sfE\Vect$ is a coCartesian fibration.
  Any functor $F\colon \del \sfA_{[n]}(0,n)^{\op} \to \sfE \Vect$ extends to the cone $\sfA_{[n]}(0,n)$ by compatibly embedding all the $F(\sigma)$ in some large inner product space $V$, and using direct sum with the orthogonal complement to construct morphisms.
  Note that we are free to change this lift by including $V$ in some larger $V'$.
  After picking such a lift, it suffices to solve the lifting problem
	\begin{equation*}
		\begin{tikzcd}
			\del \sfA_{[n]}(0,n)^{\op} \arrow[r] \arrow[d]      &   \widetilde{\Mfd}_{e} \arrow[d] \\
			\sfA_{[n]}(0,n)^{\op}  \arrow[r]  \arrow[ur,dashed]              &   \sfE \Vect \times \Mfd_{\diamond} \,.
		\end{tikzcd}
	\end{equation*}
	Since the projection to $\Mfd_{\diamond}$ is invariant under coCartesian transport along morphisms in $\sfE \Vect$, we can use coCartesian transport to the fiber over the terminal object to reduce to the lifting problem
	\begin{equation*}
		\begin{tikzcd}
			\del \sfA_{[n]}^{\op} \arrow[r] \arrow[d]      &   \PreMfd_{e,V}  \arrow[d] \\
			\sfA_{[n]}^{\op}  \arrow[r] \arrow[ur,dashed]              &   \Mfd_{\diamond}
		\end{tikzcd}
	\end{equation*}
	where $\PreMfd_{e,V}$ denotes the fiber of $\PreMfd_{e}\to \Vect$ over $V$.
  This data corresponds to a manifold $M$ stratified by $\sfA_{[n]}(0,n)$, a compatible system of vector bundles $N_{\sigma} \to \del^{\sigma}M$ and maps $N_{\sigma} \to V$, and a compatible system of maps $\del^{\sigma}M\to \del^{\sigma}I^{n}$ where we equip $I^{n}$ with the $\sfA_{[n]}(0,n)$-stratification exhibited by \Cref{lem:A=I}.

	The maps to $\del^{\sigma}I^{n}$ induce a coherent normal framing on the boundary of $M$, which can be extended to all of $M$.
  Using the collar coordinates defined by this coherent normal framing, we can extend the map $\del M \to \del I^{n}$ to a smooth map $M\to I^{n}$ transversely inducing the stratification of $M$.
  After stabilizing to a sufficiently large $V$, we can further extend this to an embedding $M\to I^{n}\times V$.
  The tubular neighborhood $N_{\sigma}$ can similarly be extended to all of $M$ by using collar coordinates near the boundary.
\end{proof}

\begin{lem}\label{lem:collapse_commutes}
	There exists a commutative square in $\Mon$ of the following form:
	\begin{equation*}
		\begin{tikzcd}[column sep= huge]
			\Mfd_{e}^{\otimes} \arrow[r, "C" ] \arrow[d, "F"] & \Sp_{\bS/} \arrow[d] \\
			\Mfd_{\diamond}^{\otimes} \arrow[r, "\Th(U\ominus T)"] & \Sp^{\otimes} \,.
		\end{tikzcd}
	\end{equation*}
\end{lem}
\begin{proof}
	The composite
	\begin{align*}
		\PreMfd_{e} \xrightarrow{C} \Sp_{\bS/} \to \Sp^{\otimes}
	\end{align*}
	can be computed by projecting away from arrow $\infty$-categories in the definition of $C$, giving
	\begin{align*}
		\PreMfd_{e}^{\otimes} \to \CW_{\ast}^{\otimes}\times_{\Simp^{\op}} \Vect^{\oplus} \xrightarrow{\Sigma^{\infty} \times J^{-1}} \Sp^{\otimes} \times_{\Simp^{\op}} \Pic(\bS) \xrightarrow{\otimes} \Sp^{\otimes} \,.
	\end{align*}
	By \Cref{lem:J_hom}, the spectrum $\Sigma^{-V}\Sigma^{\infty}N_{+}$ can also be computed as the Thom spectrum of the virtual vector bundle $N\ominus V$ on $M$, so it now suffices to show that $N\ominus V$ and $U\ominus T$ agree as monoidal functors
	\begin{align*}
		\PreMfd_{e}^{\otimes} \to \cS_{/\bZ\times BO}^{\otimes} \,.
	\end{align*}
	For an object $(M\leftarrow N \xrightarrow{e} V\times I^{k})$, the embedding $e$ gives a stable isomorphism of virtual vector bundles over $M$
	\begin{align*}
		V \ominus N \simeq \bR^{k} \ominus TM \,.
	\end{align*}
	These are clearly natural with respect to isomorphisms in $\Mfd_{e}$ and monoidal products.
  They are also natural with respect to restriction to a stratum as long as we use the coherent normal framing on $M$ induced by the map $M\to I^{k}$.
  Because the space of such normal framings is contractible, we can inductively construct the required natural transformation over the nerve of $\PreMfd_{e}^{\otimes}$.
\end{proof}

\subsection{The Pontrjagin--Thom equivalence}\label{subsec:pontrjagin-thom}
We now combine our constructions of collapse and record to a Pontrjagin--Thom equivalence.
Define $\Mfd_{e}^{\mu}$ by the following pullback of double $\infty$-categories:
\begin{equation*}
	\begin{tikzcd}
		\Mfd_{e}^{\mu} \arrow[r] \arrow[d] & \left(\cS_{\bZ\times BO}^{\otimes} \right)_{\mu/}  \arrow[d] \\
		\Mfd_{e} \arrow[r, "T\ominus U"]  & \cS_{/\bZ\times BO}^{\otimes} \, .
	\end{tikzcd}
\end{equation*}
Since the right vertical is an internal right fibration, so is the left vertical.
When we pull back along $\cA\circ \iota$ we therefore have right fibrations
\begin{align*}
	(\cA\circ\iota)^{*}\Alg(\Mfd_{e}^{\mu}) \to (\cA\circ \iota)^{*}\Alg(\Mfd_{\diamond}) \to \Simp_{s} \,,
\end{align*}
and we define $\PreBord_{e}^{\mu}$ to be the unstraightening of this composite.
Now we have a pullback of semi-simplicial spaces
\begin{equation*}
	\begin{tikzcd}
		\PreBord_{e}^{\mu} \arrow[r] \arrow[d] & \PreBord_{e} \arrow[d] \\
		\PreBord^{\mu} \arrow[r] & \PreBord
	\end{tikzcd}
\end{equation*}
where the right vertical is a trivial Kan fibration by \Cref{lem:Bord_e->Bord_is_trivial}.
Therefore, the left vertical is also a trivial Kan fibration, so $\PreBord_{e}^{\mu}$ is a quasi-unital inner Kan space, and applying $\cT$ gives an equivalence of $\infty$-categories
\begin{align*}
	\Bord_{e}^{\mu} \xrightarrow{\simeq} \Bord^{\mu} \,.
\end{align*}
\begin{theorem}
	There is an equivalence
	\begin{align*}
		\Bord^{\mu}_{e} \xrightarrow{\simeq} \Omega^{\infty}\Th(\mu)
	\end{align*}
	of $\infty$-categories, natural in $\Cat_{\infty /(U/O)}$.
\end{theorem}
\begin{proof}
	By \Cref{lem:collapse_commutes}, we have a commutative diagram
	\begin{equation*}
		\begin{tikzcd}
			\Mfd_{e}^{\mu} \arrow[r,"C_{\mu}"] \arrow[d, "R_{\mu}"] & \Sp_{\bS/} \arrow[d]\\
			\Sp_{/\Th(\mu)} \arrow[r] & \Sp
		\end{tikzcd}
	\end{equation*}
  of double $\infty$-categories.
	By \Cref{lem:double_slices} and \Cref{lem:double_slices}, we therefore get a double functor
	\begin{align*}
		\PT^{\mu}\colon \Mfd_{e}^{\mu} \to \Sp_{\bS//\Th(\mu)} \to \Omega^{\infty} \Th(\mu) \,.
	\end{align*}
	Since $\Omega^{\infty} \Th(\mu)$ is an $\infty$-category, a lax functor into it is just a regular functor.
  Pushing an $n$-simplex $\bX \colon [n] \to \Mfd_{e}^{\mu}$ forward along $\PT^{\mu}$ therefore determines a map
	\begin{align*}
		\PT^{\mu} \colon \PreBord^{\mu}_{e} \to N_{s}(\Omega^{\infty}\Th(\mu))
	\end{align*}
	to the semi-simplicial nerve.
  Quasi-units in $\Bord^{\mu}$ are given by the bimodules $\bOne_{b} \to \bOne_{b}$ corresponding to the morphism $id_{b}\colon \ast \to \mu(b,b)$, and when we apply the collapse and record construction to this, we obtain precisely the unit in $\Omega^{\infty}\Th(\mu)$. We can therefore apply $\cT$ to get a functor $\PT^{\mu}\colon \Bord^{\mu} \to \Omega^{\infty}\Th(\mu)$.
  This is essentially surjective because any object $b$ of $\mu$ is the image of the 0-simplex $\bOne_{b}$ in $\Bord^{\mu}$.

	By \Cref{prop:semi-simplicial_results}, the mapping spaces in $\Bord^{\mu}_{e}$ can be computed as the realization of right mapping spaces in $\PreBord_{e}^{\mu}$.
  We use the presentation of homotopy groups of the realization of a Kan semi-simplicial space given in \cite{lurie2011ltheory}.
  An element in the homotopy group $\pi_{n} \Bord^{\mu}(\bOne_{b}, \bOne_{c})$ can then be represented by an $(n+1)$-simplex whose $d_{0}$-face is the degenerate $n$-simplex at $\bOne_{c}$, whose 0-th vertex is $\bOne_{b}$ and where all other faces are empty.
  This data corresponds precisely to a closed manifold $M$ structured by a map $\bR^{n}\ominus TM \to \mu(b,c)$.
  We quotient this out by the homotopy relation, where a homotopy between $M$ and $N$ is given by an $(n+2)$-simplex whose $d_0$ face is the degenerate $(n+1)$-simplex at $\bOne_{c}$, whose 0-th vertex is $\bOne_{b}$, whose $d_{1}$-face is $M$, whose $d_{2}$-face is $N$ and where all other faces are empty.
  This data corresponds precisely to a $\mu(b,c)$-structured bordism between $M$ and $N$.
  In other words, we have equivalences
	\begin{align*}
		\pi_{n} \Bord^{\mu}(\bOne_{b}, \bOne_{c}) \simeq \Omega_{n}^{\mu(b,c)} \,.
	\end{align*}
	By applying the map $\PT^{\mu}$ on such representatives we obtain the classical Pontrjagin--Thom collapse
	\begin{align*}
		\Omega_{n}^{\mu(b,c)} \xrightarrow{\simeq} \pi_{n} \Omega^{\infty} \Th(\mu)(b,c) \,.
	\end{align*}
	This shows that $\PT^{\mu}$ is also fully faithful, and therefore an equivalence.
\end{proof}

\subsection{Extending the equivalence}\label{subsec:rigidity}

In this subsection, we apply some general theory of $\infty$-categories to extend the equivalence $\Bord^{\mu} \simeq \Omega^{\infty}\Th(\mu)$ to an equivalence $\TwSp^{\mu} \simeq \Flow^{\mu}$.

\begin{lem}\label{lem:zero_is_unique}
	Let $\cC_0 \subset \cC$ and $\calD_{0} \subset \calD$ be full subcategories of pointed $\infty$-categories~$\cC$ and~$\calD$, such that the complement is spanned by the zero object.
	Then any equivalence $\cC_{0} \simeq \calD_{0}$ extends uniquely to an equivalence $\cC \simeq \calD$.
\end{lem}

\begin{proof}
	We will show that given an equivalence $F\colon \cC_{0} \to \calD_{0}$, the composite $\cC_{0} \to \calD$ left Kan extends along $\cC_{0} \to \cC$ to an equivalence.
	To show that such a Kan extension exists and provides the desired equivalence,
	it suffices to show that the zero object of~$\calD$ is a colimit of the composite
	\begin{align*}
		(\cC_{0})_{/\ast} \to \cC_{0} \xrightarrow{F} \calD_{0} \to \calD \,.
	\end{align*}
	Because $\ast$ is a zero object, the first morphism is an equivalence, and $F$ is an equivalence by assumption.
	If $\calD_{0} \simeq \emptyset$, the result is trivial, so we can exclude this case from now on.
	We will show that $\calD_{0}\to \calD$ is cofinal. For $D\in \calD_{0}$, the slice $(\calD_{0})_{D/}$ has an initial object and is therefore weakly contractible.
	For $D=\ast$, the slice projection is an equivalence, so we are reduced to showing that $\calD_{0}$ is weakly contractible.
	To that end, consider any small diagram
	\begin{align*}
		k\colon \cK\to \calD_{0} \,.
	\end{align*}
	We can extend to a cone $k'$ as follows
	\begin{equation*}
		\begin{tikzcd}
			\cK \arrow[r, "k"] \arrow[d] & (\calD_{0})_{/D} \arrow[d] \\
			\cK^{\triangleleft} \arrow[r, "k'"] & \calD_{/D}
		\end{tikzcd}
	\end{equation*}
	with $k'(-\infty) = \ast$.
	Now there exists an object $C\in \calD_{0}$, which has a zero morphism $C\to \ast$.
	Composing $k'$ along this morphism, we get a new cone $k''$ extending $k$ and with $k''(-\infty) ) =C$.
	This cone now lands in the full subcategory~$\calD_{0}$.
  This shows that~$\calD_{0}$ is weakly contractible, so by cofinality, we have
	\begin{align*}
		\colim(\calD_{0}\to \calD ) = \colim(\calD \xrightarrow{id} \calD) \,.
	\end{align*}
	By the dual of \cite[Proposition 2.1.1]{nguyen2020adjoint} the latter colimit is the terminal object.
\end{proof}

\begin{rem}\label{rem:zero_is_unique}
	The previous result says that if an $\infty$-category admits zero morphisms, these are uniquely determined by their behavior under composition.
	This is easy to show in classical algebra, where a zero object $z$ of a discrete monoid $M$ satisfying $z\cdot m = z = m\cdot z$ is unique, because if $z'$ is any other such object,
	\begin{align*}
		z' = z\cdot z' = z \,.
	\end{align*}
	A similar argument works for ordinary categories with zero morphisms, which can be interpreted as a result about categories enriched in the monoidal category of pointed sets with smash product.
	For comparison, \Cref{lem:zero_is_unique} can be seen as a statement about categories enriched in $\cS_{*}$.
	An $\cS_{*}$-category $\cC_{0}$ embeds fully faithfully into an $\infty$-category of enriched presheaves
	\begin{align*}
		\yo \colon \cC_{0} \hookrightarrow \cP_{\cS_{*}}(\cC_{0}) \,.
	\end{align*}
	The latter $\infty$-category is a module over $\cS_{*}$ in $\PrL$, which by \cite{HA} means that it has a zero object.
  Then we can consider the full subcategory $U\cC$ of the underlying $\infty$-category $U\cP_{\cS_{*}}(\cC_{0})$ spanned by the image of Yoneda and the zero object.

	Given two $\cS_{*}$-categories $\cC_{0}$ and $\calD_{0}$ and an equivalence of underlying categories $U\cC_{0} \simeq U\calD_{0}$, \Cref{lem:zero_is_unique} implies that we can extend to an equivalence $U\cC \simeq U\calD$, from which one can recover an equivalence $\cP_{\cS_{*}}(\cC_{0}) \simeq \cP_{\cS_{*}}(\calD_{0})$ and therefore an equivalence of enriched categories $\cC_{0} \simeq \calD_{0}$.
\end{rem}

\begin{definition}\label{def:Psh_sigma}
	Let $\cC_{0} \subset \cC$ be a small full subcategory of a stable $\infty$-category $\cC$, such that $\cC_{0}$ contains the zero object, and is closed under suspension. We write
	\begin{align*}
		\Psh^{\Sigma}(\cC_{0}) \subset \Psh(\cC_{0})
	\end{align*}
	for the full subcategory of presheaves $F$ such that
	\begin{enumerate}
		\item $F(0) \simeq \ast$.
		\item For any object $X\in \cC_{0}$, the presheaf $X$ maps any pushout square of the form
		\begin{equation}
			\begin{tikzcd}
				X \arrow[r] \arrow[d] & 0 \arrow[d] \\
				0 \arrow[r] & \Sigma X
			\end{tikzcd}
		\end{equation}
		to a pullback in $\cS$.
	\end{enumerate}
\end{definition}

\begin{lem}\label{lem:Psh_sigma_is_stable}
	If $\cC$ is a stable presentable $\infty$-category and $\cC_{0}$ a small full subcategory containing zero and closed under shifts, then $\Psh^{\Sigma}(\cC_{0})$ is also stable and presentable.
\end{lem}

\begin{proof}
	$\Psh^{\Sigma}(\cC_{0})$ is a left-localization of the presentable $\infty$-category $\Psh(\cC_{0})$ at the collection of morphisms:
	\begin{enumerate}
		\item $\emptyset \to y(0)$.
		\item $y(0) \underset{y(X)}{\coprod} y(0) \to y(\Sigma X)$ for all $X$.
	\end{enumerate}
	This takes care of presentability.
	Furthermore,  $\Psh^{\Sigma}(\cC_{0})$ is pointed at the constant functor with value $\ast$, and the loop functor
	\begin{align*}
		\Omega \colon \Psh^{\Sigma}(\cC_{0}) \to \Psh^{\Sigma}(\cC_{0})
	\end{align*}
	is equivalent to precomposition by the equivalence $\Sigma\colon \cC_{0} \to \cC_{0}$.
	Therefore $\Psh^{\Sigma}(\cC_{0})$ is stable by \cite{HA}.
\end{proof}

Let $\cC_{0}\subset \cC$ be as in \Cref{lem:Psh_sigma_is_stable}.
Any representable presheaf on $\cC_{0}$ takes colimits to limits, and so in particular satisfies the conditions of \Cref{def:Psh_sigma}.
The Yoneda embedding therefore factors through a fully faithful functor
\begin{align*}
	y'\colon \cC_{0} \to \Psh^{\Sigma}(\cC_{0}) \,.
\end{align*}
Because $\cC$ is cocomplete, the inclusion $i\colon \cC_{0} \to \cC$ extends to a unique colimit preserving functor $\Psh(\cC_{0}) \to \cC$.
The right adjoint of this is given by the Yoneda embedding for $\cC$, followed by restriction to $\cC_{0}$.
Any presheaf on $\cC_{0}$ of the form $\Map_{\cC}(i(-), X)$ for $X\in \cC$ satisfies the conditions of \Cref{def:Psh_sigma}, so we get a restricted adjunction
\begin{align}\label{eq:Psh_sigma_adjunction}
	L : \Psh^{\Sigma}(\cC_{0}) \rightleftarrows \cC: R \,.
\end{align}
\begin{prop}\label{prop:generation_rigidity}
	Let $\cC_{0} \subset \cC$ be a small full subcategory of a stable presentable $\infty$-category such that
	\begin{enumerate}
		\item $\cC_{0}$ contains zero and is closed under shifts.
		\item The objects of $\cC_{0}$ are compact in $\cC$.
		\item The objects of $\cC_{0}$ generate $\cC$.
	\end{enumerate}
	Then the adjoint functors \eqref{eq:Psh_sigma_adjunction} are inverse equivalences.
\end{prop}
\begin{proof}
	The right adjoint $R$ is conservative because $\cC_{0}$ generates $\cC$.
  Because the objects of $\cC_{0}$ are compact, the right adjoint $R$ preserves filtered colimits, and therefore all colimits by exactness.
  Because the restriction of $L$ to the full subcategory $\cC_{0}$ is fully faithful, we have equivalences
	\begin{equation*}
		\begin{tikzcd}
			\Map_{\cC}(-, C) \arrow[r, "L"] \arrow[dr, "(\eta_{C})_{*}", "\simeq"']      &   \Map_{\cC}(L(-), L(C))  \arrow[d, "\simeq"] \\
			&   \Map_{\cC}(-, RL(C))
		\end{tikzcd}
	\end{equation*}
	of presheaves on $\cC_{0}$.
  Since $\cC_{0}$ generates $\cC$ under colimits, this implies that the unit $\eta_{C}$ is an equivalence at all $C\in \cC_{0}$.
  Since the right adjoint $R$ also preserves colimits and using generation again, this implies that the unit is an equivalence everywhere, so $L$ is fully faithful.
\end{proof}

\begin{cor}\label{cor:Flow=TwShv}
	The composite
	\begin{align*}
		\Bord^{\mu} \xrightarrow{\simeq} \Omega^{\infty}\Th(\mu) \hookrightarrow \TwShv^{\mu}
	\end{align*}
	left Kan extends to an equivalence $\Flow^{\mu} \to \TwShv^{\mu}$.
\end{cor}
\begin{proof}
	By \Cref{lem:zero_is_unique}, we can left Kan extend $\PT^{\mu}$ to an equivalence $\Bord^{\mu}_{+} \to \Omega^{\infty}\Th(\mu)_{+}$.
	The subcategories
	$\Bord^{\mu}_{+} \subset \Flow^{\mu}$ and $\Omega^{\infty}\Th(\mu)_{+}\subset \TwShv^{\mu}$ satisfy the conditions of \Cref{prop:generation_rigidity}, so we have equivalences
	\begin{align*}
		\Flow^{\mu} \simeq  \Psh^{\Sigma}( \Bord^{\mu}_{+} ) \simeq \Psh^{\Sigma}(\Omega^{\infty}\Th(\mu)_{+})\simeq \TwSp^{\mu} \,.
	\end{align*}
\end{proof}

\section*{Appendix A. Structures for Floer flow categories} \label{sec:appendix}
	Let $(X^{n},\omega)$ be a Liouville domain and $L,K\subset X$ a pair of Lagrangians which are either compact or cylindrical at infinity.
  For appropriate perturbations, \cite{large2021spectral} constructs an unstructured flow category $\bF_{LK}$ whose objects are Hamiltonian chords, and whose morphism spaces are compactified spaces of Floer strips with boundary on $L$ and $K$.
  Following \cite{bonciocat2025floer,porcelli2024spectral}, we sketch how the flow category $\bF_{LK}$ can be equipped with a natural structure arising from the index theory of the Floer equation.

	\begin{definition}
			A \emph{bundle pair} $(E,F)$ on a manifold with boundary $X$ is a complex vector bundle $E\to X$ with a totally real subbundle $F\subset E\vert_{\del X}$. A \emph{Cauchy--Riemann (CR) datum} on the strip consists of:
		\begin{enumerate}
			\item A length $T >0 $.
			\item A bundle pair $(E,F)$ on $[0,1]\times \bR$ and bundle pairs $(E_{\pm \infty}, F_{\pm \infty})$ on $[0,1]$, with equivalences
			\begin{align*}
				(E,F)\vert_{(-\infty, -T)} &\simeq (E_{-\infty}, F_{-\infty}) \times (-\infty, T) \\
				(E,F)\vert_{(T,\infty)} &\simeq (E_{\infty},F_{\infty})\times (T,\infty) \,.
			\end{align*}
			\item A $1$-form $Y \in \Omega^{(0,1)}(E)$ supported on $[-T,T] \subset \bR$.
		\end{enumerate}
	\end{definition}
	To a CR datum on the strip, we can associate a linear CR operator
	\begin{align}\label{eq:cr-operator}
		\bar{\del}_{J} + Y \colon \Gamma(E,F) \to \Omega^{(0,1)}(E) \,.
	\end{align}
	We fix a constant $\kappa$, and consider the extension of $\bar{\del}_{J} + Y$ to $W^{2,\kappa}$-Sobolev regular sections.
  This extension is a Fredholm operator.
	\begin{definition}
		A \emph{perturbation datum} for a CR datum $(T,E,F,Y)$ on a strip is a vector space $V$ and a linear map $f\colon V \to \Omega_{W^{2,\kappa}}^{(0,1)}(E)$ such that $f(v)$ is supported on $[-T,T]$ and such that the Fredholm operator $\overline{\del}_{J}+Y+f$ is surjective.
	\end{definition}
	\begin{definition}
			A \emph{CR datum on a broken strip} is a finite sequence of data on strips $(T^{j},E^{j},F^{j}, Y^{j})_{j=0}^{\ell}$ such that $(E^{j}_{\infty},F^{j}_{\infty})=(E^{j+1}_{-\infty}, F^{j+1}_{-\infty})$.
      A perturbation datum for this consists of perturbation data for each strip.
	\end{definition}

	As in \cite{porcelli2024spectral}, one can construct topological spaces $\bU^{V}(E,E')$ of CR data on broken strips with $E^{0}_{-\infty}=E$ and $E^{\ell}_{\infty}=E'$ and perturbation domains adding up to $V$.
  These are realizations of simplicial sets where the $k$-simplices are continuous families of CR data over $\Delta^{k}$ with breaking potentially occurring near the boundary.
  Each $\bU^{V}(E,E')$ carries a vector bundle $\bV^{V}(E,E')$ given by the kernel of the associated (surjective) Fredholm operator $\overline{\del}_{J}+Y+f$.
  For any isometric embedding $V\to V'$ of inner product spaces, there are stabilization maps $\bU^{V}(E,E') \to \bU^{V'}(E,E')$ which adds $V^{\perp}$ to the perturbation domain and 0 to the perturbation function.

	There are associative concatenation maps $\bU^{V}(E,E')\times \bU^{V'}(E',E'') \to \bU^{V\oplus V'}(E,E'')$, which are compatible with stabilization in the obvious way.
  The vector bundles $\bV^{V}(E,E')$ are also compatible with concatenation, in the sense that we have maps
	\begin{align*}
		\bV^{V}(E,E')\boxplus \bV^{V'}(E',E'') \to \bV^{V\oplus V'}(E,E'')
	\end{align*}
	which are also associative and compatible with stabilization.

	To unpack which categorical structure this gives, let $\IP^{\oplus}$ denote the monoidal category of finite dimensional inner product spaces and isometric embeddings.
  We have a monoidal functor
	\begin{align*}
		\perp \colon \IP \to \sfB \Vect
	\end{align*}
	defined by mapping an isometric embedding $V\to W$ to its orthogonal complement $V^{\perp}$.
  For each $n$, let $X_{n}$ denote the set of rank $n$ bundle pairs on the interval.
  The coherence exhibited in  \cite[Lemma 3.18]{porcelli2024spectral} gives a diagram of lax functors
	\begin{equation*}
		\begin{tikzcd}
			\IP^{\oplus}\times_{\Simp^{\op}} \Simp^{\op}_{X_{n}} \arrow[r, "\bV_{n}"] \arrow[d]      &   \bVect(\CW)^{\boxplus}  \arrow[d] \\
			\IP^{\oplus}  \arrow[r,"\perp"]               &  \sfB \Vect^{\oplus} \,.
		\end{tikzcd}
	\end{equation*}
	The functor $\bV_{n}$ maps a tuple $(V_{01},\dots ,V_{n-1, n})$ of inner product spaces and bundle pairs $(E_{0},\dots, E_{n})$ to the object
	\begin{align*}
		(\bV^{V_{01}}(E_{0},E_{1})\boxplus \dots \boxplus \bV^{V_{n-1n}}(E_{n-1},E_{n}), V_{01}\oplus \dots V_{n-1n})
	\end{align*}
	of $\bVect(\CW)$.
  The morphism given by $(E_{0},\dots ,E_{n}) \to (E_{0},E_{n})$ and an isometric embedding $V_{01}\oplus \dots V_{n-1n} \to V_{0n}$ with orthogonal complement $W$ is mapped to the concatenation and stabilization map
	\begin{align*}
		\bV^{V_{01}}(E_{0},E_{1})\boxplus \dots \boxplus \bV^{V_{n-1n}}(E_{n-1},E_{n}) \oplus W \to \bV^{V_{0n}}(E_{0},E_{n})\,.
	\end{align*}
	Now when we push forward along $\bVect(\CW)^{\boxplus} \to \cS_{/ \bZ\times BO}^{\otimes}$, this corresponds under adjunction to a lax functor
	\begin{align*}
		\Simp^{\op}_{X_{n}} \to \Fun(\IP,\cS_{/ (\bZ\times BO)})^{\otimes}
	\end{align*}
	where the target has the Day convolution monoidal structure.
  By further composing with the colimit functor, we obtain a $\cS_{/ (\bZ\times BO)}$-category $\Ind(\bU_{n})$.
  By \cite[Lemma 3.25]{porcelli2024spectral}, we can identify the completion of the underlying $\cS$-category $\bU_{n}$ with the homotopy pullback $BO(n)\times_{BU(n)} BO(n)$.

	 We may stabilize CR-data by adding $\bC$ to $E$, $\bR$ to $F$, and $0$ to $Y$.
   Stabilization does not change the index, and by \cite[Proposition 3.21]{porcelli2024spectral}, stabilization determines a natural transformation $\bV_{n}\to S_{n}^{*}\bV_{n+1}$ where $S_{n}\colon X_{n}\to X_{n+1}$ is the stabilization map on the set of bundle pairs.
   We let $\bV$ denote the colimit of $\bV_{n}$ as $n\to \infty$, and the associated $\cS_{/ \bZ\times BO}$-category by $\Ind(\bU)$.
   This has an underlying $\cS$-category $\bU$ which completes to $BO\times_{BU}BO$.

	\begin{definition}
		An \emph{$L,K$-boundary datum} for a CR-datum $(T,E,F,Y)$ is a map $u\colon [0,1]\times \bR \to X$ constant outside $[-T,T]$ and with boundary on $L,K$, together with an equivalence $(E,F)\simeq u^{*}(TX, TL\amalg TK)$ wich is constant in the $\bR$-direction outside $[-T,T]$.
	\end{definition}

	As in \cite{porcelli2024spectral}, there are spaces $\bU^{V}_{LK}(x,y)$ of such data on broken strips, with compatible concatenation maps and stabilization maps.
  We have forgetful maps
	\begin{align*}
		\bU^{V}_{LK}(x,y) \to \bU^{V}_{n}(x^{*}TM, y^{*}TM)
	\end{align*}
	and we pull back the vector bundles $\bV^{V}(E,E')$ along these to define vector bundles $\bV^{V}_{LK}(x,y) \to \bU^{V}_{LK}(x,y)$.
  These are compatible with stabilization and concatenation, and as for $\bV_{n}$, there exists a diagram of lax functors
	\begin{equation*}
		\begin{tikzcd}
			\IP^{\oplus}\times_{\Simp^{\op}} \Simp^{\op}_{P_{LK}} \arrow[r, "\bV_{LK}"] \arrow[d]      &   \bVect(\CW)^{\boxplus}  \arrow[d] \\
			\IP^{\oplus}  \arrow[r,"\perp"]               &  \sfB \Vect^{\oplus}\,,
		\end{tikzcd}
	\end{equation*}
	where $P_{LK}$ is the set of paths from $L$ to $K$ in $M$.
  By adjunction, this determines a $\cS_{/ \bZ\times BO}$-category
	\begin{align*}
		\Ind(\bU_{LK})\colon \Simp^{\op}_{X_{LK}} \to \Fun(\IP, \cS_{/ \bZ\times BO})^{\otimes} \xrightarrow{\colim} \cS_{/ \bZ\times BO}^{\otimes}\,.
	\end{align*}
	Again the underlying $\cS$-category completes to the homotopy pullback $L\times_{M}K$.
  Forgetting boundary data determines natural transformations $\bV_{LK}\to \bV_{n} \to \bV$, which gives maps of $\cS_{/ \bZ\times BO}$-categories $\Ind(\bU_{LK}) \to \Ind(\bU_{n}) \to \Ind(\bU)$, which on completion corresponds to maps of spaces
	\begin{align*}
		\mu \colon L\times_{M}K \xrightarrow{TL\times_{TM}TK} BO(n)\times_{BU(n)} BO(n) \to BO\times_{BU} BO \xrightarrow{B\Ind} B(\bZ\times BO) \,.
	\end{align*}
	Arguing as in \cite{bonciocat2025floer}, the map $\Omega\Ind\colon \Omega (BO\times_{BU}BO) \to\bZ\times BO$ factors through the Lagrangian difference map as
	\begin{align*}
		\Omega(BO\times_{BU} BO ) \to \Omega(U/O) \xrightarrow{\Ind} \bZ \times BO \,,
	\end{align*}
	where the last map is an equivalence.
  Note that it is still unknown whether this equivalence agrees with other models for real Bott periodicity.

	By appropriately modifying \cite[Corollary 6.15]{porcelli2024spectral}, the Floer flow category $\bF_{LK}$ admits an $LK$-orientation in the apropriately modified sense of \cite[Definition 5.1]{porcelli2024spectral}.
  In our language, this unpacks to the following data:
	\begin{enumerate}
		\item  A map of sets $\ob(\bF_{LK}) \to P_{LK}$.
		\item  A lax functor $V\colon \Simp^{\op}_P \to \IP^{\oplus}$.
		\item  A natural transformation
		\begin{equation*}
			\begin{tikzcd}
				\Simp^{\op}_{\ob(\bF_{LK})} \arrow[r, "\bF_{LK}"] \arrow[d,"{(V,f)}"]      &   \Mfd_{\diamond}^{\otimes}  \arrow[d, "I"] \arrow[dl,Rightarrow,shorten=2em] \\
				\IP^{\oplus}\times_{\Simp^{\op}} \Simp^{\op}_{P_{LK}}  \arrow[r, "\bV_{LK}"]               &   \bVect(\CW)^{\boxplus}\,.
			\end{tikzcd}
		\end{equation*}

	\end{enumerate}
	In particular, the natural transformation gives maps
	\begin{align*}
		I\bF_{LK}(p,q) \oplus V(p,q) \to \bV^{V}_{LK}(f(p),f(q))
	\end{align*}
	which encode the framing.
  After pushing forward to $\cS_{/ \bZ\times BO}$, inclusion into the colimit then determines a $\cS_{/ \bZ\times BO}$-enriched functor $I\bF_{LK} \to \Ind(\bU_{LK})$.
  This map gives $\bF_{LK}$ a $-\mu$-structure.

\printbibliography

\end{document}